\documentclass[11pt,leqno]{article}

\newcommand{\pathtotrunk}{./}

\usepackage{amsmath,amssymb,amsfonts,amsthm}
\usepackage{ifpdf}

\usepackage[all,color]{xy}
\SelectTips{cm}{}

\usepackage[section]{placeins}
\usepackage{leftidx}
\usepackage{stmaryrd} 
\usepackage{microtype}

\newcommand{\pathtodiagrams}{\pathtotrunk diagrams/}

\newcommand{\mathfig}[2]{\ensuremath{\hspace{-3pt}\begin{array}{c}%
  \raisebox{-2.5pt}{\includegraphics[width=#1\textwidth]{\pathtodiagrams #2}}%
\end{array}\hspace{-3pt}}}

\newcommand{\arxiv}[1]{\href{http://arxiv.org/abs/#1}{\tt arXiv:\nolinkurl{#1}}}
\newcommand{\doi}[1]{\href{http://dx.doi.org/#1}{{\tt DOI:#1}}}

\newcommand{\googlebooks}[1]{(preview at \href{http://books.google.com/books?id=#1}{google books})}

\theoremstyle{plain}
\newtheorem{prop}{Proposition}[subsection]
\makeatletter
\@addtoreset{prop}{section}
\makeatother

\newtheorem{thm}[prop]{Theorem}
\newtheorem{lem}[prop]{Lemma}
\newtheorem*{lem*}{Lemma}
\newtheorem{cor}[prop]{Corollary}
\newtheorem*{cor*}{Corollary}

\newtheorem{defn}[prop]{Definition}         
\newtheorem*{defn*}{Definition}             

\newtheorem{property}[prop]{Property}
\newtheorem{axiom}[prop]{Axiom}
\newtheorem{module-axiom}[prop]{Module Axiom}
\newenvironment{rem}{\noindent\textsl{Remark.}}{}  
\newtheorem{rem*}[prop]{Remark}
\numberwithin{equation}{section}

\newcounter{comment}
\newcommand{\noop}[1]{}
\newcommand{\todo}[1]{\textbf{\color[rgb]{.8,.2,.5}\small TODO: #1}}

\def\clap#1{\hbox to 0pt{\hss#1\hss}}

\def\mathclap{\mathpalette\mathclapinternal}

\def\mathclapinternal#1#2{%
\clap{$\mathsurround=0pt#1{#2}$}}

\newcommand{\Real}{\mathbb R}

\newcommand{\id}{\boldsymbol{1}}
\renewcommand{\imath}{\mathfrak{i}}
\renewcommand{\jmath}{\mathfrak{j}}

\newcommand{\HC}{\operatorname{Hoch}}
\newcommand{\HH}{\operatorname{HH}}

\newcommand{\cell}{\mathfrak{D}}

\def\bc{{\mathcal B}}
\def\btc{{\mathcal{BT}}}

\newcommand{\into}{\hookrightarrow}
\newcommand{\onto}{\twoheadrightarrow}
\newcommand{\iso}{\cong}

\newcommand{\htpy}{\simeq}

\newcommand{\xto}[1]{\xrightarrow{#1}}
\newcommand{\isoto}{\xto{\iso}}
\newcommand{\quismto}{\xrightarrow[\text{q.i.}]{\iso}}

\newcommand{\htpyto}{\xrightarrow[\text{htpy}]{\htpy}}

\newcommand{\setc}[2]{\setcl{#1}{#2}}
\newcommand{\setcl}[2]{\left\{ \left. #1 \;\right| \; #2 \right\}}

\newcommand{\bdy}{\partial}
\newcommand{\compose}{\circ}

\newcommand{\DirectSum}{\bigoplus}
\newcommand{\tensor}{\otimes}
\newcommand{\Tensor}{\bigotimes}

\newcommand{\selfarrow}{\ensuremath{\smash{\tikz[baseline]{\clip (0,0.36) rectangle (0.48,-0.16); \draw[->] (0,0.2) .. controls (0.6,0.8) and (0.6,-0.6) .. (0,0);}}}}

\newcommand{\CM}[2]{C_*(\Maps(#1 \to #2))}
\newcommand{\CD}[1]{C_*(\Diff(#1))}
\newcommand{\CH}[1]{C_*(\Homeo(#1))}

\newcommand{\cl}[1]{\underrightarrow{#1}}

\newcommand{\Set}{\text{\textbf{Set}}}
\newcommand{\Vect}{\text{\textbf{Vect}}}
\newcommand{\Kom}{\text{\textbf{Kom}}}

\newcommand{\Bord}{\operatorname{Bord}}

\newcommand{\Hom}[3]{\operatorname{Hom}_{#1}\left(#2,#3\right)}

\newcommand{\Obj}{\operatorname{Obj}}

\ifpdf
	\usepackage[pdftex,plainpages=false,hypertexnames=false,pdfpagelabels]{hyperref}
	\usepackage[pdftex]{graphicx}
\else
	\usepackage[plainpages=false,hypertexnames=false,pdfpagelabels]{hyperref}
	\usepackage{graphicx}
\fi

\usepackage{tikz}
\usetikzlibrary{shapes}
\usetikzlibrary{backgrounds}
\usetikzlibrary{decorations,decorations.pathreplacing}
\usetikzlibrary{fit,calc,through}

\makeatletter
\@ifclassloaded{beamer}{}{%
	\newtheorem{example}[prop]{Example}	
  }%
\@ifclassloaded{pnastwo}
  {}
  {
	
	\newtheorem{remark}[prop]{Remark}
	\newtheorem{lemma}[prop]{Lemma}
  }%
\makeatother

\usepackage{color}

\usepackage{xcolor}
\definecolor{dark-red}{rgb}{0.6,0.15,0.15}
\definecolor{dark-blue}{rgb}{0.15,0.15,0.55}
\definecolor{medium-blue}{rgb}{0,0,0.65}
\hypersetup{
    colorlinks, linkcolor={dark-red},
    citecolor={dark-blue}, urlcolor={medium-blue}
}

\definecolor{kw-blue-a}{rgb}{0.1,0.4,0.8}

\title{Blob Homology}

\author{Scott~Morrison and Kevin~Walker}

\def\r{\mathbb{R}}
\def\c{\mathbb{C}}
\def\t{\mathbb{T}}
\def\ebb{\mathbb{E}}

\def\k{{\bf k}}

\def\du{\sqcup}
\def\bd{\partial}
\def\sub{\subset}
\def\subeq{\subseteq}
\def\sup{\supset}
\def\setmin{\setminus}
\def\ep{\epsilon}
\def\sgl{_\mathrm{gl}}

\def\deq{\stackrel{\mathrm{def}}{=}}
\def\pd#1#2{\frac{\partial #1}{\partial #2}}
\def\lf{\cC}
\def\ot{\otimes}
\def\vphi{\varphi}
\def\inv{^{-1}}
\def\ol{\overline}
\def\BD{BD}
\def\bbc{{\mathcal{BBC}}}
\def\mbc{{\mathcal{MBC}}}
\def\vcone{\text{V-Cone}}

\def\spl{_\pitchfork}
\def\trans#1{_{\pitchfork #1}}

\def\nn#1{{{\color[rgb]{.2,.5,.6} \small [[#1]]}}}
\long\def\noop#1{}

\newcommand{\eq}[1]{\begin{displaymath}#1\end{displaymath}}
\newcommand{\eqar}[1]{\begin{eqnarray*}#1\end{eqnarray*}}

\def\semicolon{;}
\def\applytolist#1{
    \expandafter\def\csname multi#1\endcsname##1{
        \def\multiack{##1}\ifx\multiack\semicolon
            \def\next{\relax}
        \else
            \csname #1\endcsname{##1}
            \def\next{\csname multi#1\endcsname}
        \fi
        \next}
    \csname multi#1\endcsname}

\def\calc#1{\expandafter\def\csname c#1\endcsname{{\mathcal #1}}}
\applytolist{calc}QWERTYUIOPLKJHGFDSAZXCVBNM;

\def\declaremathop#1{\expandafter\DeclareMathOperator\csname #1\endcsname{#1}}
\applytolist{declaremathop}{im}{gl}{ev}{coinv}{tr}{rot}{Eq}{obj}{mor}{ob}{Rep}{Tet}{cat}{Maps}{Diff}{Homeo}{sign}{supp}{Nbd}{res}{rad}{Compat}{Coll}{Cone}{pr}{paths};

\DeclareMathOperator*{\colim}{colim}
\DeclareMathOperator*{\hocolim}{hocolim}

\DeclareMathOperator{\kone}{cone}

\title{The Blob Complex}

\begin{document}

\makeatletter
\@addtoreset{equation}{section}
\gdef\theequation{\thesection.\arabic{equation}}
\makeatother

\maketitle

\begin{abstract}
Given an $n$-manifold $M$ and an $n$-category $\cC$, we define a chain complex
(the ``blob complex") $\bc_*(M; \cC)$.
The blob complex can be thought of as a derived category analogue of the Hilbert space of a TQFT, 
and also as a generalization of Hochschild homology to $n$-categories and $n$-manifolds.
It enjoys a number of nice formal properties, including a higher dimensional
generalization of Deligne's conjecture about the action of the little disks operad on Hochschild cochains.
Along the way, we give a definition of a weak $n$-category with strong duality which
is particularly well suited for work with TQFTs.
\end{abstract}

\hypersetup{
    colorlinks, linkcolor={black},
    citecolor={dark-blue}, urlcolor={medium-blue}
}

\tableofcontents

\hypersetup{
    colorlinks, linkcolor={dark-red},
    citecolor={dark-blue}, urlcolor={medium-blue}
}


\section{Introduction}

We construct a chain complex $\bc_*(M; \cC)$ --- the ``blob complex'' --- 
associated to an $n$-manifold $M$ and a linear $n$-category $\cC$ with strong duality.
This blob complex provides a simultaneous generalization of several well known constructions:
\begin{itemize}
\item The 0-th homology $H_0(\bc_*(M; \cC))$ is isomorphic to the usual 
topological quantum field theory invariant of $M$ associated to $\cC$.
(See Proposition \ref{thm:skein-modules} later in the introduction and \S \ref{sec:constructing-a-tqft}.)
\item When $n=1$ and $\cC$ is just a 1-category (e.g.\ an associative algebra), 
the blob complex $\bc_*(S^1; \cC)$ is quasi-isomorphic to the Hochschild complex $\HC_*(\cC)$.
(See Theorem \ref{thm:hochschild} and \S \ref{sec:hochschild}.)
\item When $\cC$ is $\pi^\infty_{\leq n}(T)$, the $A_\infty$ version of the fundamental $n$-groupoid of
the space $T$ (Example \ref{ex:chains-of-maps-to-a-space}), 
$\bc_*(M; \cC)$ is homotopy equivalent to $C_*(\Maps(M\to T))$,
the singular chains on the space of maps from $M$ to $T$.
(See Theorem \ref{thm:map-recon}.)
\end{itemize}

The blob complex definition is motivated by the desire for a derived analogue of the usual TQFT Hilbert space 
(replacing the quotient of fields by local relations with some sort of resolution), 
and for a generalization of Hochschild homology to higher $n$-categories.
One can think of it as the push-out of these two familiar constructions.
More detailed motivations are described in \S \ref{sec:motivations}.

The blob complex has good formal properties, summarized in \S \ref{sec:properties}.
These include an action of $\CH{M}$, 
extending the usual $\Homeo(M)$ action on the TQFT space $H_0$ (Theorem \ref{thm:evaluation}) and a gluing 
formula allowing calculations by cutting manifolds into smaller parts (Theorem \ref{thm:gluing}).

We expect applications of the blob complex to contact topology and Khovanov homology 
but do not address these in this paper.

Throughout, we have resisted the temptation to work in the greatest possible generality.
In most of the places where we say ``set" or ``vector space", any symmetric monoidal category 
with sufficient limits and colimits would do.
We could also replace many of our chain complexes with topological spaces (or indeed, work at the generality of model categories).

{\bf Note:} For simplicity, we will assume that all manifolds are unoriented and piecewise linear, unless stated otherwise.
In fact, all the results in this paper also hold for smooth manifolds, 
as well as manifolds equipped with an orientation, spin structure, or $\mathrm{Pin}_\pm$ structure.  
We will use ``homeomorphism" as a shorthand for ``piecewise linear homeomorphism".
The reader could also interpret ``homeomorphism" to mean an isomorphism in whatever category of manifolds we happen to 
be working in (e.g.\ spin piecewise linear, oriented smooth, etc.).
In the smooth case there are additional technical details concerning corners and gluing 
which we have omitted, since 
most of the examples we are interested in require only a piecewise linear structure.

\subsection{Structure of the paper}
The subsections of the introduction explain our motivations in defining the blob complex (see \S \ref{sec:motivations}), 
summarize the formal properties of the blob complex (see \S \ref{sec:properties}), 
describe known specializations (see \S \ref{sec:specializations}), 
and outline the major results of the paper (see \S \ref{sec:structure} and \S \ref{sec:applications}).

The first part of the paper (sections \S \ref{sec:fields}--\S \ref{sec:evaluation}) gives the definition of the blob complex, 
and establishes some of its properties.
There are many alternative definitions of $n$-categories, and part of the challenge of defining the blob complex is 
simply explaining what we mean by an ``$n$-category with strong duality'' as one of the inputs.
At first we entirely avoid this problem by introducing the notion of a ``system of fields", and define the blob complex 
associated to an $n$-manifold and an $n$-dimensional system of fields.
We sketch the construction of a system of fields from a *-$1$-category and from a pivotal $2$-category.

Nevertheless, when we attempt to establish all of the observed properties of the blob complex, 
we find this situation unsatisfactory.
Thus, in the second part of the paper (\S\S \ref{sec:ncats}-\ref{sec:ainfblob}) we give yet another 
definition of an $n$-category, or rather a definition of an $n$-category with strong duality.
(Removing the duality conditions from our definition would make it more complicated rather than less.) 
We call these ``disk-like $n$-categories'', to differentiate them from previous versions.
Moreover, we find that we need analogous $A_\infty$ $n$-categories, and we define these as well following very similar axioms.
(See \S \ref{n-cat-names} below for a discussion of $n$-category terminology.)

The basic idea is that each potential definition of an $n$-category makes a choice about the ``shape" of morphisms.
We try to be as lax as possible: a disk-like $n$-category associates a 
vector space to every $B$ homeomorphic to the $n$-ball.
These vector spaces glue together associatively, and we require that there is an action of the homeomorphism groupoid.
For an $A_\infty$ $n$-category, we associate a chain complex instead of a vector space to 
each such $B$ and ask that the action of 
homeomorphisms extends to a suitably defined action of the complex of singular chains of homeomorphisms.
The axioms for an $A_\infty$ $n$-category are designed to capture two main examples: 
the blob complexes of $n$-balls labelled by a 
disk-like $n$-category, and the complex $\CM{-}{T}$ of maps to a fixed target space $T$.

In \S \ref{ssec:spherecat} we explain how $n$-categories can be viewed as objects in an $n{+}1$-category 
of sphere modules.
When $n=1$ this just the familiar 2-category of 1-categories, bimodules and intertwiners.

In \S \ref{ss:ncat_fields}  we explain how to construct a system of fields from a disk-like $n$-category 
(using a colimit along certain decompositions of a manifold into balls). 
With this in hand, we write $\bc_*(M; \cC)$ to indicate the blob complex of a manifold $M$ 
with the system of fields constructed from the $n$-category $\cC$. 
In \S \ref{sec:ainfblob} we give an alternative definition 
of the blob complex for an $A_\infty$ $n$-category on an $n$-manifold (analogously, using a homotopy colimit).
Using these definitions, we show how to use the blob complex to ``resolve" any ordinary $n$-category as an 
$A_\infty$ $n$-category, and relate the first and second definitions of the blob complex.
We use the blob complex for $A_\infty$ $n$-categories to establish important properties of the blob complex (in both variants), 
in particular the ``gluing formula" of Theorem \ref{thm:gluing} below.

The relationship between all these ideas is sketched in Figure \ref{fig:outline}.

\tikzstyle{box} = [rectangle, rounded corners, draw,outer sep = 5pt, inner sep = 5pt, line width=0.5pt]

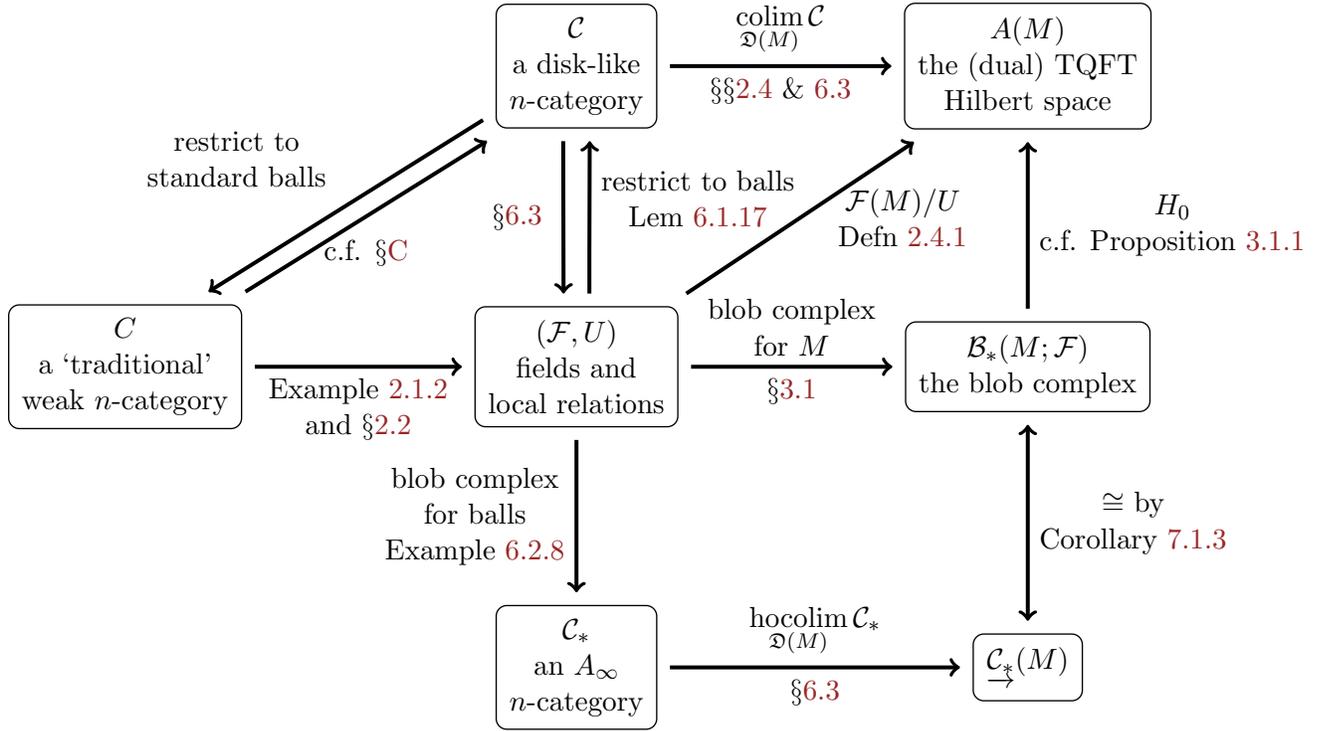
\begin{figure}[t]
{\center
\beginpgfgraphicnamed{gadgets-external}%
\begin{tikzpicture}[align=center,line width = 1.5pt]
\newcommand{\xxa}{2}
\newcommand{\xxb}{8}
\newcommand{\yya}{14}
\newcommand{\yyb}{10}
\newcommand{\yyc}{6}

\node[box] at (-4,\yyb) (tC) {$C$ \\ a `traditional' \\ weak $n$-category};
\node[box] at (\xxa,\yya) (C) {$\cC$ \\ a disk-like \\ $n$-category};
\node[box] at (\xxb,\yya) (A) {$A(M)$ \\ the (dual) TQFT \\ Hilbert space};
\node[box] at (\xxa,\yyb) (FU) {$(\cF, U)$ \\ fields and\\ local relations};
\node[box] at (\xxb,\yyb) (BC) {$\bc_*(M; \cF)$ \\ the blob complex};
\node[box] at (\xxa,\yyc) (Cs) {$\cC_*$ \\ an $A_\infty$ \\$n$-category};
\node[box] at (\xxb,\yyc) (BCs) {$\underrightarrow{\cC_*}(M)$};

\draw[->] (C) -- node[above] {$\displaystyle \colim_{\cell(M)} \cC$} node[below] {\S\S \ref{sec:constructing-a-tqft} \& \ref{ss:ncat_fields}} (A);

\draw[->] (FU) -- node[above] {blob complex \\ for $M$} node[below]{\S \ref{sec:blob-definition}} (BC);
\draw[->] (Cs) -- node[above] {$\displaystyle \hocolim_{\cell(M)} \cC_*$} node[below] {\S \ref{ss:ncat_fields}} (BCs);

\draw[->] (FU) -- node[right=10pt] {$\cF(M)/U$ \\ Defn \ref{defn:TQFT-invariant}} (A);

\draw[->] (tC) -- node[below] {Example \ref{ex:traditional-n-categories(fields)}\\ and \S \ref{sec:example:traditional-n-categories(fields)}} (FU);

\draw[->] (C.-100) -- node[left] {
	\S \ref{ss:ncat_fields}
   } (FU.100);
\draw[->] (C.210) -- node[above left=3pt] {restrict to \\ standard balls} (tC.42);
\draw[->] (tC) -- node[below=4.5pt] {c.f. \S \ref{sec:comparing-defs}} (C.220);
\draw[->] (FU.80) -- +(0,0.5) -- node[right] {restrict to balls \\ Lem \ref{lem:ncat-from-fields}} (C.-80);
\draw[->] (BC) -- node[right] {$H_0$ \\ c.f. Proposition \ref{thm:skein-modules}} (A);

\draw[->] (FU) -- node[left] {blob complex \\ for balls \\ Example \ref{ex:blob-complexes-of-balls}} (Cs);
\draw[<->] (BC) -- node[right] {$\iso$ by \\ Corollary \ref{cor:new-old}} (BCs);
\end{tikzpicture}
\endpgfgraphicnamed%
\mbox{} 
}
\caption{The main gadgets and constructions of the paper.}
\label{fig:outline}
\end{figure}

Later sections address other topics.
Section \S \ref{sec:deligne} gives
a higher dimensional generalization of the Deligne conjecture 
(that the little discs operad acts on Hochschild cochains) in terms of the blob complex.
The appendices prove technical results about $\CH{M}$ and
make connections between our definitions of $n$-categories and familiar definitions for $n=1$ and $n=2$, 
as well as relating the $n=1$ case of our $A_\infty$ $n$-categories with usual $A_\infty$ algebras. 

\subsection{Motivation}
\label{sec:motivations}

We will briefly sketch our original motivation for defining the blob complex.

As a starting point, consider TQFTs constructed via fields and local relations.
(See \S\ref{sec:tqftsviafields} or \cite{kw:tqft}.)
This gives a satisfactory treatment for semisimple TQFTs
(i.e.\ TQFTs for which the cylinder 1-category associated to an
$n{-}1$-manifold $Y$ is semisimple for all $Y$).

For non-semi-simple TQFTs, this approach is less satisfactory.
Our main motivating example (though we will not develop it in this paper)
is the $(4{+}\varepsilon)$-dimensional TQFT associated to Khovanov homology.
It associates a bigraded vector space $A_{Kh}(W^4, L)$ to a 4-manifold $W$ together
with a link $L \subset \bd W$.
The original Khovanov homology of a link in $S^3$ is recovered as $A_{Kh}(B^4, L)$.

How would we go about computing $A_{Kh}(W^4, L)$?
For the Khovanov homology of a link in $S^3$ the main tool is the exact triangle (long exact sequence)
relating resolutions of a crossing.
Unfortunately, the exactness breaks if we glue $B^4$ to itself and attempt
to compute $A_{Kh}(S^1\times B^3, L)$.
According to the gluing theorem for TQFTs, gluing along $B^3 \subset \bd B^4$
corresponds to taking a coend (self tensor product) over the cylinder category
associated to $B^3$ (with appropriate boundary conditions).
The coend is not an exact functor, so the exactness of the triangle breaks.

The obvious solution to this problem is to replace the coend with its derived counterpart, 
Hochschild homology.
This presumably works fine for $S^1\times B^3$ (the answer being the Hochschild homology
of an appropriate bimodule), but for more complicated 4-manifolds this leaves much to be desired.
If we build our manifold up via a handle decomposition, the computation
would be a sequence of derived coends.
A different handle decomposition of the same manifold would yield a different
sequence of derived coends.
To show that our definition in terms of derived coends is well-defined, we
would need to show that the above two sequences of derived coends yield 
isomorphic answers, and that the isomorphism does not depend on any
choices we made along the way.
This is probably not easy to do.

Instead, we would prefer a definition for a derived version of $A_{Kh}(W^4, L)$
which is manifestly invariant.
In other words, we want a definition that does not
involve choosing a decomposition of $W$.
After all, one of the virtues of our starting point --- TQFTs via field and local relations ---
is that it has just this sort of manifest invariance.

The solution is to replace $A_{Kh}(W^4, L)$, which is a quotient
\[
 \text{linear combinations of fields} \;\big/\; \text{local relations} ,
\]
with an appropriately free resolution (the blob complex)
\[
	\cdots\to \bc_2(W, L) \to \bc_1(W, L) \to \bc_0(W, L) .
\]
Here $\bc_0$ is linear combinations of fields on $W$,
$\bc_1$ is linear combinations of local relations on $W$,
$\bc_2$ is linear combinations of relations amongst relations on $W$,
and so on. 
We now have a long exact sequence of chain complexes relating resolutions of the link $L$ 
(c.f. Lemma \ref{lem:hochschild-exact} which shows exactness 
with respect to boundary conditions in the context of Hochschild homology).

\subsection{Formal properties}
\label{sec:properties}
The blob complex enjoys the following list of formal properties.

\begin{property}[Functoriality]
\label{property:functoriality}%
The blob complex is functorial with respect to homeomorphisms.
That is, 
for a fixed $n$-dimensional system of fields $\cF$, the association
\begin{equation*}
X \mapsto \bc_*(X; \cF)
\end{equation*}
is a functor from $n$-manifolds and homeomorphisms between them to chain 
complexes and isomorphisms between them.
\end{property}
As a consequence, there is an action of $\Homeo(X)$ on the chain complex $\bc_*(X; \cF)$; 
this action is extended to all of $C_*(\Homeo(X))$ in Theorem \ref{thm:evaluation} below.

The blob complex is also functorial with respect to $\cF$, 
although we will not address this in detail here.

\begin{property}[Disjoint union]
\label{property:disjoint-union}
The blob complex of a disjoint union is naturally isomorphic to the tensor product of the blob complexes.
\begin{equation*}
\bc_*(X_1 \du X_2) \iso \bc_*(X_1) \tensor \bc_*(X_2)
\end{equation*}
\end{property}

If an $n$-manifold $X$ contains $Y \sqcup Y^\text{op}$ as a codimension $0$ submanifold of its boundary, 
write $X_\text{gl} = X \bigcup_{Y}\selfarrow$ for the manifold obtained by gluing together $Y$ and $Y^\text{op}$.
Note that this includes the case of gluing two disjoint manifolds together.
\begin{property}[Gluing map]
\label{property:gluing-map}%
Given a gluing $X \to X_\mathrm{gl}$, there is an injective natural map
\[
	\bc_*(X) \to \bc_*(X_\mathrm{gl}) 
\]
(natural with respect to homeomorphisms, and also associative with respect to iterated gluings).
\end{property}

\begin{property}[Contractibility]
\label{property:contractibility}%
With field coefficients, the blob complex on an $n$-ball is contractible in the sense 
that it is homotopic to its $0$-th homology.
Moreover, the $0$-th homology of balls can be canonically identified with the vector spaces 
associated by the system of fields $\cF$ to balls.
\begin{equation*}
\xymatrix{\bc_*(B^n;\cF) \ar[r]^(0.4){\iso}_(0.4){\text{qi}} & H_0(\bc_*(B^n;\cF)) \ar[r]^(0.6)\iso & A_\cF(B^n)}
\end{equation*}
\end{property}

Property \ref{property:functoriality} will be immediate from the definition given in
\S \ref{sec:blob-definition}, and we'll recall it at the appropriate point there.
Properties \ref{property:disjoint-union}, \ref{property:gluing-map} and 
\ref{property:contractibility} are established in \S \ref{sec:basic-properties}.

\subsection{Specializations}
\label{sec:specializations}

The blob complex is a simultaneous generalization of the TQFT skein module construction and of Hochschild homology.

\newtheorem*{thm:skein-modules}{Proposition \ref{thm:skein-modules}}

\begin{thm:skein-modules}[Skein modules]
The $0$-th blob homology of $X$ is the usual 
(dual) TQFT Hilbert space (a.k.a.\ skein module) associated to $X$
by $\cF$.
(See \S \ref{sec:local-relations}.)
\begin{equation*}
H_0(\bc_*(X;\cF)) \iso A_{\cF}(X)
\end{equation*}
\end{thm:skein-modules}

\newtheorem*{thm:hochschild}{Theorem \ref{thm:hochschild}}

\begin{thm:hochschild}[Hochschild homology when $X=S^1$]
The blob complex for a $1$-category $\cC$ on the circle is
quasi-isomorphic to the Hochschild complex.
\begin{equation*}
\xymatrix{\bc_*(S^1;\cC) \ar[r]^(0.47){\iso}_(0.47){\text{qi}} & \HC_*(\cC).}
\end{equation*}
\end{thm:hochschild}

Proposition \ref{thm:skein-modules} is immediate from the definition, and
Theorem \ref{thm:hochschild} is established in \S \ref{sec:hochschild}.

\subsection{Structure of the blob complex}
\label{sec:structure}

In the following $\CH{X}$ is the singular chain complex of the space of homeomorphisms of $X$, fixed on $\bdy X$.

\newtheorem*{thm:CH}{Theorem \ref{thm:CH}}

\begin{thm:CH}[$C_*(\Homeo(-))$ action]
There is a chain map
\begin{equation*}
e_X: \CH{X} \tensor \bc_*(X) \to \bc_*(X).
\end{equation*}
such that
\begin{enumerate}
\item Restricted to $C_0(\Homeo(X))$ this is the action of homeomorphisms described in Property \ref{property:functoriality}. 

\item For
any codimension $0$-submanifold $Y \sqcup Y^\text{op} \subset \bdy X$ the following diagram
(using the gluing maps described in Property \ref{property:gluing-map}) commutes (up to homotopy).
\begin{equation*}
\xymatrix@C+2cm{
     \CH{X} \otimes \bc_*(X)
        \ar[r]_{e_{X}}  \ar[d]^{\gl^{\Homeo}_Y \otimes \gl_Y}  &
            \bc_*(X) \ar[d]_{\gl_Y} \\
     \CH{X \bigcup_Y \selfarrow} \otimes \bc_*(X \bigcup_Y \selfarrow) \ar[r]^<<<<<<<<<<<<{e_{(X \bigcup_Y \scalebox{0.5}{\selfarrow})}}    & \bc_*(X \bigcup_Y \selfarrow)
}
\end{equation*}
\end{enumerate}
\end{thm:CH}

\newtheorem*{thm:CH-associativity}{Theorem \ref{thm:CH-associativity}}

Further,
\begin{thm:CH-associativity}
The chain map of Theorem \ref{thm:CH} is associative, in the sense that the following diagram commutes (up to homotopy).
\begin{equation*}
\xymatrix{
\CH{X} \tensor \CH{X} \tensor \bc_*(X) \ar[r]^<<<<<{\id \tensor e_X} \ar[d]^{\compose \tensor \id} & \CH{X} \tensor \bc_*(X) \ar[d]^{e_X} \\
\CH{X} \tensor \bc_*(X) \ar[r]^{e_X} & \bc_*(X)
}
\end{equation*}
\end{thm:CH-associativity}

Since the blob complex is functorial in the manifold $X$, this is equivalent to having chain maps
$$ev_{X \to Y} : \CH{X \to Y} \tensor \bc_*(X) \to \bc_*(Y)$$
for any homeomorphic pair $X$ and $Y$, 
satisfying corresponding conditions.

In \S \ref{sec:ncats} we introduce the notion of disk-like $n$-categories, 
from which we can construct systems of fields.
Traditional $n$-categories can be converted to disk-like $n$-categories by taking string diagrams
(see \S\ref{sec:example:traditional-n-categories(fields)}).
Below, when we talk about the blob complex for a disk-like $n$-category, 
we are implicitly passing first to this associated system of fields.
Further, in \S \ref{sec:ncats} we also have the notion of a disk-like $A_\infty$ $n$-category. 
In that section we describe how to use the blob complex to 
construct disk-like $A_\infty$ $n$-categories from ordinary disk-like $n$-categories:

\newtheorem*{ex:blob-complexes-of-balls}{Example \ref{ex:blob-complexes-of-balls}}

\begin{ex:blob-complexes-of-balls}[Blob complexes of products with balls form a disk-like $A_\infty$ $n$-category]
Let $\cC$ be  an ordinary disk-like $n$-category.
Let $Y$ be an $n{-}k$-manifold. 
There is a disk-like $A_\infty$ $k$-category $\bc_*(Y;\cC)$, defined on each $m$-ball $D$, for $0 \leq m < k$, 
to be the set $$\bc_*(Y;\cC)(D) = \cC(Y \times D)$$ and on $k$-balls $D$ to be the set 
$$\bc_*(Y;\cC)(D) = \bc_*(Y \times D; \cC).$$ 
(When $m=k$ the subsets with fixed boundary conditions form a chain complex.) 
These sets have the structure of a disk-like $A_\infty$ $k$-category, with compositions coming from the gluing map in 
Property \ref{property:gluing-map} and with the action of families of homeomorphisms given in Theorem \ref{thm:evaluation}.
\end{ex:blob-complexes-of-balls}

\begin{rem}
Perhaps the most interesting case is when $Y$ is just a point; 
then we have a way of building a disk-like $A_\infty$ $n$-category from an ordinary $n$-category. 
We think of this disk-like $A_\infty$ $n$-category as a free resolution of the ordinary $n$-category.
\end{rem}

There is a version of the blob complex for $\cC$ a disk-like $A_\infty$ $n$-category
instead of an ordinary $n$-category; this is described in \S \ref{sec:ainfblob}.
The definition is in fact simpler, almost tautological, and we use a different notation, $\cl{\cC}(M)$. 
The next theorem describes the blob complex for product manifolds
in terms of the $A_\infty$ blob complex of the disk-like $A_\infty$ $n$-categories constructed as in the previous example.

\newtheorem*{thm:product}{Theorem \ref{thm:product}}

\begin{thm:product}[Product formula]
Let $W$ be a $k$-manifold and $Y$ be an $n-k$ manifold.
Let $\cC$ be an $n$-category.
Let $\bc_*(Y;\cC)$ be the disk-like $A_\infty$ $k$-category associated to $Y$ via blob homology 
(see Example \ref{ex:blob-complexes-of-balls}).
Then
\[
	\bc_*(Y\times W; \cC) \simeq \cl{\bc_*(Y;\cC)}(W).
\]
\end{thm:product}
The statement can be generalized to arbitrary fibre bundles, and indeed to arbitrary maps
(see \S \ref{ss:product-formula}).

Fix a disk-like $n$-category $\cC$, which we'll omit from the notation.
Recall that for any $(n{-}1)$-manifold $Y$, the blob complex $\bc_*(Y)$ is naturally an $A_\infty$ 1-category.
(See Appendix \ref{sec:comparing-A-infty} for the translation between disk-like $A_\infty$ $1$-categories 
and the usual algebraic notion of an $A_\infty$ category.)

\newtheorem*{thm:gluing}{Theorem \ref{thm:gluing}}

\begin{thm:gluing}[Gluing formula]
\mbox{}
\begin{itemize}
\item For any $n$-manifold $X$, with $Y$ a codimension $0$-submanifold of its boundary, the blob complex of $X$ is naturally an
$A_\infty$ module for $\bc_*(Y)$.

\item For any $n$-manifold $X_\text{gl} = X\bigcup_Y \selfarrow$, the blob complex $\bc_*(X_\text{gl})$ 
is the $A_\infty$ self-tensor product of
$\bc_*(X)$ as an $\bc_*(Y)$-bimodule:
\begin{equation*}
\bc_*(X_\text{gl}) \simeq \bc_*(X) \Tensor^{A_\infty}_{\mathclap{\bc_*(Y)}} \selfarrow
\end{equation*}
\end{itemize}
\end{thm:gluing}

Theorem \ref{thm:product} is proved in \S \ref{ss:product-formula}, and Theorem \ref{thm:gluing} in \S \ref{sec:gluing}.

\subsection{Applications}
\label{sec:applications}
Finally, we give two applications of the above machinery.

\newtheorem*{thm:map-recon}{Theorem \ref{thm:map-recon}}

\begin{thm:map-recon}[Mapping spaces]
Let $\pi^\infty_{\le n}(T)$ denote the disk-like $A_\infty$ $n$-category based on singular chains on maps 
$B^n \to T$.
(The case $n=1$ is the usual $A_\infty$-category of paths in $T$.)
Then 
\[
	\bc_*(X; \pi^\infty_{\le n}(T)) \simeq \CM{X}{T},
\]
where $C_*$ denotes singular chains.
\end{thm:map-recon}

This says that we can recover (up to homotopy) the space of maps to $T$ via blob homology from local data. 
Note that there is no restriction on the connectivity of $T$.
The proof appears in \S \ref{sec:map-recon}.

\newtheorem*{thm:deligne}{Theorem \ref{thm:deligne}}

\begin{thm:deligne}[Higher dimensional Deligne conjecture]
The singular chains of the $n$-dimensional surgery cylinder operad act on blob cochains
(up to coherent homotopy).
Since the little $n{+}1$-balls operad is a suboperad of the $n$-dimensional surgery cylinder operad,
this implies that the little $n{+}1$-balls operad acts on blob cochains of the $n$-ball.
\end{thm:deligne}
See \S \ref{sec:deligne} for a full explanation of the statement, and the proof.


\subsection{\texorpdfstring{$n$}{n}-category terminology}
\label{n-cat-names}

Section \S \ref{sec:ncats} adds to the zoo of $n$-category definitions, and the new creatures need names.
Unfortunately, we have found it difficult to come up with terminology which satisfies all
of the colleagues whom we have consulted, or even satisfies just ourselves.

One distinction we need to make is between $n$-categories which are associative in dimension $n$ and those
that are associative only up to higher homotopies.
The latter are closely related to $(\infty, n)$-categories (i.e.\ $\infty$-categories where all morphisms
of dimension greater than $n$ are invertible), but we don't want to use that name
since we think of the higher homotopies not as morphisms of the $n$-category but
rather as belonging to some auxiliary category (like chain complexes)
that we are enriching in.
We have decided to call them ``$A_\infty$ $n$-categories", since they are a natural generalization 
of the familiar $A_\infty$ 1-categories.
We also considered the names ``homotopy $n$-categories" and ``infinity $n$-categories".
When we need to emphasize that we are talking about an $n$-category which is not $A_\infty$ in this sense
we will say ``ordinary $n$-category".

Another distinction we need to make is between our style of definition of $n$-categories and
more traditional and combinatorial definitions.
We will call instances of our definition ``disk-like $n$-categories", since $n$-dimensional disks
play a prominent role in the definition.
(In general we prefer ``$k$-ball" to ``$k$-disk", but ``ball-like" doesn't roll off 
the tongue as well as ``disk-like''.)

Another thing we need a name for is the ability to rotate morphisms around in various ways.
For 2-categories, ``strict pivotal" is a standard term for what we mean. (See \cite{MR1686423, 0908.3347}, although note there the definition is only for monoidal categories; one can think of a monoidal category as a 2-category with only one $0$-morphism, then relax this requirement, to obtain the sensible notion of pivotal (or strict pivotal) for 2-categories. Compare also \cite{1009.0186} which addresses this issue explicitly.)
A more general term is ``duality", but duality comes in various flavors and degrees.
We are mainly interested in a very strong version of duality, where the available ways of
rotating $k$-morphisms correspond to all the ways of rotating $k$-balls.
We sometimes refer to this as ``strong duality", and sometimes we consider it to be implied
by ``disk-like".
(But beware: disks can come in various flavors, and some of them, such as framed disks,
don't actually imply much duality.)
Another possibility considered here was ``pivotal $n$-category", but we prefer to preserve pivotal for its usual sense. 
It will thus be a theorem that our disk-like 2-categories 
are equivalent to pivotal 2-categories, c.f. \S \ref{ssec:2-cats}.

Finally, we need a general name for isomorphisms between balls, where the balls could be
piecewise linear or smooth or Spin or framed or etc., or some combination thereof.
We have chosen to use ``homeomorphism" for the appropriate sort of isomorphism, so the reader should
keep in mind that ``homeomorphism" could mean PL homeomorphism or diffeomorphism (and so on)
depending on context.

\subsection{Thanks and acknowledgements}
We'd like to thank 
Justin Roberts (for helpful discussions in the very early stages of this work), 
Michael Freedman, 
Peter Teichner (for helping us improve an earlier version of the $n$-category definition), 
David Ben-Zvi, 
Vaughan Jones, 
Chris Schommer-Pries, 
Thomas Tradler,
Kevin Costello, 
Chris Douglas,
Alexander Kirillov,
Michael Shulman,
and
Rob Kirby
for many interesting and useful conversations. 
Peter Teichner ran a reading course based on an earlier draft of this paper, and the detailed feedback
we got from the student lecturers lead to very many improvements in later drafts.
So big thanks to
Aaron Mazel-Gee,
Nate Watson,
Alan Wilder,
Dmitri Pavlov,
Ansgar Schneider,
and
Dan Berwick-Evans.
We thank the anonymous referee for numerous suggestions which improved this paper.
During this work, Kevin Walker has been at Microsoft Station Q, and Scott Morrison has been at 
Microsoft Station Q and the Miller Institute for Basic Research at UC Berkeley. 
We'd like to thank the Aspen Center for Physics for the pleasant and productive 
environment provided there during the final preparation of this manuscript.


\section{TQFTs via fields}
\label{sec:fields}
\label{sec:tqftsviafields}

In this section we review the construction of TQFTs from fields and local relations.
For more details see \cite{kw:tqft}.
For our purposes, a TQFT is {\it defined} to be something which arises
from this construction.
This is an alternative to the more common definition of a TQFT
as a functor on cobordism categories satisfying various conditions.
A fully local (``down to points") version of the cobordism-functor TQFT definition
should be equivalent to the fields-and-local-relations definition.

A system of fields is very closely related to an $n$-category.
In one direction, Example \ref{ex:traditional-n-categories(fields)}
shows how to construct a system of fields from a (traditional) $n$-category.
We do this in detail for $n=1,2$ (\S\ref{sec:example:traditional-n-categories(fields)}) 
and more informally for general $n$.
In the other direction, 
our preferred definition of an $n$-category in \S\ref{sec:ncats} is essentially
just a system of fields restricted to balls of dimensions 0 through $n$;
one could call this the ``local" part of a system of fields.

Since this section is intended primarily to motivate
the blob complex construction of \S\ref{sec:blob-definition}, 
we suppress some technical details.
In \S\ref{sec:ncats} the analogous details are treated more carefully.

\medskip

We only consider compact manifolds, so if $Y \sub X$ is a closed codimension 0
submanifold of $X$, then $X \setmin Y$ implicitly means the closure
$\overline{X \setmin Y}$.

\subsection{Systems of fields}
\label{ss:syst-o-fields}

Let $\cM_k$ denote the category with objects 
unoriented PL manifolds of dimension
$k$ and morphisms homeomorphisms.
(We could equally well work with a different category of manifolds ---
oriented, smooth, spin, etc. --- but for simplicity we
will stick with unoriented PL.)

Fix a symmetric monoidal category $\cS$.
Fields on $n$-manifolds will be enriched over $\cS$.
Good examples to keep in mind are $\cS = \Set$ or $\cS = \Vect$.
The presentation here requires that the objects of $\cS$ have an underlying set, 
but this could probably be avoided if desired.

An $n$-dimensional {\it system of fields} in $\cS$
is a collection of functors $\cC_k : \cM_k \to \Set$ for $0 \leq k \leq n$
together with some additional data and satisfying some additional conditions, all specified below.

Before finishing the definition of fields, we give two motivating examples of systems of fields.

\begin{example}
\label{ex:maps-to-a-space(fields)}
Fix a target space $T$, and let $\cC(X)$ be the set of continuous maps
from $X$ to $T$.
\end{example}

\begin{example}
\label{ex:traditional-n-categories(fields)}
Fix an $n$-category $C$, and let $\cC(X)$ be 
the set of embedded cell complexes in $X$ with codimension-$j$ cells labeled by
$j$-morphisms of $C$.
One can think of such embedded cell complexes as dual to pasting diagrams for $C$.
This is described in more detail in \S \ref{sec:example:traditional-n-categories(fields)}.
\end{example}

Now for the rest of the definition of system of fields.
(Readers desiring a more precise definition should refer to \S\ref{ss:n-cat-def}
and replace $k$-balls with $k$-manifolds.)
\begin{enumerate}
\item There are boundary restriction maps $\cC_k(X) \to \cC_{k-1}(\bd X)$, 
and these maps comprise a natural
transformation between the functors $\cC_k$ and $\cC_{k-1}\circ\bd$.
For $c \in \cC_{k-1}(\bd X)$, we will denote by $\cC_k(X; c)$ the subset of 
$\cC(X)$ which restricts to $c$.
In this context, we will call $c$ a boundary condition.
\item The subset $\cC_n(X;c)$ of top-dimensional fields 
with a given boundary condition is an object in our symmetric monoidal category $\cS$.
(This condition is of course trivial when $\cS = \Set$.) 
If the objects are sets with extra structure (e.g. $\cS = \Vect$ or $\Kom$ (chain complexes)), 
then this extra structure is considered part of the definition of $\cC_n$.
Any maps mentioned below between fields on $n$-manifolds must be morphisms in $\cS$.
\item $\cC_k$ is compatible with the symmetric monoidal
structures on $\cM_k$, $\Set$ and $\cS$.
For $k<n$ we have $\cC_k(X \du W) \cong \cC_k(X)\times \cC_k(W)$,
compatibly with homeomorphisms and restriction to boundary.
For $k=n$ we require $\cC_n(X \du W; c\du d) \cong \cC_k(X, c)\ot \cC_k(W, d)$.
We will call the projections $\cC_k(X_1 \du X_2) \to \cC_k(X_i)$
restriction maps.
\item Gluing without corners.
Let $\bd X = Y \du Y \du W$, where $Y$ and $W$ are closed $k{-}1$-manifolds.
Let $X\sgl$ denote $X$ glued to itself along the two copies of $Y$.
Using the boundary restriction and disjoint union
maps, we get two maps $\cC_k(X) \to \cC(Y)$, corresponding to the two
copies of $Y$ in $\bd X$.
Let $\Eq_Y(\cC_k(X))$ denote the equalizer of these two maps.
(When $X$ is a disjoint union $X_1\du X_2$ the equalizer is the same as the fibered product
$\cC_k(X_1)\times_{\cC(Y)} \cC_k(X_2)$.)
Then (here's the axiom/definition part) there is an injective ``gluing" map
\[
	\Eq_Y(\cC_k(X)) \hookrightarrow \cC_k(X\sgl) ,
\]
and this gluing map is compatible with all of the above structure (actions
of homeomorphisms, boundary restrictions, disjoint union).
Furthermore, up to homeomorphisms of $X\sgl$ isotopic to the identity 
and collaring maps,
the gluing map is surjective.
We say that fields on $X\sgl$ in the image of the gluing map
are transverse to $Y$ or splittable along $Y$.
\item Gluing with corners.
Let $\bd X = (Y \du Y) \cup W$, where the two copies of $Y$ 
are disjoint from each other and $\bd(Y\du Y) = \bd W$.
Let $X\sgl$ denote $X$ glued to itself along the two copies of $Y$
(Figure \ref{fig:gluing-with-corners}).
\begin{figure}[t]
\begin{center}
\begin{tikzpicture}

\node(A) at (-4,0) {
\begin{tikzpicture}[scale=.8, fill=blue!15!white]
\filldraw[line width=1.5pt] (-.4,1) .. controls +(-1,-.1) and +(-1,0) .. (0,-1)
		.. controls +(1,0) and +(1,-.1) .. (.4,1) -- (.4,3)
		.. controls +(3,-.4) and +(3,0) .. (0,-3)
		.. controls +(-3,0) and +(-3,-.1) .. (-.4,3) -- cycle;
\node at (0,-2) {$X$};
\node (W) at (-2.7,-2) {$W$};
\node (Y1) at (-1.2,3.5) {$Y$};
\node (Y2) at (1.4,3.5) {$Y$};
\node[outer sep=2.3] (y1e) at (-.4,2) {};
\node[outer sep=2.3] (y2e) at (.4,2) {};
\node (we1) at (-2.2,-1.1) {};
\node (we2) at (-.6,-.7) {};
\draw[->] (Y1) -- (y1e);
\draw[->] (Y2) -- (y2e);
\draw[->] (W) .. controls +(0,.5) and +(-.5,-.2) .. (we1);
\draw[->] (W) .. controls +(.5,0) and +(-.2,-.5) .. (we2);
\end{tikzpicture}
};

\node(B) at (4,0) {
\begin{tikzpicture}[scale=.8, fill=blue!15!white]
\fill (0,1) .. controls +(-1,0) and +(-1,0) .. (0,-1)
		.. controls +(1,0) and +(1,0) .. (0,1) -- (0,3)
		.. controls +(3,0) and +(3,0) .. (0,-3)
		.. controls +(-3,0) and +(-3,0) .. (0,3) -- cycle;
\draw[line width=1.5pt] (0,1) .. controls +(-1,0) and +(-1,0) .. (0,-1)
		.. controls +(1,0) and +(1,0) .. (0,1);
\draw[line width=1.5pt] (0,3) .. controls +(3,0) and +(3,0) .. (0,-3)
		.. controls +(-3,0) and +(-3,0) .. (0,3);
\draw[line width=.5pt, black!65!white] (0,1) -- (0,3);
\node at (0,-2) {$X\sgl$};
\node (W) at (2.7,-2) {$W\sgl$};
\node (we1) at (2.2,-1.1) {};
\node (we2) at (.6,-.7) {};
\draw[->] (W) .. controls +(0,.5) and +(.5,-.2) .. (we1);
\draw[->] (W) .. controls +(-.5,0) and +(.2,-.5) .. (we2);
\end{tikzpicture}
};

\draw[->, red!50!green, line width=2pt] (A) -- node[above, black] {glue} (B);

\end{tikzpicture}
\end{center}
\caption{Gluing with corners}
\label{fig:gluing-with-corners}
\end{figure}
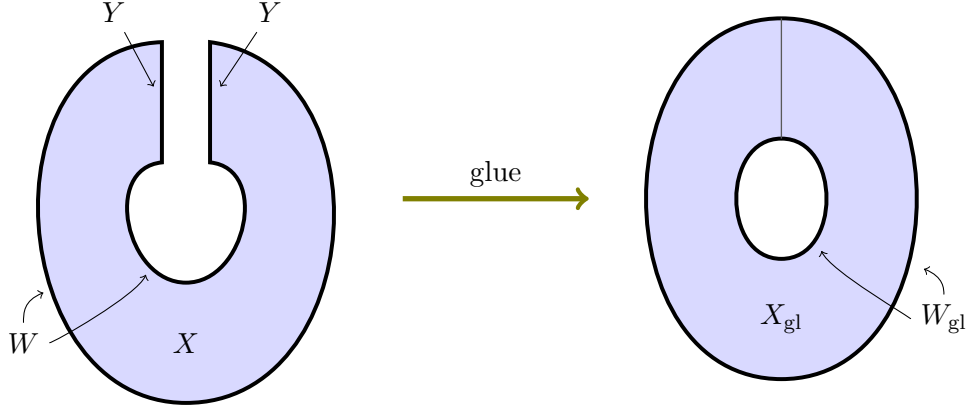
Note that $\bd X\sgl = W\sgl$, where $W\sgl$ denotes $W$ glued to itself
(without corners) along two copies of $\bd Y$.
Let $c\sgl \in \cC_{k-1}(W\sgl)$ be a be a splittable field on $W\sgl$ and let
$c \in \cC_{k-1}(W)$ be the cut open version of $c\sgl$.
Let $\cC^c_k(X)$ denote the subset of $\cC(X)$ which restricts to $c$ on $W$.
(This restriction map uses the gluing without corners map above.)
Using the boundary restriction and gluing without corners maps, 
we get two maps $\cC^c_k(X) \to \cC(Y)$, corresponding to the two
copies of $Y$ in $\bd X$.
Let $\Eq^c_Y(\cC_k(X))$ denote the equalizer of these two maps.
Then (here's the axiom/definition part) there is an injective ``gluing" map
\[
	\Eq^c_Y(\cC_k(X)) \hookrightarrow \cC_k(X\sgl, c\sgl) ,
\]
and this gluing map is compatible with all of the above structure (actions
of homeomorphisms, boundary restrictions, disjoint union).
Furthermore, up to homeomorphisms of $X\sgl$ isotopic to the identity
and collaring maps,
the gluing map is surjective.
We say that fields in the image of the gluing map
are transverse to $Y$ or splittable along $Y$.
\item Splittings.
Let $c\in \cC_k(X)$ and let $Y\sub X$ be a codimension 1 properly embedded submanifold of $X$.
Then for most small perturbations of $Y$ (e.g.\ for an open dense
subset of such perturbations, or for all perturbations satisfying
a transversality condition; c.f. Axiom \ref{axiom:splittings} much later) $c$ splits along $Y$.
(In Example \ref{ex:maps-to-a-space(fields)}, $c$ splits along all such $Y$.
In Example \ref{ex:traditional-n-categories(fields)}, $c$ splits along $Y$ so long as $Y$ 
is in general position with respect to the cell decomposition
associated to $c$.)
\item Product fields.
There are maps $\cC_{k-1}(Y) \to \cC_k(Y \times I)$, denoted
$c \mapsto c\times I$.
These maps comprise a natural transformation of functors, and commute appropriately
with all the structure maps above (disjoint union, boundary restriction, etc.).
Furthermore, if $f: Y\times I \to Y\times I$ is a fiber-preserving homeomorphism
covering $\bar{f}:Y\to Y$, then $f(c\times I) = \bar{f}(c)\times I$.
\end{enumerate}

There are two notations we commonly use for gluing.
One is 
\[
	x\sgl \deq \gl(x) \in \cC(X\sgl) , 
\]
for $x\in\cC(X)$.
The other is
\[
	x_1\bullet x_2 \deq \gl(x_1\otimes x_2) \in \cC(X\sgl) , 
\]
in the case that $X = X_1 \du X_2$, with $x_i \in \cC(X_i)$.

\medskip

Let $M$ be an $n$-manifold and $Y \subset \bd M$ be a codimension zero submanifold
of $\bd M$.
Let $M \cup (Y\times I)$ denote $M$ glued to $Y\times I$ along $Y$.
Extend the product structure on $Y\times I$ to a bicollar neighborhood of 
$Y$ inside $M \cup (Y\times I)$.
We call a homeomorphism
\[
	f: M \cup (Y\times I) \to M
\]
a {\it collaring homeomorphism} if $f$ is the identity outside of the bicollar
and $f$ preserves the fibers of the bicollar.

Using the functoriality and product field properties above, together
with collaring homeomorphisms, we can define 
{\it collar maps} $\cC(M)\to \cC(M)$.
Let $M$ and $Y \sub \bd M$ be as above.
Let $x \in \cC(M)$ be a field on $M$ and such that $\bd x$ is splittable along $\bd Y$.
Let $c$ be $x$ restricted to $Y$.
Then we have the glued field $x \bullet (c\times I)$ on $M \cup (Y\times I)$.
Let $f: M \cup (Y\times I) \to M$ be a collaring homeomorphism.
Then we call the map $x \mapsto f(x \bullet (c\times I))$ a {\it collar map}.

We call the equivalence relation generated by collar maps and
homeomorphisms isotopic to the identity {\it extended isotopy}, since the collar maps
can be thought of (informally) as the limit of homeomorphisms
which expand an infinitesimally thin collar neighborhood of $Y$ to a thicker
collar neighborhood.

\subsection{Systems of fields from \texorpdfstring{$n$}{n}-categories}
\label{sec:example:traditional-n-categories(fields)}
We now describe in more detail Example \ref{ex:traditional-n-categories(fields)}, 
systems of fields coming from embedded cell complexes labeled
by $n$-category morphisms.

Given an $n$-category $C$ with the right sort of duality,
e.g., a *-1-category (that is, a 1-category with an involution of the morphisms
reversing source and target) or a pivotal 2-category,
(\cite{MR1686423, 0908.3347,1009.0186}),
we can construct a system of fields as follows.
Roughly speaking, $\cC(X)$ will the set of all embedded cell complexes in $X$
with codimension $i$ cells labeled by $i$-morphisms of $C$.
We'll spell this out for $n=1,2$ and then describe the general case.

This way of decorating an $n$-manifold with an $n$-category is sometimes referred to
as a ``string diagram".
It can be thought of as (geometrically) dual to a pasting diagram.
One of the advantages of string diagrams over pasting diagrams is that one has more
flexibility in slicing them up in various ways.
In addition, string diagrams are traditional in quantum topology.
The diagrams predate by many years the terms ``string diagram" and 
``quantum topology", e.g. \cite{MR0281657,MR776784} 

If $X$ has boundary, we require that the cell decompositions are in general
position with respect to the boundary --- the boundary intersects each cell
transversely, so cells meeting the boundary are mere half-cells.
Put another way, the cell decompositions we consider are dual to standard cell
decompositions of $X$.

We will always assume that our $n$-categories have linear $n$-morphisms.

For $n=1$, a field on a 0-manifold $P$ is a labeling of each point of $P$ with
an object (0-morphism) of the 1-category $C$.
A field on a 1-manifold $S$ consists of
\begin{itemize}
    \item a cell decomposition of $S$ (equivalently, a finite collection
of points in the interior of $S$);
    \item a labeling of each 1-cell (and each half 1-cell adjacent to $\bd S$)
by an object (0-morphism) of $C$;
    \item a transverse orientation of each 0-cell, thought of as a choice of
``domain" and ``range" for the two adjacent 1-cells; and
    \item a labeling of each 0-cell by a 1-morphism of $C$, with
domain and range determined by the transverse orientation and the labelings of the 1-cells.
\end{itemize}

We want fields on 1-manifolds to be enriched over $\Vect$, so we also allow formal linear combinations
of the above fields on a 1-manifold $X$ so long as these fields restrict to the same field on $\bd X$.

In addition, we mod out by the relation which replaces
a 1-morphism label $a$ of a 0-cell $p$ with $a^*$ and reverse the transverse orientation of $p$.

If $C$ is a *-algebra (i.e. if $C$ has only one 0-morphism) we can ignore the labels
of 1-cells, so a field on a 1-manifold $S$ is a finite collection of points in the
interior of $S$, each transversely oriented and each labeled by an element (1-morphism)
of the algebra.

\medskip

For $n=2$, fields are just the sort of pictures based on 2-categories (e.g.\ tensor categories)
that are common in the literature.
We describe these carefully here.

A field on a 0-manifold $P$ is a labeling of each point of $P$ with
an object of the 2-category $C$.
A field of a 1-manifold is defined as in the $n=1$ case, using the 0- and 1-morphisms of $C$.
A field on a 2-manifold $Y$ consists of
\begin{itemize}
    \item a cell decomposition of $Y$ (equivalently, a graph embedded in $Y$ such
that each component of the complement is homeomorphic to a disk);
    \item a labeling of each 2-cell (and each partial 2-cell adjacent to $\bd Y$)
by a 0-morphism of $C$;
    \item a transverse orientation of each 1-cell, thought of as a choice of
``domain" and ``range" for the two adjacent 2-cells;
    \item a labeling of each 1-cell by a 1-morphism of $C$, with
domain and range determined by the transverse orientation of the 1-cell
and the labelings of the 2-cells;
    \item for each 0-cell, a homeomorphism of the boundary $R$ of a small neighborhood
of the 0-cell to $S^1$ such that the intersections of the 1-cells with $R$ are not mapped
to $\pm 1 \in S^1$
(this amounts to splitting of the link of the 0-cell into domain and range); and
    \item a labeling of each 0-cell by a 2-morphism of $C$, with domain and range
determined by the labelings of the 1-cells and the parameterizations of the previous
bullet.
\end{itemize}

As in the $n=1$ case, we allow formal linear combinations of fields on 2-manifolds, 
so long as their restrictions to the boundary coincide.

In addition, we regard the labelings as being equivariant with respect to the * structure
on 1-morphisms and pivotal structure on 2-morphisms.
That is, we mod out by the relation which flips the transverse orientation of a 1-cell 
and replaces its label $a$ by $a^*$, as well as the relation which changes the parameterization of the link
of a 0-cell and replaces its label by the appropriate pivotal conjugate.

\medskip

For general $n$, a field on a $k$-manifold $X^k$ consists of
\begin{itemize}
    \item A cell decomposition of $X$;
    \item an explicit general position homeomorphism from the link of each $j$-cell
to the boundary of the standard $(k-j)$-dimensional bihedron; and
    \item a labeling of each $j$-cell by a $(k-j)$-dimensional morphism of $C$, with
domain and range determined by the labelings of the link of $j$-cell.
\end{itemize}

It is customary when drawing string diagrams to omit identity morphisms.
In the above context, this corresponds to erasing cells which are labeled by identity morphisms.
The resulting structure might not, strictly speaking, be a cell complex.
So when we write ``cell complex" above we really mean a stratification which can be
refined to a genuine cell complex.

\subsection{Local relations}
\label{sec:local-relations}

For convenience we assume that fields are enriched over $\Vect$.

Local relations are subspaces $U(B; c)\sub \cC(B; c)$ of the fields on balls which form an ideal under gluing.
Again, we give the examples first.

\addtocounter{subsection}{-2}
\begin{example}[contd.]
For maps into spaces, $U(B; c)$ is generated by fields of the form $a-b \in \lf(B; c)$,
where $a$ and $b$ are maps (fields) which are homotopic rel boundary.
\end{example}

\begin{example}[contd.]
For $n$-category pictures, $U(B; c)$ is equal to the kernel of the evaluation map
$\lf(B; c) \to \mor(c', c'')$, where $(c', c'')$ is some (any) division of $c$ into
domain and range.
\end{example}
\addtocounter{subsection}{2}
\addtocounter{prop}{-2}

These motivate the following definition.

\begin{defn}
A {\it local relation} is a collection of subspaces $U(B; c) \sub \lf(B; c)$,
for all $n$-manifolds $B$ which are
homeomorphic to the standard $n$-ball and all $c \in \cC(\bd B)$, 
satisfying the following properties.
\begin{enumerate}
\item Functoriality: 
$f(U(B; c)) = U(B', f(c))$ for all homeomorphisms $f: B \to B'$
\item Local relations imply extended isotopy invariance: 
if $x, y \in \cC(B; c)$ and $x$ is extended isotopic 
to $y$, then $x-y \in U(B; c)$.
\item Ideal with respect to gluing:
if $B = B' \cup B''$, $x\in U(B')$, and $r\in \cC(B'')$, then $x\bullet r \in U(B)$
\end{enumerate}
\end{defn}
See \cite{kw:tqft} for further details.

\subsection{Constructing a TQFT}
\label{sec:constructing-a-tqft}

In this subsection we briefly review the construction of a TQFT from a system of fields and local relations.
As usual, see \cite{kw:tqft} for more details.

We can think of a path integral $Z(W)$ of an $n+1$-manifold 
(which we're not defining in this context; this is just motivation) as assigning to each
boundary condition $x\in \cC(\bd W)$ a complex number $Z(W)(x)$.
In other words, $Z(W)$ lies in $\c^{\lf(\bd W)}$, the vector space of linear
maps $\lf(\bd W)\to \c$.

The locality of the TQFT implies that $Z(W)$ in fact lies in a subspace
$Z(\bd W) \sub \c^{\lf(\bd W)}$ defined by local projections.
The linear dual to this subspace, $A(\bd W) = Z(\bd W)^*$,
can be thought of as finite linear combinations of fields modulo local relations.
(In other words, $A(\bd W)$ is a sort of generalized skein module.)
This is the motivation behind the definition of fields and local relations above.

In more detail, let $X$ be an $n$-manifold.
\begin{defn}
\label{defn:TQFT-invariant}
The TQFT invariant of $X$ associated to a system of fields $\cC$ and local relations $U$ is 
	$$A(X) \deq \lf(X) / U(X),$$
where $U(X) \sub \lf(X)$ is the space of local relations in $\lf(X)$:
$U(X)$ is generated by fields of the form $u\bullet r$, where
$u\in U(B)$ for some embedded $n$-ball $B\sub X$ and $r\in \cC(X\setmin B)$.
\end{defn}
The blob complex, defined in the next section, 
is in some sense the derived version of $A(X)$.
If $X$ has boundary we can similarly define $A(X; c)$ for each 
boundary condition $c\in\cC(\bd X)$.

The above construction can be extended to higher codimensions, assigning
a $k$-category $A(Y)$ to an $n{-}k$-manifold $Y$, for $0 \le k \le n$.
These invariants fit together via actions and gluing formulas.
We describe only the case $k=1$ below. We describe these extensions in the more general setting of the blob complex later, in particular in Examples \ref{ex:ncats-from-tqfts} and \ref{ex:blob-complexes-of-balls} and  in \S \ref{sec:modules}.

The construction of the $n{+}1$-dimensional part of the theory (the path integral) 
requires that the starting data (fields and local relations) satisfy additional
conditions.
(Specifically, $A(X; c)$ is finite dimensional for all $n$-manifolds $X$ and the inner products
on $A(B^n; c)$ induced by the path integral of $B^{n+1}$ are positive definite for all $c$.)
We do not assume these conditions here, so when we say ``TQFT" we mean a ``decapitated" TQFT
that lacks its $n{+}1$-dimensional part. 
Such a decapitated TQFT is sometimes also called an $n{+}\epsilon$ or 
$n{+}\frac{1}{2}$ dimensional TQFT, referring to the fact that it assigns linear maps to $n{+}1$-dimensional
mapping cylinders between $n$-manifolds, but nothing to general $n{+}1$-manifolds.

Let $Y$ be an $n{-}1$-manifold.
Define a linear 1-category $A(Y)$ as follows.
The set of objects of $A(Y)$ is $\cC(Y)$.
The morphisms from $a$ to $b$ are $A(Y\times I; a, b)$, 
where $a$ and $b$ label the two boundary components of the cylinder $Y\times I$.
Composition is given by gluing of cylinders.

Let $X$ be an $n$-manifold with boundary and consider the collection of vector spaces
$A(X; -) \deq \{A(X; c)\}$ where $c$ ranges through $\cC(\bd X)$.
This collection of vector spaces affords a representation of the category $A(\bd X)$, where
the action is given by gluing a collar $\bd X\times I$ to $X$.

Given a splitting $X = X_1 \cup_Y X_2$ of a closed $n$-manifold $X$ along an $n{-}1$-manifold $Y$,
we have left and right actions of $A(Y)$ on $A(X_1; -)$ and $A(X_2; -)$.
The gluing theorem for $n$-manifolds states that there is a natural isomorphism
\[
	A(X) \cong A(X_1; -) \otimes_{A(Y)} A(X_2; -) .
\]
A proof of this gluing formula appears in \cite{kw:tqft}, but it also becomes a 
special case of Theorem \ref{thm:gluing} by taking $0$-th homology.


\section{The blob complex}
\subsection{Definitions}
\label{sec:blob-definition}
Let $X$ be an $n$-manifold.
Let $(\cF,U)$ be a fixed system of fields and local relations.
We'll assume it is enriched over \textbf{Vect}; 
if it is not we can make it so by allowing finite
linear combinations of elements of $\cF(X; c)$, for fixed $c\in \cF(\bd X)$.


We want to replace the quotient
\[
	A(X) \deq \cF(X) / U(X)
\]
of Definition \ref{defn:TQFT-invariant} with a resolution
\[
	\cdots \to \bc_2(X) \to \bc_1(X) \to \bc_0(X) .
\]

We will define $\bc_0(X)$, $\bc_1(X)$ and $\bc_2(X)$, then give the general case $\bc_k(X)$. 
In fact, on the first pass we will intentionally describe the definition in a misleadingly simple way, 
then explain the technical difficulties, and finally give a cumbersome but complete definition in 
Definition \ref{defn:blobs}. 
If (we don't recommend it) you want to keep track of the ways in which 
this initial description is misleading, or you're reading through a second time to understand the 
technical difficulties, keep note that later we will give precise meanings to ``a ball in $X$'', 
``nested'' and ``disjoint'', that are not quite the intuitive ones. 
Moreover some of the pieces 
into which we cut manifolds below are not themselves manifolds, and it requires special attention 
to define fields on these pieces.

We of course define $\bc_0(X) = \cF(X)$.
In other words, $\bc_0(X)$ is just the vector space of all fields on $X$.

(If $X$ has nonempty boundary, instead define $\bc_0(X; c) = \cF(X; c)$ for $c \in \cF(\bdy X)$.
The blob complex $\bc_*(X; c)$ will depend on a fixed boundary condition $c\in \cF(\bdy X)$.
We'll omit such boundary conditions from the notation in the rest of this section.)

We want the vector space $\bc_1(X)$ to capture 
``the space of all local relations that can be imposed on $\bc_0(X)$".
Thus we say  a $1$-blob diagram consists of:
\begin{itemize}
\item A closed ball in $X$ (``blob") $B \sub X$.
\item A boundary condition $c \in \cF(\bdy B) = \cF(\bd(X \setmin B))$.
\item A field $r \in \cF(X \setmin B; c)$.
\item A local relation field $u \in U(B; c)$.
\end{itemize}
(See Figure \ref{blob1diagram}.) Since $c$ is implicitly determined by $u$ or $r$, we usually omit it from the notation.
\begin{figure}[t]
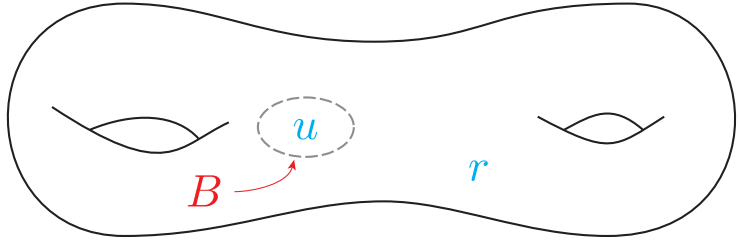
\begin{equation*}
\mathfig{.6}{definition/single-blob}
\end{equation*}\caption{A 1-blob diagram.}\label{blob1diagram}\end{figure}
In order to get the linear structure correct, we define
\[
	\bc_1(X) \deq \bigoplus_B \bigoplus_c U(B; c) \otimes \cF(X \setmin B; c) .
\]
The first direct sum is indexed by all blobs $B\subset X$, and the second
by all boundary conditions $c \in \cF(\bd B)$.
Note that $\bc_1(X)$ is spanned by 1-blob diagrams $(B, u, r)$.

Define the boundary map $\bd : \bc_1(X) \to \bc_0(X)$ by 
\[ 
	(B, u, r) \mapsto u\bullet r, 
\]
where $u\bullet r$ denotes the field on $X$ obtained by gluing $u$ to $r$.
In other words $\bd : \bc_1(X) \to \bc_0(X)$ is given by
just erasing the blob from the picture
(but keeping the blob label $u$).

Note that directly from the definition we have
\begin{prop}
\label{thm:skein-modules}
The skein module $A(X)$ is naturally isomorphic to $\bc_0(X)/\bd(\bc_1(X))) = H_0(\bc_*(X))$.
\end{prop}
This also establishes the second 
half of Property \ref{property:contractibility}.

Next, we want the vector space $\bc_2(X)$ to capture ``the space of all relations 
(redundancies, syzygies) among the 
local relations encoded in $\bc_1(X)$''.
A $2$-blob diagram comes in one of two types, disjoint and nested.
A disjoint 2-blob diagram consists of
\begin{itemize}
\item A pair of closed balls (blobs) $B_1, B_2 \sub X$ with disjoint interiors.
\item A field $r \in \cF(X \setmin (B_1 \cup B_2); c_1, c_2)$
(where $c_i \in \cF(\bd B_i)$).
\item Local relation fields $u_i \in U(B_i; c_i)$, $i=1,2$.
\end{itemize}
(See Figure \ref{blob2ddiagram}.)
\begin{figure}[t]
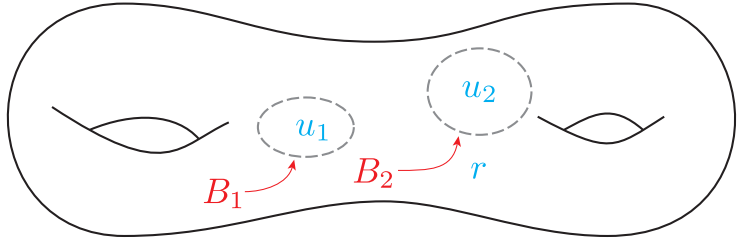
\begin{equation*}
\mathfig{.6}{definition/disjoint-blobs}
\end{equation*}\caption{A disjoint 2-blob diagram.}\label{blob2ddiagram}\end{figure}
We also identify $(B_1, B_2, u_1, u_2, r)$ with $-(B_2, B_1, u_2, u_1, r)$;
reversing the order of the blobs changes the sign.
Define $\bd(B_1, B_2, u_1, u_2, r) = 
(B_2, u_2, u_1\bullet r) - (B_1, u_1, u_2\bullet r) \in \bc_1(X)$.
In other words, the boundary of a disjoint 2-blob diagram
is the sum (with alternating signs)
of the two ways of erasing one of the blobs.
It's easy to check that $\bd^2 = 0$.

A nested 2-blob diagram consists of
\begin{itemize}
\item A pair of nested balls (blobs) $B_1 \subseteq B_2 \subseteq X$.
\item A field $r' \in \cF(B_2 \setminus B_1; c_1, c_2)$ 
(for some $c_1 \in \cF(\bdy B_1)$ and $c_2 \in \cF(\bdy B_2)$).
\item A field $r \in \cF(X \setminus B_2; c_2)$.
\item A local relation field $u \in U(B_1; c_1)$.
\end{itemize}
(See Figure \ref{blob2ndiagram}.)
\begin{figure}[t]
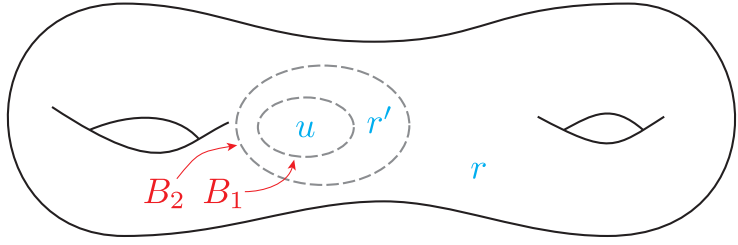
\begin{equation*}
\mathfig{.6}{definition/nested-blobs}
\end{equation*}\caption{A nested 2-blob diagram.}\label{blob2ndiagram}\end{figure}
Define $\bd(B_1, B_2, u, r', r) = (B_2, u\bullet r', r) - (B_1, u, r' \bullet r)$.
As in the disjoint 2-blob case, the boundary of a nested 2-blob is the alternating
sum of the two ways of erasing one of the blobs.
When  we erase the inner blob, the outer blob inherits the label $u\bullet r'$.
It is again easy to check that $\bd^2 = 0$. Note that the requirement that
local relations are an ideal with respect to gluing guarantees that $u\bullet r' \in U(B_2)$.

As with the $1$-blob diagrams, in order to get the linear structure correct the actual definition is 
\begin{eqnarray*}
	\bc_2(X) & \deq &
	\left( 
		\bigoplus_{B_1, B_2\; \text{disjoint}} \bigoplus_{c_1, c_2}
			U(B_1; c_1) \otimes U(B_2; c_2) \otimes \cF(X\setmin (B_1\cup B_2); c_1, c_2)
	\right)  \bigoplus \\
	&& \quad\quad  \left( 
		\bigoplus_{B_1 \subset B_2} \bigoplus_{c_1, c_2}
			U(B_1; c_1) \otimes \cF(B_2 \setmin B_1; c_1, c_2) \tensor \cF(X \setminus B_2; c_2)
	\right) .
\end{eqnarray*}

\medskip

Roughly, $\bc_k(X)$ is generated by configurations of $k$ blobs, pairwise disjoint or nested, 
along with fields on all the components that the blobs divide $X$ into. 
Blobs which have no other blobs inside are called `twig blobs', 
and the fields on the twig blobs must be local relations.
The boundary is the alternating sum of erasing one of the blobs.
In order to describe this general case in full detail, we must give a more precise description of
which configurations of balls inside $X$ we permit.
These configurations are generated by two operations:
\begin{itemize}
\item For any (possibly empty) configuration of blobs on an $n$-ball $D$, we can add
$D$ itself as an outermost blob.
(This is used in the proof of Proposition \ref{bcontract}.)
\item If $X\sgl$ is obtained from $X$ by gluing, then any permissible configuration of blobs
on $X$ gives rise to a permissible configuration on $X\sgl$.
(This is necessary for Proposition \ref{blob-gluing}.)
\end{itemize}
Combining these two operations can give rise to configurations of blobs whose complement in $X$ is not
a manifold.
Thus we will need to be more careful when speaking of a field $r$ on the complement of the blobs.

\begin{example} \label{sin1x-example}
Consider the four subsets of $\Real^3$,
\begin{align*}
A & = [0,1] \times [0,1] \times [0,1] \\
B & = [0,1] \times [-1,0] \times [0,1] \\
C & = [-1,0] \times \setc{(y,z)}{e^{-1/z^2} \sin(1/z) \leq y \leq 1, z \in [0,1]} \\
D & = [-1,0] \times \setc{(y,z)}{-1 \leq y \leq e^{-1/z^2} \sin(1/z), z \in [0,1]}.
\end{align*}
Here $A \cup B = [0,1] \times [-1,1] \times [0,1]$ and $C \cup D = [-1,0] \times [-1,1] \times [0,1]$. 
Now, $\{A\}$ is a valid configuration of blobs in $A \cup B$, 
and $\{D\}$ is a valid configuration of blobs in $C \cup D$, 
so we must allow $\{A, D\}$ as a configuration of blobs in $[-1,1]^2 \times [0,1]$. 
Note however that the complement is not a manifold. See Figure \ref{fig:blocks}.
\end{example}

\begin{figure}[t]
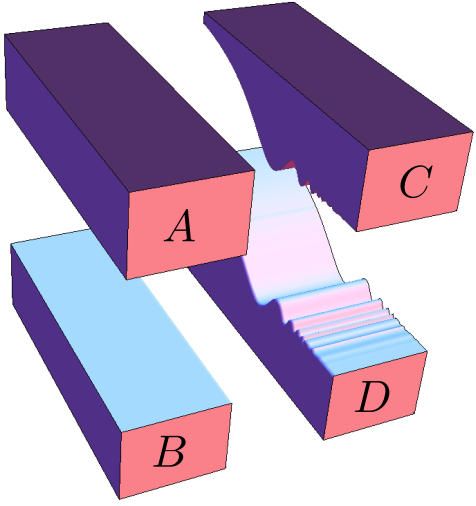
\begin{equation*}
\mathfig{.4}{definition/blocks}
\end{equation*}\caption{The subsets $A$, $B$, $C$ and $D$ from Example \ref{sin1x-example}. The pair $\{A, D\}$ is a valid configuration of blobs, even though the complement is not a manifold.}\label{fig:blocks}\end{figure}

\begin{defn}
\label{defn:gluing-decomposition}
A \emph{gluing decomposition} of an $n$-manifold $X$ is a sequence of manifolds 
$M_0 \to M_1 \to \cdots \to M_m = X$ such that each $M_k$ is obtained from $M_{k-1}$ 
by gluing together some disjoint pair of homeomorphic $n{-}1$-manifolds in the boundary of $M_{k-1}$.
If, in addition, $M_0$ is a disjoint union of balls, we call it a \emph{ball decomposition}.
\end{defn}

Let $M_0 \to M_1 \to \cdots \to M_m = X$ be a gluing decomposition of $X$, 
and let $M_0^0,\ldots,M_0^k$ be the connected components of $M_0$.
We say that a field 
$a\in \cF(X)$ is splittable along the decomposition if $a$ is the image 
under gluing and disjoint union of fields $a_i \in \cF(M_0^i)$, $0\le i\le k$.
Note that if $a$ is splittable in this sense then it makes sense to talk about the restriction of $a$ to any
component $M'_j$ of any $M_j$ of the decomposition.

In the example above, note that
\[
	A \sqcup B \sqcup C \sqcup D \to (A \cup B) \sqcup (C \cup D) \to A \cup B \cup C \cup D
\]
is a  ball decomposition, but other sequences of gluings starting from $A \sqcup B \sqcup C \sqcup D$
have intermediate steps which are not manifolds.

We'll now slightly restrict the possible configurations of blobs.
\begin{defn}
\label{defn:configuration}
A configuration of $k$ blobs in $X$ is an ordered collection of $k$ subsets $\{B_1, \ldots, B_k\}$ 
of $X$ such that there exists a gluing decomposition $M_0  \to \cdots \to M_m = X$ of $X$ 
with the property that 
for each subset $B_i$ there is some $0 \leq l \leq m$ and some connected component $M_l'$ of 
$M_l$ which is a ball, such that $B_i$ is the image of $M_l'$ in $X$. 
We say that such a gluing decomposition 
is \emph{compatible} with the configuration. 
A blob $B_i$ is a twig blob if no other blob $B_j$ is a strict subset of it. 
\end{defn}
In particular, this implies what we said about blobs above: 
that for any two blobs in a configuration of blobs in $X$, 
they either have disjoint interiors, or one blob is contained in the other. 
We describe these as disjoint blobs and nested blobs. 
Note that nested blobs may have boundaries that overlap, or indeed coincide. 
Blobs may meet the boundary of $X$.
Further, note that blobs need not actually be embedded balls in $X$, since parts of the 
boundary of the ball $M_l'$ may have been glued together.

Note that often the gluing decomposition for a configuration of blobs may just be the trivial one: 
if the boundaries of all the blobs cut $X$ into pieces which are all manifolds, 
we can just take $M_0$ to be these pieces, and $M_1 = X$.

In the initial informal definition of a $k$-blob diagram above, we allowed any 
collection of $k$ balls which were pairwise disjoint or nested. 
We now further require that the balls are a configuration in the sense of Definition \ref{defn:configuration}. 
We also specified a local relation on each twig blob, and a field on the complement of the twig blobs; 
this is unsatisfactory because that complement need not be a manifold. Thus, the official definitions are
\begin{defn}
\label{defn:blob-diagram}
A $k$-blob diagram on $X$ consists of
\begin{itemize}
\item a configuration $\{B_1, \ldots, B_k\}$ of $k$ blobs in $X$,
\item and a field $r \in \cF(X)$ which is splittable along some gluing decomposition compatible with that configuration,
\end{itemize}
such that
the restriction $u_i$ of $r$ to each twig blob $B_i$ lies in the subspace 
$U(B_i) \subset \cF(B_i)$. 
(See Figure \ref{blobkdiagram}.) 
More precisely, each twig blob $B_i$ is the image of some ball $M_l'$ as above, 
and it is really the restriction to $M_l'$ that must lie in the subspace $U(M_l')$.
\end{defn}
\begin{figure}[t]
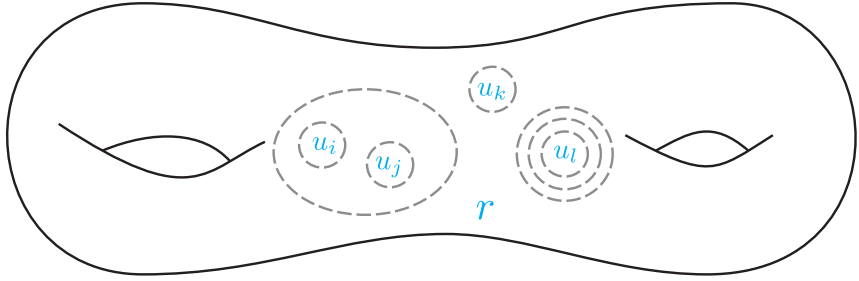
\begin{equation*}
\mathfig{.7}{definition/k-blobs}
\end{equation*}\caption{A $k$-blob diagram.}\label{blobkdiagram}\end{figure}

\begin{defn}
\label{defn:blobs}
The $k$-th vector space $\bc_k(X)$ of the \emph{blob complex} of $X$ is the direct sum over all 
configurations of $k$ blobs in $X$ of the vector space of $k$-blob diagrams with that configuration, 
modulo identifying the vector spaces for configurations that only differ by a permutation of the blobs 
by the sign of that permutation. 
The differential $\bc_k(X) \to \bc_{k-1}(X)$ is, as above, the signed sum of ways of 
forgetting one blob from the configuration, preserving the field $r$:
\begin{equation*}
\bdy(\{B_1, \ldots B_k\}, r) = \sum_{i=1}^{k} (-1)^{i+1} (\{B_1, \ldots, \widehat{B_i}, \ldots, B_k\}, r)
\end{equation*}
\end{defn}
We readily see that if a gluing decomposition is compatible with some configuration of blobs, 
then it is also compatible with any configuration obtained by forgetting some blobs, 
ensuring that the differential in fact lands in the space of $k{-}1$-blob diagrams.
A slight compensation to the complication of the official definition arising from attention 
to splitting is that the differential now just preserves the entire field $r$ without 
having to say anything about gluing together fields on smaller components.

Note that Property \ref{property:functoriality}, that the blob complex is functorial with respect to homeomorphisms, 
is immediately obvious from the definition.
A homeomorphism acts in an obvious way on blobs and on fields.

\begin{remark} \label{blobsset-remark} \rm
We note that blob diagrams in $X$ have a structure similar to that of a simplicial set,
but with simplices replaced by a more general class of combinatorial shapes.
Let $P$ be the minimal set of (isomorphisms classes of) polyhedra which is closed under products
and cones, and which contains the point.
We can associate an element $p(b)$ of $P$ to each blob diagram $b$ 
(equivalently, to each rooted tree) according to the following rules:
\begin{itemize}
\item $p(\emptyset) = pt$, where $\emptyset$ denotes a 0-blob diagram or empty tree;
\item $p(a \du b) = p(a) \times p(b)$, where $a \du b$ denotes the distant (non-overlapping) union 
of two blob diagrams (equivalently, join two trees at the roots); and
\item $p(\bar{b}) = \kone(p(b))$, where $\bar{b}$ is obtained from $b$ by adding an outer blob which 
encloses all the others (equivalently, add a new edge to the root, with the new vertex becoming the root
of the new tree).
\end{itemize}
For example, a diagram of $k$ strictly nested blobs corresponds to a $k$-simplex, while
a diagram of $k$ disjoint blobs corresponds to a $k$-cube.
(When the fields come from an $n$-category, this correspondence works best if we think of each 
twig label $u_i$ as having the form
$x - s(e(x))$, where $x$ is an arbitrary field on $B_i$, $e: \cF(B_i) \to C$ is the evaluation map, 
and $s:C \to \cF(B_i)$ is some fixed section of $e$.)

For lack of a better name, 
we'll call elements of $P$ cone-product polyhedra, 
and say that blob diagrams have the structure of a cone-product set (analogous to simplicial set).
\end{remark}


\subsection{Basic properties}
\label{sec:basic-properties}

In this section we complete the proofs of Properties \ref{property:disjoint-union}--\ref{property:contractibility}.
Throughout the paper, where possible, we prove results using Properties \ref{property:functoriality}--\ref{property:contractibility}, 
rather than the actual definition of the blob complex.
This allows the possibility of future improvements on or alternatives to our definition.
In fact, we hope that there may be a characterization of the blob complex in 
terms of Properties \ref{property:functoriality}--\ref{property:contractibility}, but at this point we are unaware of one.

Recall Property \ref{property:disjoint-union}, 
that there is a natural isomorphism $\bc_*(X \du Y) \cong \bc_*(X) \otimes \bc_*(Y)$.

\begin{proof}[Proof of Property \ref{property:disjoint-union}]
Given blob diagrams $b_1$ on $X$ and $b_2$ on $Y$, we can combine them
(putting the $b_1$ blobs before the $b_2$ blobs in the ordering) to get a
blob diagram $(b_1, b_2)$ on $X \du Y$.
Because of the blob reordering relations, all blob diagrams on $X \du Y$ arise this way.
In the other direction, any blob diagram on $X\du Y$ is equal (up to sign)
to one that puts $X$ blobs before $Y$ blobs in the ordering, and so determines
a pair of blob diagrams on $X$ and $Y$.
These two maps are compatible with our sign conventions.
(We follow the usual convention for tensors products of complexes, 
as in e.g. \cite{MR1438306}: $d(a \tensor b) = da \tensor b + (-1)^{\deg(a)} a \tensor db$.)
The two maps are inverses of each other.
\end{proof}

For the next proposition we will temporarily restore $n$-manifold boundary
conditions to the notation.

Suppose that for all $c \in \cC(\bd B^n)$
we have a splitting $s: H_0(\bc_*(B^n; c)) \to \bc_0(B^n; c)$
of the quotient map
$p: \bc_0(B^n; c) \to H_0(\bc_*(B^n; c))$.
For example, this is always the case if the coefficient ring is a field.
Then
\begin{prop} \label{bcontract}
For all $c \in \cC(\bd B^n)$ the natural map $p: \bc_*(B^n; c) \to H_0(\bc_*(B^n; c))$
is a chain homotopy equivalence
with inverse $s: H_0(\bc_*(B^n; c)) \to \bc_*(B^n; c)$.
Here we think of $H_0(\bc_*(B^n; c))$ as a 1-step complex concentrated in degree 0.
\end{prop}
\begin{proof}
By assumption $p\circ s = \id$, so all that remains is to find a degree 1 map
$h : \bc_*(B^n; c) \to \bc_*(B^n; c)$ such that $\bd h + h\bd = \id - s \circ p$.
For $i \ge 1$, define $h_i : \bc_i(B^n; c) \to \bc_{i+1}(B^n; c)$ by adding
an $(i{+}1)$-st blob equal to all of $B^n$.
In other words, add a new outermost blob which encloses all of the others.
Define $h_0 : \bc_0(B^n; c) \to \bc_1(B^n; c)$ by setting $h_0(x)$ equal to
the 1-blob with blob $B^n$ and label $x - s(p(x)) \in U(B^n; c)$.
\end{proof}
This proves Property \ref{property:contractibility} (the second half of the 
statement of this Property was immediate from the definitions).
Note that even when there is no splitting $s$, we can let $h_0 = 0$ and get a homotopy
equivalence to the 2-step complex $U(B^n; c) \to \cC(B^n; c)$.

For fields based on $n$-categories, $H_0(\bc_*(B^n; c)) \cong \mor(c', c'')$,
where $(c', c'')$ is some (any) splitting of $c$ into domain and range.

\begin{cor} \label{disj-union-contract}
If $X$ is a disjoint union of $n$-balls, then $\bc_*(X; c)$ is contractible.
\end{cor}

\begin{proof}
This follows from Properties \ref{property:disjoint-union} and \ref{property:contractibility}.
\end{proof}

We define the {\it support} of a blob diagram $b$, $\supp(b) \sub X$, 
to be the union of the blobs of $b$.
For $y \in \bc_*(X)$ with $y = \sum c_i b_i$ ($c_i$ a non-zero number, $b_i$ a blob diagram),
we define $\supp(y) \deq \bigcup_i \supp(b_i)$.

\noop{ 
For future use we prove the following lemma.

\begin{lemma} \label{support-shrink}
Let $L_* \sub \bc_*(X)$ be a subcomplex generated by some
subset of the blob diagrams on $X$, and let $f: L_* \to L_*$
be a chain map which does not increase supports and which induces an isomorphism on
$H_0(L_*)$.
Then $f$ is homotopic (in $\bc_*(X)$) to the identity $L_*\to L_*$.
\end{lemma}

\begin{proof}
We will use the method of acyclic models.
Let $b$ be a blob diagram of $L_*$, let $S\sub X$ be the support of $b$, and let
$r$ be the restriction of $b$ to $X\setminus S$.
Note that $S$ is a disjoint union of balls.
Assign to $b$ the acyclic (in positive degrees) subcomplex $T(b) \deq r\bullet\bc_*(S)$.
Note that if a diagram $b'$ is part of $\bd b$ then $T(b') \sub T(b)$.
Both $f$ and the identity are compatible with $T$ (in the sense of acyclic models, \S\ref{sec:moam}), 
so $f$ and the identity map are homotopic.
\end{proof}
} 

For the next proposition we will temporarily restore $n$-manifold boundary
conditions to the notation. Let $X$ be an $n$-manifold, with $\bd X = Y \cup Y \cup Z$.
Gluing the two copies of $Y$ together yields an $n$-manifold $X\sgl$
with boundary $Z\sgl$.
Given compatible fields (boundary conditions) $a$, $b$ and $c$ on $Y$, $Y$ and $Z$,
we have the blob complex $\bc_*(X; a, b, c)$.
If $b = a$, then we can glue up blob diagrams on
$X$ to get blob diagrams on $X\sgl$.
This proves Property \ref{property:gluing-map}, which we restate here in more detail.

\begin{prop} \label{blob-gluing}
There is a natural chain map
\eq{
    \gl: \bigoplus_a \bc_*(X; a, a, c) \to \bc_*(X\sgl; c\sgl).
}
The sum is over all fields $a$ on $Y$ compatible at their
($n{-}2$-dimensional) boundaries with $c$.
``Natural" means natural with respect to the actions of homeomorphisms.
In degree zero the map agrees with the gluing map coming from the underlying system of fields.
\end{prop}

This map is very far from being an isomorphism, even on homology.
We eliminate this deficit in \S\ref{sec:gluing} below.


\section{Hochschild homology when \texorpdfstring{$n=1$}{n=1}}
\label{sec:hochschild}

\subsection{Outline}

So far we have provided no evidence that blob homology is interesting in degrees 
greater than zero.
In this section we analyze the blob complex in dimension $n=1$.

Recall (\S \ref{sec:example:traditional-n-categories(fields)}) 
that from a *-1-category $C$ we can construct a system of fields $\cC$.
In this section we prove that $\bc_*(S^1, \cC)$ is homotopy equivalent to the 
Hochschild complex of $C$.
Thus the blob complex is a natural generalization of something already
known to be interesting in higher homological degrees.

It is also worth noting that the original idea for the blob complex came from trying
to find a more ``local" description of the Hochschild complex.

\medskip

Let $C$ be a *-1-category.
Then specializing the definition of the associated system of fields from \S \ref{sec:example:traditional-n-categories(fields)} above to the case $n=1$ we have:
\begin{itemize}
\item $\cC(pt) = \ob(C)$ .
\item Let $R$ be a 1-manifold and $c \in \cC(\bd R)$.
Then an element of $\cC(R; c)$ is a collection of (transversely oriented)
points in the interior
of $R$, each labeled by a morphism of $C$.
The intervals between the points are labeled by objects of $C$, consistent with
the boundary condition $c$ and the domains and ranges of the point labels.
\item There is an evaluation map $e: \cC(I; a, b) \to \mor(a, b)$ given by
composing the morphism labels of the points.
Note that we also need the * of *-1-category here in order to make all the morphisms point
the same way.
\item For $x \in \mor(a, b)$ let $\chi(x) \in \cC(I; a, b)$ be the field with a single
point (at some standard location) labeled by $x$.
Then the kernel of the evaluation map $U(I; a, b)$ is generated by things of the
form $y - \chi(e(y))$.
Thus we can, if we choose, restrict the blob twig labels to things of this form.
\end{itemize}

We want to show that $\bc_*(S^1)$ is homotopy equivalent to the
Hochschild complex of $C$.
In order to prove this we will need to extend the 
definition of the blob complex to allow points to also
be labeled by elements of $C$-$C$-bimodules.
(See Subsections \ref{moddecss} and \ref{ssec:spherecat} for a more general version of this construction that applies in all dimensions.)

Fix points $p_1, \ldots, p_k \in S^1$ and $C$-$C$-bimodules $M_1, \ldots M_k$.
We define a blob-like complex $K_*(S^1, (p_i), (M_i))$.
The fields have elements of $M_i$ labeling 
the fixed points $p_i$ and elements of $C$ labeling other (variable) points.
As before, the regions between the marked points are labeled by
objects of $C$.
The blob twig labels lie in kernels of evaluation maps.
(The range of these evaluation maps is a tensor product (over $C$) of $M_i$'s,
corresponding to the $p_i$'s that lie within the twig blob.)
Let $K_*(M) = K_*(S^1, (*), (M))$, where $* \in S^1$ is some standard base point.
In other words, fields for $K_*(M)$ have an element of $M$ at the fixed point $*$
and elements of $C$ at variable other points.

In the theorems, propositions and lemmas below we make various claims
about complexes being homotopy equivalent.
In all cases the complexes in question are free (and hence projective), 
so it suffices to show that they are quasi-isomorphic.

We claim that
\begin{thm}
\label{thm:hochschild}
The blob complex $\bc_*(S^1; C)$ on the circle is homotopy equivalent to the
usual Hochschild complex for $C$.
\end{thm}

This follows from two results.
First, we see that
\begin{lem}
\label{lem:module-blob}%
The complex $K_*(C)$ (here $C$ is being thought of as a
$C$-$C$-bimodule, not a category) is homotopy equivalent to the blob complex
$\bc_*(S^1; C)$.
\end{lem}
The proof appears below.

Next, we show that for any $C$-$C$-bimodule $M$,
\begin{prop} \label{prop:hoch}
The complex $K_*(M)$ is homotopy equivalent to $\HC_*(M)$, the usual
Hochschild complex of $M$.
\end{prop}
\begin{proof}
Recall that the usual Hochschild complex of $M$ is uniquely determined,
up to quasi-isomorphism, by the following properties:
\begin{enumerate}
\item \label{item:hochschild-additive}%
$\HC_*(M_1 \oplus M_2) \cong \HC_*(M_1) \oplus \HC_*(M_2)$.
\item \label{item:hochschild-exact}%
An exact sequence $0 \to M_1 \into M_2 \onto M_3 \to 0$ gives rise to an
exact sequence $0 \to \HC_*(M_1) \into \HC_*(M_2) \onto \HC_*(M_3) \to 0$.
\item \label{item:hochschild-coinvariants}%
$\HH_0(M)$ is isomorphic to the coinvariants of $M$, $\coinv(M) =
M/\langle cm-mc \rangle$.
\item \label{item:hochschild-free}%
$\HC_*(C\otimes C)$ is contractible.
(Here $C\otimes C$ denotes
the free $C$-$C$-bimodule with one generator.)
That is, $\HC_*(C\otimes C)$ is
quasi-isomorphic to its $0$-th homology (which in turn, by \ref{item:hochschild-coinvariants}
above, is just $C$) via the quotient map $\HC_0 \onto \HH_0$.
\end{enumerate}
(Together, these just say that Hochschild homology is ``the derived functor of coinvariants".)
We'll first recall why these properties are characteristic.

Take some $C$-$C$ bimodule $M$, and choose a free resolution
\begin{equation*}
\cdots \to F_2 \xrightarrow{f_2} F_1 \xrightarrow{f_1} F_0.
\end{equation*}
We will show that for any functor $\cP$ satisfying properties
\ref{item:hochschild-additive}, \ref{item:hochschild-exact},
\ref{item:hochschild-coinvariants} and \ref{item:hochschild-free}, there
is a quasi-isomorphism
$$\cP_*(M) \iso \coinv(F_*).$$
Observe that there's a quotient map $\pi: F_0 \onto M$, and by
construction the cone of the chain map $\pi: F_* \to M$ is acyclic. 
Now construct the total complex $\cP_i(F_j)$, with $i,j \geq 0$, graded by $i+j$. 
We have two chain maps
\begin{align*}
\cP_i(F_*) & \xrightarrow{\cP_i(\pi)} \cP_i(M) \\
\intertext{and}
\cP_*(F_j) & \xrightarrow{\cP_0(F_j) \onto H_0(\cP_*(F_j))} \coinv(F_j).
\end{align*}
The cone of each chain map is acyclic.
In the first case, this is because the ``rows" indexed by $i$ are acyclic since $\cP_i$ is exact.
In the second case, this is because the ``columns" indexed by $j$ are acyclic, since $F_j$ is free.
Because the cones are acyclic, the chain maps are quasi-isomorphisms.
Composing one with the inverse of the other, we obtain the desired quasi-isomorphism
$$\cP_*(M) \quismto \coinv(F_*).$$

%

Proposition \ref{prop:hoch} then follows from the following lemmas, 
establishing that $K_*$ has precisely these required properties.
\begin{lem}
\label{lem:hochschild-additive}%
Directly from the definition, $K_*(M_1 \oplus M_2) \cong K_*(M_1) \oplus K_*(M_2)$.
\end{lem}
\begin{lem}
\label{lem:hochschild-exact}%
An exact sequence $0 \to M_1 \into M_2 \onto M_3 \to 0$ gives rise to an
exact sequence $0 \to K_*(M_1) \into K_*(M_2) \onto K_*(M_3) \to 0$.
\end{lem}
\begin{lem}
\label{lem:hochschild-coinvariants}%
$H_0(K_*(M))$ is isomorphic to the coinvariants of $M$.
\end{lem}
\begin{lem}
\label{lem:hochschild-free}%
$K_*(C\otimes C)$ is quasi-isomorphic to $H_0(K_*(C \otimes C)) \iso C$.
\end{lem}

The remainder of this section is devoted to proving Lemmas
\ref{lem:module-blob},
\ref{lem:hochschild-exact}, \ref{lem:hochschild-coinvariants} and
\ref{lem:hochschild-free}.
\end{proof}

\subsection{Technical details}
\begin{proof}[Proof of Lemma \ref{lem:module-blob}]
We show that $K_*(C)$ is quasi-isomorphic to $\bc_*(S^1)$.
$K_*(C)$ differs from $\bc_*(S^1)$ only in that the base point *
is always a labeled point in $K_*(C)$, while in $\bc_*(S^1)$ it may or may not be.
In particular, there is an inclusion map $i: K_*(C) \to \bc_*(S^1)$.

We want to define a homotopy inverse to the above inclusion, but before doing so
we must replace $\bc_*(S^1)$ with a homotopy equivalent subcomplex.
Let $J_* \sub \bc_*(S^1)$ be the subcomplex where * does not lie on the boundary
of any blob.
Note that the image of $i$ is contained in $J_*$.
Note also that in $\bc_*(S^1)$ (away from $J_*$) 
a blob diagram could have multiple (nested) blobs whose
boundaries contain *, on both the right and left of *.

We claim that $J_*$ is homotopy equivalent to $\bc_*(S^1)$.
Let $F_*^\ep \sub \bc_*(S^1)$ be the subcomplex where either
(a) the point * is not on the boundary of any blob or
(b) there are no labeled points or blob boundaries within distance $\ep$ of *,
other than blob boundaries at * itself.
Note that all blob diagrams are in some $F_*^\ep$ for $\ep$ sufficiently small.
Let $b$ be a blob diagram in $F_*^\ep$.
Define $f(b)$ to be the result of moving any blob boundary points which lie on *
to distance $\ep$ from *.
(Move right or left so as to shrink the blob.)
Extend to get a chain map $f: F_*^\ep \to F_*^\ep$.
By Corollary \ref{disj-union-contract}, 
$f$ is homotopic to the identity.
(Use the facts that $f$ factors though a map from a disjoint union of balls
into $S^1$, and that $f$ is the identity in degree 0.)
Since the image of $f$ is in $J_*$, and since any blob chain is in $F_*^\ep$
for $\ep$ sufficiently small, we have that $J_*$ is homotopic to all of $\bc_*(S^1)$.

We now define a homotopy inverse $s: J_* \to K_*(C)$ to the inclusion $i$.
If $y$ is a field defined on a neighborhood of *, define $s(y) = y$ if
* is a labeled point in $y$.
Otherwise, define $s(y)$ to be the result of adding a label 1 (identity morphism) at *.
Extending linearly, we get the desired map $s: J_* \to K_*(C)$.
It is easy to check that $s$ is a chain map and $s \circ i = \id$.
What remains is to show that $i \circ s$ is homotopic to the identity.

Let $N_\ep$ denote the ball of radius $\ep$ around *.
Let $L_*^\ep \sub J_*$ be the subcomplex 
spanned by blob diagrams
where there are no labeled points
in $N_\ep$, except perhaps $*$, and $N_\ep$ is either disjoint from or contained in 
every blob in the diagram.
Note that for any chain $x \in J_*$, $x \in L_*^\ep$ for sufficiently small $\ep$.

We define a degree $1$ map $j_\ep: L_*^\ep \to L_*^\ep$ as follows.
Let $x \in L_*^\ep$ be a blob diagram.
If $*$ is not contained in any twig blob, we define $j_\ep(x)$ by adding 
$N_\ep$ as a new twig blob, with label $y - s(y)$ where $y$ is the restriction
of $x$ to $N_\ep$.
If $*$ is contained in a twig blob $B$ with label $u=\sum z_i$, 
write $y_i$ for the restriction of $z_i$ to $N_\ep$, and let
$x_i$ be equal to $x$ on $S^1 \setmin B$, equal to $z_i$ on $B \setmin N_\ep$,
and have an additional blob $N_\ep$ with label $y_i - s(y_i)$.
Define $j_\ep(x) = \sum x_i$.

It is not hard to show that on $L_*^\ep$
\[
	\bd j_\ep  + j_\ep \bd = \id - i \circ s .
\]
(To get the signs correct here, we add $N_\ep$ as the first blob.)
Since for $\ep$ small enough $L_*^\ep$ captures all of the
homology of $J_*$, 
it follows that the mapping cone of $i \circ s$ is acyclic and therefore (using the fact that
these complexes are free) $i \circ s$ is homotopic to the identity.
\end{proof}

\begin{proof}[Proof of Lemma \ref{lem:hochschild-exact}]
We now prove that $K_*$ is an exact functor.

As a warm-up, we prove
that the functor on $C$-$C$ bimodules
\begin{equation*}
M \mapsto \ker(C \tensor M \tensor C \xrightarrow{c_1 \tensor m \tensor c_2 \mapsto c_1 m c_2} M)
\end{equation*}
is exact.
Suppose we have a short exact sequence of $C$-$C$ bimodules $$\xymatrix{0 \ar[r] & K \ar@{^{(}->}[r]^f & E \ar@{->>}[r]^g & Q \ar[r] & 0}.$$
We'll write $\hat{f}$ and $\hat{g}$ for the image of $f$ and $g$ under the functor, so 
\[
	\hat{f}(\textstyle\sum_i a_i \tensor k_i \tensor b_i) = 
						\textstyle\sum_i a_i \tensor f(k_i) \tensor b_i ,
\]
and similarly for $\hat{g}$.
Most of what we need to check is easy.
Suppose we have $\sum_i (a_i \tensor k_i \tensor b_i) \in \ker(C \tensor K \tensor C \to K)$, 
assuming without loss of generality that $\{a_i \tensor b_i\}_i$ is linearly independent in $C \tensor C$, 
and $\hat{f}(a \tensor k \tensor b) = 0 \in \ker(C \tensor E \tensor C \to E)$.
We must then have $f(k_i) = 0 \in E$ for each $i$, which implies $k_i=0$ itself. 
If $\sum_i (a_i \tensor e_i \tensor b_i) \in \ker(C \tensor E \tensor C \to E)$ 
is in the image of $\ker(C \tensor K \tensor C \to K)$ under $\hat{f}$, 
again by assuming the set  $\{a_i \tensor b_i\}_i$ is linearly independent we can deduce that each
$e_i$ is in the image of the original $f$, and so is in the kernel of the original $g$, 
and so $\hat{g}(\sum_i a_i \tensor e_i \tensor b_i) = 0$.
If $\hat{g}(\sum_i a_i \tensor e_i \tensor b_i) = 0$, then each $g(e_i) = 0$, so $e_i = f(\widetilde{e_i})$ 
for some $\widetilde{e_i} \in K$, and $\sum_i a_i \tensor e_i \tensor b_i = \hat{f}(\sum_i a_i \tensor \widetilde{e_i} \tensor b_i)$.
Finally, the interesting step is in checking that any $q = \sum_i a_i \tensor q_i \tensor b_i$ 
such that $\sum_i a_i q_i b_i = 0$ is in the image of $\ker(C \tensor E \tensor C \to C)$ under $\hat{g}$.
For each $i$, we can find $\widetilde{q_i}$ so $g(\widetilde{q_i}) = q_i$.
However $\sum_i a_i \widetilde{q_i} b_i$ need not be zero.
Consider then $$\widetilde{q} = \sum_i \left(a_i \tensor \widetilde{q_i} \tensor b_i\right) - 1 \tensor \left(\sum_i a_i \widetilde{q_i} b_i\right) \tensor 1.$$ Certainly
$\widetilde{q} \in \ker(C \tensor E \tensor C \to E)$.
Further,
\begin{align*}
\hat{g}(\widetilde{q}) & = \sum_i \left(a_i \tensor g(\widetilde{q_i}) \tensor b_i\right) - 1 \tensor \left(\sum_i a_i g(\widetilde{q_i}) b_i\right) \tensor 1 \\
                       & = q - 0
\end{align*}
(here we used that $g$ is a map of $C$-$C$ bimodules, and that $\sum_i a_i q_i b_i = 0$).

Similar arguments show that the functors
\begin{equation}
\label{eq:ker-functor}%
M \mapsto \ker(C^{\tensor k} \tensor M \tensor C^{\tensor l} \to M)
\end{equation}
are all exact too.
Moreover, tensor products of such functors with each
other and with $C$ or $\ker(C^{\tensor k} \to C)$ (e.g., producing the functor $M \mapsto \ker(M \tensor C \to M)
\tensor C \tensor \ker(C \tensor C \to M)$) are all still exact.

Finally, then we see that the functor $K_*$ is simply an (infinite)
direct sum of copies of this sort of functor.
The direct sum is indexed by
configurations of nested blobs and of labels; for each such configuration, we have one of 
the above tensor product functors,
with the labels of twig blobs corresponding to tensor factors as in \eqref{eq:ker-functor} 
or $\ker(C^{\tensor k} \to C)$ (depending on whether they contain a marked point $p_i$), and all other labelled points corresponding
to tensor factors of $C$ and $M$.
\end{proof}
\begin{proof}[Proof of Lemma \ref{lem:hochschild-coinvariants}]
We show that $H_0(K_*(M))$ is isomorphic to the coinvariants of $M$.

We define a map $\ev: K_0(M) \to M$.
If $x \in K_0(M)$ has the label $m \in M$ at $*$, and labels $c_i \in C$ at the other 
labeled points of $S^1$, reading clockwise from $*$,
we set $\ev(x) = m c_1 \cdots c_k$.
We can think of this as $\ev : M \tensor C^{\tensor k} \to M$, for each direct summand of 
$K_0(M)$ indexed by a configuration of labeled points.

There is a quotient map $\pi: M \to \coinv{M}$.
We claim that the composition $\pi \compose \ev$ is well-defined on the quotient $H_0(K_*(M))$; 
i.e.\ that $\pi(\ev(\bd y)) = 0$ for all $y \in K_1(M)$.
There are two cases, depending on whether the blob of $y$ contains the point *.
If it doesn't, then
suppose $y$ has label $m$ at $*$, labels $c_i$ at other labeled points outside the blob, 
and the field inside the blob is a sum, with the $j$-th term having
labeled points $d_{j,i}$.
Then $\sum_j d_{j,1} \tensor \cdots \tensor d_{j,k_j} \in \ker(\DirectSum_k C^{\tensor k} \to C)$, and so
$\ev(\bdy y) = 0$, because $$C^{\tensor \ell_1} \tensor \ker(\DirectSum_k C^{\tensor k} \to C) \tensor C^{\tensor \ell_2} \subset \ker(\DirectSum_k C^{\tensor k} \to C).$$
Similarly, if $*$ is contained in the blob, then the blob label is a sum, with the 
$j$-th term have labelled points $d_{j,i}$ to the left of $*$, $m_j$ at $*$, and $d_{j,i}'$ to the right of $*$,
and there are labels $c_i$ at the labeled points outside the blob.
We know that
$$\sum_j d_{j,1} \tensor \cdots \tensor d_{j,k_j} \tensor m_j \tensor d_{j,1}' \tensor \cdots \tensor d_{j,k'_j}' \in \ker(\DirectSum_{k,k'} C^{\tensor k} \tensor M \tensor C^{\tensor k'} \to M),$$
and so
\begin{align*}
\pi\left(\ev(\bdy y)\right) & = \pi\left(\sum_j m_j d_{j,1}' \cdots d_{j,k'_j}' c_1 \cdots c_k d_{j,1} \cdots d_{j,k_j}\right) \\
            & = \pi\left(\sum_j d_{j,1} \cdots d_{j,k_j} m_j d_{j,1}' \cdots d_{j,k'_j}' c_1 \cdots c_k\right) \\
            & = 0
\end{align*}
where this time we use the fact that we're mapping to $\coinv{M}$, not just $M$.

The map $\pi \compose \ev: H_0(K_*(M)) \to \coinv{M}$ is clearly 
surjective ($\ev$ surjects onto $M$); we now show that it's injective.
This is equivalent to showing that 
\[
	\ev\inv(\ker(\pi)) \sub \bd K_1(M) .
\]
The above inclusion follows from
\[
	\ker(\ev) \sub \bd K_1(M)
\]
and
\[
	\ker(\pi) \sub \ev(\bd K_1(M)) .
\]
Let $x = \sum x_i$ be in the kernel of $\ev$, where each $x_i$ is a configuration of 
labeled points in $S^1$.
Since the sum is finite, we can find an interval (blob) $B$ in $S^1$
such that for each $i$ the $C$-labeled points of $x_i$ all lie to the right of the 
base point *.
Let $y_i$ be the restriction of $x_i$ to $B$ and $y = \sum y_i$.
Let $r$ be the ``empty" field on $S^1 \setmin B$.
It follows that $y \in U(B)$ and 
\[
	\bd(B, y, r) = x .
\]
$\ker(\pi)$ is generated by elements of the form $cm - mc$.
As shown in Figure \ref{fig:hochschild-1-chains}, $cm - mc$ lies in $\ev(\bd K_1(M))$.
\end{proof}

\begin{proof}[Proof of Lemma \ref{lem:hochschild-free}]
We show that $K_*(C\otimes C)$ is
quasi-isomorphic to the 0-step complex $C$.
We'll do this in steps, establishing quasi-isomorphisms and homotopy equivalences
$$K_*(C \tensor C) \quismto K'_* \htpyto K''_* \quismto C.$$

Let $K'_* \sub K_*(C\otimes C)$ be the subcomplex where the label of
the point $*$ is $1 \otimes 1 \in C\otimes C$.
We will show that the inclusion $i: K'_* \to K_*(C\otimes C)$ is a quasi-isomorphism.

Fix a small $\ep > 0$.
Let $N_\ep$ be the ball of radius $\ep$ around $* \in S^1$.
Let $K_*^\ep \sub K_*(C\otimes C)$ be the subcomplex
generated by blob diagrams $b$ such that $N_\ep$ is either disjoint from
or contained in each blob of $b$, and the only labeled point inside $N_\ep$ is $*$.
For a field $y$ on $N_\ep$, let $s_\ep(y)$ be the equivalent picture with~$*$
labeled by $1\otimes 1$ and the only other labeled points at distance $\pm\ep/2$ from $*$.
(See Figure \ref{fig:sy}.)
Note that $y - s_\ep(y) \in U(N_\ep)$. 
Let $\sigma_\ep: K_*^\ep \to K_*^\ep$ be the chain map
given by replacing the restriction $y$ to $N_\ep$ of each field
appearing in an element of  $K_*^\ep$ with $s_\ep(y)$.
Note that $\sigma_\ep(x) \in K'_*$.
\begin{figure}[t]
\begin{align*}
y & = \mathfig{0.2}{hochschild/y} &
s_\ep(y) & = \mathfig{0.2}{hochschild/sy}
\end{align*}
\caption{Defining $s_\ep$.}
\label{fig:sy}
\end{figure}

Define a degree 1 map $j_\ep : K_*^\ep \to K_*^\ep$ as follows.
Let $x \in K_*^\ep$ be a blob diagram.
If $*$ is not contained in any twig blob, $j_\ep(x)$ is obtained by adding $N_\ep$ to
$x$ as a new twig blob, with label $y - s_\ep(y)$, where $y$ is the restriction of $x$ to $N_\ep$.
If $*$ is contained in a twig blob $B$ with label $u = \sum z_i$, $j_\ep(x)$ is obtained as follows.
Let $y_i$ be the restriction of $z_i$ to $N_\ep$.
Let $x_i$ be equal to $x$ outside of $B$, equal to $z_i$ on $B \setmin N_\ep$,
and have an additional blob $N_\ep$ with label $y_i - s_\ep(y_i)$.
Define $j_\ep(x) = \sum x_i$.
Note that if $x \in K'_* \cap K_*^\ep$ then $j_\ep(x) \in K'_*$ also.

The key property of $j_\ep$ is
\eq{
    \bd j_\ep + j_\ep \bd = \id - \sigma_\ep.
}
(Again, to get the correct signs, $N_\ep$ must be added as the first blob.)
If $j_\ep$ were defined on all of $K_*(C\otimes C)$, this would show that $\sigma_\ep$
is a homotopy inverse to the inclusion $K'_* \to K_*(C\otimes C)$.
One strategy would be to try to stitch together various $j_\ep$ for progressively smaller
$\ep$ and show that $K'_*$ is homotopy equivalent to $K_*(C\otimes C)$.
Instead, we'll be less ambitious and just show that
$K'_*$ is quasi-isomorphic to $K_*(C\otimes C)$.

If $x$ is a cycle in $K_*(C\otimes C)$, then for sufficiently small $\ep$ we have
$x \in K_*^\ep$.
(This is true for any chain in $K_*(C\otimes C)$, since chains are sums of
finitely many blob diagrams.)
Then $x$ is homologous to $\sigma_\ep(x)$, which is in $K'_*$, so the inclusion map
$K'_* \sub K_*(C\otimes C)$ is surjective on homology.
If $y \in K_*(C\otimes C)$ and $\bd y = x \in K_*(C\otimes C)$, then $y \in K_*^\ep$ for some $\ep$
and
\eq{
    \bd y = \bd (\sigma_\ep(y) + j_\ep(x)) .
}
Since $\sigma_\ep(y) + j_\ep(x) \in K'_*$, it follows that the inclusion map is injective on homology.
This completes the proof that $K'_*$ is quasi-isomorphic to $K_*(C\otimes C)$.

Let $K''_* \sub K'_*$ be the subcomplex of $K'_*$ where $*$ is not contained in any blob.
We will show that the inclusion $i: K''_* \to K'_*$ is a homotopy equivalence.

First, a lemma:  Let $G''_*$ and $G'_*$ be defined similarly to $K''_*$ and $K'_*$, except with
$S^1$ replaced by some neighborhood $N$ of $* \in S^1$.
($G''_*$ and $G'_*$ depend on $N$, but that is not reflected in the notation.)
Then $G''_*$ and $G'_*$ are both contractible
and the inclusion $G''_* \sub G'_*$ is a homotopy equivalence.
For $G'_*$ the proof is the same as in Lemma \ref{bcontract}, except that the splitting
$G'_0 \to H_0(G'_*)$ concentrates the point labels at two points to the right and left of $*$.
For $G''_*$ we note that any cycle is supported away from $*$.
Thus any cycle lies in the image of the normal blob complex of a disjoint union
of two intervals, which is contractible by Lemma \ref{bcontract} and Corollary \ref{disj-union-contract}.
Finally, it is easy to see that the inclusion
$G''_* \to G'_*$ induces an isomorphism on $H_0$.

Next we construct a degree 1 map (homotopy) $h: K'_* \to K'_*$ such that
for all $x \in K'_*$ we have
\eq{
    x - \bd h(x) - h(\bd x) \in K''_* .
}
Since $K'_0 = K''_0$, we can take $h_0 = 0$.
Let $x \in K'_1$, with single blob $B \sub S^1$.
If $* \notin B$, then $x \in K''_1$ and we define $h_1(x) = 0$.
If $* \in B$, then we work in the image of $G'_*$ and $G''_*$ (with $B$ playing the role of $N$ above).
Choose $x'' \in G''_1$ such that $\bd x'' = \bd x$.
Since $G'_*$ is contractible, there exists $y \in G'_2$ such that $\bd y = x - x''$.
Define $h_1(x) = y$.
The general case is similar, except that we have to take lower order homotopies into account.
Let $x \in K'_k$.
If $*$ is not contained in any of the blobs of $x$, then define $h_k(x) = 0$.
Otherwise, let $B$ be the outermost blob of $x$ containing $*$.
We can decompose $x = x' \bullet p$, 
where $x'$ is supported on $B$ and $p$ is supported away from $B$.
So $x' \in G'_l$ for some $l \le k$.
Choose $x'' \in G''_l$ such that $\bd x'' = \bd (x' - h_{l-1}\bd x')$.
Choose $y \in G'_{l+1}$ such that $\bd y = x' - x'' - h_{l-1}\bd x'$.
Define $h_k(x) = y \bullet p$.
This completes the proof that $i: K''_* \to K'_*$ is a homotopy equivalence.

Finally, we show that $K''_*$ is contractible with $H_0\cong C$.
This is similar to the proof of Proposition \ref{bcontract}, but a bit more
complicated since there is no single blob which contains the support of all blob diagrams
in $K''_*$.
Let $x$ be a cycle of degree greater than zero in $K''_*$.
The union of the supports of the diagrams in $x$ does not contain $*$, so there exists a
ball $B \subset S^1$ containing the union of the supports and not containing $*$.
Adding $B$ as an outermost blob to each summand of $x$ gives a chain $y$ with $\bd y = x$.
Thus $H_i(K''_*) \cong 0$ for $i> 0$ and $K''_*$ is contractible.

To see that $H_0(K''_*) \cong C$, consider the map $p: K''_0 \to C$ which sends a 0-blob
diagram to the product of its labeled points.
$p$ is clearly surjective.
It's also easy to see that $p(\bd K''_1) = 0$.
Finally, if $p(y) = 0$ then there exists a blob $B \sub S^1$ which contains
all of the labeled points (other than *) of all of the summands of $y$.
This allows us to construct $x\in K''_1$ such that $\bd x = y$.
(The label of $B$ is the restriction of $y$ to $B$.)
It follows that $H_0(K''_*) \cong C$.
\end{proof}

\subsection{An explicit chain map in low degrees}

For purposes of illustration, we describe an explicit chain map
$\HC_*(M) \to K_*(M)$
between the Hochschild complex and the blob complex (with bimodule point)
for degree $\le 2$.
This map can be completed to a homotopy equivalence, though we will not prove that here.
There are of course many such maps; what we describe here is one of the simpler possibilities.

Recall that in low degrees $\HC_*(M)$ is
\[
	\cdots \stackrel{\bd}{\to} M \otimes C\otimes C \stackrel{\bd}{\to} 
			M \otimes C \stackrel{\bd}{\to} M
\]
with
\eqar{
	\bd(m\otimes a)  & = & ma - am \\
	\bd(m\otimes a \otimes b) & = & ma\otimes b - m\otimes ab + bm \otimes a .
}
In degree 0, we send $m\in M$ to the 0-blob diagram $\mathfig{0.04}{hochschild/0-chains}$; the base point
in $S^1$ is labeled by $m$ and there are no other labeled points.
In degree 1, we send $m\ot a$ to the sum of two 1-blob diagrams
as shown in Figure \ref{fig:hochschild-1-chains}.

\begin{figure}[t]
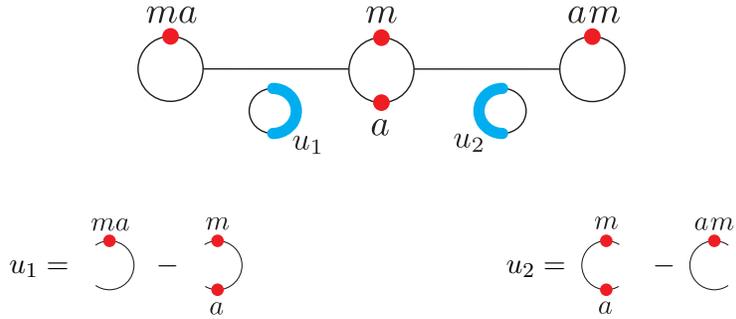

\begin{equation*}
\mathfig{0.4}{hochschild/1-chains}
\end{equation*}
\begin{align*}
u_1 & = \mathfig{0.05}{hochschild/u_1-1} - \mathfig{0.05}{hochschild/u_1-2} & u_2 & = \mathfig{0.05}{hochschild/u_2-1} - \mathfig{0.05}{hochschild/u_2-2} 
\end{align*}
\caption{The image of $m \tensor a$ in the blob complex.}
\label{fig:hochschild-1-chains}
\end{figure}

\begin{figure}[t]
\begin{equation*}
\mathfig{0.6}{hochschild/2-chains-0}
\end{equation*}
\caption{The 0-chains in the image of $m \tensor a \tensor b$.}
\label{fig:hochschild-2-chains-0}
\end{figure}
\begin{figure}[t]
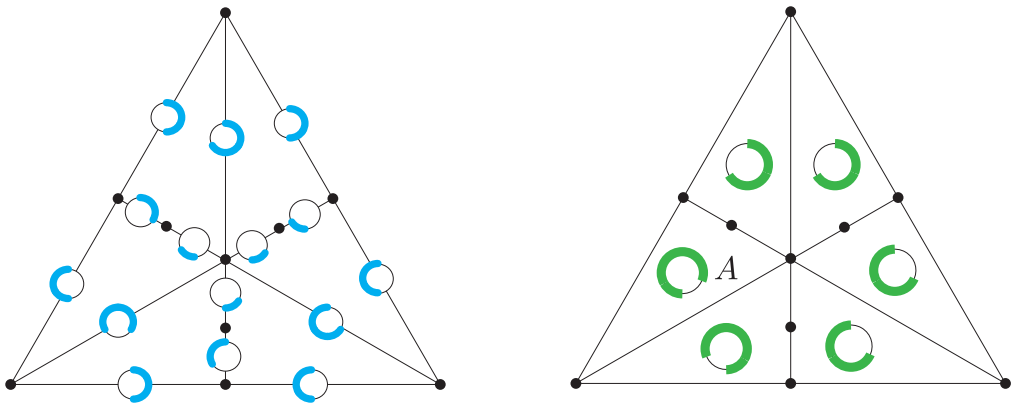

\begin{equation*}
\mathfig{0.4}{hochschild/2-chains-1} \qquad \mathfig{0.4}{hochschild/2-chains-2}
\end{equation*}
\caption{The 1- and 2-chains in the image of $m \tensor a \tensor b$.
Only the supports of the blobs are shown, but see Figure \ref{fig:hochschild-example-2-cell} for an example of a $2$-cell label.}
\label{fig:hochschild-2-chains-12}
\end{figure}

\begin{figure}[t]
\begin{equation*}
A = \mathfig{0.1}{hochschild/v_1} + \mathfig{0.1}{hochschild/v_2} + \mathfig{0.1}{hochschild/v_3} + \mathfig{0.1}{hochschild/v_4}
\end{equation*}
\begin{align*}
v_1 & = \mathfig{0.05}{hochschild/v_1-1} -  \mathfig{0.05}{hochschild/v_1-2} &  v_2 & = \mathfig{0.05}{hochschild/v_2-1} -  \mathfig{0.05}{hochschild/v_2-2} \\ 
v_3 & = \mathfig{0.05}{hochschild/v_3-1} -  \mathfig{0.05}{hochschild/v_3-2} &  v_4 & = \mathfig{0.05}{hochschild/v_4-1} -  \mathfig{0.05}{hochschild/v_4-2}
\end{align*}
\caption{One of the 2-cells from Figure \ref{fig:hochschild-2-chains-12}.}
\label{fig:hochschild-example-2-cell}
\end{figure}

In degree 2, we send $m\ot a \ot b$ to the sum of 24 ($=6\cdot4$) 2-blob diagrams as shown in
Figures \ref{fig:hochschild-2-chains-0} and \ref{fig:hochschild-2-chains-12}.
In Figure \ref{fig:hochschild-2-chains-12} the 1- and 2-blob diagrams are indicated only by their support.
We leave it to the reader to determine the labels of the 1-blob diagrams.
Each 2-cell in the figure is labeled by a ball $V$ in $S^1$ which contains the support of all
1-blob diagrams in its boundary.
Such a 2-cell corresponds to a sum of the 2-blob diagrams obtained by adding $V$
as an outer (non-twig) blob to each of the 1-blob diagrams in the boundary of the 2-cell.
Figure \ref{fig:hochschild-example-2-cell} shows this explicitly for the 2-cell
labeled $A$ in Figure \ref{fig:hochschild-2-chains-12}.
Note that the (blob complex) boundary of this sum of 2-blob diagrams is
precisely the sum of the 1-blob diagrams corresponding to the boundary of the 2-cell.
(Compare with the proof of \ref{bcontract}.)


\section{Action of \texorpdfstring{$\CH{X}$}{C*(Homeo(M))}}
\label{sec:evaluation}

In this section we extend the action of homeomorphisms on $\bc_*(X)$
to an action of {\it families} of homeomorphisms.
That is, for each pair of homeomorphic manifolds $X$ and $Y$
we define a chain map
\[
    e_{XY} : \CH{X, Y} \otimes \bc_*(X) \to \bc_*(Y) ,
\]
where $C_*(\Homeo(X, Y))$ is the singular chains on the space
of homeomorphisms from $X$ to $Y$.
(If $X$ and $Y$ have non-empty boundary, these families of homeomorphisms
are required to restrict to a fixed homeomorphism on the boundaries.)
These actions (for various $X$ and $Y$) are compatible with gluing.
See \S \ref{ss:emap-def} for a more precise statement.

The most convenient way to prove that maps $e_{XY}$ with the desired properties exist is to 
introduce a homotopy equivalent alternate version of the blob complex, $\btc_*(X)$,
which is more amenable to this sort of action.
Recall from Remark \ref{blobsset-remark} that blob diagrams
have the structure of a cone-product set.
Blob diagrams can also be equipped with a natural topology, which converts this
cone-product set into a cone-product space.
Taking singular chains of this space we get $\btc_*(X)$.
The details are in \S \ref{ss:alt-def}.
We also prove a useful result (Lemma \ref{small-blobs-b}) which says that we can assume that
blobs are small with respect to any fixed open cover.

%

\subsection{Alternative definitions of the blob complex}
\label{ss:alt-def}

\newcommand\sbc{\bc^{\cU}}

In this subsection we define a subcomplex (small blobs) and supercomplex (families of blobs)
of the blob complex, and show that they are both homotopy equivalent to $\bc_*(X)$.

\medskip

If $b$ is a blob diagram in $\bc_*(X)$, recall from \S \ref{sec:basic-properties} that the {\it support} of $b$, denoted
$\supp(b)$ or $|b|$, is the union of the blobs of $b$.
More generally, we say that a chain $a\in \bc_k(X)$ is supported on $S$ if
$a = a'\bullet r$, where $a'\in \bc_k(S)$ and $r\in \bc_0(X\setmin S)$.

Similarly, if $f: P\times X\to X$ is a family of homeomorphisms and $Y\sub X$, we say that $f$ is 
{\it supported on $Y$} if $f(p, x) = f(p', x)$ for all $x\in X\setmin Y$ and all $p,p'\in P$.
We will sometimes abuse language and talk about ``the" support of $f$,
again denoted $\supp(f)$ or $|f|$, to mean some particular choice of $Y$ such that
$f$ is supported on $Y$.

If $f: M \cup (Y\times I) \to M$ is a collaring homeomorphism
(cf.\ the end of \S \ref{ss:syst-o-fields}),
we say that $f$ is supported on $S\sub M$ if $f(x) = x$ for all $x\in M\setmin S$.

\medskip

Fix $\cU$, an open cover of $X$.
Define the ``small blob complex" $\bc^{\cU}_*(X)$ to be the subcomplex of $\bc_*(X)$ 
generated by blob diagrams such that every blob is contained in some open set of $\cU$, 
and moreover each field labeling a region cut out by the blobs is splittable 
into fields on smaller regions, each of which is contained in some open set of $\cU$.

\begin{lemma}[Small blobs] \label{small-blobs-b}  \label{thm:small-blobs}
The inclusion $i: \bc^{\cU}_*(X) \into \bc_*(X)$ is a homotopy equivalence.
\end{lemma}

\begin{proof}
Since both complexes are free, it suffices to show that the inclusion induces
an isomorphism of homotopy groups.
To show this it in turn suffices to show that for any finitely generated 
pair $(C_*, D_*)$, with $D_*$ a subcomplex of $C_*$ such that 
\[
	(C_*, D_*) \sub (\bc_*(X), \sbc_*(X))
\]
we can find a homotopy $h:C_*\to \bc_*(X)$ such that $h(D_*) \sub \sbc_*(X)$
and
\[
	h\bd(x) + \bd h(x) + x \in \sbc_*(X)
\]
for all $x\in C_*$.

By the splittings axiom for fields, any field is splittable into small pieces.
It follows that $\sbc_0(X) = \bc_0(X)$.
Accordingly, we define $h_0 = 0$.

Next we define $h_1$.
Let $b\in C_1$ be a 1-blob diagram.
Let $B$ be the blob of $b$.
We will construct a 1-chain $s(b)\in \sbc_1(X)$ such that $\bd(s(b)) = \bd b$
and the support of $s(b)$ is contained in $B$.
(If $B$ is not embedded in $X$, then we implicitly work in some stage of a decomposition
of $X$ where $B$ is embedded.
See Definition \ref{defn:configuration} and preceding discussion.)
It then follows from Corollary \ref{disj-union-contract} that we can choose
$h_1(b) \in \bc_2(X)$ such that $\bd(h_1(b)) = s(b) - b$.

Roughly speaking, $s(b)$ consists of a series of 1-blob diagrams implementing a series
of small collar maps, plus a shrunken version of $b$.
The composition of all the collar maps shrinks $B$ to a ball which is small with respect to $\cU$.

Let $\cV_1$ be an auxiliary open cover of $X$, subordinate to $\cU$ and 
fine enough that a condition stated later in this proof is satisfied.
Let $b = (B, u, r)$, with $u = \sum a_i$ the label of $B$, and $a_i\in \bc_0(B)$.
Choose a sequence of collar maps $\bar{f}_j:B\cup\text{collar}\to B$ satisfying conditions 
specified at the end of this paragraph.
Let $f_j:B\to B$ be the restriction of $\bar{f}_j$ to $B$; $f_j$ maps $B$ homeomorphically to 
a slightly smaller submanifold of $B$.
Let $g_j = f_1\circ f_2\circ\cdots\circ f_j$.
Let $g$ be the last of the $g_j$'s.
Choose the sequence $\bar{f}_j$ so that 
$g(B)$ is contained in an open set of $\cV_1$ and
$g_{j-1}(|f_j|)$ is also contained in an open set of $\cV_1$.

There are 1-blob diagrams $c_{ij} \in \bc_1(B)$ such that $c_{ij}$ is compatible with $\cV_1$
(more specifically, $|c_{ij}| = g_{j-1}(|f_j|)$)
and $\bd c_{ij} = g_{j-1}(a_i) - g_{j}(a_i)$.
Define
\[
	s(b) = \sum_{i,j} c_{ij} + g(b)
\]
and choose $h_1(b) \in \bc_2(X)$ such that 
\[
	\bd(h_1(b)) = s(b) - b .
\]

Next we define $h_2$.
Let $b\in C_2$ be a 2-blob diagram.
Let $B = |b|$, either a ball or a union of two balls.
By possibly working in a decomposition of $X$, we may assume that the ball(s)
of $B$ are disjointly embedded.
We will construct a 2-chain $s(b)\in \sbc_2(X)$ such that
\[
	\bd(s(b)) = \bd(h_1(\bd b) + b) = s(\bd b)
\]
and the support of $s(b)$ is contained in $B$.
It then follows from Corollary \ref{disj-union-contract} that we can choose
$h_2(b) \in \bc_2(X)$ such that $\bd(h_2(b)) = s(b) - b - h_1(\bd b)$.

Similarly to the construction of $h_1$ above, 
$s(b)$ consists of a series of 2-blob diagrams implementing a series
of small collar maps, plus a shrunken version of $b$.
The composition of all the collar maps shrinks $B$ to a sufficiently small 
disjoint union of balls.

Let $\cV_2$ be an auxiliary open cover of $X$, subordinate to $\cU$ and 
fine enough that a condition stated later in the proof is satisfied.
As before, choose a sequence of collar maps $f_j$ 
such that each has support
contained in an open set of $\cV_1$ and the composition of the corresponding collar homeomorphisms
yields an embedding $g:B\to B$ such that $g(B)$ is contained in an open set of $\cV_1$.
Let $g_j:B\to B$ be the embedding at the $j$-th stage.

Fix $j$.
We will construct a 2-chain $d_j$ such that $\bd d_j = g_{j-1}(s(\bd b)) - g_{j}(s(\bd b))$.
Let $s(\bd b) = \sum e_k$, and let $\{p_m\}$ be the 0-blob diagrams
appearing in the boundaries of the $e_k$.
As in the construction of $h_1$, we can choose 1-blob diagrams $q_m$ such that
$\bd q_m = g_{j-1}(p_m) - g_j(p_m)$ and $|q_m|$ is contained in an open set of $\cV_1$.
If $x$ is a sum of $p_m$'s, we denote the corresponding sum of $q_m$'s by $q(x)$.

Now consider, for each $k$, $g_{j-1}(e_k) - q(\bd e_k)$.
This is a 1-chain whose boundary is $g_j(\bd e_k)$.
The support of $e_k$ is $g_{j-1}(V)$ for some $V\in \cV_1$, and
the support of $q(\bd e_k)$ is contained in a union $V'$ of finitely many open sets
of $\cV_1$, all of which contain the support of $f_j$.
We now reveal the mysterious condition (mentioned above) which $\cV_1$ satisfies:
the union of $g_{j-1}(V)$ and $V'$, for all of the finitely many instances
arising in the construction of $h_2$, lies inside a disjoint union of balls $U$
such that each individual ball lies in an open set of $\cV_2$.
(In this case there are either one or two balls in the disjoint union.)
For any fixed open cover $\cV_2$ this condition can be satisfied by choosing $\cV_1$ 
to be a sufficiently fine cover.
It follows from Corollary \ref{disj-union-contract} that we can choose 
$x_k \in \bc_2(X)$ with $\bd x_k = g_{j-1}(e_k) - g_j(e_k) - q(\bd e_k)$
and with $\supp(x_k) = U$.
We can now take $d_j \deq \sum x_k$.
It is clear that $\bd d_j = \sum (g_{j-1}(e_k) - g_j(e_k)) = g_{j-1}(s(\bd b)) - g_{j}(s(\bd b))$, as desired.

We now define 
\[
	s(b) = \sum d_j + g(b),
\]
where $g$ is the composition of all the $f_j$'s.
It is easy to verify that $s(b) \in \sbc_2$, $\supp(s(b)) = \supp(b)$, and 
$\bd(s(b)) = s(\bd b)$.
If follows that we can choose $h_2(b)\in \bc_2(X)$ such that $\bd(h_2(b)) = s(b) - b - h_1(\bd b)$.
This completes the definition of $h_2$.

The general case $h_l$ is similar.
When constructing the analogue of $x_k$ above, we will need to find a disjoint union of balls $U$
which contains finitely many open sets from $\cV_{l-1}$
such that each ball is contained in some open set of $\cV_l$.
For sufficiently fine $\cV_{l-1}$ this will be possible.
Since $C_*$ is finite, the process terminates after finitely many, say $r$, steps.
We take $\cV_r = \cU$.
\end{proof}

\medskip

Next we define the cone-product space version of the blob complex, $\btc_*(X)$.
First we must specify a topology on the set of $k$-blob diagrams, $\BD_k$.
We give $\BD_k$ the finest topology such that
\begin{itemize}
\item For any $b\in \BD_k$ the action map $\Homeo(X) \to \BD_k$, $f \mapsto f(b)$ is continuous.
\item The gluing maps $\BD_k(M)\to \BD_k(M\sgl)$ are continuous.
\item For balls $B$, the map $U(B) \to \BD_1(B)$, $u\mapsto (B, u, \emptyset)$, is continuous,
where $U(B) \sub \bc_0(B)$ inherits its topology from $\bc_0(B)$ and the topology on
$\bc_0(B)$ comes from the generating set $\BD_0(B)$.
\end{itemize}

We can summarize the above by saying that in the typical continuous family
$P\to \BD_k(X)$, $p\mapsto \left(B_i(p), u_i(p), r(p)\right)$, $B_i(p)$ and $r(p)$ are induced by a map
$P\to \Homeo(X)$, with the twig blob labels $u_i(p)$ varying independently. 
(``Varying independently'' means that after pulling back via the family of homeomorphisms to the original twig blob, 
one sees a continuous family of labels.)
We note that while we've decided not to allow the blobs $B_i(p)$ to vary independently of the field $r(p)$,
if we did allow this it would not affect the truth of the claims we make below.
In particular, such a definition of $\btc_*(X)$ would result in a homotopy equivalent complex.

Next we define $\btc_*(X)$ to be the total complex of the double complex (denoted $\btc_{**}$) 
whose $(i,j)$ entry is $C_j(\BD_i)$, the singular $j$-chains on the space of $i$-blob diagrams.
The vertical boundary of the double complex,
denoted $\bd_t$, is the singular boundary, and the horizontal boundary, denoted $\bd_b$, is
the blob boundary. Following the usual sign convention, we have $\bd = \bd_b + (-1)^i \bd_t$.

We will regard $\bc_*(X)$ as the subcomplex $\btc_{*0}(X) \sub \btc_{**}(X)$.
The main result of this subsection is

\begin{lemma} \label{lem:bc-btc}
The inclusion $\bc_*(X) \sub \btc_*(X)$ is a homotopy equivalence
\end{lemma}

Before giving the proof we need a few preliminary results.

\begin{lemma} \label{bt-contract}
$\btc_*(B^n)$ is contractible (acyclic in positive degrees).
\end{lemma}
\begin{proof}
We will construct a contracting homotopy $h: \btc_*(B^n)\to \btc_{*+1}(B^n)$.

We will assume a splitting $s:H_0(\btc_*(B^n))\to \btc_0(B^n)$
of the quotient map $q:\btc_0(B^n)\to H_0(\btc_*(B^n))$.
Let $\rho = s\circ q$.

For $x\in \btc_{ij}$ with $i\ge 1$ define
\[
	h(x) = e(x) ,
\]
where
\[
	e: \btc_{ij}\to\btc_{i+1,j}
\]
adds an outermost blob, equal to all of $B^n$, to the $j$-parameter family of blob diagrams.
Note that for fixed $i$, $e$ is a chain map, i.e. $\bd_t e = e \bd_t$.

A generator $y\in \btc_{0j}$ is a map $y:P\to \BD_0$, where $P$ is some $j$-dimensional polyhedron.
We define $r(y)\in \btc_{0j}$ to be the constant function $\rho\circ y : P\to \BD_0$. 
Let $c(r(y))\in \btc_{0,j+1}$ be the constant map from the cone of $P$ to $\BD_0$ taking
the same value (namely $r(y(p))$, for any $p\in P$).
Let $e(y - r(y)) \in \btc_{1j}$ denote the $j$-parameter family of 1-blob diagrams
whose value at $p\in P$ is the blob $B^n$ with label $y(p) - r(y(p))$.
Now define, for $y\in \btc_{0j}$,
\[
	h(y) = e(y - r(y)) - c(r(y)) .
\]

We must now verify that $h$ does the job it was intended to do.
For $x\in \btc_{ij}$ with $i\ge 2$ we have
\begin{align*}
	\bd h(x) + h(\bd x) &= \bd(e(x)) + e(\bd x) && \\
			&= \bd_b(e(x)) + (-1)^{i+1} \bd_t(e(x)) + e(\bd_b x) + (-1)^i e(\bd_t x) && \\
			&= \bd_b(e(x)) + e(\bd_b x) && \text{(since $\bd_t(e(x)) = e(\bd_t x)$)} \\
		 	&= x . &&
\end{align*}
For $x\in \btc_{1j}$ we have
\begin{align*}
	\bd h(x) + h(\bd x) &= \bd_b(e(x)) + \bd_t(e(x)) + e(\bd_b x - r(\bd_b x)) - c(r(\bd_b x)) - e(\bd_t x) && \\
			&= \bd_b(e(x)) + e(\bd_b x) && \text{(since $r(\bd_b x) = 0$)} \\
			&= x . &&
\end{align*}
For $x\in \btc_{0j}$ with $j\ge 1$ we have
\begin{align*}
	\bd h(x) + h(\bd x) &= \bd_b(e(x - r(x))) - \bd_t(e(x - r(x))) + \bd_t(c(r(x))) + 
											e(\bd_t x - r(\bd_t x)) - c(r(\bd_t x)) \\
			&= x - r(x) + \bd_t(c(r(x))) - c(r(\bd_t x)) \\
			&= x - r(x) + r(x) \\
			&= x.
\end{align*}
Here we have used the fact that $\bd_b(c(r(x))) = 0$ since $c(r(x))$ is a $0$-blob diagram, 
as well as that $\bd_t(e(r(x))) = e(r(\bd_t x))$
and $\bd_t(c(r(x))) - c(r(\bd_t x))  = r(x)$.

For $x\in \btc_{00}$ we have
\begin{align*}
	\bd h(x) + h(\bd x) &= \bd_b(e(x - r(x))) + \bd_t(c(r(x))) \\
			&= x - r(x) + r(x) - r(x)\\
			&= x - r(x). \qedhere
\end{align*}
\end{proof}

\begin{lemma} \label{btc-prod}
For manifolds $X$ and $Y$, we have $\btc_*(X\du Y) \simeq \btc_*(X)\otimes\btc_*(Y)$.
\end{lemma}
\begin{proof}
This follows from the Eilenberg-Zilber theorem and the fact that
\begin{align*}
	\BD_k(X\du Y) & \cong \coprod_{i+j=k} \BD_i(X)\times\BD_j(Y) . \qedhere
\end{align*}
\end{proof}

For $S\sub X$, we say that $a\in \btc_k(X)$ is {\it supported on $S$}
if there exist $a'\in \btc_k(S)$
and $r\in \btc_0(X\setmin S)$ such that $a = a'\bullet r$.

\newcommand\sbtc{\btc^{\cU}}
Let $\cU$ be an open cover of $X$.
Let $\sbtc_*(X)\sub\btc_*(X)$ be the subcomplex generated by
$a\in \btc_*(X)$ such that there is a decomposition $X = \cup_i D_i$
such that each $D_i$ is a ball contained in some open set of $\cU$ and
$a$ is splittable along this decomposition.
In other words, $a$ can be obtained by gluing together pieces, each of which
is small with respect to $\cU$.

\begin{lemma} \label{small-top-blobs}
For any open cover $\cU$ of $X$, the inclusion $\sbtc_*(X)\sub\btc_*(X)$
is a homotopy equivalence.
\end{lemma}
\begin{proof}
This follows from a combination of Lemma \ref{extension_lemma_c} and the techniques of
the proof of Lemma \ref{small-blobs-b}.

It suffices to show that we can deform a finite subcomplex $C_*$ of $\btc_*(X)$ into $\sbtc_*(X)$
(relative to any designated subcomplex of $C_*$ already in $\sbtc_*(X)$).
The first step is to replace families of general blob diagrams with families 
of blob diagrams that are small with respect to $\cU$.
(If $f:P \to \BD_k$ is the family then for all $p\in P$ we have that $f(p)$ is a diagram in which the blobs are small.)
This is done as in the proof of Lemma \ref{small-blobs-b}; the technique of the proof works in families.
Each such family is homotopic to a sum of families which can be a ``lifted" to $\Homeo(X)$.
That is, $f:P \to \BD_k$ has the form $f(p) = g(p)(b)$ for some $g:P\to \Homeo(X)$ and $b\in \BD_k$.
(We are ignoring a complication related to twig blob labels, which might vary
independently of $g$, but this complication does not affect the conclusion we draw here.)
We now apply Lemma \ref{extension_lemma_c} to get families which are supported 
on balls $D_i$ contained in open sets of $\cU$.
\end{proof}

\begin{proof}[Proof of Lemma \ref{lem:bc-btc}]
Armed with the above lemmas, we can now proceed similarly to the proof of Lemma \ref{small-blobs-b}.

It suffices to show that for any finitely generated pair of subcomplexes 
$(C_*, D_*) \sub (\btc_*(X), \bc_*(X))$
we can find a homotopy $h:C_*\to \btc_{*+1}(X)$ such that $h(D_*) \sub \bc_{*+1}(X)$
and $x + h\bd(x) + \bd h(x) \in \bc_*(X)$ for all $x\in C_*$.

By Lemma \ref{small-top-blobs}, we may assume that $C_* \sub \btc_*^\cU(X)$ for some
cover $\cU$ of our choosing.
We choose $\cU$ fine enough so that each generator of $C_*$ is supported on a disjoint union of balls.
(This is possible since the original $C_*$ was finite and therefore had bounded dimension.)

Since $\bc_0(X) = \btc_0(X)$, we can take $h_0 = 0$.

Let $b \in C_1$ be a generator.
Since $b$ is supported in a disjoint union of balls,
we can find $s(b)\in \bc_1$ with $\bd (s(b)) = \bd b$
(by Corollary \ref{disj-union-contract}), and also $h_1(b) \in \btc_2(X)$
such that $\bd (h_1(b)) = s(b) - b$
(by Lemmas \ref{bt-contract} and \ref{btc-prod}).

Now let $b$ be a generator of $C_2$.
If $\cU$ is fine enough, there is a disjoint union of balls $V$
on which $b + h_1(\bd b)$ is supported.
Since $\bd(b + h_1(\bd b)) = s(\bd b) \in \bc_1(X)$, we can find
$s(b)\in \bc_2(X)$ with $\bd(s(b)) = \bd(b + h_1(\bd b))$ (by Corollary \ref{disj-union-contract}).
By Lemmas \ref{bt-contract} and \ref{btc-prod}, we can now find
$h_2(b) \in \btc_3(X)$, also supported on $V$, such that $\bd(h_2(b)) = s(b) - b - h_1(\bd b)$

The general case, $h_k$, is similar.

Note that it is possible to make the various choices above so that the homotopies we construct
are fixed on $\bc_* \sub \btc_*$.
It follows that we may assume that
the homotopy inverse to the inclusion constructed above is the identity on $\bc_*$.
Note that the complex of all homotopy inverses with this property is contractible, 
so the homotopy inverse is well-defined up to a contractible set of choices.
\end{proof}


\subsection{Action of \texorpdfstring{$\CH{X}$}{C*(Homeo(M))}}
\label{ss:emap-def}

Let  $C_*(\Homeo(X \to Y))$ denote the singular chain complex of
the space of homeomorphisms
between the $n$-manifolds $X$ and $Y$ 
(any given singular chain extends a fixed homeomorphism $\bd X \to \bd Y$).
We also will use the abbreviated notation $\CH{X} \deq \CH{X \to X}$.
(For convenience, we will permit the singular cells generating $\CH{X \to Y}$ to be more general
than simplices --- they can be based on any cone-product polyhedron (see Remark \ref{blobsset-remark}).)

\begin{thm}  \label{thm:CH} \label{thm:evaluation}%
For $n$-manifolds $X$ and $Y$ there is a chain map
\eq{
    e_{XY} : \CH{X \to Y} \otimes \bc_*(X) \to \bc_*(Y) ,
}
well-defined up to coherent homotopy,
such that
\begin{enumerate}
\item on $C_0(\Homeo(X \to Y)) \otimes \bc_*(X)$ it agrees with the obvious action of 
$\Homeo(X, Y)$ on $\bc_*(X)$  described in Property \ref{property:functoriality}, and
\item for any compatible splittings $X\to X\sgl$ and $Y\to Y\sgl$, 
the following diagram commutes up to homotopy
\begin{equation*}
\xymatrix@C+2cm{
      \CH{X \to Y} \otimes \bc_*(X)
        \ar[r]_(.6){e_{XY}}  \ar[d]^{\gl \otimes \gl}   &
            \bc_*(Y)\ar[d]^{\gl} \\
     \CH{X\sgl, Y\sgl} \otimes \bc_*(X\sgl) \ar[r]_(.6){e_{X\sgl Y\sgl}}   & 	\bc_*(Y\sgl)  
}
\end{equation*}
\end{enumerate}
\end{thm}

\begin{proof}
In light of Lemma \ref{lem:bc-btc}, it suffices to prove the theorem with 
$\bc_*$ replaced by $\btc_*$.
In fact, for $\btc_*$ we get a sharper result: we can omit
the ``up to homotopy" qualifiers.

Let $f\in C_k(\Homeo(X \to Y))$, $f:P^k\to \Homeo(X \to Y)$ and $a\in \btc_{ij}(X)$, 
$a:Q^j \to \BD_i(X)$.
Define $e_{XY}(f\ot a)\in \btc_{i,j+k}(Y)$ by
\begin{align*}
	e_{XY}(f\ot a) : P\times Q &\to \BD_i(Y) \\
	(p,q) &\mapsto f(p)(a(q))  .
\end{align*}
It is clear that this agrees with the previously defined $C_0(\Homeo(X \to Y))$ action on $\btc_*$,
and it is also easy to see that the diagram in item 2 of the statement of the theorem
commutes on the nose.
\end{proof}

\begin{thm}
\label{thm:CH-associativity}
The $\CH{X \to Y}$ actions defined above are associative.
That is, the following diagram commutes up to coherent homotopy:
\[ \xymatrix@C=5pt{
& \CH{Y\to Z} \ot \bc_*(Y) \ar[drr]^{e_{YZ}} & &\\
\CH{X \to Y} \ot \CH{Y \to Z} \ot \bc_*(X) \ar[ur]^{e_{XY}\ot\id} \ar[dr]_{\mu\ot\id} & & & \bc_*(Z) \\
& \CH{X \to Z} \ot \bc_*(X) \ar[urr]_{e_{XZ}} & &
} \]
Here $\mu:\CH{X\to Y} \ot \CH{Y \to Z}\to \CH{X \to Z}$ is the map induced by composition
of homeomorphisms.
\end{thm}
\begin{proof}
The corresponding diagram for $\btc_*$ commutes on the nose.
\end{proof}

\begin{remark} \label{collar-map-action-remark} \rm
Like $\Homeo(X)$, collar maps also have a natural topology (see discussion following Axiom \ref{axiom:families}),
and by adjusting the topology on blob diagrams we can arrange that families of collar maps
act naturally on $\btc_*(X)$.
\end{remark}



\def\xxpar#1#2{\smallskip\noindent{\bf #1} {\it #2} \smallskip}
\def\mmpar#1#2#3{\smallskip\noindent{\bf #1} (#2). {\it #3} \smallskip}

\section{\texorpdfstring{$n$}{n}-categories and their modules}
\label{sec:ncats}

\subsection{Definition of \texorpdfstring{$n$}{n}-categories}
\label{ss:n-cat-def}

Before proceeding, we need more appropriate definitions of $n$-categories, 
$A_\infty$ $n$-categories, as well as modules for these, and tensor products of these modules.
(As is the case throughout this paper, by ``$n$-category" we mean some notion of
a ``weak" $n$-category with ``strong duality".)

Compared to other definitions in the literature,
the definitions presented below tie the categories more closely to the topology
and avoid combinatorial questions about, for example, finding a minimal sufficient
collection of generalized associativity axioms; we prefer maximal sets of axioms to minimal sets.
It is easy to show that examples of topological origin
(e.g.\ categories whose morphisms are maps into spaces or decorated balls, or bordism categories)
satisfy our axioms.
To show that examples of a more purely algebraic origin satisfy our axioms, 
one would typically need the combinatorial
results that we have avoided here.

See \S\ref{n-cat-names} for a discussion of $n$-category terminology.


\medskip

The axioms for an $n$-category are spread throughout this section.
Collecting these together, an $n$-category is a gadget satisfying Axioms \ref{axiom:morphisms}, 
\ref{nca-boundary}, \ref{axiom:composition},  \ref{nca-assoc}, \ref{axiom:product}, \ref{axiom:extended-isotopies} and  \ref{axiom:splittings}.
For an enriched $n$-category we add Axiom \ref{axiom:enriched}.
For an $A_\infty$ $n$-category, we replace 
Axiom \ref{axiom:extended-isotopies} with Axiom \ref{axiom:families}.

Strictly speaking, before we can state the axioms for $k$-morphisms we need all the axioms 
for $k{-}1$-morphisms.
Readers who prefer things to be presented in a strictly logical order should read this 
subsection $n{+}1$ times, first setting $k=0$, then $k=1$, and so on until they reach $k=n$.

\medskip

There are many existing definitions of $n$-categories, with various intended uses.
In any such definition, there are sets of $k$-morphisms for each $0 \leq k \leq n$.
Generally, these sets are indexed by instances of a certain typical shape. 
Some $n$-category definitions model $k$-morphisms on the standard bihedron (interval, bigon, and so on).
Other definitions have a separate set of 1-morphisms for each interval $[0,l] \sub \r$, 
a separate set of 2-morphisms for each rectangle $[0,l_1]\times [0,l_2] \sub \r^2$,
and so on.
(This allows for strict associativity; see \cite{ulrike-tillmann-2008,0909.2212}.)
Still other definitions (see, for example, \cite{MR2094071})
model the $k$-morphisms on more complicated combinatorial polyhedra.

For our definition, we will allow our $k$-morphisms to have {\it any} shape, so long as it is 
homeomorphic to the standard $k$-ball.
Thus we associate a set of $k$-morphisms $\cC_k(X)$ to any $k$-manifold $X$ homeomorphic 
to the standard $k$-ball.

Below, we will use ``a $k$-ball" to mean any $k$-manifold which is homeomorphic to the 
standard $k$-ball.
We {\it do not} assume that such $k$-balls are equipped with a 
preferred homeomorphism to the standard $k$-ball.
The same applies to ``a $k$-sphere" below.

Given a homeomorphism $f:X\to Y$ between $k$-balls (not necessarily fixed on 
the boundary), we want a corresponding
bijection of sets $f:\cC_k(X)\to \cC_k(Y)$.
(This will imply ``strong duality", among other things.) Putting these together, we have

\begin{axiom}[Morphisms]
\label{axiom:morphisms}
For each $0 \le k \le n$, we have a functor $\cC_k$ from 
the category of $k$-balls and 
homeomorphisms to the category of sets and bijections.
\end{axiom}

(Note: We often omit the subscript $k$.)

We are being deliberately vague about what flavor of $k$-balls
we are considering.
They could be unoriented or oriented or Spin or $\mbox{Pin}_\pm$.
They could be PL or smooth.
(If smooth, ``homeomorphism" should be read ``diffeomorphism", and we would need
to be fussier about corners and boundaries.)
For each flavor of manifold there is a corresponding flavor of $n$-category.
For simplicity, we will concentrate on the case of PL unoriented manifolds.

(An interesting open question is whether the techniques of this paper can be adapted to topological
manifolds and plain, merely continuous homeomorphisms.
The main obstacles are proving a version of Lemma \ref{basic_adaptation_lemma} and adapting the
transversality arguments used in Lemma \ref{lem:colim-injective}.)

An ambitious reader may want to keep in mind two other classes of balls.
The first is balls equipped with a map to some other space $Y$ (c.f. \cite{MR2079378}). 
This will be used below (see the end of \S \ref{ss:product-formula}) to describe the blob complex of a fiber bundle with
base space $Y$.
The second is balls equipped with sections of the tangent bundle, or the frame
bundle (i.e.\ framed balls), or more generally some partial flag bundle associated to the tangent bundle.
These can be used to define categories with less than the ``strong" duality we assume here,
though we will not develop that idea in this paper.

Next we consider domains and ranges of morphisms (or, as we prefer to say, boundaries
of morphisms).
The 0-sphere is unusual among spheres in that it is disconnected.
Correspondingly, for 1-morphisms it makes sense to distinguish between domain and range.
(Actually, this is only true in the oriented case, with 1-morphisms parameterized
by {\it oriented} 1-balls.)
For $k>1$ and in the presence of strong duality the division into domain and range makes less sense.
For example, in a pivotal tensor category, there are natural isomorphisms $\Hom{}{A}{B \tensor C} \isoto \Hom{}{B^* \tensor A}{C}$, etc. 
(sometimes called ``Frobenius reciprocity''), which canonically identify all the morphism spaces which have the same boundary.
We prefer not to make the distinction in the first place.

Instead, we will combine the domain and range into a single entity which we call the 
boundary of a morphism.
Morphisms are modeled on balls, so their boundaries are modeled on spheres.
In other words, we need to extend the functors $\cC_{k-1}$ from balls to spheres, for 
$1\le k \le n$.
At first it might seem that we need another axiom 
(more specifically, additional data) for this, but in fact once we have
all the axioms in this subsection for $0$ through $k-1$ we can use a colimit
construction, as described in \S\ref{ss:ncat-coend} below, to extend $\cC_{k-1}$
to spheres (and any other manifolds):

\begin{lem}
\label{lem:spheres}
For each $1 \le k \le n$, we have a functor $\cl{\cC}_{k-1}$ from 
the category of $k{-}1$-spheres and 
homeomorphisms to the category of sets and bijections.
\end{lem}

We postpone the proof of this result until after we've actually given all the axioms.
Note that defining this functor for fixed $k$ only requires the data described in Axiom \ref{axiom:morphisms} at level $k$, 
along with the data described in the other axioms for smaller values of $k$. 

Of course, Lemma \ref{lem:spheres}, as stated, is satisfied by the trivial functor.
What we really mean is that there exists a functor which interacts with the other data of $\cC$ as specified 
in the axioms below.

\begin{axiom}[Boundaries]\label{nca-boundary}
For each $k$-ball $X$, we have a map of sets $\bd: \cC_k(X)\to \cl{\cC}_{k-1}(\bd X)$.
These maps, for various $X$, comprise a natural transformation of functors.
\end{axiom}

Note that the first ``$\bd$" above is part of the data for the category, 
while the second is the ordinary boundary of manifolds.
Given $c\in\cl{\cC}(\bd(X))$, we will write $\cC(X; c)$ for $\bd^{-1}(c)$, those morphisms with specified boundary $c$.

\medskip

In order to simplify the exposition we have concentrated on the case of 
unoriented PL manifolds and avoided the question of what exactly we mean by 
the boundary of a manifold with extra structure, such as an oriented manifold.
In general, all manifolds of dimension less than $n$ should be equipped with the germ
of a thickening to dimension $n$, and this germ should carry whatever structure we have 
on $n$-manifolds.
In addition, lower dimensional manifolds should be equipped with a framing
of their normal bundle in the thickening; the framing keeps track of which
side (iterated) bounded manifolds lie on.
For example, the boundary of an oriented $n$-ball
should be an $n{-}1$-sphere equipped with an orientation of its once stabilized tangent
bundle and a choice of direction in this bundle indicating
which side the $n$-ball lies on.

\medskip

We have just argued that the boundary of a morphism has no preferred splitting into
domain and range, but the converse meets with our approval.
That is, given compatible domain and range, we should be able to combine them into
the full boundary of a morphism.
The following lemma will follow from the colimit construction used to define $\cl{\cC}_{k-1}$
on spheres.

\begin{lem}[Boundary from domain and range]
\label{lem:domain-and-range}
Let $S = B_1 \cup_E B_2$, where $S$ is a $k{-}1$-sphere $(1\le k\le n)$,
$B_i$ is a $k{-}1$-ball, and $E = B_1\cap B_2$ is a $k{-}2$-sphere (Figure \ref{blah3}).
Let $\cC(B_1) \times_{\cl{\cC}(E)} \cC(B_2)$ denote the fibered product of the 
two maps $\bd: \cC(B_i)\to \cl{\cC}(E)$.
Then we have an injective map
\[
	\gl_E : \cC(B_1) \times_{\cl{\cC}(E)} \cC(B_2) \into \cl{\cC}(S)
\]
which is natural with respect to the actions of homeomorphisms.
(When $k=1$ we stipulate that $\cl{\cC}(E)$ is a point, so that the above fibered product
becomes a normal product.)
\end{lem}

\begin{figure}[t] \centering
\begin{tikzpicture}[
]
\node[fill=black, circle, label=below:$E$, inner sep=1.5pt](S) at (0,0) {};
\node[fill=black, circle, label=above:$E$, inner sep=1.5pt](N) at (0,2) {};
\draw (S) arc  (-90:90:1);
\draw (N) arc  (90:270:1);
\node[left] at (-1,1) {$B_1$};
\node[right] at (1,1) {$B_2$};
\end{tikzpicture}
\caption{Combining two balls to get a full boundary.}\label{blah3}\end{figure}
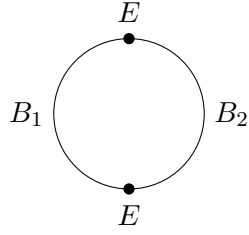

Note that we insist on injectivity above. 
The lemma follows from Definition \ref{def:colim-fields} and Lemma \ref{lem:colim-injective}.

We do not insist on surjectivity of the gluing map, since this is not satisfied by all of the examples
we are trying to axiomatize.
If our $k$-morphisms $\cC(X)$ are labeled cell complexes embedded in $X$ (c.f. Example \ref{ex:traditional-n-categories} below), then a $k$-morphism is
in the image of the gluing map precisely when the cell complex is in general position
with respect to $E$. On the other hand, in categories based on maps to a target space (c.f. Example \ref{ex:maps-to-a-space} below) the gluing map is always surjective.

If $S$ is a 0-sphere (the case $k=1$ above), then $S$ can be identified with the {\it disjoint} union
of two 0-balls $B_1$ and $B_2$ and the colimit construction $\cl{\cC}(S)$ can be identified
with the (ordinary, not fibered) product $\cC(B_1) \times \cC(B_2)$.

Let $\cl{\cC}(S)\trans E$ denote the image of $\gl_E$.
We will refer to elements of $\cl{\cC}(S)\trans E$ as ``splittable along $E$" or ``transverse to $E$". 
When the gluing map is surjective every such element is splittable.

If $X$ is a $k$-ball and $E \sub \bd X$ splits $\bd X$ into two $k{-}1$-balls $B_1$ and $B_2$
as above, then we define $\cC(X)\trans E = \bd^{-1}(\cl{\cC}(\bd X)\trans E)$.

We will call the projection $\cl{\cC}(S)\trans E \to \cC(B_i)$ given by the composition
$$\cl{\cC}(S)\trans E \xrightarrow{\gl^{-1}} \cC(B_1) \times \cC(B_2) \xrightarrow{\pr_i} \cC(B_i)$$
a {\it restriction} map and write $\res_{B_i}(a)$
(or simply $\res(a)$ when there is no ambiguity), for $a\in \cl{\cC}(S)\trans E$.
More generally, we also include under the rubric ``restriction map"
the boundary maps of Axiom \ref{nca-boundary} above,
another class of maps introduced after Axiom \ref{nca-assoc} below, as well as any composition
of restriction maps.
In particular, we have restriction maps $\cC(X)\trans E \to \cC(B_i)$
defined as the composition of the boundary with the first restriction map described above:
$$
\cC(X) \trans E \xrightarrow{\bdy} \cl{\cC}(\bdy X)\trans E \xrightarrow{\res} \cC(B_i)
.$$
These restriction maps can be thought of as 
domain and range maps, relative to the choice of splitting $\bd X = B_1 \cup_E B_2$.
\noop{These restriction maps in fact have their image in the subset $\cC(B_i)\trans E$,
and so to emphasize this we will sometimes write the restriction map as $\cC(X)\trans E \to \cC(B_i)\trans E$.}

Next we consider composition of morphisms.
For $n$-categories which lack strong duality, one usually considers
$k$ different types of composition of $k$-morphisms, each associated to a different ``direction".
(For example, vertical and horizontal composition of 2-morphisms.)
In the presence of strong duality, these $k$ distinct compositions are subsumed into 
one general type of composition which can be in any direction.

\begin{axiom}[Composition]
\label{axiom:composition}
Let $B = B_1 \cup_Y B_2$, where $B$, $B_1$ and $B_2$ are $k$-balls ($1\le k\le n$)
and $Y = B_1\cap B_2$ is a $k{-}1$-ball (Figure \ref{blah5}).
Let $E = \bd Y$, which is a $k{-}2$-sphere.
Note that each of $B$, $B_1$ and $B_2$ has its boundary split into two $k{-}1$-balls by $E$.
We have restriction (domain or range) maps $\cC(B_i)\trans E \to \cC(Y)$.
Let $\cC(B_1)\trans E \times_{\cC(Y)} \cC(B_2)\trans E$ denote the fibered product of these two maps. 
We have a map
\[
	\gl_Y : \cC(B_1)\trans E \times_{\cC(Y)} \cC(B_2)\trans E \to \cC(B)\trans E
\]
which is natural with respect to the actions of homeomorphisms, and also compatible with restrictions
to the intersection of the boundaries of $B$ and $B_i$.
If $k < n$
we require that $\gl_Y$ is injective.
\end{axiom}

\begin{figure}[t] \centering
\begin{tikzpicture}[
				x=1.5cm,y=1.5cm]
\node[fill=black, circle, label=below:$E$, inner sep=2pt](S) at (0,0) {};
\node[fill=black, circle, label=above:$E$, inner sep=2pt](N) at (0,2) {};
\draw (S) arc  (-90:90:1);
\draw (N) arc  (90:270:1);
\draw (N) -- (S);
\node[left] at (-1/4,1) {$B_1$};
\node[right] at (1/4,1) {$B_2$};
\node at (1/6,3/2)  {$Y$};
\end{tikzpicture}
\caption{From two balls to one ball.}\label{blah5}\end{figure}
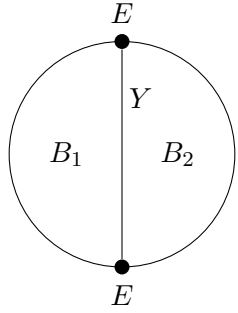

\begin{axiom}[Strict associativity] \label{nca-assoc}
The composition (gluing) maps above are strictly associative.
Given any splitting of a ball $B$ into smaller balls
$$\bigsqcup B_i \to B,$$ 
any sequence of gluings (in the sense of Definition \ref{defn:gluing-decomposition}, where all the intermediate steps are also disjoint unions of balls) yields the same result.
\end{axiom}

\begin{figure}[t]
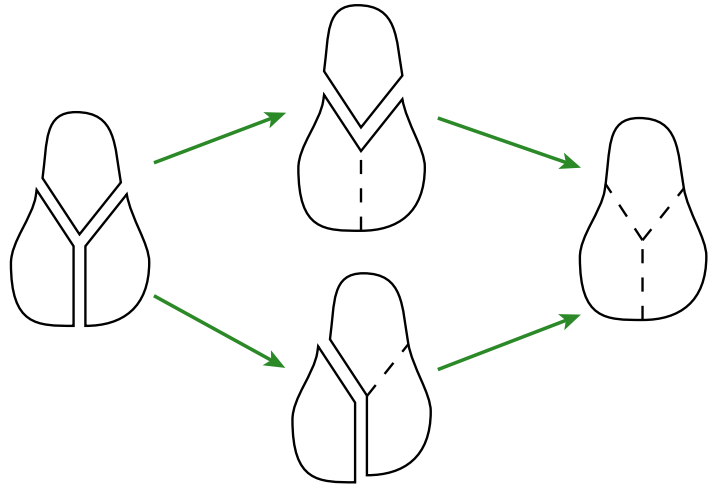

$$\mathfig{.65}{ncat/strict-associativity}$$
\caption{An example of strict associativity.}\label{blah6}\end{figure}

We'll use the notation  $a\bullet b$ for the glued together field $\gl_Y(a, b)$.
In the other direction, we will call the projection from $\cC(B)\trans E$ to $\cC(B_i)\trans E$ 
a restriction map (one of many types of map so called) and write $\res_{B_i}(a)$ for $a\in \cC(B)\trans E$.

We will write $\cC(B)\trans Y$ for the image of $\gl_Y$ in $\cC(B)$.
We will call elements of $\cC(B)\trans Y$ morphisms which are 
``splittable along $Y$'' or ``transverse to $Y$''.
We have $\cC(B)\trans Y \sub \cC(B)\trans E \sub \cC(B)$.

More generally, let $\alpha$ be a splitting of $X$ into smaller balls.
Let $\cC(X)_\alpha \sub \cC(X)$ denote the image of the iterated gluing maps from 
the smaller balls to $X$.
We  say that elements of $\cC(X)_\alpha$ are morphisms which are ``splittable along $\alpha$".
In situations where the splitting is notationally anonymous, we will write
$\cC(X)\spl$ for the morphisms which are splittable along (a.k.a.\ transverse to)
the unnamed splitting.
If $\beta$ is a ball decomposition of $\bd X$, we define $\cC(X)_\beta \deq \bd\inv(\cl{\cC}(\bd X)_\beta)$;
this can also be denoted $\cC(X)\spl$ if the context contains an anonymous
decomposition of $\bd X$ and no competing splitting of $X$.

The above two composition axioms are equivalent to the following one,
which we state in slightly vague form.

\xxpar{Multi-composition:}
{Given any splitting $B_1 \sqcup \cdots \sqcup B_m \to B$ of a $k$-ball
into small $k$-balls, there is a 
map from an appropriate subset (like a fibered product) 
of $\cC(B_1)\spl\times\cdots\times\cC(B_m)\spl$ to $\cC(B)\spl$,
and these various $m$-fold composition maps satisfy an
operad-type strict associativity condition (Figure \ref{fig:operad-composition}).}

\begin{figure}[t]
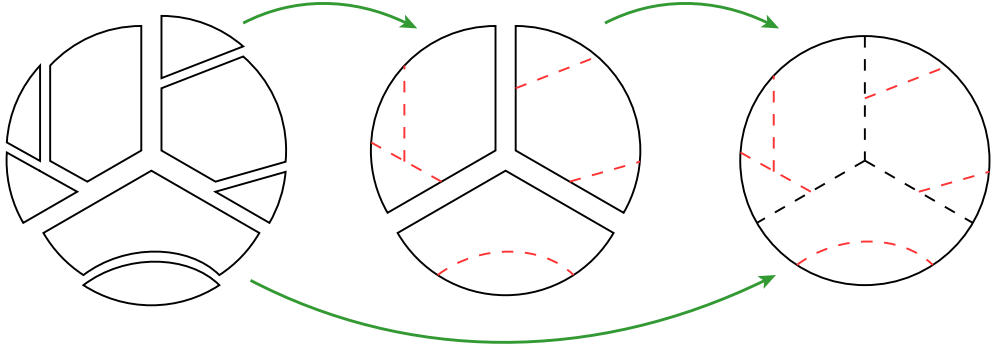

$$\mathfig{.8}{ncat/operad-composition}$$
\caption{Operad composition and associativity}\label{fig:operad-composition}\end{figure}

The next axiom is related to identity morphisms, though that might not be immediately obvious.

\begin{axiom}[Product (identity) morphisms, preliminary version]
For each $k$-ball $X$ and $m$-ball $D$, with $k+m \le n$, there is a map $\cC(X)\to \cC(X\times D)$, 
usually denoted $a\mapsto a\times D$ for $a\in \cC(X)$.
These maps must satisfy the following conditions.
\begin{enumerate}
\item
If $f:X\to X'$ and $\tilde{f}:X\times D \to X'\times D'$ are homeomorphisms such that the diagram
\[ \xymatrix{
	X\times D \ar[r]^{\tilde{f}} \ar[d]_{\pi} & X'\times D' \ar[d]^{\pi} \\
	X \ar[r]^{f} & X'
} \]
commutes, then we have 
\[
	\tilde{f}(a\times D) = f(a)\times D' .
\]
\item
Product morphisms are compatible with gluing (composition) in both factors:
\[
	(a'\times D)\bullet(a''\times D) = (a'\bullet a'')\times D
\]
and
\[
	(a\times D')\bullet(a\times D'') = a\times (D'\bullet D'') .
\]
\item
Product morphisms are associative:
\[
	(a\times D)\times D' = a\times (D\times D') .
\]
(Here we are implicitly using functoriality and the obvious homeomorphism
$(X\times D)\times D' \to X\times(D\times D')$.)
\item
Product morphisms are compatible with restriction:
\[
	\res_{X\times E}(a\times D) = a\times E
\]
for $E\sub \bd D$ and $a\in \cC(X)$.
\end{enumerate}
\end{axiom}

We will need to strengthen the above preliminary version of the axiom to allow
for products which are ``pinched" in various ways along their boundary.
(See Figure \ref{pinched_prods}.)
\begin{figure}[t]
$$
\begin{tikzpicture}[baseline=0]
\begin{scope}
\path[clip] (0,0) arc (135:45:4) arc (-45:-135:4);
\draw[kw-blue-a,line width=2pt] (0,0) arc (135:45:4) arc (-45:-135:4);
\foreach \x in {0, 0.5, ..., 6} {
	\draw[green!50!brown] (\x,-2) -- (\x,2);
}
\end{scope}
\draw[kw-blue-a,line width=1.5pt] (0,-3) -- (5.66,-3);
\draw[->,red,line width=2pt] (2.83,-1.5) -- (2.83,-2.5);
\end{tikzpicture}
\qquad \qquad
\begin{tikzpicture}[baseline=-0.15cm]
\begin{scope}
\path[clip] (0,1) arc (90:135:8 and 4)  arc (-135:-90:8 and 4) -- cycle;
\draw[kw-blue-a,line width=2pt] (0,1) arc (90:135:8 and 4)  arc (-135:-90:8 and 4) -- cycle;
\foreach \x in {-6, -5.5, ..., 0} {
	\draw[green!50!brown] (\x,-2) -- (\x,2);
}
\end{scope}
\draw[kw-blue-a,line width=1.5pt] (-5.66,-3.15) -- (0,-3.15);
\draw[->,red,line width=2pt] (-2.83,-1.5) -- (-2.83,-2.5);
\end{tikzpicture}
$$
\caption{Examples of pinched products}\label{pinched_prods}
\end{figure}
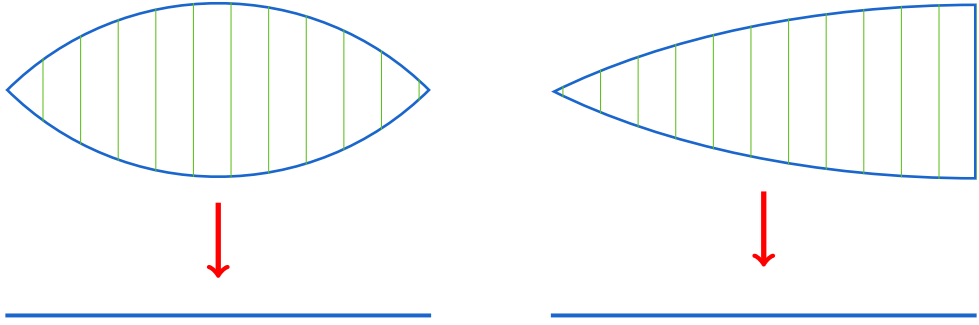
The need for a strengthened version will become apparent in Appendix \ref{sec:comparing-defs}
where we construct a traditional 2-category from a disk-like 2-category.
For example, ``half-pinched" products of 1-balls are used to construct weak identities for 1-morphisms
in 2-categories (see \S\ref{ssec:2-cats}).
We also need fully-pinched products to define collar maps below (see Figure \ref{glue-collar}).

Define a {\it pinched product} to be a map
\[
	\pi: E\to X
\]
such that $E$ is a $k{+}m$-ball, $X$ is a $k$-ball ($m\ge 1$), and $\pi$ is locally modeled
on a standard iterated degeneracy map
\[
	d: \Delta^{k+m}\to\Delta^k .
\]
(We thank Kevin Costello for suggesting this approach.)

Note that for each interior point $x\in X$, $\pi\inv(x)$ is an $m$-ball,
and for each boundary point $x\in\bd X$, $\pi\inv(x)$ is a ball of dimension
$l \le m$, with $l$ depending on $x$.
It is easy to see that a composition of pinched products is again a pinched product.
A {\it sub pinched product} is a sub-$m$-ball $E'\sub E$ such that the restriction
$\pi:E'\to \pi(E')$ is again a pinched product.
A {union} of pinched products is a decomposition $E = \cup_i E_i$
such that each $E_i\sub E$ is a sub pinched product.
(See Figure \ref{pinched_prod_unions}.)
\begin{figure}[t]
$$
\begin{tikzpicture}[baseline=0]
\begin{scope}
\path[clip] (0,0) arc (135:45:4) arc (-45:-135:4);
\draw[kw-blue-a,line width=2pt] (0,0) arc (135:45:4) arc (-45:-135:4);
\draw[kw-blue-a] (0,0) -- (5.66,0);
\foreach \x in {0, 0.5, ..., 6} {
	\draw[green!50!brown] (\x,-2) -- (\x,2);
}
\end{scope}
\end{tikzpicture}
\qquad
\begin{tikzpicture}[baseline=0]
\begin{scope}
\path[clip] (0,-1) rectangle (4,1);
\draw[kw-blue-a,line width=2pt] (0,-1) rectangle (4,1);
\draw[kw-blue-a] (0,0) -- (5,0);
\foreach \x in {0, 0.5, ..., 6} {
	\draw[green!50!brown] (\x,-2) -- (\x,2);
}
\end{scope}
\end{tikzpicture}
\qquad
\begin{tikzpicture}[baseline=0]
\begin{scope}
\path[clip] (0,0) arc (135:45:4) arc (-45:-135:4);
\draw[kw-blue-a,line width=2pt] (0,0) arc (135:45:4) arc (-45:-135:4);
\draw[kw-blue-a] (2.83,3) circle (3);
\foreach \x in {0, 0.5, ..., 6} {
	\draw[green!50!brown] (\x,-2) -- (\x,2);
}
\end{scope}
\end{tikzpicture}
$$
$$
\begin{tikzpicture}[baseline=0]
\begin{scope}
\path[clip] (0,-1) rectangle (4,1);
\draw[kw-blue-a,line width=2pt] (0,-1) rectangle (4,1);
\draw[kw-blue-a] (0,-1) -- (4,1);
\foreach \x in {0, 0.5, ..., 6} {
	\draw[green!50!brown] (\x,-2) -- (\x,2);
}
\end{scope}
\end{tikzpicture}
\qquad
\begin{tikzpicture}[baseline=0]
\begin{scope}
\path[clip] (0,-1) rectangle (5,1);
\draw[kw-blue-a,line width=2pt] (0,-1) rectangle (5,1);
\draw[kw-blue-a] (1,-1) .. controls  (2,-1) and (3,1) .. (4,1);
\foreach \x in {0, 0.5, ..., 6} {
	\draw[green!50!brown] (\x,-2) -- (\x,2);
}
\end{scope}
\end{tikzpicture}
\qquad
\begin{tikzpicture}[baseline=0]
\begin{scope}
\path[clip] (0,0) arc (135:45:4) arc (-45:-135:4);
\draw[kw-blue-a,line width=2pt] (0,0) arc (135:45:4) arc (-45:-135:4);
\draw[kw-blue-a] (2.82,-5) -- (2.83,5);
\foreach \x in {0, 0.5, ..., 6} {
	\draw[green!50!brown] (\x,-2) -- (\x,2);
}
\end{scope}
\end{tikzpicture}
$$
\caption{Six examples of unions of pinched products}\label{pinched_prod_unions}
\end{figure}
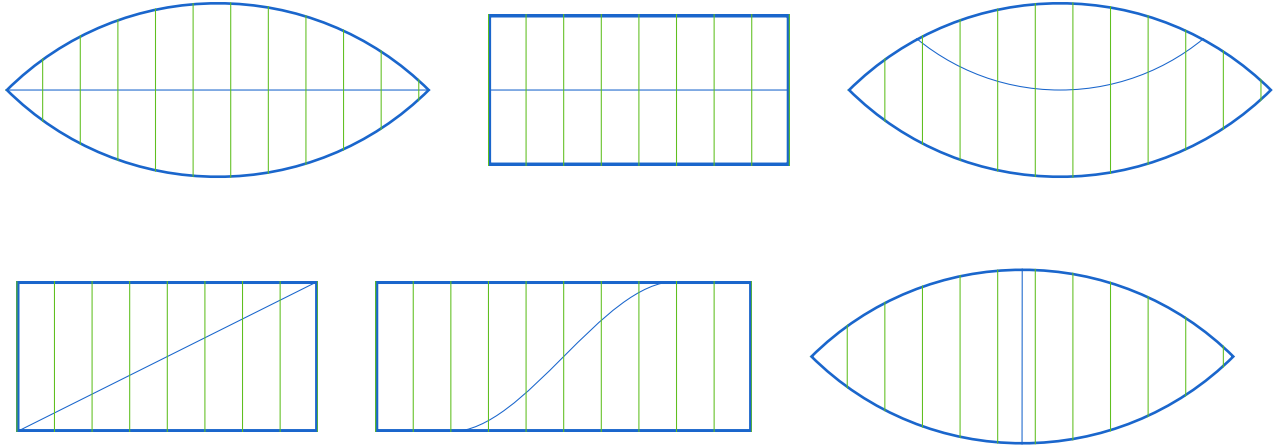

Note that $\bd X$ has a (possibly trivial) subdivision according to 
the dimension of $\pi\inv(x)$, $x\in \bd X$.
Let $\cC(X)\trans{}$ denote the morphisms which are splittable along this subdivision.

The product axiom will give a map $\pi^*:\cC(X)\trans{}\to \cC(E)$ for each pinched product
$\pi:E\to X$.
Morphisms in the image of $\pi^*$ will be called product morphisms.
Before stating the axiom, we illustrate it in our two motivating examples of $n$-categories.
In the case where $\cC(X) = \{f: X\to T\}$, we define $\pi^*(f) = f\circ\pi$.
In the case where $\cC(X)$ is the set of all labeled embedded cell complexes $K$ in $X$, 
define $\pi^*(K) = \pi\inv(K)$, with each codimension $i$ cell $\pi\inv(c)$ labeled by the
same (traditional) $i$-morphism as the corresponding codimension $i$ cell $c$.

\begin{axiom}[Product (identity) morphisms]
\label{axiom:product}
For each pinched product $\pi:E\to X$, with $X$ a $k$-ball and $E$ a $k{+}m$-ball ($m\ge 1$),
there is a map $\pi^*:\cC(X)\trans{}\to \cC(E)$.
These maps must satisfy the following conditions.
\begin{enumerate}
\item
If $\pi:E\to X$ and $\pi':E'\to X'$ are pinched products, and
if $f:X\to X'$ and $\tilde{f}:E \to E'$ are maps such that the diagram
\[ \xymatrix{
	E \ar[r]^{\tilde{f}} \ar[d]_{\pi} & E' \ar[d]^{\pi'} \\
	X \ar[r]^{f} & X'
} \]
commutes, then we have 
\[
	\pi'^*\circ f = \tilde{f}\circ \pi^*.
\]
\item
Product morphisms are compatible with gluing (composition).
Let $\pi:E\to X$, $\pi_1:E_1\to X_1$, and $\pi_2:E_2\to X_2$ 
be pinched products with $E = E_1\cup E_2$.
(See Figure \ref{pinched_prod_unions}.)  
Note that $X_1$ and $X_2$ can be identified with subsets of $X$, 
but $X_1 \cap X_2$ might not be codimension 1, and indeed we might have $X_1 = X_2 = X$.
We assume that there is a decomposition of $X$ into balls which is compatible with
$X_1$ and $X_2$.
Let $a\in \cC(X)\trans{}$, and let $a_i$ denote the restriction of $a$ to $X_i\sub X$.
(We assume that $a$ is splittable with respect to the above decomposition of $X$ into balls.)
Then 
\[
	\pi^*(a) = \pi_1^*(a_1)\bullet \pi_2^*(a_2) .
\]
\item
Product morphisms are associative.
If $\pi:E\to X$ and $\rho:D\to E$ are pinched products then
\[
	\rho^*\circ\pi^* = (\pi\circ\rho)^* .
\]
\item
Product morphisms are compatible with restriction.
If we have a commutative diagram
\[ \xymatrix{
	D \ar@{^(->}[r] \ar[d]_{\rho} & E \ar[d]^{\pi} \\
	Y \ar@{^(->}[r] & X
} \]
such that $\rho$ and $\pi$ are pinched products, then
\[
	\res_D\circ\pi^* = \rho^*\circ\res_Y .
\]
\end{enumerate}
\end{axiom}

\medskip



The next axiom says, roughly, that we have strict associativity in dimension $n$, 
even when we reparametrize our $n$-balls.

\begin{axiom}[\textup{\textbf{[preliminary]}} Isotopy invariance in dimension $n$]
\label{axiom:isotopy-preliminary}
Let $X$ be an $n$-ball, $b \in \cC(X)$, and $f: X\to X$ be a homeomorphism which 
acts trivially on the restriction $\bd b$ of $b$ to $\bd X$.
(Keep in mind the important special case where $f$ restricted to $\bd X$ is the identity.)
Suppose furthermore that $f$ is isotopic to the identity through homeomorphisms which act
trivially on $\bd b$.
Then $f(b) = b$.
In particular, homeomorphisms which are isotopic to the identity rel boundary act trivially on 
all of $\cC(X)$.
\end{axiom}

This axiom needs to be strengthened to force product morphisms to act as the identity.
Let $X$ be an $n$-ball and $Y\sub\bd X$ be an $n{-}1$-ball.
Let $J$ be a 1-ball (interval).
Let $s_{Y,J}: X\cup_Y (Y\times J) \to X$ be a collaring homeomorphism
(see the end of \S\ref{ss:syst-o-fields}).
Here we use $Y\times J$ with boundary entirely pinched.
We define a map
\begin{eqnarray*}
	\psi_{Y,J}: \cC(X) &\to& \cC(X) \\
	a & \mapsto & s_{Y,J}(a \bullet ((a|_Y)\times J)) .
\end{eqnarray*}
(See Figure \ref{glue-collar}.)
\begin{figure}[t]
\begin{equation*}
\begin{tikzpicture}
\def\rad{1}
\def\srad{0.75}
\def\gap{4.5}
\foreach \i in {0, 1, 2} {
	\node(\i) at ($\i*(\gap,0)$) [draw, circle through = {($\i*(\gap,0)+(\rad,0)$)}] {};
	\node(\i-small) at (\i.east) [circle through={($(\i.east)+(\srad,0)$)}] {};
	\foreach \n in {1,2} {
		\fill (intersection \n of \i-small and \i) node(\i-intersection-\n) {} circle (2pt);
	}
}

\begin{scope}[decoration={brace,amplitude=10,aspect=0.5}]
	\draw[decorate] (0-intersection-1.east) -- (0-intersection-2.east);
\end{scope}
\node[right=1mm] at (0.east) {$a$};
\draw[->] ($(0.east)+(0.75,0)$) -- ($(1.west)+(-0.2,0)$);

\draw (1-small)  circle (\srad);
\foreach \theta in {90, 72, ..., -90} {
	\draw[kw-blue-a] (1) -- ($(1)+(\rad,0)+(\theta:\srad)$);
}
\filldraw[fill=white] (1) circle (\rad);
\foreach \n in {1,2} {
	\fill (intersection \n of 1-small and 1) circle (2pt);
}
\node[below] at (1-small.south) {$a \times J$};
\draw[->] ($(1.east)+(1,0)$) -- ($(2.west)+(-0.2,0)$);

\begin{scope}
\path[clip] (2) circle (\rad);
\draw[clip] (2.east) circle (\srad);
\foreach \y in {1, 0.86, ..., -1} {
	\draw[kw-blue-a] ($(2)+(-1,\y) $)-- ($(2)+(1,\y)$);
}
\end{scope}
\end{tikzpicture}
\end{equation*}
\begin{equation*}
\xymatrix@C+2cm{\cC(X) \ar[r]^(0.45){\text{glue}} & \cC(X \cup \text{collar}) \ar[r]^(0.55){\text{homeo}} & \cC(X)}
\end{equation*}

\caption{Extended homeomorphism.}\label{glue-collar}\end{figure}
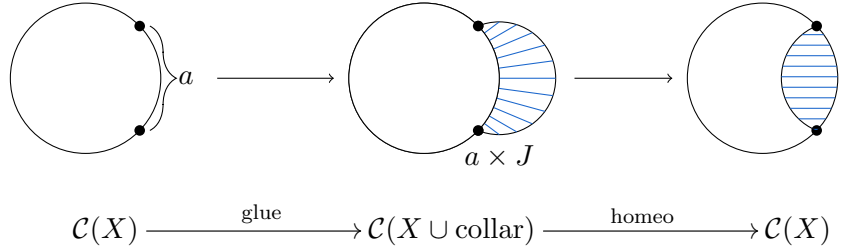
We call a map of this form a {\it collar map}.
It can be thought of as the action of the inverse of
a map which projects a collar neighborhood of $Y$ onto $Y$,
or as the limit of homeomorphisms $X\to X$ which expand a very thin collar of $Y$
to a larger collar.
We call the equivalence relation generated by collar maps and homeomorphisms
isotopic (rel boundary) to the identity {\it extended isotopy}.

The revised axiom is

\begin{axiom}[Extended isotopy invariance in dimension $n$]
\label{axiom:extended-isotopies}
Let $X$ be an $n$-ball, $b \in \cC(X)$, and $f: X\to X$ be a homeomorphism which 
acts trivially on the restriction $\bd b$ of $b$ to $\bd X$.
Suppose furthermore that $f$ is isotopic to the identity through homeomorphisms which
act trivially on $\bd b$.
Then $f(b) = b$.
In addition, collar maps act trivially on $\cC(X)$.
\end{axiom}

\medskip

We need one additional axiom.
It says, roughly, that given a $k$-ball $X$, $k<n$, and $c\in \cC(X)$, there exist sufficiently many splittings of $c$.
We use this axiom in the proofs of \ref{lem:d-a-acyclic} and \ref{lem:colim-injective}.
The analogous axiom for systems of fields is used in the proof of \ref{small-blobs-b}.
All of the examples of (disk-like) $n$-categories we consider in this paper satisfy the axiom, but
nevertheless we feel that it is too strong.
In the future we would like to see this provisional version of the axiom replaced by something less restrictive.

We give two alternate versions of the axiom, one better suited for smooth examples, and one better suited to PL examples.

\begin{axiom}[Splittings]
\label{axiom:splittings}
Let $c\in \cC_k(X)$, with $0\le k < n$.
Let $s = \{X_i\}$ be a splitting of X (so $X = \cup_i X_i$).
Let $\Homeo_\bd(X)$ denote homeomorphisms of $X$ which restrict to the identity on $\bd X$.
\begin{itemize}
\item (Alternative 1) Consider the set of homeomorphisms $g:X\to X$ such that $c$ splits along $g(s)$.
Then this subset of $\Homeo(X)$ is open and dense.
Furthermore, if $s$ restricts to a splitting $\bd s$ of $\bd X$, and if $\bd c$ splits along $\bd s$, then the
intersection of the set of such homeomorphisms $g$ with $\Homeo_\bd(X)$ is open and dense in $\Homeo_\bd(X)$.
\item (Alternative 2) Then there exists an embedded cell complex $S_c \sub X$, called the string locus of $c$,
such that if the splitting $s$ is transverse to $S_c$ then $c$ splits along $s$.
\end{itemize}
\end{axiom}

We note some consequences of Axiom \ref{axiom:splittings}.

First, some preliminary definitions.
If $P$ is a poset let $P\times I$ denote the product poset, where $I = \{0, 1\}$ with ordering $0\le 1$.
Let $\Cone(P)$ denote $P$ adjoined an additional object $v$ (the vertex of the cone) with $p\le v$ for all objects $p$ of $P$.
Finally, let $\vcone(P)$ denote $P\times I \cup \Cone(P)$, where we identify $P\times \{0\}$ with the base of the cone.
We call $P\times \{1\}$ the base of $\vcone(P)$.
(See Figure \ref{vcone-fig}.)
\begin{figure}[t]
\centering
\begin{tikzpicture}
	[kw node/.style={circle,fill=orange!70},
	kw arrow/.style={-latex, very thick, blue!70, shorten >=.06cm, shorten <=.06cm},
	kw label/.style={cca},
	]

	\definecolor{cca}{rgb}{.1,.4,.3};

	\node at (0,0) {
		\begin{tikzpicture}	
			\draw 
				(0,0) node[kw node](p1){}
				(1,.5) node[kw node](p2){}
				(2,0) node[kw node](p3){};
			
			\draw[kw arrow] (p1) -- (p3);
			\draw[kw arrow] (p2) -- (p3);
			\draw[kw arrow] (p1) -- (p2);
			
			\draw[kw label] (1,-.6) node{(a)};
		\end{tikzpicture}
	};
	
	\node at (7,0) {
		\begin{tikzpicture}	
			\draw 
				(0,0) node[kw node](p1){}
				++(0,2.5) node[kw node](q1){}
				(1,.5) node[kw node](p2){}
				++(0,2.5) node[kw node](q2){}
				(2,0)  node[kw node](p3){}
				++(0,2.5) node[kw node](q3){}
				;
			
			\draw[kw arrow] (p1) -- (p3);
			\draw[kw arrow] (p2) -- (p3);
			\draw[kw arrow] (p1) -- (p2);
			\draw[kw arrow] (q1) -- (q3);
			\draw[kw arrow] (q2) -- (q3);
			\draw[kw arrow] (q1) -- (q2);
			\draw[kw arrow] (p1) -- (q1);
			\draw[kw arrow] (p2) -- (q2);
			\draw[kw arrow] (p3) -- (q3);

			\draw[kw label] (1,-.6) node{(b)};
		\end{tikzpicture}
	};
	
	\node at (0,-5) {
		\begin{tikzpicture}	
			\draw 
				(0,0) node[kw node](p1){}
				(1,.5) node[kw node](p2){}
				++(0,2.5) node[kw node](v){}
				(2,0)  node[kw node](p3){}
				;
			
			\draw[kw arrow] (p1) -- (p3);
			\draw[kw arrow] (p2) -- (p3);
			\draw[kw arrow] (p1) -- (p2);
			\draw[kw arrow] (p1) -- (v);
			\draw[kw arrow] (p2) -- (v);
			\draw[kw arrow] (p3) -- (v);

			\draw[kw label] (1,-.6) node{(c)};
		\end{tikzpicture}
	};
	
	\node at (7,-5) {
		\begin{tikzpicture}	
			\draw 
				(0,0) node[kw node](p1){}
				++(-2,2.5) node[kw node](q1){}
				(1,.5) node[kw node](p2){}
				++(-2,2.5) node[kw node](q2){}
				++(4,0) node[kw node](v){}
				(2,0)  node[kw node](p3){}
				++(-2,2.5) node[kw node](q3){}
				;
			
			\draw[kw arrow] (p1) -- (p3);
			\draw[kw arrow] (p2) -- (p3);
			\draw[kw arrow] (p1) -- (p2);
			\draw[kw arrow] (p1) -- (v);
			\draw[kw arrow] (p2) -- (v);
			\draw[kw arrow] (p3) -- (v);
			\draw[kw arrow] (q1) -- (q3);
			\draw[kw arrow] (q2) -- (q3);
			\draw[kw arrow] (q1) -- (q2);
			\draw[kw arrow] (p1) -- (q1);
			\draw[kw arrow] (p2) -- (q2);
			\draw[kw arrow] (p3) -- (q3);

			\draw[kw label] (1,-.6) node{(d)};
		\end{tikzpicture}
	};
	
\end{tikzpicture}
\caption{(a) $P$, (b) $P\times I$, (c) $\Cone(P)$, (d) $\vcone(P)$}
\label{vcone-fig}
\end{figure}

\begin{lem}
\label{lemma:vcones}
Let $c\in \cC_k(X)$, with $0\le k < n$, and
let $P$ be a finite poset of splittings of $c$.
Then we can embed $\vcone(P)$ into the splittings of $c$, with $P$ corresponding to the base of $\vcone(P)$.
Furthermore, if $q$ is any decomposition of $X$, then we can take the vertex of $\vcone(P)$ to be $q$ up to a small perturbation.
\end{lem}

\begin{proof}
After a small perturbation, we may assume that $q$ is simultaneously transverse to all the splittings in $P$, and
(by Axiom \ref{axiom:splittings}) that $c$ splits along $q$.
We can now choose, for each splitting $p$ in $P$, a common refinement $p'$ of $p$ and $q$.
This constitutes the middle part ($P\times \{0\}$ above) of $\vcone(P)$.
\end{proof}

\begin{cor}
For any $c\in \cC_k(X)$, the geometric realization of the poset of splittings of $c$ is contractible.
\end{cor}

\begin{proof}
In the geometric realization, V-Cones become actual cones, allowing us to contract any cycle.
\end{proof}


\medskip

This completes the definition of an $n$-category.
Next we define enriched $n$-categories.

\medskip

Most of the examples of $n$-categories we are interested in are enriched in the following sense.
The various sets of $n$-morphisms $\cC(X; c)$, for all $n$-balls $X$ and
all $c\in \cl{\cC}(\bd X)$, have the structure of an object in some appropriate auxiliary category
(e.g.\ vector spaces, or modules over some ring, or chain complexes),
and all the structure maps of the $n$-category are compatible with the auxiliary
category structure.
Note that this auxiliary structure is only in dimension $n$; if $\dim(Y) < n$ then 
$\cC(Y; c)$ is just a plain set.

First we must specify requirements for the auxiliary category.
It should have a {\it distributive monoidal structure} in the sense of 
\cite{1010.4527}.
This means that there is a monoidal structure $\otimes$ and also coproduct $\oplus$,
and these two structures interact in the appropriate way.
Examples include 
\begin{itemize}
\item vector spaces (or $R$-modules or chain complexes) with tensor product and direct sum; and
\item topological spaces with product and disjoint union.
\end{itemize}
For convenience, we will also assume that the objects of our auxiliary category are sets with extra structure.
(Otherwise, stating the axioms for identity morphisms becomes more cumbersome.)

Before stating the revised axioms for an $n$-category enriched in a distributive monoidal category,
we need a preliminary definition.
Once we have the above $n$-category axioms for $n{-}1$-morphisms, we can define the 
category $\bbc$ of {\it $n$-balls with boundary conditions}.
Its objects are pairs $(X, c)$, where $X$ is an $n$-ball and $c \in \cl\cC(\bd X)$ is the ``boundary condition".
The morphisms from $(X, c)$ to $(X', c')$, denoted $\Homeo(X; c \to X'; c')$, are
homeomorphisms $f:X\to X'$ such that $f|_{\bd X}(c) = c'$.
 
\begin{axiom}[Enriched $n$-categories]
\label{axiom:enriched}
Let $\cS$ be a distributive symmetric monoidal category.
An $n$-category enriched in $\cS$ satisfies the above $n$-category axioms for $k=0,\ldots,n-1$,
and modifies the axioms for $k=n$ as follows:
\begin{itemize}
\item Morphisms. We have a functor $\cC_n$ from $\bbc$ ($n$-balls with boundary conditions) to $\cS$.
\item Composition. Let $B = B_1\cup_Y B_2$ as in Axiom \ref{axiom:composition}.
Let $Y_i = \bd B_i \setmin Y$.  
Note that $\bd B = Y_1\cup Y_2$.
Let $c_i \in \cC(Y_i)$ with $\bd c_1 = \bd c_2 = d \in \cl\cC(E)$.
Then we have a map
\[
	\gl_Y : \bigoplus_c \cC(B_1; c_1 \bullet c) \otimes \cC(B_2; c_2\bullet c) \to \cC(B; c_1\bullet c_2),
\]
where the sum is over $c\in\cC(Y)$ such that $\bd c = d$.
This map is natural with respect to the action of homeomorphisms and with respect to restrictions.
\end{itemize}
\end{axiom}

\medskip

When the enriching category $\cS$ is chain complexes or topological spaces,
or more generally an appropriate sort of $\infty$-category,
we can modify the extended isotopy axiom \ref{axiom:extended-isotopies}
to require that families of homeomorphisms act
and obtain what we shall call an $A_\infty$ $n$-category.

\noop{
We believe that abstract definitions should be guided by diverse collections
of concrete examples, and a lack of diversity in our present collection of examples of $A_\infty$ $n$-categories
makes us reluctant to commit to an all-encompassing general definition.
Instead, we will give a relatively narrow definition which covers the examples we consider in this paper.
After stating it, we will briefly discuss ways in which it can be made more general.
}

Recall the category $\bbc$ of balls with boundary conditions.
Note that the set of morphisms $\Homeo(X;c \to X'; c')$ from $(X, c)$ to $(X', c')$ is a topological space.
Let $\cS$ be an appropriate $\infty$-category (e.g.\ chain complexes)
and let $\cJ$ be an $\infty$-functor from topological spaces to $\cS$
(e.g.\ the singular chain functor $C_*$).

\begin{axiom}[\textup{\textbf{[$A_\infty$ replacement for Axiom \ref{axiom:extended-isotopies}]}} Families of homeomorphisms act in dimension $n$.]
\label{axiom:families}
For each pair of $n$-balls $X$ and $X'$ and each pair $c\in \cl{\cC}(\bd X)$ and $c'\in \cl{\cC}(\bd X')$ we have an $\cS$-morphism
\[
	\cJ(\Homeo(X;c \to X'; c')) \ot \cC(X; c) \to \cC(X'; c') .
\]
Similarly, we have an $\cS$-morphism
\[
	\cJ(\Coll(X,c)) \ot \cC(X; c) \to \cC(X; c),
\]
where $\Coll(X,c)$ denotes the space of collar maps.
(See below for further discussion.)
These action maps are required to be associative up to coherent homotopy,
and also compatible with composition (gluing) in the sense that
a diagram like the one in Theorem \ref{thm:CH} commutes.
\end{axiom}

We now describe the topology on $\Coll(X; c)$.
We retain notation from the above definition of collar map (after Axiom \ref{axiom:isotopy-preliminary}).
Each collaring homeomorphism $X \cup (Y\times J) \to X$ determines a map from points $p$ of $\bd X$ to
(possibly length zero) embedded intervals in $X$ terminating at $p$.
If $p \in Y$ this interval is the image of $\{p\}\times J$.
If $p \notin Y$ then $p$ is assigned the length zero interval $\{p\}$.
Such collections of intervals have a natural topology, and $\Coll(X; c)$ inherits its topology from this.
Note in particular that parts of the collar are allowed to shrink continuously to zero length.
(This is the real content; if nothing shrinks to zero length then the action of families of collar
maps follows from the action of families of homeomorphisms and compatibility with gluing.)

The $k=n$ case of Axiom \ref{axiom:morphisms} posits a {\it strictly} associative action of {\it sets}
$\Homeo(X;c\to X'; c') \times \cC(X; c) \to \cC(X'; c')$, and at first it might seem that this would force the above
action of $\cJ(\Homeo(X;c\to X'; c'))$ to be strictly associative as well (assuming the two actions are compatible).
In fact, compatibility implies less than this.
For simplicity, assume that $\cJ$ is $C_*$, the singular chains functor.
(This is the example most relevant to this paper.)
Then compatibility implies that the action of $C_*(\Homeo(X;c\to X'; c'))$ agrees with the action
of $C_0(\Homeo(X;c\to X'; c'))$ coming from Axiom \ref{axiom:morphisms}, so we only require associativity in degree zero.
And indeed, this is true for our main example of an $A_\infty$ $n$-category based on the blob construction (see Example \ref{ex:blob-complexes-of-balls} below).
Stating this sort of compatibility for general $\cS$ and $\cJ$ requires further assumptions, 
such as the forgetful functor from $\cS$ to sets having a left adjoint, and $\cS$ having an internal Hom.

An alternative (due to Peter Teichner) is to say that Axiom \ref{axiom:families} 
supersedes the $k=n$ case of Axiom \ref{axiom:morphisms}; in dimension $n$ we just have a
functor $\bbc \to \cS$ of $A_\infty$ 1-categories.
(This assumes some prior notion of $A_\infty$ 1-category.)
We are not currently aware of any examples which require this sort of greater generality, so we think it best
to refrain from settling on a preferred version of the axiom until
we have a greater variety of examples to guide the choice.

Note that if we think of an ordinary 1-category as an $A_\infty$ 1-category where $k$-morphisms are identities for $k>1$,
then Axiom \ref{axiom:families} implies Axiom \ref{axiom:extended-isotopies}.

Another variant of the above axiom would be to drop the ``up to homotopy" and require a strictly associative action. 
In fact, the alternative construction $\btc_*(X)$ of the blob complex described in \S \ref{ss:alt-def} 
gives $n$-categories as in Example \ref{ex:blob-complexes-of-balls} which satisfy this stronger axiom. 
For future reference we make the following definition.

\begin{defn}
A {\em strict $A_\infty$ $n$-category} is one in which the actions of Axiom \ref{axiom:families} are strictly associative.
\end{defn}

\noop{
Note that if we take homology of chain complexes, we turn an $A_\infty$ $n$-category
into a ordinary $n$-category (enriched over graded groups).
In a different direction, if we enrich over topological spaces instead of chain complexes,
we get a space version of an $A_\infty$ $n$-category, with $\Homeo_\bd(X)$ acting 
instead of  $C_*(\Homeo_\bd(X))$.
Taking singular chains converts such a space type $A_\infty$ $n$-category into a chain complex
type $A_\infty$ $n$-category.
}

\medskip

We define a $j$ times monoidal $n$-category to be an $(n{+}j)$-category $\cC$ where
$\cC(X)$ is a trivial 1-element set if $X$ is a $k$-ball with $k<j$.
See Example \ref{ex:bord-cat}.

\medskip

The alert reader will have already noticed that our definition of an (ordinary) $n$-category
is extremely similar to our definition of a system of fields.
There are two differences.
First, for the $n$-category definition we restrict our attention to balls
(and their boundaries), while for fields we consider all manifolds.
Second,  in the category definition we directly impose isotopy
invariance in dimension $n$, while in the fields definition we 
instead remember a subspace of local relations which contain differences of isotopic fields. 
(Recall that the compensation for this complication is that we can demand that the gluing map for fields is injective.)
Thus
\begin{lem}
\label{lem:ncat-from-fields}
A system of fields and local relations $(\cF,U)$ determines an $n$-category $\cC_ {\cF,U}$ simply by restricting our attention to
balls and, at level $n$, quotienting out by the local relations:
\begin{align*}
\cC_{\cF,U}(B^k) & = \begin{cases}\cF(B) & \text{when $k<n$,} \\ \cF(B) / U(B) & \text{when $k=n$.}\end{cases}
\end{align*}
\end{lem}
This $n$-category can be thought of as the local part of the fields.
Conversely, given a disk-like $n$-category we can construct a system of fields via 
a colimit construction; see \S \ref{ss:ncat_fields} below.

\medskip

In the $n$-category axioms above we have intermingled data and properties for expository reasons.
Here's a summary of the definition which segregates the data from the properties.
We also remind the reader of the inductive nature of the definition: All the data for $k{-}1$-morphisms must be in place
before we can describe the data for $k$-morphisms.

An $n$-category consists of the following data:
\begin{itemize}
\item functors $\cC_k$ from $k$-balls to sets, $0\le k\le n$ (Axiom \ref{axiom:morphisms});
\item boundary natural transformations $\cC_k \to \cl{\cC}_{k-1} \circ \bd$ (Axiom \ref{nca-boundary});
\item ``composition'' or ``gluing'' maps $\gl_Y : \cC(B_1)\trans E \times_{\cC(Y)} \cC(B_2)\trans E \to \cC(B_1\cup_Y B_2)\trans E$ (Axiom \ref{axiom:composition});
\item ``product'' or ``identity'' maps $\pi^*:\cC(X)\to \cC(E)$ for each pinched product $\pi:E\to X$ (Axiom \ref{axiom:product});
\item if enriching in an auxiliary category, additional structure on $\cC_n(X; c)$ (Axiom \ref{axiom:enriched});
\item in the $A_\infty$ case, actions of the topological spaces of homeomorphisms preserving boundary conditions
and collar maps (Axiom \ref{axiom:families}).
\end{itemize}
The above data must satisfy the following conditions.
\begin{itemize}
\item The gluing maps are compatible with actions of homeomorphisms and boundary 
restrictions (Axiom \ref{axiom:composition}).
\item For $k<n$ the gluing maps are injective (Axiom \ref{axiom:composition}).
\item The gluing maps are strictly associative (Axiom \ref{nca-assoc}).
\item The product maps are associative and also compatible with homeomorphism actions, gluing and restriction (Axiom \ref{axiom:product}).
\item If enriching in an auxiliary category, all of the data should be compatible 
with the auxiliary category structure on $\cC_n(X; c)$ (Axiom \ref{axiom:enriched}).
\item The possible splittings of a morphism satisfy various conditions (Axiom \ref{axiom:splittings}).
\item For ordinary categories, invariance of $n$-morphisms under extended isotopies 
and collar maps (Axiom \ref{axiom:extended-isotopies}).
\end{itemize}

\subsection{Examples of \texorpdfstring{$n$}{n}-categories}
\label{ss:ncat-examples}

We now describe several classes of examples of $n$-categories satisfying our axioms.
We typically specify only the morphisms; the rest of the data for the category
(restriction maps, gluing, product morphisms, action of homeomorphisms) is usually obvious.

\begin{example}[Maps to a space]
\rm
\label{ex:maps-to-a-space}%
Let $T$ be a topological space.
We define $\pi_{\leq n}(T)$, the fundamental $n$-category of $T$, as follows.
For $X$ a $k$-ball with $k < n$, define $\pi_{\leq n}(T)(X)$ to be the set of 
all continuous maps from $X$ to $T$.
For $X$ an $n$-ball define $\pi_{\leq n}(T)(X)$ to be continuous maps from $X$ to $T$ modulo
homotopies fixed on $\bd X$.
(Note that homotopy invariance implies isotopy invariance.)
For $a\in \cC(X)$ define the product morphism $a\times D \in \cC(X\times D)$ to
be $a\circ\pi_X$, where $\pi_X : X\times D \to X$ is the projection.
\end{example}

\noop{
Recall we described a system of fields and local relations based on maps to $T$ in Example \ref{ex:maps-to-a-space(fields)} above.
Constructing a system of fields from $\pi_{\leq n}(T)$ recovers that example.
\nn{shouldn't this go elsewhere?  we haven't yet discussed constructing a system of fields from
an n-cat}
}

\begin{example}[Maps to a space, with a fiber] \label{ex:maps-with-fiber}
\rm
\label{ex:maps-to-a-space-with-a-fiber}%
We can modify the example above, by fixing a
closed $m$-manifold $F$, and defining $\pi^{\times F}_{\leq n}(T)(X) = \Maps(X \times F \to T)$, 
otherwise leaving the definition in Example \ref{ex:maps-to-a-space} unchanged.
Taking $F$ to be a point recovers the previous case.
\end{example}

\begin{example}[Linearized, twisted, maps to a space]
\rm
\label{ex:linearized-maps-to-a-space}%
We can linearize Examples \ref{ex:maps-to-a-space} and \ref{ex:maps-to-a-space-with-a-fiber} as follows.
Let $\alpha$ be an $(n{+}m{+}1)$-cocycle on $T$ with values in a ring $R$
(have in mind the trivial cocycle).
For $X$ of dimension less than $n$ define $\pi^{\alpha, \times F}_{\leq n}(T)(X)$ as before, ignoring $\alpha$.
For $X$ an $n$-ball and $c\in \Maps(\bdy X \times F \to T)$ define $\pi^{\alpha, \times F}_{\leq n}(T)(X; c)$ to be
the $R$-module of finite linear combinations of continuous maps from $X\times F$ to $T$,
modulo the relation that if $a$ is homotopic to $b$ (rel boundary) via a homotopy
$h: X\times F\times I \to T$, then $a = \alpha(h)b$.
(In order for this to be well-defined we must choose $\alpha$ to be zero on degenerate simplices.
Alternatively, we could equip the balls with fundamental classes.)
\end{example}

\begin{example}[$n$-categories from TQFTs]
\rm
\label{ex:ncats-from-tqfts}%
Let $\cF$ be a TQFT in the sense of \S\ref{sec:fields}: an $n$-dimensional 
system of fields (also denoted $\cF$) and local relations.
Let $W$ be an $n{-}j$-manifold.
Define the $j$-category $\cF(W)$ as follows.
If $X$ is a $k$-ball with $k<j$, let $\cF(W)(X) \deq \cF(W\times X)$.
If $X$ is a $j$-ball and $c\in \cl{\cF(W)}(\bd X)$,
let $\cF(W)(X; c) \deq A_\cF(W\times X; c)$.
\end{example}

This last example generalizes Lemma \ref{lem:ncat-from-fields} above which produced an $n$-category from an $n$-dimensional system of fields and local relations. Taking $W$ to be the point recovers that statement.

The next example is only intended to be illustrative, as we don't specify 
which definition of a ``traditional $n$-category with strong duality" we intend.
\begin{example}[Traditional $n$-categories]
\rm
\label{ex:traditional-n-categories}
Given a ``traditional $n$-category with strong duality" $C$
define $\cC(X)$, for $X$ a $k$-ball with $k < n$,
to be the set of all $C$-labeled embedded cell complexes of $X$ (c.f. \S \ref{sec:fields}).
For $X$ an $n$-ball and $c\in \cl{\cC}(\bd X)$, define $\cC(X; c)$ to be finite linear
combinations of $C$-labeled embedded cell complexes of $X$
modulo the kernel of the evaluation map.
Define a product morphism $a\times D$, for $D$ an $m$-ball, to be the product of the cell complex of $a$ with $D$,
with each cell labelled according to the corresponding cell for $a$.
(These two cells have the same codimension.)
More generally, start with an $n{+}m$-category $C$ and a closed $m$-manifold $F$.
Define $\cC(X)$, for $\dim(X) < n$,
to be the set of all $C$-labeled embedded cell complexes of $X\times F$.
Define $\cC(X; c)$, for $X$ an $n$-ball,
to be the dual Hilbert space $A(X\times F; c)$.
(See \S\ref{sec:constructing-a-tqft}.)
\end{example}

\begin{example}[The bordism $n$-category of $d$-manifolds, ordinary version]
\label{ex:bord-cat}
\rm
\label{ex:bordism-category}
For a $k$-ball $X$, $k<n$, define $\Bord^{n,d}(X)$ to be the set of all $(d{-}n{+}k)$-dimensional PL
submanifolds $W$ of $X\times \Real^\infty$ such that $\bd W = W \cap \bd X \times \Real^\infty$.
For an $n$-ball $X$ define $\Bord^{n,d}(X)$ to be homeomorphism classes (rel boundary) of such $d$-dimensional submanifolds;
we identify $W$ and $W'$ if $\bd W = \bd W'$ and there is a homeomorphism
$W \to W'$ which restricts to the identity on the boundary.
For $n=1$ we have the familiar bordism 1-category of $d$-manifolds.
The case $n=d$ captures the $n$-categorical nature of bordisms.
The case $n > 2d$ captures the full symmetric monoidal $n$-category structure.
\end{example}
\begin{rem}
Working with the smooth bordism category would require careful attention to either collars, corners or halos.
\end{rem}




\begin{example}[Chains (or space) of maps to a space]
\rm
\label{ex:chains-of-maps-to-a-space}
We can modify Example \ref{ex:maps-to-a-space} above to define the fundamental $A_\infty$ $n$-category $\pi^\infty_{\le n}(T)$ of a topological space $T$.
For a $k$-ball $X$, with $k < n$, the set $\pi^\infty_{\leq n}(T)(X)$ is just $\Maps(X \to T)$.
Define $\pi^\infty_{\leq n}(T)(X; c)$ for an $n$-ball $X$ and $c \in \pi^\infty_{\leq n}(T)(\bdy X)$ to be the chain complex
\[
	C_*(\Maps_c(X \to T)),
\]
where $\Maps_c$ denotes continuous maps restricting to $c$ on the boundary,
and $C_*$ denotes singular chains.
Alternatively, if we take the $n$-morphisms to be simply $\Maps_c(X \to T)$, 
we get an $A_\infty$ $n$-category enriched over spaces.
\end{example}

See also Theorem \ref{thm:map-recon} below, recovering $C_*(\Maps(M \to T))$ up to 
homotopy as the blob complex of $M$ with coefficients in $\pi^\infty_{\le n}(T)$.

Instead of using the TQFT invariant $\cA$ as in Example \ref{ex:ncats-from-tqfts} above, we can turn an $n$-dimensional system of fields and local relations into an $A_\infty$ $n$-category using the blob complex. With a codimension $k$ fiber, we obtain an $A_\infty$ $k$-category:

\begin{example}[Blob complexes of balls (with a fiber)]
\rm
\label{ex:blob-complexes-of-balls}
Fix an $n{-}k$-dimensional manifold $F$ and an $n$-dimensional system of fields $\cE$.
We will define an $A_\infty$ $k$-category $\cC$.
When $X$ is an $m$-ball, with $m<k$, define $\cC(X) = \cE(X\times F)$.
When $X$ is a $k$-ball,
define $\cC(X; c) = \bc^\cE_*(X\times F; c)$
where $\bc^\cE_*$ denotes the blob complex based on $\cE$.
\end{example}

This example will be used in Theorem \ref{thm:product} below, which allows us to compute the blob complex of a product.
Notice that with $F$ a point, the above example is a construction turning an ordinary 
$n$-category $\cC$ into an $A_\infty$ $n$-category.
We think of this as providing a ``free resolution" 
of the ordinary $n$-category. 
In fact, there is also a trivial, but mostly uninteresting, way to do this: 
we can think of each vector space associated to an $n$-ball as a chain complex concentrated in degree $0$, 
and let $\CH{B}$ act trivially. 

Beware that the ``free resolution" of the ordinary $n$-category $\pi_{\leq n}(T)$ 
is not the $A_\infty$ $n$-category $\pi^\infty_{\leq n}(T)$.
It's easy to see that with $n=0$, the corresponding system of fields is just 
linear combinations of connected components of $T$, and the local relations are trivial.
There's no way for the blob complex to magically recover all the data of $\pi^\infty_{\leq 0}(T) \iso C_* T$.

\begin{example}[The bordism $n$-category of $d$-manifolds, $A_\infty$ version]
\rm
\label{ex:bordism-category-ainf}
As in Example \ref{ex:bord-cat}, for $X$ a $k$-ball, $k<n$, we define $\Bord^{n,d}_\infty(X)$
to be the set of all $(d{-}n{+}k)$-dimensional
submanifolds $W$ of $X\times \Real^\infty$ such that $\bd W = W \cap \bd X \times \Real^\infty$.
For an $n$-ball $X$ with boundary condition $c$ 
define $\Bord^{n,d}_\infty(X; c)$ to be the space of all $d$-dimensional
submanifolds $W$ of $X\times \Real^\infty$ such that 
$W$ coincides with $c$ at $\bd X \times \Real^\infty$.
(The topology on this space is induced by ambient isotopy rel boundary.
This is homotopy equivalent to a disjoint union of copies $\mathrm{B}\!\Homeo(W')$, where
$W'$ runs though representatives of homeomorphism types of such manifolds.)
\end{example}

Let $\cE\cB_n$ be the operad of smooth embeddings of $k$ (little)
copies of the standard $n$-ball $B^n$ into another (big) copy of $B^n$.
(We require that the interiors of the little balls be disjoint, but their 
boundaries are allowed to meet.
Note in particular that the space for $k=1$ contains a copy of $\Diff(B^n)$, namely
the embeddings of a ``little" ball with image all of the big ball $B^n$.
(But note also that this inclusion is not
necessarily a homotopy equivalence.))
The operad $\cE\cB_n$ is homotopy equivalent to the standard framed little $n$-ball operad:
by shrinking the little balls (precomposing them with dilations), 
we see that both operads are homotopic to the space of $k$ framed points
in $B^n$.
It is easy to see that $n$-fold loop spaces $\Omega^n(T)$  have
an action of $\cE\cB_n$.

\begin{example}[$E_n$ algebras]
\rm
\label{ex:e-n-alg}
Let $A$ be an $\cE\cB_n$-algebra.
Note that this implies a $\Diff(B^n)$ action on $A$, 
since $\cE\cB_n$ contains a copy of $\Diff(B^n)$.
We will define a strict $A_\infty$ $n$-category $\cC^A$.
(We enrich in topological spaces, though this could easily be adapted to, say, chain complexes.)
If $X$ is a ball of dimension $k<n$, define $\cC^A(X)$ to be a point.
In other words, the $k$-morphisms are trivial for $k<n$.
If $X$ is an $n$-ball, we define $\cC^A(X)$ via a colimit construction.
(Plain colimit, not homotopy colimit.)
Let $J$ be the category whose objects are embeddings of a disjoint union of copies of 
the standard ball $B^n$ into $X$, and whose morphisms are given by engulfing some of the 
embedded balls into a single larger embedded ball.
To each object of $J$ we associate $A^{\times m}$ (where $m$ is the number of balls), and
to each morphism of $J$ we associate a morphism coming from the $\cE\cB_n$ action on $A$.
Alternatively and more simply, we could define $\cC^A(X)$ to be 
$\Diff(B^n\to X)\times A$ modulo the diagonal action of $\Diff(B^n)$.
The remaining data for the $A_\infty$ $n$-category 
--- composition and $\Diff(X\to X')$ action ---
also comes from the $\cE\cB_n$ action on $A$.

Conversely, one can show that a disk-like strict $A_\infty$ $n$-category $\cC$, where the $k$-morphisms
$\cC(X)$ are trivial (single point) for $k<n$, gives rise to 
an $\cE\cB_n$-algebra.
Let $A = \cC(B^n)$, where $B^n$ is the standard $n$-ball.
We must define maps
\[
	\cE\cB_n^k \times A \times \cdots \times A \to A ,
\]
where $\cE\cB_n^k$ is the $k$-th space of the $\cE\cB_n$ operad.
Let $(b, a_1,\ldots,a_k)$ be a point of $\cE\cB_n^k \times A \times \cdots \times A \to A$.
The $i$-th embedding of $b$ together with $a_i$ determine an element of $\cC(B_i)$, 
where $B_i$ denotes the $i$-th little ball.
Using composition of $n$-morphsims in $\cC$, and padding the spaces between the little balls with the 
(essentially unique) identity $n$-morphism of $\cC$, we can construct a well-defined element
of $\cC(B^n) = A$.

If we apply the homotopy colimit construction of the next subsection to this example, 
we get an instance of Lurie's topological chiral homology construction.
\end{example}

\subsection{From balls to manifolds}
\label{ss:ncat_fields} \label{ss:ncat-coend}
In this section we show how to extend an $n$-category $\cC$ as described above 
(of either the ordinary or $A_\infty$ variety) to an invariant of manifolds, which we denote by $\cl{\cC}$.
This extension is a certain colimit, and the arrow in the notation is intended as a reminder of this.

In the case of ordinary $n$-categories, this construction factors into a construction of a 
system of fields and local relations, followed by the usual TQFT definition of a 
vector space invariant of manifolds given as Definition \ref{defn:TQFT-invariant}.
For an $A_\infty$ $n$-category, $\cl{\cC}$ is defined using a homotopy colimit instead.
Recall that we can take an ordinary $n$-category $\cC$ and pass to the ``free resolution", 
an $A_\infty$ $n$-category $\bc_*(\cC)$, by computing the blob complex of balls 
(recall Example \ref{ex:blob-complexes-of-balls} above).
We will show in Corollary \ref{cor:new-old} below that the homotopy colimit invariant 
for a manifold $M$ associated to this $A_\infty$ $n$-category is actually the 
same as the original blob complex for $M$ with coefficients in $\cC$.

Recall that we've already anticipated this construction Subsection \ref{ss:n-cat-def}, 
inductively defining $\cl{\cC}$ on $k$-spheres in terms of $\cC$ on $k$-balls, 
so that we can state the boundary axiom for $\cC$ on $k+1$-balls.

\medskip

We will first define the {\it decomposition poset} $\cell(W)$ for any $k$-manifold $W$, for $1 \leq k \leq n$. 
An $n$-category $\cC$ provides a functor from this poset to the category of sets, 
and we  will define $\cl{\cC}(W)$ as a suitable colimit 
(or homotopy colimit in the $A_\infty$ case) of this functor. 
We'll later give a more explicit description of this colimit.
In the case that the $n$-category $\cC$ is enriched (e.g. associates vector spaces or chain 
complexes to $n$-balls with boundary data), 
then the resulting colimit is also enriched, that is, the set associated to $W$ splits into 
subsets according to boundary data, and each of these subsets has the appropriate structure 
(e.g. a vector space or chain complex).

Recall (Definition \ref{defn:gluing-decomposition}) that a {\it ball decomposition} of $W$ is a 
sequence of gluings $M_0\to M_1\to\cdots\to M_m = W$ such that $M_0$ is a disjoint union of balls
$\du_a X_a$.
Abusing notation, we let $X_a$ denote both the ball (component of $M_0$) and
its image in $W$ (which is not necessarily a ball --- parts of $\bd X_a$ may have been glued together).
Define a {\it permissible decomposition} of $W$ to be a map
\[
	\coprod_a X_a \to W,
\]
which can be completed to a ball decomposition $\du_a X_a = M_0\to\cdots\to M_m = W$.
We further require that $\du_a (X_a \cap \bd W) \to \bd W$ 
can be completed to a (not necessarily ball) decomposition of $\bd W$.
(So, for example, in Example \ref{sin1x-example} if we take $W = B\cup C\cup D$ then $B\du C\du D \to W$
is not allowed since $D\cap \bd W$ is not a submanifold.)
Roughly, a permissible decomposition is like a ball decomposition where we don't care in which order the balls
are glued up to yield $W$, so long as there is some (non-pathological) way to glue them.

(Every smooth or PL manifold has a ball decomposition, but certain topological manifolds (e.g.\ non-smoothable
topological 4-manifolds) do not have ball decompositions.
For such manifolds we have only the empty colimit.)

We want the category (poset) of decompositions of $W$ to be small, so when we say decomposition we really
mean isomorphism class of decomposition.
Isomorphisms are defined in the obvious way: a collection of homeomorphisms $M_i\to M_i'$ which commute
with the gluing maps $M_i\to M_{i+1}$ and $M'_i\to M'_{i+1}$.

Given permissible decompositions $x = \{X_a\}$ and $y = \{Y_b\}$ of $W$, we say that $x$ is a refinement
of $y$, or write $x \le y$, if there is a ball decomposition $\du_a X_a = M_0\to\cdots\to M_m = W$
with $\du_b Y_b = M_i$ for some $i$,
and with $M_0, M_1, \ldots, M_i$ each being a disjoint union of balls.

\begin{defn}
The poset $\cell(W)$ has objects the permissible decompositions of $W$, 
and a unique morphism from $x$ to $y$ if and only if $x$ is a refinement of $y$.
See Figure \ref{partofJfig}.
\end{defn}

\begin{figure}[t]
\begin{equation*}
\mathfig{.63}{ncat/zz2}
\end{equation*}
\caption{A small part of $\cell(W)$}
\label{partofJfig}
\end{figure}

An $n$-category $\cC$ determines 
a functor $\psi_{\cC;W}$ from $\cell(W)$ to the category of sets 
(possibly with additional structure if $k=n$).
Let $x = \{X_a\}$ be a permissible decomposition of $W$ (i.e.\ object of $\cD(W)$).
We will define $\psi_{\cC;W}(x)$ to be a certain subset of $\prod_a \cC(X_a)$.
Roughly speaking, $\psi_{\cC;W}(x)$ is the subset where the restriction maps from
$\cC(X_a)$ and $\cC(X_b)$ agree whenever some part of $\bd X_a$ is glued to some part of $\bd X_b$.
(Keep in mind that perhaps $a=b$.)
Since we allow decompositions in which the intersection of $X_a$ and $X_b$ might be messy 
(see Example \ref{sin1x-example}), we must define $\psi_{\cC;W}(x)$ in a more roundabout way.

Inductively, we may assume that we have already defined the colimit $\cl\cC(M)$ for $k{-}1$-manifolds $M$.
(To start the induction, we define $\cl\cC(M)$, where $M = \du_a P_a$ is a 0-manifold and each $P_a$ is
a 0-ball, to be $\prod_a \cC(P_a)$.)
We also assume, inductively, that we have gluing and restriction maps for colimits of $k{-}1$-manifolds.
Gluing and restriction maps for colimits of $k$-manifolds will be defined later in this subsection.

Let $\du_a X_a = M_0\to\cdots\to M_m = W$ be a ball decomposition compatible with $x$.
Let $\bd M_i = N_i \cup Y_i \cup Y'_i$, where $Y_i$ and $Y'_i$ are glued together to produce $M_{i+1}$.
We will define $\psi_{\cC;W}(x)$ to be the subset of $\prod_a \cC(X_a)$ which satisfies a series of conditions
related to the gluings $M_{i-1} \to M_i$, $1\le i \le m$.
By Axiom \ref{nca-boundary}, we have a map
\[
	\prod_a \cC(X_a) \to \cl\cC(\bd M_0) .
\]
The first condition is that the image of $\psi_{\cC;W}(x)$ in $\cl\cC(\bd M_0)$ is splittable
along $\bd Y_0$ and $\bd Y'_0$, and that the restrictions to $\cl\cC(Y_0)$ and $\cl\cC(Y'_0)$ agree
(with respect to the identification of $Y_0$ and $Y'_0$ provided by the gluing map). 

On the subset of $\prod_a \cC(X_a)$ which satisfies the first condition above, we have a restriction
map to $\cl\cC(N_0)$ which we can compose with the gluing map 
$\cl\cC(N_0) \to \cl\cC(\bd M_1)$.
The second condition is that the image of $\psi_{\cC;W}(x)$ in $\cl\cC(\bd M_1)$ is splittable
along $\bd Y_1$ and $\bd Y'_1$, and that the restrictions to $\cl\cC(Y_1)$ and $\cl\cC(Y'_1)$ agree
(with respect to the identification of $Y_1$ and $Y'_1$ provided by the gluing map). 
The $i$-th condition is defined similarly.
Note that these conditions depend only on the boundaries of elements of $\prod_a \cC(X_a)$.

We define $\psi_{\cC;W}(x)$ to be the subset of $\prod_a \cC(X_a)$ which satisfies the 
above conditions for all $i$ and also all 
ball decompositions compatible with $x$.
(If $x$ is a nice, non-pathological cell decomposition, then it is easy to see that gluing
compatibility for one ball decomposition implies gluing compatibility for all other ball decompositions.
Rather than try to prove a similar result for arbitrary
permissible decompositions, we instead require compatibility with all ways of gluing up the decomposition.)

If $x$ is a refinement of $y$, the map $\psi_{\cC;W}(x) \to \psi_{\cC;W}(y)$ 
is given by the composition maps of $\cC$.
This completes the definition of the functor $\psi_{\cC;W}$.

If $k=n$ in the above definition and we are enriching in some auxiliary category, 
we need to say a bit more.
We can rewrite the colimit as
\[  
	\psi_{\cC;W}(x) \deq \coprod_\beta \prod_a \cC(X_a; \beta) ,
\]  
where $\beta$ runs through 
boundary conditions on $\du_a X_a$ which are compatible with gluing as specified above
and $\cC(X_a; \beta)$
means the subset of $\cC(X_a)$ whose restriction to $\bd X_a$ agrees with $\beta$.
If we are enriching over $\cS$ and $k=n$, then $\cC(X_a; \beta)$ is an object in 
$\cS$ and the coproduct and product in the above expression should be replaced by the appropriate
operations in $\cS$ (e.g. direct sum and tensor product if $\cS$ is Vect).

Finally, we construct $\cl{\cC}(W)$ as the appropriate colimit of $\psi_{\cC;W}$:

\begin{defn}[System of fields functor]
\label{def:colim-fields}
If $\cC$ is an $n$-category enriched in sets or vector spaces, $\cl{\cC}(W)$ is the usual colimit of the functor $\psi_{\cC;W}$.
That is, for each decomposition $x$ there is a map
$\psi_{\cC;W}(x)\to \cl{\cC}(W)$, these maps are compatible with the refinement maps
above, and $\cl{\cC}(W)$ is universal with respect to these properties.
\end{defn}

\begin{defn}[System of fields functor, $A_\infty$ case]
When $\cC$ is an $A_\infty$ $n$-category, $\cl{\cC}(W)$ for $W$ a $k$-manifold with $k < n$ 
is defined as above, as the colimit of $\psi_{\cC;W}$.
When $W$ is an $n$-manifold, the chain complex $\cl{\cC}(W)$ is the homotopy colimit of the functor $\psi_{\cC;W}$.
\end{defn}


\medskip

We must now define restriction maps $\bd : \cl{\cC}(W) \to \cl{\cC}(\bd W)$ and gluing maps.

Let $y\in \cl{\cC}(W)$.
Choose a representative of $y$ in the colimit: a permissible decomposition $\du_a X_a \to W$ and elements
$y_a \in \cC(X_a)$.
By assumption, $\du_a (X_a \cap \bd W) \to \bd W$ can be completed to a decomposition of $\bd W$.
Let $r(y_a) \in \cl\cC(X_a \cap \bd W)$ be the restriction.
Choose a representative of $r(y_a)$ in the colimit $\cl\cC(X_a \cap \bd W)$: a permissible decomposition
$\du_b Q_{ab} \to X_a \cap \bd W$ and elements $z_{ab} \in \cC(Q_{ab})$.
Then $\du_{ab} Q_{ab} \to \bd W$ is a permissible decomposition of $\bd W$ and $\{z_{ab}\}$ represents
an element of $\cl{\cC}(\bd W)$.  Define $\bd y$ to be this element.
It is not hard to see that it is independent of the various choices involved.

Note that since we have already (inductively) defined gluing maps for colimits of $k{-}1$-manifolds,
we can also define restriction maps from $\cl{\cC}(W)\trans{}$ to $\cl{\cC}(Y)$ where $Y$ is a codimension 0 
submanifold of $\bd W$.

Next we define gluing maps for colimits of $k$-manifolds.
Let $W = W_1 \cup_Y W_2$.
Let $y_i \in \cl\cC(W_i)$ and assume that the restrictions of $y_1$ and $y_2$ to $\cl\cC(Y)$ agree.
We want to define $y_1\bullet y_2 \in \cl\cC(W)$.
Choose a permissible decomposition $\du_a X_{ia} \to W_i$ and elements 
$y_{ia} \in \cC(X_{ia})$ representing $y_i$.
It might not be the case that $\du_{ia} X_{ia} \to W$ is a permissible decomposition of $W$,
since intersections of the pieces with $\bd W$ might not be well-behaved.
However, using the fact that $\bd y_i$ splits along $\bd Y$ and applying Axiom \ref{axiom:splittings},
we can choose the decomposition $\du_{a} X_{ia}$ so that its restriction to $\bd W_i$ is a refinement
of the splitting along $\bd Y$, and this implies that the combined decomposition $\du_{ia} X_{ia}$
is permissible.
We can now define the gluing $y_1\bullet y_2$ in the obvious way, and a further application of Axiom \ref{axiom:splittings}
shows that this is independent of the choices of representatives of $y_i$.

\medskip

We now give more concrete descriptions of the above colimits.

In the non-enriched case (e.g.\ $k<n$), where each $\cC(X_a; \beta)$ is just a set,
the colimit is
\[
	\cl{\cC}(W,c) = \left( \coprod_x \coprod_\beta \prod_a \cC(X_a; \beta) \right) \Bigg/ \sim ,
\]
where $x$ runs through decompositions of $W$, and $\sim$ is the obvious equivalence relation 
induced by refinement and gluing.
If $\cC$ is enriched over, for example, vector spaces and $W$ is an $n$-manifold, 
we can take
\begin{equation*}
	\cl{\cC}(W,c) = \left( \bigoplus_x \bigoplus_\beta \bigotimes_a \cC(X_a; \beta) \right) \Bigg/ K,
\end{equation*}
where $K$ is the vector space spanned by elements $a - g(a)$, with
$a\in \psi_{\cC;W,c}(x)$ for some decomposition $x$, and $g: \psi_{\cC;W,c}(x)
\to \psi_{\cC;W,c}(y)$ is the value of $\psi_{\cC;W,c}$ on some antirefinement $x \leq y$.

In the $A_\infty$ case, enriched over chain complexes, the concrete description of the homotopy colimit
is more involved.
We will describe two different (but homotopy equivalent) versions of the homotopy colimit of $\psi_{\cC;W}$.
The first is the usual one, which works for any indexing category.
The second construction, which we call the {\it local} homotopy colimit,
is more closely related to the blob complex
construction of \S \ref{sec:blob-definition} and takes advantage of local (gluing) properties
of the indexing category $\cell(W)$.

Define an $m$-sequence in $W$ to be a sequence $x_0 \le x_1 \le \dots \le x_m$ of permissible decompositions of $W$.
Such sequences (for all $m$) form a simplicial set in $\cell(W)$.
Define $\cl{\cC}(W)$ as a vector space via
\[
	\cl{\cC}(W) = \bigoplus_{(x_i)} \psi_{\cC;W}(x_0)[m] ,
\]
where the sum is over all $m$ and all $m$-sequences $(x_i)$, and each summand is degree shifted by $m$. 
Elements of a summand indexed by an $m$-sequence will be call $m$-simplices.
We endow $\cl{\cC}(W)$ with a differential which is the sum of the differential of the $\psi_{\cC;W}(x_0)$
summands plus another term using the differential of the simplicial set of $m$-sequences.
More specifically, if $(a, \bar{x})$ denotes an element in the $\bar{x}$
summand of $\cl{\cC}(W)$ (with $\bar{x} = (x_0,\dots,x_k)$), define
\[
	\bd (a, \bar{x}) = (\bd a, \bar{x}) + (-1)^{\deg{a}} (g(a), d_0(\bar{x})) + (-1)^{\deg{a}} \sum_{j=1}^k (-1)^{j} (a, d_j(\bar{x})) ,
\]
where $d_j(\bar{x}) = (x_0,\dots,x_{j-1},x_{j+1},\dots,x_k)$ and $g: \psi_\cC(x_0)\to \psi_\cC(x_1)$
is the usual gluing map coming from the antirefinement $x_0 \le x_1$.

We can think of this construction as starting with a disjoint copy of a complex for each
permissible decomposition (the 0-simplices).
Then we glue these together with mapping cylinders coming from gluing maps
(the 1-simplices).
Then we kill the extra homology we just introduced with mapping 
cylinders between the mapping cylinders (the 2-simplices), and so on.

Next we describe the local homotopy colimit.
This is similar to the usual homotopy colimit, but using
a cone-product set (Remark \ref{blobsset-remark}) in place of a simplicial set.
The cone-product $m$-polyhedra for the set are pairs $(x, E)$, where $x$ is a decomposition of $W$
and $E$ is an $m$-blob diagram such that each blob is a union of balls of $x$.
(Recall that this means that the interiors of
each pair of blobs (i.e.\ balls) of $E$ are either disjoint or nested.)
To each $(x, E)$ we associate the chain complex $\psi_{\cC;W}(x)$, shifted in degree by $m$.
The boundary has a term for omitting each blob of $E$.
If we omit an innermost blob then we replace $x$ by the formal difference $x - \gl(x)$, where
$\gl(x)$ is obtained from $x$ by gluing together the balls of $x$ contained in the blob we are omitting.
The gluing maps of $\cC$ give us a maps from $\psi_{\cC;W}(x)$ to $\psi_{\cC;W}(\gl(x))$.

One can show that the usual hocolimit and the local hocolimit are homotopy equivalent using an 
Eilenberg-Zilber type subdivision argument.

\medskip

$\cl{\cC}(W)$ is functorial with respect to homeomorphisms of $k$-manifolds. 
Restricting to $k$-spheres, we have now proved Lemma \ref{lem:spheres}.

\begin{lem}
\label{lem:colim-injective}
Let $W$ be a manifold of dimension $j<n$.  Then for each
decomposition $x$ of $W$ the natural map $\psi_{\cC;W}(x)\to \cl{\cC}(W)$ is injective.
\end{lem}
\begin{proof}
$\cl{\cC}(W)$ is a colimit of a diagram of sets, and each of the arrows in the diagram is
injective.
Concretely, the colimit is the disjoint union of the sets (one for each decomposition of $W$),
modulo the relation which identifies the domain of each of the injective maps
with its image.

To save ink and electrons we will simplify notation and write $\psi(x)$ for $\psi_{\cC;W}(x)$.

Suppose $a, \hat{a}\in \psi(x)$ have the same image in $\cl{\cC}(W)$ but $a\ne \hat{a}$.
Then there exist
\begin{itemize}
\item decompositions $x = x_0, x_1, \ldots , x_{k-1}, x_k = x$ and $v_1,\ldots, v_k$ of $W$;
\item anti-refinements $v_i\to x_i$ and $v_i\to x_{i-1}$; and
\item elements $a_i\in \psi(x_i)$ and $b_i\in \psi(v_i)$, with $a_0 = a$ and $a_k = \hat{a}$, 
such that $b_i$ and $b_{i+1}$ both map to (glue up to) $a_i$.
\end{itemize}
In other words, we have a zig-zag of equivalences starting at $a$ and ending at $\hat{a}$.
The idea of the proof is to produce a similar zig-zag where everything antirefines to the same
disjoint union of balls, and then invoke Axiom \ref{nca-assoc} which ensures associativity.

Let $z$ be a decomposition of $W$ which is in general position with respect to all of the 
$x_i$'s and $v_i$'s.
There exist decompositions $x'_i$ and $v'_i$ (for all $i$) such that
\begin{itemize}
\item $x'_i$ antirefines to $x_i$ and $z$;
\item $v'_i$ antirefines to $x'_i$, $x'_{i-1}$ and $v_i$;
\item $b_i$ is the image of some $b'_i\in \psi(v'_i)$; and
\item $a_i$ is the image of some $a'_i\in \psi(x'_i)$, which in turn is the image
of $b'_i$ and $b'_{i+1}$.
\end{itemize}
(This is possible by Axiom \ref{axiom:splittings}.)
Now consider the diagrams
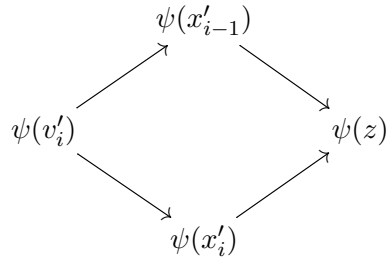
\[ \xymatrix{
	& \psi(x'_{i-1}) \ar[rd] & \\
	\psi(v'_i) \ar[ru] \ar[rd] & & \psi(z) \\
	& \psi(x'_i) \ar[ru] &
} \]
The associativity axiom applied to this diagram implies that $a'_{i-1}$ and $a'_i$
map to the same element $c\in \psi(z)$.
Therefore $a'_0$ and $a'_k$ both map to $c$.
But $a'_0$ and $a'_k$ are both elements of $\psi(x'_0)$ (because $x'_k = x'_0$).
So by the injectivity clause of the composition axiom, we must have that $a'_0 = a'_k$.
But this implies that $a = a_0 = a_k = \hat{a}$, contrary to our assumption that $a\ne \hat{a}$.
\end{proof}


\subsection{Modules}
\label{sec:modules}

\tikzset{marked/.style={line width=3pt,red}}

Next we define ordinary and $A_\infty$ $n$-category modules.
The definition will be very similar to that of $n$-categories,
but with $k$-balls replaced by {\it marked $k$-balls,} defined below.

Our motivating example comes from an $(m{-}n{+}1)$-dimensional manifold $W$ with boundary
in the context of an $m{+}1$-dimensional TQFT.
Such a $W$ gives rise to a module for the $n$-category associated to $\bd W$ (see Example \ref{ex:ncats-from-tqfts}).
This will be explained in more detail as we present the axioms.

Throughout, we fix an $n$-category $\cC$.
For all but one axiom, it doesn't matter whether $\cC$ is an ordinary $n$-category or an $A_\infty$ $n$-category.
We state the final axiom, regarding actions of homeomorphisms, differently in the two cases.

Define a {\it marked $k$-ball} to be a pair $(B, N)$ homeomorphic to the pair
$$(\text{standard $k$-ball}, \text{northern hemisphere in boundary of standard $k$-ball}).$$
We call $B$ the ball and $N$ the marking.
A homeomorphism between marked $k$-balls is a homeomorphism of balls which
restricts to a homeomorphism of markings.

\begin{module-axiom}[Module morphisms] \label{module-axiom-funct}
{For each $1 \le k \le n$, we have a functor $\cM_k$ from 
the category of marked $k$-balls and 
homeomorphisms to the category of sets and bijections.}
\end{module-axiom}

(As with $n$-categories, we will usually omit the subscript $k$.)

For example, let $\cD$ be the TQFT which assigns to a $k$-manifold $N$ the set 
of maps from $N$ to $T$ (for $k\le m$), modulo homotopy (and possibly linearized) if $k=m$
(see Example \ref{ex:maps-with-fiber}).
Let $W$ be an $(m{-}n{+}1)$-dimensional manifold with boundary.
Let $\cC$ be the $n$-category with $\cC(X) \deq \cD(X\times \bd W)$.
Let $\cM(B, N) \deq \cD((B\times \bd W)\cup (N\times W))$.
(The union is along $N\times \bd W$.)
See Figure \ref{blah15}.

\begin{figure}[t]
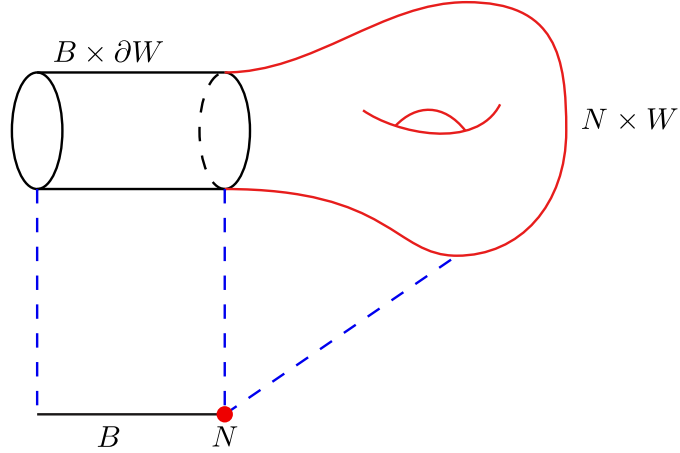

$$\mathfig{.55}{ncat/boundary-collar}$$
\caption{From manifold with boundary collar to marked ball}\label{blah15}\end{figure}

Define the boundary of a marked $k$-ball $(B, N)$ to be the pair $(\bd B \setmin N, \bd N)$.
Call such a thing a {\it marked $k{-}1$-hemisphere}.
(A marked $k{-}1$-hemisphere is, of course, just a $k{-}1$-ball with its entire boundary marked.
We call it a hemisphere instead of a ball because it plays a role analogous
to the $k{-}1$-spheres in the $n$-category definition.)

\begin{lem}
\label{lem:hemispheres}
{For each $1 \le k \le n$, we have a functor $\cl\cM_{k-1}$ from 
the category of marked $k$-hemispheres and 
homeomorphisms to the category of sets and bijections.}
\end{lem}
The proof is exactly analogous to that of Lemma \ref{lem:spheres}, and we omit the details.
We use the same type of colimit construction.

In our example, $\cl\cM(H) = \cD(H\times\bd W \cup \bd H\times W)$.

\begin{module-axiom}[Module boundaries]
{For each marked $k$-ball $M$ we have a map of sets $\bd: \cM(M)\to \cl\cM(\bd M)$.
These maps, for various $M$, comprise a natural transformation of functors.}
\end{module-axiom}

Given $c\in\cl\cM(\bd M)$, let $\cM(M; c) \deq \bd^{-1}(c)$.

If the $n$-category $\cC$ is enriched over some other category (e.g.\ vector spaces),
then for each marked $n$-ball $M=(B,N)$ and $c\in \cC(\bd B \setminus N)$, the set $\cM(M; c)$ should be an object in that category.

\begin{lem}[Boundary from domain and range]
\label{lem:module-boundary}
{Let $H = M_1 \cup_E M_2$, where $H$ is a marked $k{-}1$-hemisphere ($1\le k\le n$),
$M_i$ is a marked $k{-}1$-ball, and $E = M_1\cap M_2$ is a marked $k{-}2$-hemisphere.
Let $\cM(M_1) \times_{\cM(E)} \cM(M_2)$ denote the fibered product of the 
two maps $\bd: \cM(M_i)\to \cl\cM(E)$.
Then we have an injective map
\[
	\gl_E : \cM(M_1) \times_{\cl\cM(E)} \cM(M_2) \hookrightarrow \cl\cM(H)
\]
which is natural with respect to the actions of homeomorphisms.}
\end{lem}
This is in exact analogy with Lemma \ref{lem:domain-and-range}, and illustrated in Figure \ref{fig:module-boundary}.
\begin{figure}[t]
\begin{equation*}
\begin{tikzpicture}[baseline=0]
\coordinate (a) at (0,1);
\coordinate (b) at (4,1);
\draw[marked] (a) arc (180:0:2);
\draw (b) -- (a);
\node at (2,2) {$M_1$};

\draw (0,0) node[fill, circle, red] {} -- (4,0) node[fill,circle,red] {};
\node at (-0.6,0) {$E$};

\draw[marked] (0,-1) arc(-180:0:2);
\draw (4,-1) -- (0,-1);
\node at (2,-2) {$M_2$};
\end{tikzpicture}
\qquad \qquad \qquad
\begin{tikzpicture}[baseline=0]
\draw[marked] (0,0) node[black] {$H$} circle (2);
\end{tikzpicture}
\end{equation*}\caption{The marked hemispheres and marked balls from Lemma \ref{lem:module-boundary}.}
\label{fig:module-boundary}
\end{figure}
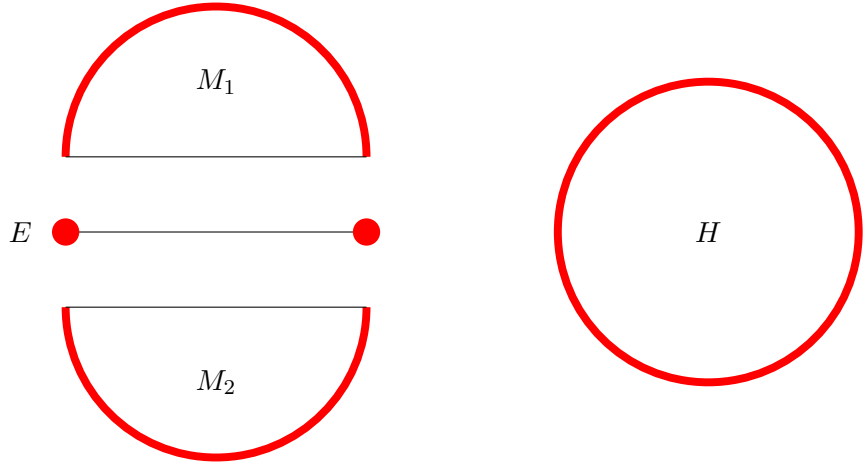

Let $\cl\cM(H)\trans E$ denote the image of $\gl_E$.
We will refer to elements of $\cl\cM(H)\trans E$ as ``splittable along $E$" or ``transverse to $E$". 

\noop{ 
\begin{lem}[Module to category restrictions]
{For each marked $k$-hemisphere $H$ there is a restriction map
$\cl\cM(H)\to \cC(H)$.
($\cC(H)$ means apply $\cC$ to the underlying $k$-ball of $H$.)
These maps comprise a natural transformation of functors.}
\end{lem}
}	

It follows from the definition of the colimit $\cl\cM(H)$ that
given any (unmarked) $k{-}1$-ball $Y$ in the interior of $H$ there is a restriction map
from a subset $\cl\cM(H)_{\trans{\bdy Y}}$ of $\cl\cM(H)$ to $\cC(Y)$.
Combining this with the boundary map $\cM(B,N) \to \cl\cM(\bd(B,N))$, we also have a restriction
map from a subset $\cM(B,N)_{\trans{\bdy Y}}$ of $\cM(B,N)$ to $\cC(Y)$ whenever $Y$ is in the interior of $\bd B \setmin N$.
This fact will be used below.

\noop{ 
Note that combining the various boundary and restriction maps above
(for both modules and $n$-categories)
we have for each marked $k$-ball $(B, N)$ and each $k{-}1$-ball $Y\sub \bd B \setmin N$
a natural map from a subset of $\cM(B, N)$ to $\cC(Y)$.
This subset $\cM(B,N)\trans{\bdy Y}$ is the subset of morphisms which are appropriately splittable (transverse to the
cutting submanifolds).
This fact will be used below.
} 

In our example, the various restriction and gluing maps above come from
restricting and gluing maps into $T$.

We require two sorts of composition (gluing) for modules, corresponding to two ways
of splitting a marked $k$-ball into two (marked or plain) $k$-balls.
(See Figure \ref{zzz3}.)

\begin{figure}[t]
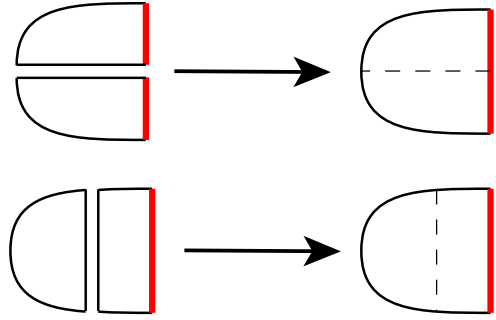

\begin{equation*}
\mathfig{.4}{ncat/zz3}
\end{equation*}
\caption{Module composition (top); $n$-category action (bottom).}
\label{zzz3}
\end{figure}

First, we can compose two module morphisms to get another module morphism.

\begin{module-axiom}[Module composition]
{Let $M = M_1 \cup_Y M_2$, where $M$, $M_1$ and $M_2$ are marked $k$-balls (with $2\le k\le n$)
and $Y = M_1\cap M_2$ is a marked $k{-}1$-ball.
Let $E = \bd Y$, which is a marked $k{-}2$-hemisphere.
Note that each of $M$, $M_1$ and $M_2$ has its boundary split into two marked $k{-}1$-balls by $E$.
We have restriction (domain or range) maps $\cM(M_i)\trans E \to \cM(Y)$.
Let $\cM(M_1) \trans E \times_{\cM(Y)} \cM(M_2) \trans E$ denote the fibered product of these two maps. 
Then (axiom) we have a map
\[
	\gl_Y : \cM(M_1) \trans E \times_{\cM(Y)} \cM(M_2) \trans E \to \cM(M) \trans E
\]
which is natural with respect to the actions of homeomorphisms, and also compatible with restrictions
to the intersection of the boundaries of $M$ and $M_i$.
If $k < n$ we require that $\gl_Y$ is injective.}
\end{module-axiom}

Second, we can compose an $n$-category morphism with a module morphism to get another
module morphism.
We'll call this the action map to distinguish it from the other kind of composition.

\begin{module-axiom}[$n$-category action]
{Let $M = X \cup_Y M'$, where $M$ and $M'$ are marked $k$-balls ($1\le k\le n$),
$X$ is a plain $k$-ball,
and $Y = X\cap M'$ is a $k{-}1$-ball.
Let $E = \bd Y$, which is a $k{-}2$-sphere.
We have restriction maps $\cM(M') \trans E \to \cC(Y)$ and $\cC(X) \trans E\to \cC(Y)$.
Let $\cC(X)\trans E \times_{\cC(Y)} \cM(M') \trans E$ denote the fibered product of these two maps. 
Then (axiom) we have a map
\[
	\gl_Y :\cC(X)\trans E \times_{\cC(Y)} \cM(M')\trans E \to \cM(M) \trans E
\]
which is natural with respect to the actions of homeomorphisms, and also compatible with restrictions
to the intersection of the boundaries of $X$ and $M'$.
If $k < n$ we require that $\gl_Y$ is injective.}
\end{module-axiom}

\begin{module-axiom}[Strict associativity]
The composition and action maps above are strictly associative.
Given any decomposition of a large marked ball into smaller marked and unmarked balls
any sequence of pairwise gluings yields (via composition and action maps) the same result.
\end{module-axiom}

Note that the above associativity axiom applies to mixtures of module composition,
action maps and $n$-category composition.
See Figure \ref{zzz1b}.

\begin{figure}[t]
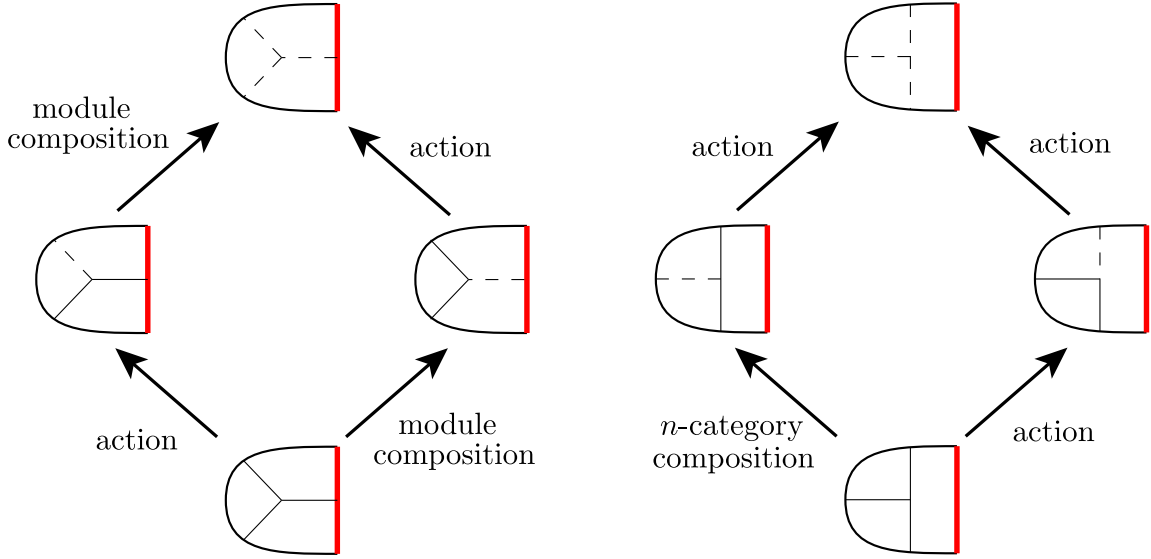

\begin{equation*}
\mathfig{0.49}{ncat/zz0} \mathfig{0.49}{ncat/zz1}
\end{equation*}
\caption{Two examples of mixed associativity}
\label{zzz1b}
\end{figure}

The above three axioms are equivalent to the following axiom,
which we state in slightly vague form.

\xxpar{Module multi-composition:}
{Given any splitting 
\[
	X_1 \sqcup\cdots\sqcup X_p \sqcup M_1\sqcup\cdots\sqcup M_q \to M
\]
of a marked $k$-ball $M$
into small (marked and plain) $k$-balls $M_i$ and $X_j$, there is a 
map from an appropriate subset (like a fibered product) 
of 
\[
	\cC(X_1)\times\cdots\times\cC(X_p) \times \cM(M_1)\times\cdots\times\cM(M_q) 
\]
to $\cM(M)$,
and these various multifold composition maps satisfy an
operad-type strict associativity condition.}

The above operad-like structure is analogous to the swiss cheese operad
\cite{MR1718089}.

\medskip

We can define marked pinched products $\pi:E\to M$ of marked balls similarly to the 
plain ball case. A marked pinched product $\pi: E \to M$ is a pinched product (that is, locally modeled on degeneracy maps) which restricts to a map between the markings which is also a pinched product, and in a neighborhood of the markings is the product of the map between the markings with an interval.  (See Figure \ref{fig:marked-pinched-products}.)

\begin{figure}[ht]
\begin{equation*}
\begin{tikzpicture}
\draw (0,2) -- (2,2.5);
\draw (0,2) -- (2,1.5);
\draw[marked] (2,1.5) -- (2,2.5);
\draw (0,0) -- (2,0) node[red,circle,fill,inner sep=2pt] {};
\draw[->] (1,1.5) -- (1,0.25);
\path[clip] (0,2) -- (2,2.5) -- (2,1.5) -- cycle;
\foreach \x in {0, 0.25, ..., 1.75} {
	\draw[green!50!brown] (\x,1) -- (\x,3);
}
\end{tikzpicture}
\qquad \qquad \qquad
\begin{tikzpicture}
\draw (2,2.5) -- (0,2.5) -- (0,1.5) -- (2,1.5);
\draw[marked] (2,1.5) -- (2,2.5);
\draw (0,0) -- (2,0) node[red,circle,fill,inner sep=2pt] {};
\draw[->] (1,1.2) -- (1,0.25);
\path[clip] (2,2.5) -- (0,2.5) -- (0,1.5) -- (2,1.5);
\foreach \x in {0, 0.25, ..., 1.75} {
	\draw[green!50!brown] (\x,1) -- (\x,3);
}
\end{tikzpicture}
\end{equation*}
\caption{Two examples of marked pinched products.}
\label{fig:marked-pinched-products}
\end{figure}
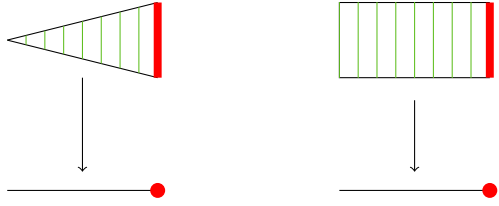

Note that a marked pinched product can be decomposed into either
two marked pinched products or a plain pinched product and a marked pinched product.
 (See Figure \ref{fig:decomposing-marked-pinched-products}.)
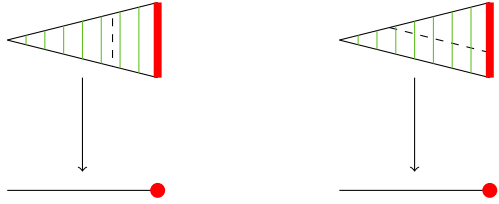
\begin{figure}[ht]
\begin{equation*}
\begin{tikzpicture}
\draw (0,2) -- (2,2.5);
\draw (0,2) -- (2,1.5);
\draw[marked] (2,1.5) -- (2,2.5);
\draw (0,0) -- (2,0) node[red,circle,fill,inner sep=2pt] {};
\draw[->] (1,1.5) -- (1,0.25);
\path[clip] (0,2) -- (2,2.5) -- (2,1.5) -- cycle;
\draw[dashed] (1.4,2.5) -- (1.4,1.5);
\foreach \x in {0, 0.25, ..., 1.75} {
	\draw[green!50!brown] (\x,1) -- (\x,3);
}
\end{tikzpicture}
\qquad \qquad \qquad
\begin{tikzpicture}
\draw (0,2) -- (2,2.5);
\draw (0,2) -- (2,1.5);
\draw[dashed] (0.666,2.166) -- (2,1.833);
\draw[marked] (2,1.5) -- (2,2.5);
\draw (0,0) -- (2,0) node[red,circle,fill,inner sep=2pt] {};
\draw[->] (1,1.5) -- (1,0.25);
\path[clip] (0,2) -- (2,2.5) -- (2,1.5) -- cycle;
\foreach \x in {0, 0.25, ..., 1.75} {
	\draw[green!50!brown] (\x,1) -- (\x,3);
}
\end{tikzpicture}
\end{equation*}
\caption{Two examples of decompositions of marked pinched products.}
\label{fig:decomposing-marked-pinched-products}
\end{figure}

\begin{module-axiom}[Product (identity) morphisms]
For each pinched product $\pi:E\to M$, with $M$ a marked $k$-ball and $E$ a marked
$k{+}m$-ball ($m\ge 1$),
there is a map $\pi^*:\cM(M)\to \cM(E)$.
These maps must satisfy the following conditions.
\begin{enumerate}
\item
If $\pi:E\to M$ and $\pi':E'\to M'$ are marked pinched products, and
if $f:M\to M'$ and $\tilde{f}:E \to E'$ are maps such that the diagram
\[ \xymatrix{
	E \ar[r]^{\tilde{f}} \ar[d]_{\pi} & E' \ar[d]^{\pi'} \\
	M \ar[r]^{f} & M'
} \]
commutes, then we have 
\[
	\pi'^*\circ f = \tilde{f}\circ \pi^*.
\]
\item
Product morphisms are compatible with module composition and module action.
Let $\pi:E\to M$, $\pi_1:E_1\to M_1$, and $\pi_2:E_2\to M_2$ 
be pinched products with $E = E_1\cup E_2$.
Let $a\in \cM(M)$, and let $a_i$ denote the restriction of $a$ to $M_i\sub M$.
Then 
\[
	\pi^*(a) = \pi_1^*(a_1)\bullet \pi_2^*(a_2) .
\]
Similarly, if $\rho:D\to X$ is a pinched product of plain balls and
$E = D\cup E_1$, then
\[
	\pi^*(a) = \rho^*(a')\bullet \pi_1^*(a_1),
\]
where $a'$ is the restriction of $a$ to $D$.
\item
Product morphisms are associative.
If $\pi:E\to M$ and $\rho:D\to E$ are marked pinched products then
\[
	\rho^*\circ\pi^* = (\pi\circ\rho)^* .
\]
\item
Product morphisms are compatible with restriction.
If we have a commutative diagram
\[ \xymatrix{
	D \ar@{^(->}[r] \ar[d]_{\rho} & E \ar[d]^{\pi} \\
	Y \ar@{^(->}[r] & M
} \]
such that $\rho$ and $\pi$ are pinched products, then
\[
	\res_D\circ\pi^* = \rho^*\circ\res_Y .
\]
($Y$ could be either a marked or plain ball.)
\end{enumerate}
\end{module-axiom}

As in the $n$-category definition, once we have product morphisms we can define
collar maps $\cM(M)\to \cM(M)$.
Note that there are two cases:
the collar could intersect the marking of the marked ball $M$, in which case
we use a product on a morphism of $\cM$; or the collar could be disjoint from the marking,
in which case we use a product on a morphism of $\cC$.

In our example, elements $a$ of $\cM(M)$ are maps to $T$, and $\pi^*(a)$ is the pullback of
$a$ along the map associated to $\pi$.

\medskip


The remaining module axioms are very similar to their counterparts in \S\ref{ss:n-cat-def}.

\begin{module-axiom}[Extended isotopy invariance in dimension $n$]
\label{ei-module-axiom}
Let $M$ be a marked $n$-ball, $b \in \cM(M)$, and $f: M\to M$ be a homeomorphism which 
acts trivially on the restriction $\bd b$ of $b$ to $\bd M$.
Suppose furthermore that $f$ is isotopic to the identity through homeomorphisms which
act trivially on $\bd b$.
Then $f(b) = b$.
In addition, collar maps act trivially on $\cM(M)$.
\end{module-axiom}

We emphasize that the $\bd M$ above (and below) means boundary in the marked $k$-ball sense.
In other words, if $M = (B, N)$ then we require only that isotopies are fixed 
on $\bd B \setmin N$.

\begin{module-axiom}[Splittings]
Let $c\in \cM_k(M)$, with $1\le k < n$.
Let $s = \{X_i\}$ be a splitting of M (so $M = \cup_i X_i$, and each $X_i$ is either a marked ball or a plain ball).
Let $\Homeo_\bd(M)$ denote homeomorphisms of $M$ which restrict to the identity on $\bd M$.
\begin{itemize}
\item (Alternative 1) Consider the set of homeomorphisms $g:M\to M$ such that $c$ splits along $g(s)$.
Then this subset of $\Homeo(M)$ is open and dense.
Furthermore, if $s$ restricts to a splitting $\bd s$ of $\bd M$, and if $\bd c$ splits along $\bd s$, then the
intersection of the set of such homeomorphisms $g$ with $\Homeo_\bd(M)$ is open and dense in $\Homeo_\bd(M)$.
\item (Alternative 2) Then there exists an embedded cell complex $S_c \sub M$, called the string locus of $c$,
such that if the splitting $s$ is transverse to $S_c$ then $c$ splits along $s$.
\end{itemize}
\end{module-axiom}

We define the 
category $\mbc$ of {\it marked $n$-balls with boundary conditions} as follows.
Its objects are pairs $(M, c)$, where $M$ is a marked $n$-ball and $c \in \cl\cM(\bd M)$ is the ``boundary condition".
The morphisms from $(M, c)$ to $(M', c')$, denoted $\Homeo(M; c \to M'; c')$, are
homeomorphisms $f:M\to M'$ such that $f|_{\bd M}(c) = c'$.

Let $\cS$ be a distributive symmetric monoidal category, and assume that $\cC$ is enriched in $\cS$.
A $\cC$-module enriched in $\cS$ is defined analogously to \ref{axiom:enriched}.
The top-dimensional part of the module $\cM_n$ is required to be a functor from $\mbc$ to $\cS$.
The top-dimensional gluing maps (module composition and $n$-category action) are $\cS$-maps whose
domain is a direct sub of tensor products, as in \ref{axiom:enriched}.

If $\cC$ is an $A_\infty$ $n$-category (see \ref{axiom:families}), we replace module axiom \ref{ei-module-axiom}
with the following axiom.
Retain notation from \ref{axiom:families}.

\begin{module-axiom}[Families of homeomorphisms act in dimension $n$.]
For each pair of marked $n$-balls $M$ and $M'$ and each pair $c\in \cl{\cM}(\bd M)$ and $c'\in \cl{\cM}(\bd M')$ 
we have an $\cS$-morphism
\[
	\cJ(\Homeo(M;c \to M'; c')) \ot \cM(M; c) \to \cM(M'; c') .
\]
Similarly, we have an $\cS$-morphism
\[
	\cJ(\Coll(M,c)) \ot \cM(M; c) \to \cM(M; c),
\]
where $\Coll(M,c)$ denotes the space of collar maps.
These action maps are required to be associative up to coherent homotopy,
and also compatible with composition (gluing) in the sense that
a diagram like the one in Theorem \ref{thm:CH} commutes.
\end{module-axiom}

\medskip

Note that the above axioms imply that an $n$-category module has the structure
of an $n{-}1$-category.
More specifically, let $J$ be a marked 1-ball, and define $\cE(X)\deq \cM(X\times J)$,
where $X$ is a $k$-ball and in the product $X\times J$ we pinch 
above the non-marked boundary component of $J$.
(More specifically, we collapse $X\times P$ to a single point, where
$P$ is the non-marked boundary component of $J$.)
Then $\cE$ has the structure of an $n{-}1$-category.

All marked $k$-balls are homeomorphic, unless $k = 1$ and our manifolds
are oriented or Spin (but not unoriented or $\text{Pin}_\pm$).
In this case ($k=1$ and oriented or Spin), there are two types
of marked 1-balls, call them left-marked and right-marked,
and hence there are two types of modules, call them right modules and left modules.
In all other cases ($k>1$ or unoriented or $\text{Pin}_\pm$),
there is no left/right module distinction.

\medskip

We now give some examples of modules over ordinary and $A_\infty$ $n$-categories.

\begin{example}[Examples from TQFTs]
\rm
Continuing Example \ref{ex:ncats-from-tqfts}, with $\cF$ a TQFT, $W$ an $n{-}j$-manifold,
and $\cF(W)$ the $j$-category associated to $W$.
Let $Y$ be an $(n{-}j{+}1)$-manifold with $\bd Y = W$.
Define a $\cF(W)$ module $\cF(Y)$ as follows.
If $M = (B, N)$ is a marked $k$-ball with $k<j$ let 
$\cF(Y)(M)\deq \cF((B\times W) \cup (N\times Y))$.
If $M = (B, N)$ is a marked $j$-ball and $c\in \cl{\cF(Y)}(\bd M)$ let
$\cF(Y)(M)\deq A_\cF((B\times W) \cup (N\times Y); c)$.
\end{example}

\begin{example}[Examples from the blob complex] \label{bc-module-example}
\rm
In the previous example, we can instead define
$\cF(Y)(M)\deq \bc_*((B\times W) \cup (N\times Y), c; \cF)$ (when $\dim(M) = n$)
and get a module for the $A_\infty$ $n$-category associated to $\cF$ as in 
Example \ref{ex:blob-complexes-of-balls}.
\end{example}

\begin{example}
\rm
Suppose $S$ is a topological space, with a subspace $T$.
We can define a module $\pi_{\leq n}(S,T)$ so that on each marked $k$-ball $(B,N)$ 
for $k<n$ the set $\pi_{\leq n}(S,T)(B,N)$ consists of all continuous maps of pairs 
$(B,N) \to (S,T)$ and on each marked $n$-ball $(B,N)$ it consists of all 
such maps modulo homotopies fixed on $\bdy B \setminus N$.
This is a module over the fundamental $n$-category $\pi_{\leq n}(S)$ of $S$, from Example \ref{ex:maps-to-a-space}.
\end{example}
Modifications corresponding to Examples \ref{ex:maps-to-a-space-with-a-fiber} and 
\ref{ex:linearized-maps-to-a-space} are also possible, and there is an $A_\infty$ version analogous to 
Example \ref{ex:chains-of-maps-to-a-space} given by taking singular chains.

\subsection{Modules as boundary labels (colimits for decorated manifolds)}
\label{moddecss}

Fix an ordinary $n$-category or $A_\infty$ $n$-category  $\cC$.
Let $W$ be a $k$-manifold ($k\le n$),
let $\{Y_i\}$ be a collection of disjoint codimension 0 submanifolds of $\bd W$,
and let $\cN = (\cN_i)$ be an assignment of a $\cC$ module $\cN_i$ to each $Y_i$.

We will define a set $\cC(W, \cN)$ using a colimit construction very similar to 
the one appearing in \S \ref{ss:ncat_fields} above.
(If $k = n$ and our $n$-categories are enriched, then
$\cC(W, \cN)$ will have additional structure; see below.)

Define a permissible decomposition of $W$ to be a map
\[
	\left(\bigsqcup_a X_a\right) \sqcup \left(\bigsqcup_{i,b} M_{ib}\right)  \to W,
\]
where each $X_a$ is a plain $k$-ball disjoint, in $W$, from $\cup Y_i$, and
each $M_{ib}$ is a marked $k$-ball intersecting $Y_i$  (once mapped into $W$),
with $M_{ib}\cap Y_i$ being the marking, which extends to a ball decomposition in the sense of Definition \ref{defn:gluing-decomposition}.
(See Figure \ref{mblabel}.)
\begin{figure}[t]
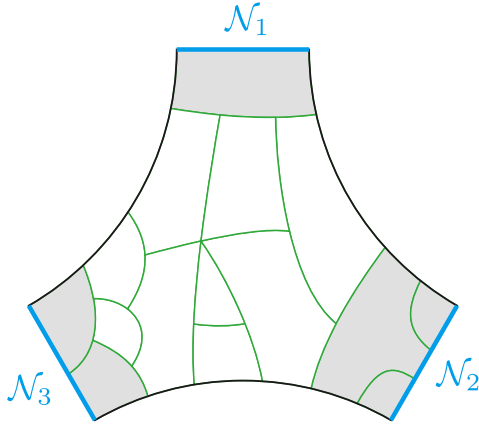

\begin{equation*}
\mathfig{.4}{ncat/mblabel}
\end{equation*}
\caption{A permissible decomposition of a manifold
whose boundary components are labeled by $\cC$ modules $\{\cN_i\}$.
Marked balls are shown shaded, plain balls are unshaded.}\label{mblabel}
\end{figure}
Given permissible decompositions $x$ and $y$, we say that $x$ is a refinement
of $y$, or write $x \le y$, if each ball of $y$ is a union of balls of $x$.
This defines a partial ordering $\cell(W)$, which we will think of as a category.
(The objects of $\cell(D)$ are permissible decompositions of $W$, and there is a unique
morphism from $x$ to $y$ if and only if $x$ is a refinement of $y$.)

The collection of modules $\cN$ determines 
a functor $\psi_\cN$ from $\cell(W)$ to the category of sets 
(possibly with additional structure if $k=n$).
For a decomposition $x = (X_a, M_{ib})$ in $\cell(W)$, define $\psi_\cN(x)$ to be the subset
\[
	\psi_\cN(x) \sub \left(\prod_a \cC(X_a)\right) \times \left(\prod_{ib} \cN_i(M_{ib})\right)
\]
such that the restrictions to the various pieces of shared boundaries amongst the
$X_a$ and $M_{ib}$ all agree.
If $x$ is a refinement of $y$, define a map $\psi_\cN(x)\to\psi_\cN(y)$
via the gluing (composition or action) maps from $\cC$ and the $\cN_i$.

We now define the set $\cC(W, \cN)$ to be the colimit of the functor $\psi_\cN$.
(As in \S\ref{ss:ncat-coend}, if $k=n$ we take a colimit in whatever
category we are enriching over, and if additionally we are in the $A_\infty$ case, 
then we use a homotopy colimit.)

\medskip

If $D$ is an $m$-ball, $0\le m \le n-k$, then we can similarly define
$\cC(D\times W, \cN)$, where in this case $\cN_i$ labels the submanifold 
$D\times Y_i \sub \bd(D\times W)$.
It is not hard to see that the assignment $D \mapsto \cC(D\times W, \cN)$
has the structure of an $n{-}k$-category.

\medskip

We will use a simple special case of the above 
construction to define tensor products 
of modules.
Let $\cM_1$ and $\cM_2$ be modules for an $n$-category $\cC$.
(If $k=1$ and our manifolds are oriented, then one should be 
a left module and the other a right module.)
Choose a 1-ball $J$, and label the two boundary points of $J$ by $\cM_1$ and $\cM_2$.
Define the tensor product $\cM_1 \tensor \cM_2$ to be the 
$n{-}1$-category associated as above to $J$ with its boundary labeled by $\cM_1$ and $\cM_2$.
This of course depends (functorially)
on the choice of 1-ball $J$.

We will define a more general self tensor product (categorified coend) below.

\subsection{Morphisms of modules}
\label{ss:module-morphisms}

Modules are collections of functors together with some additional data, so we define morphisms
of modules to be collections of natural transformations which are compatible with this
additional data.

More specifically, let $\cX$ and $\cY$ be $\cC$ modules, i.e.\ collections of functors
$\{\cX_k\}$ and $\{\cY_k\}$, for $1\le k\le n$, from marked $k$-balls to sets 
as in Module Axiom \ref{module-axiom-funct}.
A morphism $g:\cX\to\cY$ is a collection of natural transformations $g_k:\cX_k\to\cY_k$
satisfying:
\begin{itemize}
\item Each $g_k$ commutes with $\bd$.
\item Each $g_k$ commutes with gluing (module composition and $\cC$ action).
\item Each $g_k$ commutes with taking products.
\item In the top dimension $k=n$, $g_n$ preserves whatever additional structure we are enriching over (e.g.\ vector
spaces).
In the $A_\infty$ case (e.g.\ enriching over chain complexes) $g_n$ should live in 
an appropriate derived hom space, as described below.
\end{itemize}

We will be mainly interested in the case $n=1$ and enriched over chain complexes,
since this is the case that's relevant to the generalized Deligne conjecture of \S\ref{sec:deligne}.
So we treat this case in more detail.

First we explain the remark about derived hom above.
Let $L$ be a marked 1-ball and let $\cl{\cX}(L)$ denote the local homotopy colimit construction
associated to $L$ by $\cX$ and $\cC$.
(See \S \ref{ss:ncat_fields} and \S \ref{moddecss}.)
Define $\cl{\cY}(L)$ similarly.
For $K$ an unmarked 1-ball let $\cl{\cC}(K)$ denote the local homotopy colimit
construction associated to $K$ by $\cC$.
Then we have an injective gluing map
\[
	\gl: \cl{\cX}(L) \ot \cl{\cC}(K) \to \cl{\cX}(L\cup K) 
\]
which is also a chain map.
(For simplicity we are suppressing mention of boundary conditions on the unmarked 
boundary components of the 1-balls.)
We define $\hom_\cC(\cX \to \cY)$ to be a collection of (graded linear) natural transformations
$g: \cl{\cX}(L)\to \cl{\cY}(L)$ such that the following diagram commutes for all $L$ and $K$:
\[ \xymatrix{
	\cl{\cX}(L) \ot \cl{\cC}(K) \ar[r]^{\gl} \ar[d]_{g\ot \id} & \cl{\cX}(L\cup K) \ar[d]^{g}\\
	\cl{\cY}(L) \ot \cl{\cC}(K) \ar[r]^{\gl} & \cl{\cY}(L\cup K)
} \]

The usual differential on graded linear maps between chain complexes induces a differential
on $\hom_\cC(\cX \to \cY)$, giving it the structure of a chain complex.

Let $\cZ$ be another $\cC$ module.
We define a chain map
\[
	a: \hom_\cC(\cX \to \cY) \ot (\cX \ot_\cC \cZ) \to \cY \ot_\cC \cZ
\]
as follows.
Recall that the tensor product $\cX \ot_\cC \cZ$  depends on a choice of interval $J$, labeled
by $\cX$ on one boundary component and $\cZ$ on the other.
Because we are using the {\it local} homotopy colimit, any generator
$D\ot x\ot \bar{c}\ot z$ of $\cX \ot_\cC \cZ$ can be written (perhaps non-uniquely) as a gluing
$(D'\ot x \ot \bar{c}') \bullet (D''\ot \bar{c}''\ot z)$, for some decomposition $J = L'\cup L''$
and with $D'\ot x \ot \bar{c}'$ a generator of $\cl{\cX}(L')$ and 
$D''\ot \bar{c}''\ot z$ a generator of $\cl{\cZ}(L'')$.
(Such a splitting exists because the blob diagram $D$ can be split into left and right halves, 
since no blob can include both the leftmost and rightmost intervals in the underlying decomposition.
This step would fail if we were using the usual hocolimit instead of the local hocolimit.)
We now define
\[
	a: g\ot (D\ot x\ot \bar{c}\ot z) \mapsto g(D'\ot x \ot \bar{c}')\bullet (D''\ot \bar{c}''\ot z) .
\]
This does not depend on the choice of splitting $D = D'\bullet D''$ because $g$ commutes with gluing.

\subsection{The \texorpdfstring{$n{+}1$}{n+1}-category of sphere modules}
\label{ssec:spherecat}

In this subsection we define $n{+}1$-categories $\cS$ of ``sphere modules".
The objects are $n$-categories, the $k$-morphisms are $k{-}1$-sphere modules for $1\le k \le n$,
and the $n{+}1$-morphisms are intertwiners.
With future applications in mind, we treat simultaneously the big $n{+}1$-category
of all $n$-categories and all sphere modules and also subcategories thereof.
When $n=1$ this is closely related to the familiar $2$-category consisting of 
algebras, bimodules and intertwiners, or a subcategory of that.
(More generally, we can replace algebras with linear 1-categories.)
The ``bi" in ``bimodule" corresponds to the fact that a 0-sphere consists of two points.
The sphere module $n{+}1$-category is a natural generalization of the 
algebra-bimodule-intertwiner 2-category to higher dimensions.

Another possible name for this $n{+}1$-category is the $n{+}1$-category of defects.
The $n$-categories are thought of as representing field theories, and the 
$0$-sphere modules are codimension 1 defects between adjacent theories.
In general, $m$-sphere modules are codimension $m{+}1$ defects;
the link of such a defect is an $m$-sphere decorated with defects of smaller codimension.

\medskip


For simplicity, we will assume that $n$-categories are enriched over $\c$-vector spaces.

The $1$- through $n$-dimensional parts of $\cS$ are various sorts of modules, and we describe
these first.
The $n{+}1$-dimensional part of $\cS$ consists of intertwiners
of  $1$-category modules associated to decorated $n$-balls.
We will see below that in order for these $n{+}1$-morphisms to satisfy all of
the axioms of an $n{+}1$-category (in particular, duality requirements), we will have to assume
that our $n$-categories and modules have non-degenerate inner products.
(In other words, we need to assume some extra duality on the $n$-categories and modules.)

\medskip

Our first task is to define an $n$-category $m$-sphere module, for $0\le m \le n-1$.
These will be defined in terms of certain classes of marked balls, very similarly
to the definition of $n$-category modules above.
(This, in turn, is very similar to our definition of $n$-category.)
Because of this similarity, we only sketch the definitions below.

We start with $0$-sphere modules, which also could reasonably be called (categorified) bimodules.
(For $n=1$ they are precisely bimodules in the usual, uncategorified sense.)
We prefer the more awkward term ``0-sphere module" to emphasize the analogy
with the higher sphere modules defined below.

Define a $0$-marked $k$-ball, $1\le k \le n$, to be a pair  $(X, M)$ homeomorphic to the standard
$(B^k, B^{k-1})$.
See Figure \ref{feb21a}.
Another way to say this is that $(X, M)$ is homeomorphic to $B^{k-1}\times([-1,1], \{0\})$.

\begin{figure}[t]
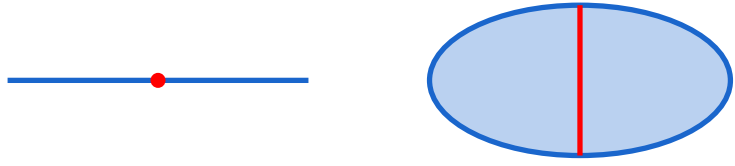

$$\tikz[baseline,line width=2pt]{\draw[kw-blue-a] (-2,0)--(2,0); \fill[red] (0,0) circle (0.1);} \qquad \qquad \tikz[baseline,line width=2pt]{\draw[kw-blue-a][fill=kw-blue-a!30!white] (0,0) circle (2 and 1); \draw[red] (0,1)--(0,-1);}$$
\caption{0-marked 1-ball and 0-marked 2-ball}
\label{feb21a}
\end{figure}

The $0$-marked balls can be cut into smaller balls in various ways.
We only consider those decompositions in which the smaller balls are either
$0$-marked (i.e. intersect the $0$-marking of the large ball in a disc) 
or plain (don't intersect the $0$-marking of the large ball).
We can also take the boundary of a $0$-marked ball, which is a $0$-marked sphere.

Fix $n$-categories $\cA$ and $\cB$.
These will label the two halves of a $0$-marked $k$-ball.

An $n$-category $0$-sphere module $\cM$ over the $n$-categories $\cA$ and $\cB$ is 
a collection of functors $\cM_k$ from the category
of $0$-marked $k$-balls, $1\le k \le n$,
(with the two halves labeled by $\cA$ and $\cB$) to the category of sets.
If $k=n$ these sets should be enriched to the extent $\cA$ and $\cB$ are.
Given a decomposition of a $0$-marked $k$-ball $X$ into smaller balls $X_i$, we have
morphism sets $\cA_k(X_i)$ (if $X_i$ lies on the $\cA$-labeled side)
or $\cB_k(X_i)$ (if $X_i$ lies on the $\cB$-labeled side)
or $\cM_k(X_i)$ (if $X_i$ intersects the marking and is therefore a smaller 0-marked ball).
Corresponding to this decomposition we have a composition (or ``gluing") map
from the product (fibered over the boundary data) of these various sets into $\cM_k(X)$.

\medskip

Part of the structure of an $n$-category 0-sphere module $\cM$  is captured by saying it is
a collection $\cD^{ab}$ of $n{-}1$-categories, indexed by pairs $(a, b)$ of objects (0-morphisms)
of $\cA$ and $\cB$.
Let $J$ be some standard 0-marked 1-ball (i.e.\ an interval with a marked point in its interior).
Given a $j$-ball $X$, $0\le j\le n-1$, we define
\[
	\cD(X) \deq \cM(X\times J) .
\]
The product is pinched over the boundary of $J$.
The set $\cD$ breaks into ``blocks" according to the restrictions to the pinched points of $X\times J$
(see Figure \ref{feb21b}).
These restrictions are 0-morphisms $(a, b)$ of $\cA$ and $\cB$.

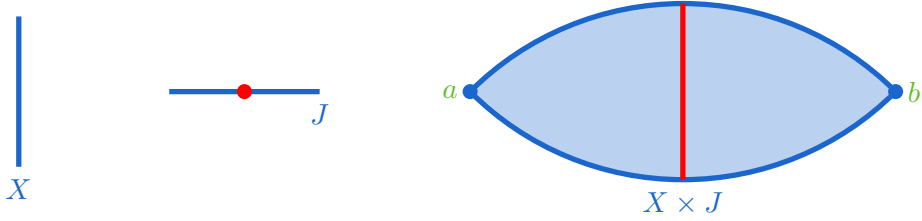
\begin{figure}[t] \centering
\begin{tikzpicture}[kw-blue-a,line width=2pt]
\draw (0,1) -- (0,-1) node[below] {$X$};

\draw (2,0) -- (4,0) node[below] {$J$};
\fill[red] (3,0) circle (0.1);

\draw[fill=kw-blue-a!30!white] (6,0) node(a) {} arc (135:90:4) node(top) {} arc (90:45:4) node(b) {} arc (-45:-90:4) node(bottom) {} arc(-90:-135:4);
\draw[red] (top.center) -- (bottom.center);
\fill (a) circle (0.1) node[left] {\color{green!50!brown} $a$};
\fill (b) circle (0.1) node[right] {\color{green!50!brown} $b$};

\path (bottom) node[below]{$X \times J$};

\end{tikzpicture}
\caption{The pinched product $X\times J$}
\label{feb21b}
\end{figure}

More generally, consider an interval with interior marked points, and with the complements
of these points labeled by $n$-categories $\cA_i$ ($0\le i\le l$) and the marked points labeled
by $\cA_i$-$\cA_{i+1}$ 0-sphere modules $\cM_i$.
(See Figure \ref{feb21c}.)
To this data we can apply the coend construction as in \S\ref{moddecss} above
to obtain an $\cA_0$-$\cA_l$ $0$-sphere module and, forgetfully, an $n{-}1$-category.
This amounts to a definition of taking tensor products of $0$-sphere modules over $n$-categories.

\begin{figure}[t] \centering
\begin{tikzpicture}[baseline,line width = 2pt]
\draw[kw-blue-a] (0,0) -- (6,0);
\foreach \x/\n in {0.5/0,1.5/1,3/2,4.5/3,5.5/4} {
	\path (\x,0)  node[below] {\color{green!50!brown}$\cA_{\n}$};
}
\foreach \x/\n in {1/0,2/1,4/2,5/3} {
	\fill[red] (\x,0) circle (0.1) node[above] {\color{green!50!brown}$\cM_{\n}$};
}
\end{tikzpicture}
\qquad
\qquad
\begin{tikzpicture}[baseline,line width = 2pt]
\draw[kw-blue-a] (0,0) circle (2);
\foreach \q/\n in {-45/0,90/1,180/2} {
	\path (\q:2.4)  node {\color{green!50!brown}$\cA_{\n}$};
}
\foreach \q/\n in {60/0,120/1,-120/2} {
	\fill[red] (\q:2) circle (0.1);
	\path (\q:2.4) node {\color{green!50!brown}$\cM_{\n}$};
}
\end{tikzpicture}
\caption{Marked and labeled 1-manifolds}
\label{feb21c}
\end{figure}
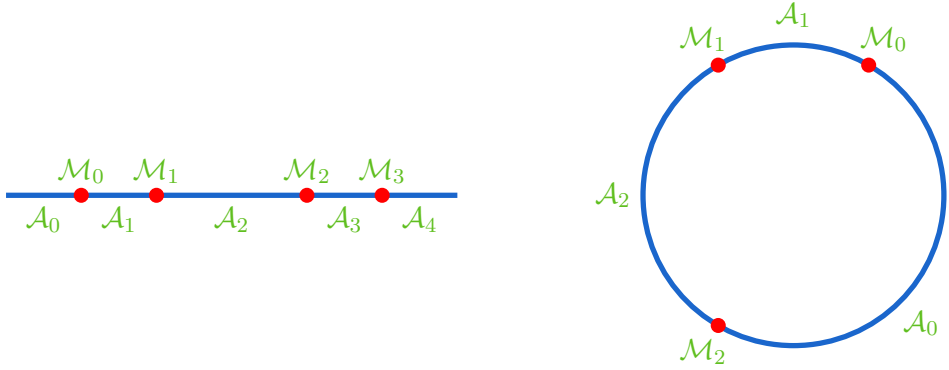

We could also similarly mark and label a circle, obtaining an $n{-}1$-category
associated to the marked and labeled circle.
(See Figure \ref{feb21c}.)
If the circle is divided into two intervals, we can think of this $n{-}1$-category
as the 2-sided tensor product of the two 0-sphere modules associated to the two intervals.

\medskip

Next we define $n$-category 1-sphere modules.
These are just representations of (modules for) $n{-}1$-categories associated to marked and labeled 
circles (1-spheres) which we just introduced.

Equivalently, we can define 1-sphere modules in terms of 1-marked $k$-balls, $2\le k\le n$.
Fix a marked (and labeled) circle $S$.
Let $C(S)$ denote the cone of $S$, a marked 2-ball (Figure \ref{feb21d}).
A 1-marked $k$-ball is anything homeomorphic to $B^j \times C(S)$, $0\le j\le n-2$, 
where $B^j$ is the standard $j$-ball.
A 1-marked $k$-ball can be decomposed in various ways into smaller balls, which are either 
(a) smaller 1-marked $k$-balls, (b) 0-marked $k$-balls, or (c) plain $k$-balls.
(See Figure \ref{subdividing1marked}.)
We now proceed as in the above module definitions.

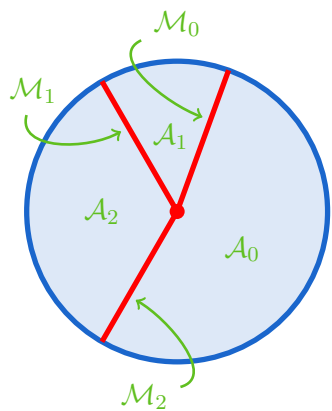
\begin{figure}[t] \centering
\begin{tikzpicture}[baseline,line width = 2pt]
\draw[kw-blue-a][fill=kw-blue-a!15!white] (0,0) circle (2);
\fill[red] (0,0) circle (0.1);
\foreach \qm/\qa/\n in {70/-30/0, 120/95/1, -120/180/2} {
	\draw[red] (0,0) -- (\qm:2);
	\path (\qa:1) node {\color{green!50!brown} $\cA_\n$};
	\path (\qm+20:2.5) node(M\n) {\color{green!50!brown} $\cM_\n$};
	\draw[line width=1pt, green!50!brown, ->] (M\n.\qm+135) to[out=\qm+135,in=\qm+90] (\qm+5:1.3);
}
\end{tikzpicture}
\caption{Cone on a marked circle, the prototypical 1-marked ball}
\label{feb21d}
\end{figure}

\begin{figure}[t] \centering
\begin{tikzpicture}[baseline,line width = 2pt]
\draw[kw-blue-a][fill=kw-blue-a!15!white] (0,0) circle (2);
\fill[red] (0,0) circle (0.1);
\foreach \qm/\qa/\n in {70/-30/0, 120/95/1, -120/180/2} {
	\draw[red] (0,0) -- (\qm:2);
}

\begin{scope}[black, thin]
\clip (0,0) circle (2);
\draw (0:1) -- (90:1) -- (180:1) -- (270:1) -- cycle;
\draw (90:1) -- (90:2.1);
\draw (180:1) -- (180:2.1);
\draw (270:1) -- (270:2.1);
\draw (0:1) -- (15:2.1);
\draw (0:1) -- (315:1.5) -- (270:1);
\draw (315:1.5) -- (315:2.1);
\end{scope}

\node(0marked) at (2.5,2.25) {$0$-marked ball};
\node(1marked) at (3.5,1) {$1$-marked ball};
\node(plain) at (3,-1) {plain ball};
\draw[line width=1pt, green!50!brown, ->] (0marked.270) to[out=270,in=45] (50:1.1);
\draw[line width=1pt, green!50!brown, ->] (1marked.225) to[out=270,in=45] (0.4,0.1);
\draw[line width=1pt, green!50!brown, ->] (plain.90) to[out=135,in=45] (-45:1);

\end{tikzpicture}
\caption{Subdividing a $1$-marked ball into plain, $0$-marked and $1$-marked balls.}
\label{subdividing1marked}
\end{figure}
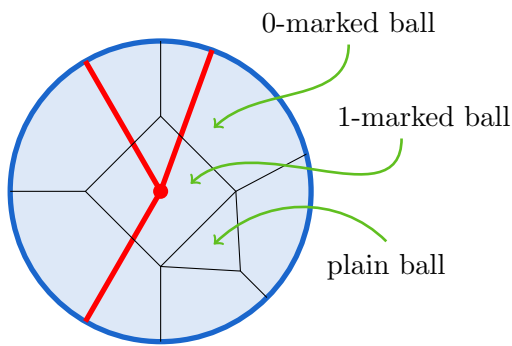

A $n$-category 1-sphere module is, among other things, an $n{-}2$-category $\cD$ with
\[
	\cD(X) \deq \cM(X\times C(S)) .
\]
The product is pinched over the boundary of $C(S)$.
$\cD$ breaks into ``blocks" according to the restriction to the 
image of $\bd C(S) = S$ in $X\times C(S)$.

More generally, consider a 2-manifold $Y$ 
(e.g.\ 2-ball or 2-sphere) marked by an embedded 1-complex $K$.
The components of $Y\setminus K$ are labeled by $n$-categories, 
the edges of $K$ are labeled by 0-sphere modules, 
and the 0-cells of $K$ are labeled by 1-sphere modules.
We can now apply the coend construction and obtain an $n{-}2$-category.
If $Y$ has boundary then this $n{-}2$-category is a module for the $n{-}1$-category
associated to the (marked, labeled) boundary of $Y$.
In particular, if $\bd Y$ is a 1-sphere then we get a 1-sphere module as defined above.

\medskip

It should now be clear how to define $n$-category $m$-sphere modules for $0\le m \le n-1$.
For example, there is an $n{-}2$-category associated to a marked, labeled 2-sphere,
and a 2-sphere module is a representation of such an $n{-}2$-category.

\medskip

We can now define the $n$-or-less-dimensional part of our $n{+}1$-category $\cS$.
Choose some collection of $n$-categories, then choose some collections of 0-sphere modules between
these $n$-categories, then choose some collection of 1-sphere modules for the various
possible marked 1-spheres labeled by the $n$-categories and 0-sphere modules, and so on.
Let $L_i$ denote the collection of $i{-}1$-sphere modules we have chosen.
(For convenience, we declare a $(-1)$-sphere module to be an $n$-category.)
There is a wide range of possibilities.
The set $L_0$ could contain infinitely many $n$-categories or just one.
For each pair of $n$-categories in $L_0$, $L_1$ could contain no 0-sphere modules at all or 
it could contain several.
The only requirement is that each $k$-sphere module be a module for a $k$-sphere $n{-}k$-category
constructed out of labels taken from $L_j$ for $j<k$.

We remind the reader again that $\cS$ depends on 
the choice of $L_i$ above as well as the choice of 
families of inner products described below.

We now define $\cS(X)$, for $X$ a ball of dimension at most $n$, to be the set of all 
cell-complexes $K$ embedded in $X$, with the codimension-$j$ parts of $(X, K)$ labeled
by elements of $L_j$.
As described above, we can think of each decorated $k$-ball as defining a $k{-}1$-sphere module
for the $n{-}k{+}1$-category associated to its decorated boundary.
Thus the $k$-morphisms of $\cS$ (for $k\le n$) can be thought 
of as $n$-category $k{-}1$-sphere modules 
(generalizations of bimodules).
On the other hand, we can equally well think of the $k$-morphisms as decorations on $k$-balls, 
and from this point of view it is clear that they satisfy all of the axioms of an
$n{+}1$-category.
(All of the axioms for the less-than-$n{+}1$-dimensional part of an $n{+}1$-category, that is.)

\medskip

Next we define the $n{+}1$-morphisms of $\cS$.
The construction of the 0- through $n$-morphisms was easy and tautological, but the 
$n{+}1$-morphisms will require some effort and combinatorial topology, as well as additional
duality assumptions on the lower morphisms. 
These are required because we define the spaces of $n{+}1$-morphisms by 
making arbitrary choices of incoming and outgoing boundaries for each $n{+}1$-ball. 
The additional duality assumptions are needed to prove independence of our definition from these choices.

Let $X$ be an $n{+}1$-ball, and let $c$ be a decoration of its boundary
by a cell complex labeled by 0- through $n$-morphisms, as above.
Choose an $n{-}1$-sphere $E\sub \bd X$, transverse to $c$, which divides
$\bd X$ into ``incoming" and ``outgoing" boundary $\bd_-X$ and $\bd_+X$.
Let $E_c$ denote $E$ decorated by the restriction of $c$ to $E$.
Recall from above the associated 1-category $\cS(E_c)$.
We can also have $\cS(E_c)$ modules $\cS(\bd_-X_c)$ and $\cS(\bd_+X_c)$.
Define
\[
	\cS(X; c; E) \deq \hom_{\cS(E_c)}(\cS(\bd_-X_c), \cS(\bd_+X_c)) .
\]

We will show that if the sphere modules are equipped with a ``compatible family of 
non-degenerate inner products", then there is a coherent family of isomorphisms
$\cS(X; c; E) \cong \cS(X; c; E')$ for all pairs of choices $E$ and $E'$.
This will allow us to define $\cS(X; c)$ independently of the choice of $E$.

First we must define ``inner product", ``non-degenerate" and ``compatible".
Let $Y$ be a decorated $n$-ball, and $\ol{Y}$ its mirror image.
(We assume we are working in the unoriented category.)
Let $Y\cup\ol{Y}$ denote the decorated $n$-sphere obtained by gluing $Y$ and $\ol{Y}$
along their common boundary.
An {\it inner product} on $\cS(Y)$ is a dual vector
\[
	z_Y : \cS(Y\cup\ol{Y}) \to \c.
\]
We will also use the notation
\[
	\langle a, b\rangle \deq z_Y(a\bullet b) \in \c .
\]
An inner product induces a linear map
\begin{eqnarray*}
	\varphi: \cS(Y) &\to& \cS(Y)^* \\
	a &\mapsto& \langle a, \cdot \rangle
\end{eqnarray*}
which satisfies, for all morphisms $e$ of $\cS(\bd Y)$,
\[
	\varphi(ae)(b) = \langle ae, b \rangle = z_Y(a\bullet e\bullet b) = 
			\langle a, eb \rangle = \varphi(a)(eb) .
\]
In other words, $\varphi$ is a map of $\cS(\bd Y)$ modules.
An inner product is {\it non-degenerate} if $\varphi$ is an isomorphism.
This implies that $\cS(Y; c)$ is finite dimensional for all boundary conditions $c$.
(One can think of these inner products as giving some duality in dimension $n{+}1$;
heretofore we have only assumed duality in dimensions 0 through $n$.)

Next we define compatibility.
Let $Y = Y_1\cup Y_2$ with $D = Y_1\cap Y_2$.
Let $X_1$ and $X_2$ be the two components of $Y\times I$ cut along
$D\times I$, in both cases using the pinched product.
(Here we are overloading notation and letting $D$ denote both a decorated and an undecorated
manifold.)
We have $\bd X_i = Y_i \cup \ol{Y}_i \cup (D\times I)$
(see Figure \ref{jun23a}).
\begin{figure}[t]
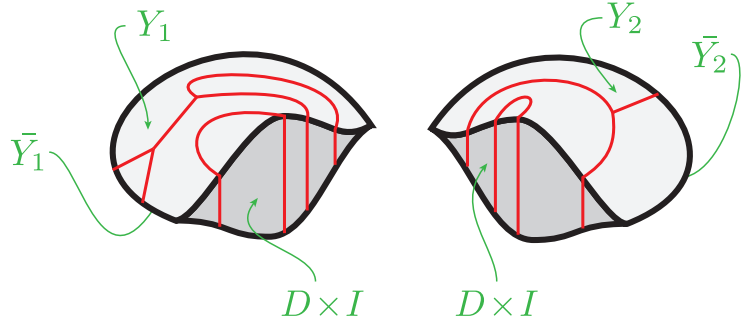

\begin{equation*}
\mathfig{.6}{ncat/YxI-sliced}
\end{equation*}
\caption{$Y\times I$ sliced open}
\label{jun23a}
\end{figure}
Given $a_i\in \cS(Y_i)$, $b_i\in \cS(\ol{Y}_i)$ and $v\in\cS(D\times I)$
which agree on their boundaries, we can evaluate
\[
	z_{Y_i}(a_i\bullet b_i\bullet v) \in \c .
\]
(This requires a choice of homeomorphism $Y_i \cup \ol{Y}_i \cup (D\times I) \cong
Y_i \cup \ol{Y}_i$, but the value of $z_{Y_i}$ is independent of this choice.)
We can think of $z_{Y_i}$ as giving a function
\[
	\psi_i : \cS(Y_i) \ot \cS(\ol{Y}_i) \to \cS(D\times I)^* 
					\stackrel{\varphi\inv}{\longrightarrow} \cS(D\times I) .
\]
We can now finally define a family of inner products to be {\it compatible} if
for all decompositions $Y = Y_1\cup Y_2$ as above and all $a_i\in \cS(Y_i)$, $b_i\in \cS(\ol{Y}_i)$
we have
\[
	z_Y(a_1\bullet a_2\bullet b_1\bullet b_2) = 
				z_{D\times I}(\psi_1(a_1\ot b_1)\bullet \psi_2(a_2\ot b_2)) .
\]
In other words, the inner product on $Y$ is determined by the inner products on
$Y_1$, $Y_2$ and $D\times I$.

Now we show how to unambiguously identify $\cS(X; c; E)$ and $\cS(X; c; E')$ for any
two choices of $E$ and $E'$.
Consider first the case where $\bd X$ is decomposed as three $n$-balls $A$, $B$ and $C$,
with $E = \bd(A\cup B)$ and $E' = \bd A$.
We must provide an isomorphism between $\cS(X; c; E) = \hom(\cS(C), \cS(A\cup B))$
and $\cS(X; c; E') = \hom(\cS(C\cup \ol{B}), \cS(A))$.
Let $D = B\cap A$.
Then as above we can construct a map
\[
	\psi: \cS(B)\ot\cS(\ol{B}) \to \cS(D\times I) .
\]
Given $f\in \hom(\cS(C), \cS(A\cup B))$ we define $f'\in \hom(\cS(C\cup \ol{B}), \cS(A))$
to be the composition
\[
	\cS(C\cup \ol{B}) \stackrel{f\ot\id}{\longrightarrow}
		\cS(A\cup B\cup \ol{B})  \stackrel{\id\ot\psi}{\longrightarrow}
			\cS(A\cup(D\times I)) \stackrel{\cong}{\longrightarrow} \cS(A) .
\]
(See Figure \ref{jun23b}.)
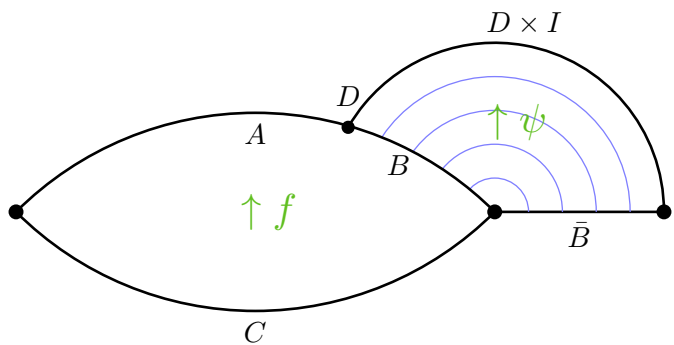
\begin{figure}[t]
$$
\begin{tikzpicture}[baseline,line width = 1pt,x=1.5cm,y=1.5cm]
\draw (0,0) node(R) {}
	-- (0.75,0) node[below] {$\bar{B}$}
	--(1.5,0)  node[circle,fill=black,inner sep=2pt] {}
	arc (0:80:1.5) node[above] {$D \times I$}
	arc (80:180:1.5);
\foreach \r in {0.3, 0.6, 0.9, 1.2} {
	\draw[blue!50, line width = 0.5pt] (\r,0) arc (0:180:\r);
}
\draw[fill=white]
	(R) node[circle,fill=black,inner sep=2pt] {}
	arc (45:65:3) node[below] {$B$}
	arc (65:90:3) node[below] {$A$}
	arc (90:135:3) node[circle,fill=black,inner sep=2pt] {}
	arc (-135:-90:3) node[below] {$C$}
	arc (-90:-45:3);
\draw[fill]  (150:1.5) circle (2pt) node[above=4pt] {$D$};
\node[green!50!brown] at (-2,0) {\scalebox{1.4}{$\uparrow f$}};
\node[green!50!brown] at (0.2,0.8) {\scalebox{1.4}{$\uparrow \psi$}};
\end{tikzpicture}
$$
\caption{Moving $B$ from top to bottom}
\label{jun23b}
\end{figure}
Let $D' = B\cap C$.
Using the inner products there is an adjoint map
\[
	\psi^\dagger: \cS(D'\times I) \to \cS(\ol{B})\ot\cS(B) .
\]
Given $f'\in \hom(\cS(C\cup \ol{B}), \cS(A))$ we define $f\in \hom(\cS(C), \cS(A\cup B))$
to be the composition
\[
	\cS(C) \stackrel{\cong}{\longrightarrow}
		\cS(C\cup(D'\times I)) \stackrel{\id\ot\psi^\dagger}{\longrightarrow}
			\cS(C\cup \ol{B}\cup B)   \stackrel{f'\ot\id}{\longrightarrow}
				\cS(A\cup B) .
\]
(See Figure \ref{jun23c}.)
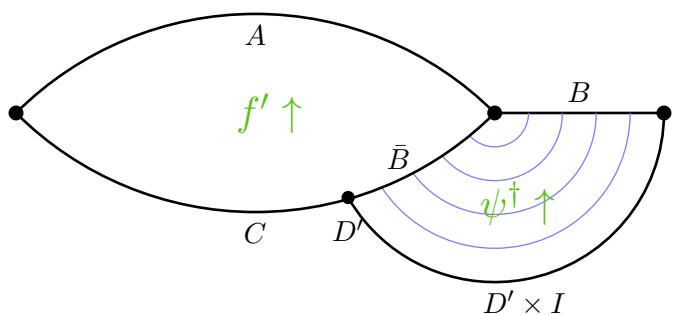
\begin{figure}[t]
\begin{equation*}
\begin{tikzpicture}[baseline,line width = 1pt,x=1.5cm,y=-1.5cm]
\draw (0,0) node(R) {}
	-- (0.75,0) node[above] {$B$}
	--(1.5,0)  node[circle,fill=black,inner sep=2pt] {}
	arc (0:80:1.5) node[below] {$D' \times I$}
	arc (80:180:1.5);
\foreach \r in {0.3, 0.6, 0.9, 1.2} {
	\draw[blue!50, line width = 0.5pt] (\r,0) arc (0:180:\r);
}
\draw[fill=white]
	(R) node[circle,fill=black,inner sep=2pt] {}
	arc (45:65:3) node[above] {$\bar{B}$}
	arc (65:90:3) node[below] {$C$}
	arc (90:135:3) node[circle,fill=black,inner sep=2pt] {}
	arc (-135:-90:3) node[below] {$A$}
	arc (-90:-45:3);
\draw[fill]  (150:1.5) circle (2pt) node[below=4pt] {$D'$};
\node[green!50!brown] at (-2,0) {\scalebox{1.4}{$f'\uparrow $}};
\node[green!50!brown] at (0.2,0.8) {\scalebox{1.4}{$\psi^\dagger \uparrow $}};
\end{tikzpicture}
\end{equation*}
\caption{Moving $B$ from bottom to top}
\label{jun23c}
\end{figure}
It is not hard too show that the above two maps are mutually inverse.

\begin{lem} \label{equator-lemma}
Any two choices of $E$ and $E'$ are related by a series of modifications as above.
\end{lem}

\begin{proof}
(Sketch)
$E$ and $E'$ are isotopic, and any isotopy is 
homotopic to a composition of small isotopies which are either
(a) supported away from $E$, or (b) modify $E$ in the simple manner described above.
\end{proof}

It follows from the lemma that we can construct an isomorphism
between $\cS(X; c; E)$ and $\cS(X; c; E')$ for any pair $E$, $E'$.
This construction involves a choice of simple ``moves" (as above) to transform
$E$ to $E'$.
We must now show that the isomorphism does not depend on this choice.
We will show below that it suffices to check two ``movie moves".

The first movie move is to push $E$ across an $n$-ball $B$ as above, then push it back.
The result is equivalent to doing nothing.
As we remarked above, the isomorphisms corresponding to these two pushes are mutually
inverse, so we have invariance under this movie move.

The second movie move replaces two successive pushes in the same direction,
across $B_1$ and $B_2$, say, with a single push across $B_1\cup B_2$.
(See Figure \ref{jun23d}.)
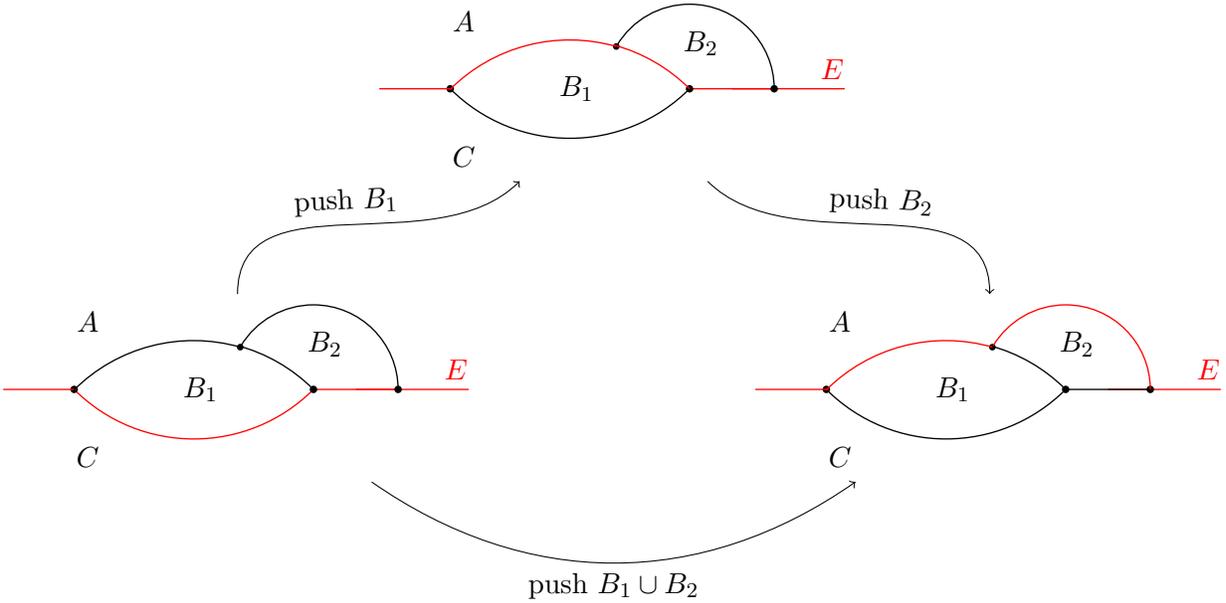
\begin{figure}[t]
\begin{tikzpicture}
\node(L) {
\scalebox{0.5}{
\begin{tikzpicture}[baseline,line width = 1pt,x=1.5cm,y=1.5cm]
\draw[red] (0.75,0) -- +(2,0);
\draw[red] (0,0) node(R) {}
	-- (0.75,0) node[below] {}
	--(1.5,0)  node[circle,fill=black,inner sep=2pt] {};
\draw[fill]  (150:1.5) circle (2pt) node[above=4pt] {};
\draw (1.5,0) arc (0:149:1.5);
\draw[red]
	(R) node[circle,fill=black,inner sep=2pt] {}
	arc (-45:-135:3) node[circle,fill=black,inner sep=2pt] {};
\draw[red] (-5.5,0) -- (-4.2,0);
\draw (R) arc (45:75:3);
\draw (150:1.5) arc (74:135:3);
\node at (-2,0) {\scalebox{2.0}{$B_1$}};
\node at (0.2,0.8) {\scalebox{2.0}{$B_2$}};
\node at (-4,1.2) {\scalebox{2.0}{$A$}};
\node at (-4,-1.2) {\scalebox{2.0}{$C$}};
\node[red] at (2.53,0.35) {\scalebox{2.0}{$E$}};
\end{tikzpicture}
}
};
\node(M) at (5,4) {
\scalebox{0.5}{
\begin{tikzpicture}[baseline,line width = 1pt,x=1.5cm,y=1.5cm]
\draw[red] (0.75,0) -- +(2,0);
\draw[red] (0,0) node(R) {}
	-- (0.75,0) node[below] {}
	--(1.5,0)  node[circle,fill=black,inner sep=2pt] {};
\draw[fill]  (150:1.5) circle (2pt) node[above=4pt] {};
\draw(1.5,0) arc (0:149:1.5);
\draw
	(R) node[circle,fill=black,inner sep=2pt] {}
	arc (-45:-135:3) node[circle,fill=black,inner sep=2pt] {};
\draw[red] (-5.5,0) -- (-4.2,0);
\draw[red] (R) arc (45:75:3);
\draw[red] (150:1.5) arc (74:135:3);
\node at (-2,0) {\scalebox{2.0}{$B_1$}};
\node at (0.2,0.8) {\scalebox{2.0}{$B_2$}};
\node at (-4,1.2) {\scalebox{2.0}{$A$}};
\node at (-4,-1.2) {\scalebox{2.0}{$C$}};
\node[red] at (2.53,0.35) {\scalebox{2.0}{$E$}};
\end{tikzpicture}
}
};
\node(R) at (10,0) {
\scalebox{0.5}{
\begin{tikzpicture}[baseline,line width = 1pt,x=1.5cm,y=1.5cm]
\draw[red] (0.75,0) -- +(2,0);
\draw (0,0) node(R) {}
	-- (0.75,0) node[below] {}
	--(1.5,0)  node[circle,fill=black,inner sep=2pt] {};
\draw[fill]  (150:1.5) circle (2pt) node[above=4pt] {};
\draw[red] (1.5,0) arc (0:149:1.5);
\draw
	(R) node[circle,fill=black,inner sep=2pt] {}
	arc (-45:-135:3) node[circle,fill=black,inner sep=2pt] {};
\draw[red] (-5.5,0) -- (-4.2,0);
\draw (R) arc (45:75:3);
\draw[red] (150:1.5) arc (74:135:3);
\node at (-2,0) {\scalebox{2.0}{$B_1$}};
\node at (0.2,0.8) {\scalebox{2.0}{$B_2$}};
\node at (-4,1.2) {\scalebox{2.0}{$A$}};
\node at (-4,-1.2) {\scalebox{2.0}{$C$}};
\node[red] at (2.53,0.35) {\scalebox{2.0}{$E$}};
\end{tikzpicture}
}
};
\draw[->] (L) to[out=90,in=225] node[sloped, above] {push $B_1$} (M);
\draw[->] (M)  to[out=-45,in=90] node[sloped, above] {push $B_2$} (R);
\draw[->] (L) to[out=-35,in=-145] node[sloped, below] {push $B_1 \cup B_2$} (R);
\end{tikzpicture}
\caption{A movie move}
\label{jun23d}
\end{figure}
Invariance under this movie move follows from the compatibility of the inner
product for $B_1\cup B_2$ with the inner products for $B_1$ and $B_2$.


If $n\ge 2$, these two movie moves suffice:

\begin{lem}
Assume $n\ge 2$ and fix $E$ and $E'$ as above.
Then any two sequences of elementary moves connecting $E$ to $E'$
are related by a sequence of the two movie moves defined above.
\end{lem}

\begin{proof}
(Sketch)
Consider a two parameter family of diffeomorphisms (one parameter family of isotopies) 
of $\bd X$.
Up to homotopy,
such a family is homotopic to a family which can be decomposed 
into small families which are either
(a) supported away from $E$, 
(b) have boundaries corresponding to the two movie moves above.
Finally, observe that the space of $E$'s is simply connected.
(This fails for $n=1$.)
\end{proof}

For $n=1$ we have to check an additional ``global" relation corresponding to 
rotating the 0-sphere $E$ around the 1-sphere $\bd X$.
But if $n=1$, then we are in the case of ordinary algebroids and bimodules,
and this is just the well-known ``Frobenius reciprocity" result for bimodules \cite{MR1424954}.

\medskip

We have now defined $\cS(X; c)$ for any $n{+}1$-ball $X$ with boundary decoration $c$.
We must also define, for any homeomorphism $X\to X'$, an action $f: \cS(X; c) \to \cS(X', f(c))$.
Choosing an equator $E\sub \bd X$ we have 
\[
	\cS(X; c) \cong \cS(X; c; E) \deq \hom_{\cS(E_c)}(\cS(\bd_-X_c), \cS(\bd_+X_c)) .
\]
We define $f: \cS(X; c) \to \cS(X', f(c))$ to be the tautological map
\[
	f: \cS(X; c; E) \to \cS(X'; f(c); f(E)) .
\]
It is easy to show that this is independent of the choice of $E$.
Note also that this map depends only on the restriction of $f$ to $\bd X$.
In particular, if $F: X\to X$ is the identity on $\bd X$ then $f$ acts trivially, as required by
Axiom \ref{axiom:extended-isotopies}.

We define product $n{+}1$-morphisms to be identity maps of modules.

To define (binary) composition of $n{+}1$-morphisms, choose the obvious common equator
then compose the module maps.
The proof that this composition rule is associative is similar to the proof of Lemma \ref{equator-lemma}.

\medskip

We end this subsection with some remarks about Morita equivalence of disk-like $n$-categories.
Recall that two 1-categories $\cC$ and $\cD$ are Morita equivalent if and only if they are equivalent
objects in the 2-category of (linear) 1-categories, bimodules, and intertwiners.
Similarly, we define two disk-like $n$-categories to be Morita equivalent if they are equivalent objects in the
$n{+}1$-category of sphere modules.

Because of the strong duality enjoyed by disk-like $n$-categories, the data for such an equivalence lives only in 
dimensions 1 and $n+1$ (the middle dimensions come along for free).
The $n{+}1$-dimensional part of the data must be invertible and satisfy
identities corresponding to Morse cancellations in $n$-manifolds.
We will treat this in detail for the $n=2$ case; the case for general $n$ is very similar.

Let $\cC$ and $\cD$ be (unoriented) disk-like 2-categories.
Let $\cS$ denote the 3-category of 2-category sphere modules.
The 1-dimensional part of the data for a Morita equivalence between $\cC$ and $\cD$ is a 0-sphere module $\cM = {}_\cC\cM_\cD$ 
(categorified bimodule) connecting $\cC$ and $\cD$.
Because of the full unoriented symmetry, this can also be thought of as a 
0-sphere module ${}_\cD\cM_\cC$ connecting $\cD$ and $\cC$.

We want $\cM$ to be an equivalence, so we need 2-morphisms in $\cS$ 
between ${}_\cC\cM_\cD \otimes_\cD {}_\cD\cM_\cC$ and the identity 0-sphere module ${}_\cC\cC_\cC$, and similarly
with the roles of $\cC$ and $\cD$ reversed.
These 2-morphisms come for free, in the sense of not requiring additional data, since we can take them to be the labeled 
cell complexes (cups and caps) in $B^2$ shown in Figure \ref{morita-fig-1}.

\definecolor{C}{named}{orange}
\definecolor{D}{named}{blue}
\definecolor{M}{named}{purple}

\begin{figure}[t]

$$
\begin{tikzpicture}
\node(L) at (0,0) {\tikz{
	\draw[C] (0,0) -- node[below] {$\cC$} (1,0);
	\draw[D] (1,0) -- node[below] {$\cD$} (2,0);
	\draw[C] (2,0) -- node[below] {$\cC$} (3,0);
	\node[M, fill, circle, inner sep=2pt, label=$\cM$] at (1,0) {};
	\node[M, fill, circle, inner sep=2pt, label=$\cM$] at (2,0) {};
}};

\node(R) at (6,0) {\tikz{
	\draw[C] (0,0) -- node[below] {$\cC$} (3,0);
	\node[label={\phantom{$\cM$}}] at (1.5,0) {};
}};

\node at (-1,-1.5) { $\leftidx{_\cC}{(\cM \tensor_\cD \cM)}{_\cC}$ };
\node at (7,-1.5) { $\leftidx{_\cC}{\cC}{_\cC}$ };

\draw[->] (L) to[out=35, in=145] node[below] {$w$} node[above] { \tikz{
	\draw (0,0) circle (16pt);
	\path[clip] (0,0) circle (16pt);
	\draw[fill=C!20] (0,0) circle (16pt);
	\draw[M,fill=D!20,line width=2pt] (0,-0.5) circle (16pt);
}}(R);

\draw[->] (R) to[out=-145, in=-35] node[above] {$x$} node[below] { \tikz{
	\draw (0,0) circle (16pt);
	\path[clip] (0,0) circle (16pt);
	\draw[fill=C!20] (0,0) circle (16pt);
	\draw[M,fill=D!20,line width=2pt] (0,0.5) circle (16pt);
}}(L);

\end{tikzpicture}
$$
$$
\begin{tikzpicture}
\node(L) at (0,0) {\tikz{
	\draw[D] (0,0) -- node[below] {$\cD$} (1,0);
	\draw[C] (1,0) -- node[below] {$\cC$} (2,0);
	\draw[D] (2,0) -- node[below] {$\cD$} (3,0);
	\node[M, fill, circle, inner sep=2pt, label=$\cM$] at (1,0) {};
	\node[M, fill, circle, inner sep=2pt, label=$\cM$] at (2,0) {};
}};

\node(R) at (6,0) {\tikz{
	\draw[D] (0,0) -- node[below] {$\cD$} (3,0);
	\node[label={\phantom{$\cM$}}] at (1.5,0) {};
}};

\node at (-1,-1.5) { $\leftidx{_\cD}{(\cM \tensor_\cC \cM)}{_\cD}$ };
\node at (7,-1.5) { $\leftidx{_\cD}{\cD}{_\cD}$ };

\draw[->] (L) to[out=35, in=145] node[below] {$y$} node[above] { \tikz{
	\draw (0,0) circle (16pt);
	\path[clip] (0,0) circle (16pt);
	\draw[fill=D!20] (0,0) circle (16pt);
	\draw[M,fill=C!20,line width=2pt] (0,-0.5) circle (16pt);
}}(R);

\draw[->] (R) to[out=-145, in=-35] node[above] {$z$} node[below] { \tikz{
	\draw (0,0) circle (16pt);
	\path[clip] (0,0) circle (16pt);
	\draw[fill=D!20] (0,0) circle (16pt);
	\draw[M,fill=C!20,line width=2pt] (0,0.5) circle (16pt);
}}(L);

\end{tikzpicture}
$$
\caption{Cups and caps for free}\label{morita-fig-1}
\end{figure}
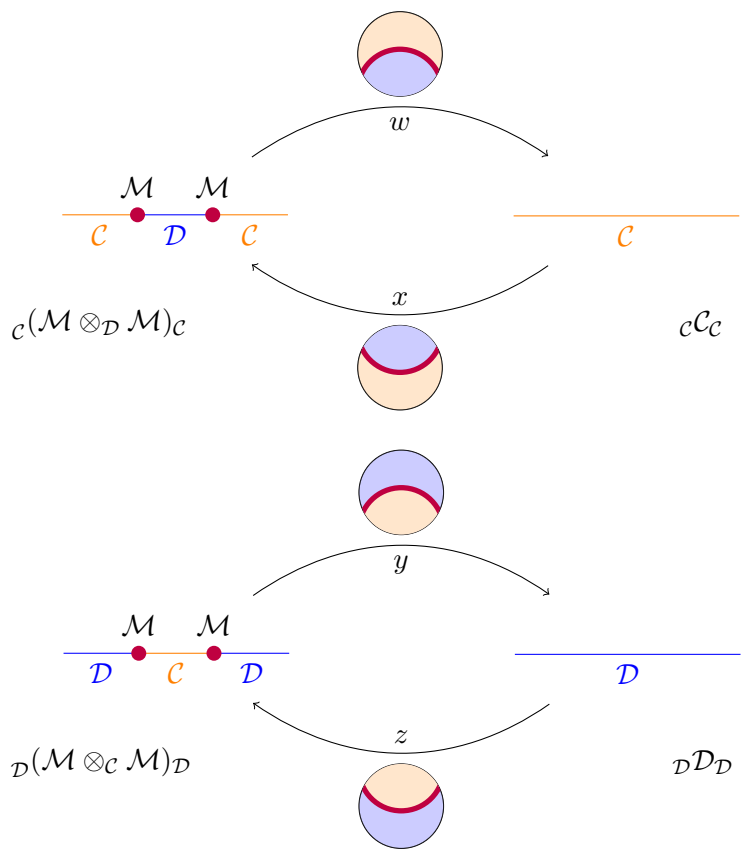

We want the 2-morphisms from the previous paragraph to be equivalences, so we need 3-morphisms
between various compositions of these 2-morphisms and various identity 2-morphisms.
Recall that the 3-morphisms of $\cS$ are intertwiners between representations of 1-categories associated
to decorated circles.
Figure \ref{morita-fig-2} 
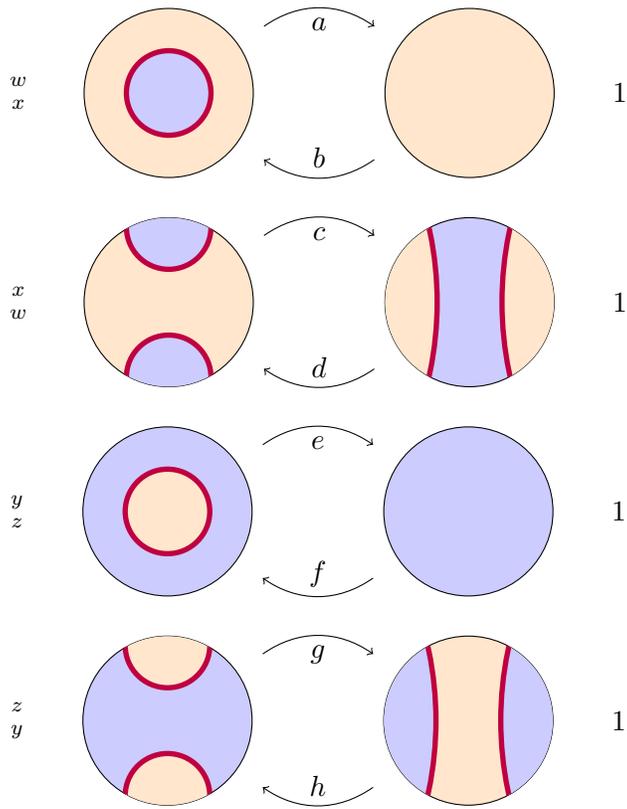
\begin{figure}[t]
$$
\begin{tikzpicture}
\node(L) at (0,0) {\tikz{
	\draw[fill=C!20] (0,0) circle (32pt);
	\draw[M,fill=D!20,line width=2pt] (0,0) circle (16pt);
}};
\node(R) at (4,0) {\tikz{
	\draw[fill=C!20] (0,0) circle (32pt);
}};
\draw[->] (L) to[out=35, in=145] node[below] {$a$} (R);
\draw[->] (R) to[out=-145, in=-35] node[above] {$b$} (L);
\node at (-2,0) {$w \atop x$};
\node at (6,0) {$1$};
\end{tikzpicture}
$$
$$
\begin{tikzpicture}
\node(L) at (0,0) {\tikz{
	\draw[fill=C!20] (0,0) circle (32pt);
	\path[clip] (0,0) circle (32pt);
	\draw[M,fill=D!20,line width=2pt] (0,1) circle (16pt);
	\draw[M,fill=D!20,line width=2pt] (0,-1) circle (16pt);
}};
\node(R) at (4,0) {\tikz{
	\draw[fill=D!20] (0,0) circle (32pt);
	\path[clip] (0,0) circle (32pt);
	\draw[M,fill=C!20,line width=2pt] (5,0) circle (130pt);
	\draw[M,fill=C!20,line width=2pt] (-5,0) circle (130pt);
}};
\draw[->] (L) to[out=35, in=145] node[below] {$c$} (R);
\draw[->] (R) to[out=-145, in=-35] node[above] {$d$} (L);
\node at (-2,0) {$x \atop w$};
\node at (6,0) {$1$};
\end{tikzpicture}
$$
$$
\begin{tikzpicture}
\node(L) at (0,0) {\tikz{
	\draw[fill=D!20] (0,0) circle (32pt);
	\draw[M,fill=C!20,line width=2pt] (0,0) circle (16pt);
}};
\node(R) at (4,0) {\tikz{
	\draw[fill=D!20] (0,0) circle (32pt);
}};
\draw[->] (L) to[out=35, in=145] node[below] {$e$} (R);
\draw[->] (R) to[out=-145, in=-35] node[above] {$f$} (L);
\node at (-2,0) {$y \atop z$};
\node at (6,0) {$1$};
\end{tikzpicture}
$$
$$
\begin{tikzpicture}
\node(L) at (0,0) {\tikz{
	\draw[fill=D!20] (0,0) circle (32pt);
	\path[clip] (0,0) circle (32pt);
	\draw[M,fill=C!20,line width=2pt] (0,1) circle (16pt);
	\draw[M,fill=C!20,line width=2pt] (0,-1) circle (16pt);
}};
\node(R) at (4,0) {\tikz{
	\draw[fill=C!20] (0,0) circle (32pt);
	\path[clip] (0,0) circle (32pt);
	\draw[M,fill=D!20,line width=2pt] (5,0) circle (130pt);
	\draw[M,fill=D!20,line width=2pt] (-5,0) circle (130pt);
}};
\draw[->] (L) to[out=35, in=145] node[below] {$g$} (R);
\draw[->] (R) to[out=-145, in=-35] node[above] {$h$} (L);
\node at (-2,0) {$z \atop y$};
\node at (6,0) {$1$};
\end{tikzpicture}
$$

\caption{Intertwiners for a Morita equivalence}\label{morita-fig-2}
\end{figure}
shows the intertwiners we need.
Each decorated 2-ball in that figure determines a representation of the 1-category associated to the decorated circle
on the boundary.
This is the 3-dimensional part of the data for the Morita equivalence.
(Note that, by symmetry, the $c$ and $d$ arrows of Figure \ref{morita-fig-2} 
are the same (up to rotation), as the $h$ and $g$ arrows.)

In order for these 3-morphisms to be equivalences, 
they must be invertible (i.e.\ $a=b\inv$, $c=d\inv$, $e=f\inv$) and in addition
they must satisfy identities corresponding to Morse cancellations on 2-manifolds.
These are illustrated in Figure \ref{morita-fig-3}.
\begin{figure}[t]
$$
\begin{tikzpicture}
\node(L) at (0,0) {\tikz{
\draw[fill=C!20] (0,0) circle (32pt);
\path[clip] (0,0) circle (32pt);
\draw[M,fill=D!20,line width=2pt] (-5,0) circle (130pt);
}};
\node(C) at (4,0) {\tikz{
\draw[fill=C!20] (0,0) circle (32pt);
\path[clip] (0,0) circle (32pt);
\draw[M,fill=D!20,line width=2pt] (-5,0) circle (130pt);
\draw[M,fill=D!20,line width=2pt] (0.25,0) circle (6pt);
}};
\node(R) at (8,0) {\tikz{
\draw[fill=C!20] (0,0) circle (32pt);
\path[clip] (0,0) circle (32pt);
\draw[M,line width=4pt] (-0.75,2) .. controls +(0,-2) and +(0,0.5) .. (0.2,0) .. controls +(0,-0.5) and +(0,2) .. (-0.75,-2) -- (-5,-2) -- (-5,2) -- cycle;
\path[clip] (-0.75,2) .. controls +(0,-2) and +(0,0.5) .. (0.2,0) .. controls +(0,-0.5) and +(0,2) .. (-0.75,-2) -- (-5,-2) -- (-5,2)  -- cycle;
\path[fill=D!20] (-5,-2) rectangle (5,2); 
}};
\draw[<-] (L) to[out=35, in=145] node[above] {$a$} (C);
\draw[<-] (C) to[out=35, in=145] node[above] {$d$} (R);
\draw[<-] (R) to[out=-145, in=-35] node[below] {$c$} (C);
\draw[<-] (C) to[out=-145, in=-35] node[below] {$b$} (L);
\end{tikzpicture}
$$
$$
\begin{tikzpicture}
\node(L) at (0,0) {\tikz{
\draw[fill=D!20] (0,0) circle (32pt);
\path[clip] (0,0) circle (32pt);
\draw[M,fill=C!20,line width=2pt] (-5,0) circle (130pt);
}};
\node(C) at (4,0) {\tikz{
\draw[fill=D!20] (0,0) circle (32pt);
\path[clip] (0,0) circle (32pt);
\draw[M,fill=C!20,line width=2pt] (-5,0) circle (130pt);
\draw[M,fill=C!20,line width=2pt] (0.25,0) circle (6pt);
}};
\node(R) at (8,0) {\tikz{
\draw[fill=D!20] (0,0) circle (32pt);
\path[clip] (0,0) circle (32pt);
\draw[M,line width=4pt] (-0.75,2) .. controls +(0,-2) and +(0,0.5) .. (0.2,0) .. controls +(0,-0.5) and +(0,2) .. (-0.75,-2) -- (-5,-2) -- (-5,2) -- cycle;
\path[clip] (-0.75,2) .. controls +(0,-2) and +(0,0.5) .. (0.2,0) .. controls +(0,-0.5) and +(0,2) .. (-0.75,-2) -- (-5,-2) -- (-5,2)  -- cycle;
\path[fill=C!20] (-5,-2) rectangle (5,2); 
}};
\draw[<-] (L) to[out=35, in=145] node[above] {$e$} (C);
\draw[<-] (C) to[out=35, in=145] node[above] {$c$} (R);
\draw[<-] (R) to[out=-145, in=-35] node[below] {$d$} (C);
\draw[<-] (C) to[out=-145, in=-35] node[below] {$f$} (L);
\end{tikzpicture}
$$
\caption{Identities for intertwiners}\label{morita-fig-3}
\end{figure}
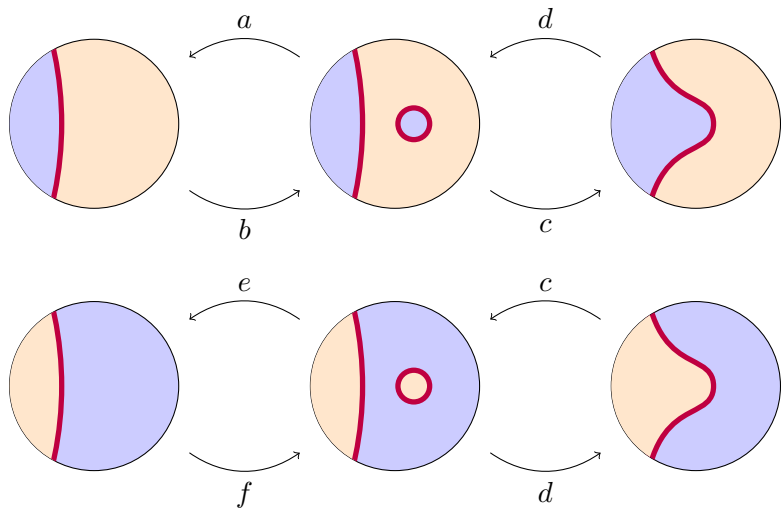
Each line shows a composition of two intertwiners which we require to be equal to the identity intertwiner.
The modules corresponding to the leftmost and rightmost disks in the figure can be identified via the obvious isotopy.

For general $n$, we start with an $n$-category 0-sphere module $\cM$ which is the data for the 1-dimensional
part of the Morita equivalence.
For $2\le k \le n$, the $k$-dimensional parts of the Morita equivalence are various decorated $k$-balls with submanifolds
labeled by $\cC$, $\cD$ and $\cM$; no additional data is needed for these parts.
The $n{+}1$-dimensional part of the equivalence is given by certain intertwiners, and these intertwiners must 
be invertible and satisfy
identities corresponding to Morse cancellations in $n$-manifolds. 

\noop{
One way of thinking of these conditions is as follows.
Given a decorated $n{+}1$-manifold, with a codimension 1 submanifold labeled by $\cM$ and 
codimension 0 submanifolds labeled by $\cC$ and $\cD$, we can make any local modification we like without 
changing
}

If $\cC$ and $\cD$ are Morita equivalent $n$-categories, then it is easy to show that for any $n-j$-manifold
$Y$ the $j$-categories $\cC(Y)$ and $\cD(Y)$ are Morita equivalent.
When $j=0$ this means that the TQFT Hilbert spaces $\cC(Y)$ and $\cD(Y)$ are isomorphic 
(if we are enriching over vector spaces).

\noop{ 
More specifically, the 1-dimensional part of the data is a 0-sphere module $\cM = {}_\cCM_\cD$ 
(categorified bimodule) connecting $\cC$ and $\cD$.
From $\cM$ we can construct various $k$-sphere modules $N^k_{j,E}$ for $0 \le k \le n$, $0\le j \le k$, and $E = \cC$ or $\cD$.
$N^k_{j,E}$ can be thought of as the graph of an index $j$ Morse function on the $k$-ball $B^k$
(so the graph lives in $B^k\times I = B^{k+1}$).
The positive side of the graph is labeled by $E$, the negative side by $E'$
(where $\cC' = \cD$ and $\cD' = \cC$), and the codimension-1 
submanifold separating the positive and negative regions is labeled by $\cM$.
We think of $N^k_{j,E}$ as a $k{+}1$-morphism connecting 
We plan on treating this in more detail in a future paper.
\nn{should add a few more details}
}


\section{The blob complex for \texorpdfstring{$A_\infty$}{A-infinity} \texorpdfstring{$n$}{n}-categories}
\label{sec:ainfblob}
Given an $A_\infty$ $n$-category $\cC$ and an $n$-manifold $M$, we make the following 
anticlimactically tautological definition of the blob
complex.
\begin{defn}
The blob complex $\bc_*(M;\cC)$ of an $n$-manifold $M$ with coefficients in 
an $A_\infty$ $n$-category $\cC$ is the homotopy colimit $\cl{\cC}(M)$ of \S\ref{ss:ncat_fields}.
\end{defn}

We will show below 
in Corollary \ref{cor:new-old}
that when $\cC$ is obtained from a system of fields $\cE$ 
as the blob complex of an $n$-ball (see Example \ref{ex:blob-complexes-of-balls}), 
$\cl{\cC}(M)$ is homotopy equivalent to
our original definition of the blob complex $\bc_*(M;\cE)$.


%

\subsection{A product formula}
\label{ss:product-formula}

Given an $n$-dimensional system of fields $\cE$ and a $n{-}k$-manifold $F$, recall from 
Example \ref{ex:blob-complexes-of-balls} that there is an  $A_\infty$ $k$-category $\cC_F$ 
defined by $\cC_F(X) = \cE(X\times F)$ if $\dim(X) < k$ and
$\cC_F(X) = \bc_*(X\times F;\cE)$ if $\dim(X) = k$.

\begin{thm} \label{thm:product}
Let $Y$ be a $k$-manifold which admits a ball decomposition
(e.g.\ any triangulable manifold).
Then there is a homotopy equivalence between ``old-fashioned" (blob diagrams) 
and ``new-fangled" (hocolimit) blob complexes
\[
	\cB_*(Y \times F) \htpy \cl{\cC_F}(Y) .
\]\end{thm}

\begin{proof}
We will use the concrete description of the homotopy colimit from \S\ref{ss:ncat_fields}.

First we define a map 
\[
	\psi: \cl{\cC_F}(Y) \to \bc_*(Y\times F;\cE) .
\]
On 0-simplices of the hocolimit 
we just glue together the various blob diagrams on $X_i\times F$
(where $X_i$ is a component of a permissible decomposition of $Y$) to get a blob diagram on
$Y\times F$.
For simplices of dimension 1 and higher we define the map to be zero.
It is easy to check that this is a chain map.

In the other direction, we will define (in the next few paragraphs) 
a subcomplex $G_*\sub \bc_*(Y\times F;\cE)$ and a map
\[
	\phi: G_* \to \cl{\cC_F}(Y) .
\]

Given a decomposition $K$ of $Y$ into $k$-balls $X_i$, let $K\times F$ denote the corresponding
decomposition of $Y\times F$ into the pieces $X_i\times F$.

Let $G_*\sub \bc_*(Y\times F;\cE)$ be the subcomplex generated by blob diagrams $a$ such that there
exists a decomposition $K$ of $Y$ such that $a$ splits along $K\times F$.
It follows from Lemma \ref{thm:small-blobs} that $\bc_*(Y\times F; \cE)$ 
is homotopic to a subcomplex of $G_*$.
(If the blobs of $a$ are small with respect to a sufficiently fine cover then their
projections to $Y$ are contained in some disjoint union of balls.)
Note that the image of $\psi$ is equal to $G_*$.

We will define $\phi: G_* \to \cl{\cC_F}(Y)$ using the method of acyclic models.
Let $a$ be a generator of $G_*$.
Let $D(a)$ denote the subcomplex of $\cl{\cC_F}(Y)$ generated by all $(b, \ol{K})$
where $b$ is a generator appearing
in an iterated boundary of $a$ (this includes $a$ itself)
and $b$ splits along $K_0\times F$.
(Recall that $\ol{K} = (K_0,\ldots,K_l)$ denotes a chain of decompositions;
see \S\ref{ss:ncat_fields}.)
By $(b, \ol{K})$ we really mean $(b^\sharp, \ol{K})$, where $b^\sharp$ is 
$b$ split according to $K_0\times F$.
To simplify notation we will just write plain $b$ instead of $b^\sharp$.
Roughly speaking, $D(a)$ consists of 0-simplices which glue up to give
$a$ (or one of its iterated boundaries), 1-simplices which connect all the 0-simplices, 
2-simplices which kill the homology created by the 
1-simplices, and so on.
More formally,
 
\begin{lemma} \label{lem:d-a-acyclic}
$D(a)$ is acyclic in positive degrees.
\end{lemma}

\begin{proof}
Let $P(a)$ denote the finite cone-product polyhedron composed of $a$ and its iterated boundaries.
(See Remark \ref{blobsset-remark}.)
We can think of $D(a)$ as a cell complex equipped with an obvious
map $p: D(a) \to P(a)$ which forgets the second factor.
For each cell $b$ of $P(a)$, let $I(b) = p\inv(b)$.
It suffices to show that each $I(b)$ is acyclic and more generally that
each intersection $I(b)\cap I(b')$ is acyclic.

If $I(b)\cap I(b')$ is nonempty then then as a cell complex it is isomorphic to
$(b\cap b') \times E(b, b')$, where $E(b, b')$ consists of those simplices
$\ol{K} = (K_0,\ldots,K_l)$ such that both $b$ and $b'$ split along $K_0\times F$.
(Here we are thinking of $b$ and $b'$ as both blob diagrams and also faces of $P(a)$.)
So it suffices to show that $E(b, b')$ is acyclic.

Let $K$ and $K'$ be two decompositions of $Y$ (i.e.\ 0-simplices) in $E(b, b')$.
We want to find 1-simplices which connect $K$ and $K'$.
We might hope that $K$ and $K'$ have a common refinement, but this is not necessarily
the case.
(Consider the $x$-axis and the graph of $y = e^{-1/x^2} \sin(1/x)$ in $\r^2$.)
However, we {\it can} find another decomposition $L$ such that $L$ shares common
refinements with both $K$ and $K'$. (For instance, in the example above, $L$ can be the graph of $y=x^2-1$.)
This follows from Axiom \ref{axiom:splittings}, which in turn follows from the
splitting axiom for the system of fields $\cE$.
Let $KL$ and $K'L$ denote these two refinements.
Then 1-simplices associated to the four anti-refinements
$KL\to K$, $KL\to L$, $K'L\to L$ and $K'L\to K'$
give the desired chain connecting $(a, K)$ and $(a, K')$
(see Figure \ref{zzz4}).
(In the language of Lemma \ref{lemma:vcones}, this is $\vcone(K \du K')$.)

\begin{figure}[t] \centering
\begin{tikzpicture}
\foreach \x/\label in {-3/K, 0/L, 3/K'} {
	\node(\label) at (\x,0) {$\label$};
}
\foreach \x/\la/\lb in {-1.5/K/L, 1.5/K'/L} {
	\node(\la \lb) at (\x,-1.5) {$\la \lb$};
	\draw[->] (\la \lb) -- (\la);
	\draw[->] (\la \lb) -- (\lb); 
}
\end{tikzpicture}
\caption{Connecting $K$ and $K'$ via $L$}
\label{zzz4}
\end{figure}
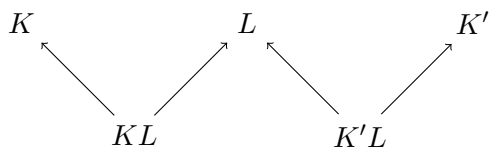

Consider next a 1-cycle in $E(b, b')$, such as one arising from
a different choice of decomposition $L'$ in place of $L$ above.
By Lemma \ref{lemma:vcones} we can fill in this 1-cycle with 2-simplices.
Choose a decomposition $M$ which has common refinements with each of 
$K$, $KL$, $L$, $K'L$, $K'$, $K'L'$, $L'$ and $KL'$.
(We also require that $KLM$ antirefines to $KM$, etc.)
Then we have 2-simplices, as shown in Figure \ref{zzz5}, which do the trick.
(Each small triangle in Figure \ref{zzz5} can be filled with a 2-simplex.)

\begin{figure}[t] \centering
\begin{tikzpicture}
\node(M) at (0,0) {$M$};
\foreach \angle/\label in {0/K', 45/K'L, 90/L, 135/KL, 180/K, 225/KL', 270/L', 315/K'L'} {
	\node(\label) at (\angle:4) {$\label$};
}
\foreach \label in {K', L, K, L'} {
	\node(\label M) at ($(M)!0.6!(\label)$) {$\label M$};
	\draw[->] (\label M)--(M);
	\draw[->] (\label M)--(\label);
}
\foreach \k in {K, K'} {
	\foreach \l in {L, L'} {
		\node(\k \l M) at (intersection cs: first line={(\k M)--(\l)}, second line={(\l M)--(\k)}) {$\k \l M$};
		\draw[->] (\k \l M)--(M);
		\draw[->] (\k \l M)--(\k \l );
		\draw[->] (\k \l M)--(\k M);
		\draw[->] (\k \l M)--(\l);
		\draw[->] (\k \l M)--(\l M);
		\draw[->] (\k \l M)--(\k);
	}
}
\draw[->] (K'L') to[bend right=10] (K');
\draw[->] (K'L') to[bend left=10] (L');
\draw[->] (KL') to[bend left=10] (K);
\draw[->] (KL') to[bend right=10] (L');
\draw[->] (K'L) to[bend left=10] (K');
\draw[->] (K'L) to[bend right=10] (L);
\draw[->] (KL) to[bend right=10] (K);
\draw[->] (KL) to[bend left=10] (L);
\end{tikzpicture}
\caption{Filling in $K$-$KL$-$L$-$K'L$-$K'$-$K'L'$-$L'$-$KL'$-$K$}
\label{zzz5}
\end{figure}
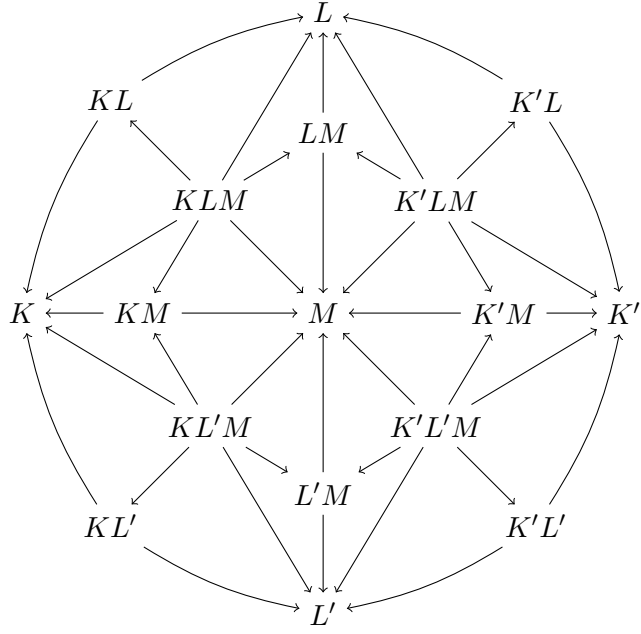

Continuing in this way we see that $D(a)$ is acyclic.
By Lemma \ref{lemma:vcones} we can fill in any cycle with a V-Cone.
\end{proof}

We are now in a position to apply the method of acyclic models to get a map
$\phi:G_* \to \cl{\cC_F}(Y)$.
We may assume that $\phi(a)$ has the form $(a, K) + r$, where $(a, K)$ is a 0-simplex
and $r$ is a sum of simplices of dimension 1 or higher.

We now show that $\phi\circ\psi$ and $\psi\circ\phi$ are homotopic to the identity.

First, $\psi\circ\phi$ is the identity on the nose:
\[
	\psi(\phi(a)) = \psi((a,K)) + \psi(r) = a + 0.
\]
Roughly speaking, $(a, K)$ is just $a$ chopped up into little pieces, and 
$\psi$ glues those pieces back together, yielding $a$.
We have $\psi(r) = 0$ since $\psi$ is zero on $(\ge 1)$-simplices.
 
Second, $\phi\circ\psi$ is the identity up to homotopy by another argument based on the method of acyclic models.
To each generator $(b, \ol{K})$ of $\cl{\cC_F}(Y)$ we associate the acyclic subcomplex $D(b)$ defined above.
Both the identity map and $\phi\circ\psi$ are compatible with this
collection of acyclic subcomplexes, so by the usual method of acyclic models argument these two maps
are homotopic.

This concludes the proof of Theorem \ref{thm:product}.
\end{proof}


If $Y$ has dimension $k-m$, then we have an $m$-category $\cC_{Y\times F}$ whose value at
a $j$-ball $X$ is either $\cE(X\times Y\times F)$ (if $j<m$) or $\bc_*(X\times Y\times F)$
(if $j=m$).
(See Example \ref{ex:blob-complexes-of-balls}.)
Similarly we have an $m$-category whose value at $X$ is $\cl{\cC_F}(X\times Y)$.
These two categories are equivalent, but since we do not define functors between
disk-like $n$-categories in this paper we are unable to say precisely
what ``equivalent" means in this context.
We hope to include this stronger result in a future paper.

\medskip

Taking $F$ in Theorem \ref{thm:product} to be a point, we obtain the following corollary.

\begin{cor}
\label{cor:new-old}
Let $\cE$ be a system of fields (with local relations) and let $\cC_\cE$ be the $A_\infty$
$n$-category obtained from $\cE$ by taking the blob complex of balls.
Then for all $n$-manifolds $Y$ the old-fashioned and new-fangled blob complexes are
homotopy equivalent:
\[
	\bc^\cE_*(Y) \htpy \cl{\cC_\cE}(Y) .
\]
\end{cor}

\medskip

Theorem \ref{thm:product} extends to the case of general fiber bundles
\[
	F \to E \to Y ,
\]
and indeed even to the case of general maps
\[
	M\to Y .
\]
We outline two approaches to these generalizations.
The first is somewhat tautological, while the second is more amenable to
calculation.

We can generalize the definition of a $k$-category by replacing the categories
of $j$-balls ($j\le k$) with categories of $j$-balls $D$ equipped with a map $p:D\to Y$
(c.f. \cite{MR2079378}).
Call this a {\it $k$-category over $Y$}.
A fiber bundle $F\to E\to Y$ gives an example of a $k$-category over $Y$:
assign to $p:D\to Y$ the blob complex $\bc_*(p^*(E))$, when $\dim(D) = k$,
or the fields $\cE(p^*(E))$, when $\dim(D) < k$.
(Here $p^*(E)$ denotes the pull-back bundle over $D$.)
Let $\cF_E$ denote this $k$-category over $Y$.
We can adapt the homotopy colimit construction (based on decompositions of $Y$ into balls) to
get a chain complex $\cl{\cF_E}(Y)$.

\begin{thm}
Let $F \to E \to Y$ be a fiber bundle and let $\cF_E$ be the $k$-category over $Y$ defined above.
Then
\[
	\bc_*(E) \simeq \cl{\cF_E}(Y) .
\]
\qed
\end{thm}

\begin{proof}
The proof is nearly identical to the proof of Theorem \ref{thm:product}, so we will only give a sketch which 
emphasizes the few minor changes that need to be made.

As before, we define a map
\[
	\psi: \cl{\cF_E}(Y) \to \bc_*(E) .
\]
The 0-simplices of the homotopy colimit $\cl{\cF_E}(Y)$ are glued up to give an element of $\bc_*(E)$.
Simplices of positive degree are sent to zero.

Let $G_* \sub \bc_*(E)$ be the image of $\psi$.
By Lemma \ref{thm:small-blobs}, $\bc_*(Y\times F; \cE)$ 
is homotopic to a subcomplex of $G_*$.
We will define a homotopy inverse of $\psi$ on $G_*$, using acyclic models.
To each generator $a$ of $G_*$ we assign an acyclic subcomplex $D(a) \sub \cl{\cF_E}(Y)$ which consists of
0-simplices which map via $\psi$ to $a$, plus higher simplices (as described in the proof of Theorem \ref{thm:product})
which insure that $D(a)$ is acyclic.
\end{proof}

We can generalize this result still further by noting that it is not really necessary
for the definition of $\cF_E$ that $E\to Y$ be a fiber bundle.
Let $M\to Y$ be a map, with $\dim(M) = n$ and $\dim(Y) = k$.
Call a map $D^j\to Y$ ``good" with respect to $M$ if the fibered product
$D\widetilde{\times} M$ is a manifold of dimension $n-k+j$ with a collar structure along the boundary of $D$.
(If $D\to Y$ is an embedding then $D\widetilde{\times} M$ is just the part of $M$
lying above $D$.)
We can define a $k$-category $\cF_M$ based on maps of balls into $Y$ which are good with respect to $M$.
We can again adapt the homotopy colimit construction to
get a chain complex $\cl{\cF_M}(Y)$.
The proof of Theorem \ref{thm:product} again goes through essentially unchanged 
to show that
\[
	\bc_*(M) \simeq \cl{\cF_M}(Y) .
\]

\medskip

In the second approach we use a decorated colimit (as in \S \ref{ssec:spherecat}) 
and various sphere modules based on $F \to E \to Y$
or $M\to Y$, instead of an undecorated colimit with fancier $k$-categories over $Y$.
Information about the specific map to $Y$ has been taken out of the categories
and put into sphere modules and decorations.

Let $F \to E \to Y$ be a fiber bundle as above.
Choose a decomposition $Y = \cup X_i$
such that the restriction of $E$ to $X_i$ is homeomorphic to a product $F\times X_i$,
and choose trivializations of these products as well.

Let $\cF$ be the $k$-category associated to $F$.
To each codimension-1 face $X_i\cap X_j$ we have a bimodule ($S^0$-module) for $\cF$.
More generally, to each codimension-$m$ face we have an $S^{m-1}$-module for a $(k{-}m{+}1)$-category
associated to the (decorated) link of that face.
We can decorate the strata of the decomposition of $Y$ with these sphere modules and form a 
colimit as in \S \ref{ssec:spherecat}.
This colimit computes $\bc_*(E)$.

There is a similar construction for general maps $M\to Y$.


\subsection{A gluing theorem}
\label{sec:gluing}

Next we prove a gluing theorem. Throughout this section fix a particular $n$-dimensional system of fields $\cE$ and local relations. Each blob complex below is  with respect to this $\cE$.
Let $X$ be a closed $k$-manifold with a splitting $X = X'_1\cup_Y X'_2$.
We will need an explicit collar on $Y$, so rewrite this as
$X = X_1\cup (Y\times J) \cup X_2$.
Given this data we have:
\begin{itemize}
\item An $A_\infty$ $n{-}k$-category $\bc(X)$, which assigns to an $m$-ball
$D$ fields on $D\times X$ (for $m+k < n$) or the blob complex $\bc_*(D\times X; c)$
(for $m+k = n$).
(See Example \ref{ex:blob-complexes-of-balls}.)
\item An $A_\infty$ $n{-}k{+}1$-category $\bc(Y)$, defined similarly.
\item Two $\bc(Y)$ modules $\bc(X_1)$ and $\bc(X_2)$, which assign to a marked
$m$-ball $(D, H)$ either fields on $(D\times Y) \cup (H\times X_i)$ (if $m+k < n$)
or the blob complex $\bc_*((D\times Y) \cup (H\times X_i))$ (if $m+k = n$).
(See Example \ref{bc-module-example}.)
\item The tensor product $\bc(X_1) \otimes_{\bc(Y), J} \bc(X_2)$, which is
an $A_\infty$ $n{-}k$-category.
(See \S \ref{moddecss}.)
\end{itemize}

It is the case that the $n{-}k$-categories $\bc(X)$ and $\bc(X_1) \otimes_{\bc(Y), J} \bc(X_2)$
are equivalent for all $k$, but since we do not develop a definition of functor between $n$-categories
in this paper, we cannot state this precisely.
(It will appear in a future paper.)
So we content ourselves with

\begin{thm}
\label{thm:gluing}
Suppose $X$ is an $n$-manifold, and $X = X_1\cup (Y\times J) \cup X_2$ (i.e. take $k=n$ in the above discussion). 
Then $\bc(X)$ is homotopy equivalent to the $A_\infty$ tensor product $\bc(X_1) \otimes_{\bc(Y), J} \bc(X_2)$.
\end{thm}

\begin{proof}
The proof is similar to that of Theorem \ref{thm:product}.
We give a short sketch with emphasis on the differences from 
the proof of Theorem \ref{thm:product}.

Let $\cT$ denote the chain complex $\bc(X_1) \otimes_{\bc(Y), J} \bc(X_2)$.
Recall that this is a homotopy colimit based on decompositions of the interval $J$.

We define a map $\psi:\cT\to \bc_*(X)$.
On 0-simplices it is given
by gluing the pieces together to get a blob diagram on $X$.
On simplices of dimension 1 and greater $\psi$ is zero.

The image of $\psi$ is the subcomplex $G_*\sub \bc(X)$ generated by blob diagrams which split
over some decomposition of $J$.
It follows from Lemma \ref{thm:small-blobs} that $\bc_*(X)$ is homotopic to 
a subcomplex of $G_*$. 

Next we define a map $\phi:G_*\to \cT$ using the method of acyclic models.
As in the proof of Theorem \ref{thm:product}, we assign to a generator $a$ of $G_*$
an acyclic subcomplex which is (roughly) $\psi\inv(a)$.
The proof of acyclicity is easier in this case since any pair of decompositions of $J$ have
a common refinement.

The proof that these two maps are homotopy inverse to each other is the same as in
Theorem \ref{thm:product}.
\end{proof}

\medskip

\subsection{Reconstructing mapping spaces}
\label{sec:map-recon}

The next theorem shows how to reconstruct a mapping space from local data.
Let $T$ be a topological space, let $M$ be an $n$-manifold, 
and recall the $A_\infty$ $n$-category $\pi^\infty_{\leq n}(T)$ 
of Example \ref{ex:chains-of-maps-to-a-space}.
Think of $\pi^\infty_{\leq n}(T)$ as encoding everything you would ever
want to know about spaces of maps of $k$-balls into $T$ ($k\le n$).
To simplify notation, let $\cT = \pi^\infty_{\leq n}(T)$.

\begin{thm}
\label{thm:map-recon}
The blob complex for $M$ with coefficients in the fundamental $A_\infty$ $n$-category for $T$ 
is quasi-isomorphic to singular chains on maps from $M$ to $T$.
$$\cB^\cT(M) \simeq C_*(\Maps(M\to T)).$$
\end{thm}
\begin{rem}
Lurie has shown in \cite[Theorem 3.8.6]{0911.0018} that the topological chiral homology 
of an $n$-manifold $M$ with coefficients in a certain $E_n$ algebra constructed from $T$ recovers 
the same space of singular chains on maps from $M$ to $T$, with the additional hypothesis that $T$ is $n{-}1$-connected.
This extra hypothesis is not surprising, in view of the idea described in Example \ref{ex:e-n-alg} 
that an $E_n$ algebra is roughly equivalent data to an $A_\infty$ $n$-category which 
is trivial at levels 0 through $n-1$.
Ricardo Andrade also told us about a similar result.

Specializing still further, Theorem \ref{thm:map-recon} is related to the classical result that for connected spaces $T$
we have $HH_*(C_*(\Omega T)) \cong H_*(LT)$, that is, the Hochschild homology of based loops in $T$ is isomorphic
to the homology of the free loop space of $T$ (see \cite{MR793184} and \cite{MR842427}).
Theorem \ref{thm:map-recon} says that for any space $T$ (connected or not) we have
$\bc_*(S^1; C_*(\pi^\infty_{\le 1}(T))) \simeq C_*(LT)$.
Here $C_*(\pi^\infty_{\le 1}(T))$ denotes the singular chain version of the fundamental infinity-groupoid of $T$, 
whose objects are points in $T$ and morphism chain complexes are $C_*(\paths(t_1 \to t_2))$ for $t_1, t_2 \in T$.
If $T$ is connected then the $A_\infty$ 1-category $C_*(\pi^\infty_{\le 1}(T))$ is Morita equivalent to the
$A_\infty$ algebra $C_*(\Omega T)$; 
the bimodule for the equivalence is the singular chains of the space of paths which start at the base point of $T$.
Theorem \ref{thm:hochschild} holds for $A_\infty$ 1-categories (though we do not prove that in this paper),
which then implies that
\[
	Hoch_*(C_*(\Omega T)) \simeq Hoch_*(C_*(\pi^\infty_{\le 1}(T)))
			\simeq \bc_*(S^1; C_*(\pi^\infty_{\le 1}(T))) \simeq C_*(LT) .
\]
\end{rem}

\begin{proof}[Proof of Theorem \ref{thm:map-recon}]
The proof is again similar to that of Theorem \ref{thm:product}.

We begin by constructing a chain map $\psi: \cB^\cT(M) \to C_*(\Maps(M\to T))$.

Recall that 
the 0-simplices of the homotopy colimit $\cB^\cT(M)$ 
are a direct sum of chain complexes with the summands indexed by
decompositions of $M$ which have their $n{-}1$-skeletons labeled by $n{-}1$-morphisms
of $\cT$.
Since $\cT = \pi^\infty_{\leq n}(T)$, this means that the summands are indexed by pairs
$(K, \vphi)$, where $K$ is a decomposition of $M$ and $\vphi$ is a continuous
map from the $n{-}1$-skeleton of $K$ to $T$.
The summand indexed by $(K, \vphi)$ is
\[
	\bigotimes_b D_*(b, \vphi),
\]
where $b$ runs through the $n$-cells of $K$ and $D_*(b, \vphi)$ denotes
chains of maps from $b$ to $T$ compatible with $\vphi$.
We can take the product of these chains of maps to get chains of maps from
all of $M$ to $K$.
This defines $\psi$ on 0-simplices.

We define $\psi$ to be zero on $(\ge1)$-simplices.
It is not hard to see that this defines a chain map from 
$\cB^\cT(M)$ to $C_*(\Maps(M\to T))$.

The image of $\psi$ is the subcomplex $G_*\sub C_*(\Maps(M\to T))$ generated by 
families of maps whose support is contained in a disjoint union of balls.
It follows from Lemma \ref{extension_lemma_c} 
that $C_*(\Maps(M\to T))$ is homotopic to a subcomplex of $G_*$.

We will define a map $\phi:G_*\to \cB^\cT(M)$ via acyclic models.
Let $a$ be a generator of $G_*$.
Define $D(a)$ to be the subcomplex of $\cB^\cT(M)$ generated by all 
pairs $(b, \ol{K})$, where $b$ is a generator appearing in an iterated boundary of $a$
and $\ol{K}$ is an index of the homotopy colimit $\cB^\cT(M)$.
(See the proof of Theorem \ref{thm:product} for more details.)
The same proof as of Lemma \ref{lem:d-a-acyclic} shows that $D(a)$ is acyclic.
By the usual acyclic models nonsense, there is a (unique up to homotopy)
map $\phi:G_*\to \cB^\cT(M)$ such that $\phi(a)\in D(a)$.
Furthermore, we may choose $\phi$ such that for all $a$ 
\[
	\phi(a) = (a, K) + r
\]
where $(a, K)$ is a 0-simplex and $r$ is a sum of simplices of dimension 1 and greater.

It is now easy to see that $\psi\circ\phi$ is the identity on the nose.
Another acyclic models argument shows that $\phi\circ\psi$ is homotopic to the identity.
(See the proof of Theorem \ref{thm:product} for more details.)
\end{proof}


\section{Higher-dimensional Deligne conjecture}
\label{sec:deligne}
In this section we prove a higher dimensional version of the Deligne conjecture
about the action of the little disks operad on Hochschild cochains.
The first several paragraphs lead up to a precise statement of the result
(Theorem \ref{thm:deligne} below).
Then we give the proof.


The usual Deligne conjecture (proved variously in \cite{MR1805894, MR1328534, MR2064592, hep-th/9403055, MR1805923}) gives a map
\[
	C_*(LD_k)\otimes \overbrace{Hoch^*(C, C)\otimes\cdots\otimes Hoch^*(C, C)}^{\text{$k$ copies}}
			\to  Hoch^*(C, C) .
\]
Here $LD_k$ is the $k$-th space of the little disks operad and $Hoch^*(C, C)$ denotes Hochschild
cochains.

We now reinterpret $C_*(LD_k)$ and $Hoch^*(C, C)$ in such a way as to make the generalization to
higher dimensions clear.

The little disks operad is homotopy equivalent to configurations of little bigons inside a big bigon,
as shown in Figure \ref{delfig1}.
We can think of such a configuration as encoding a sequence of surgeries, starting at the bottommost interval
of Figure \ref{delfig1} and ending at the topmost interval.
\begin{figure}[t]
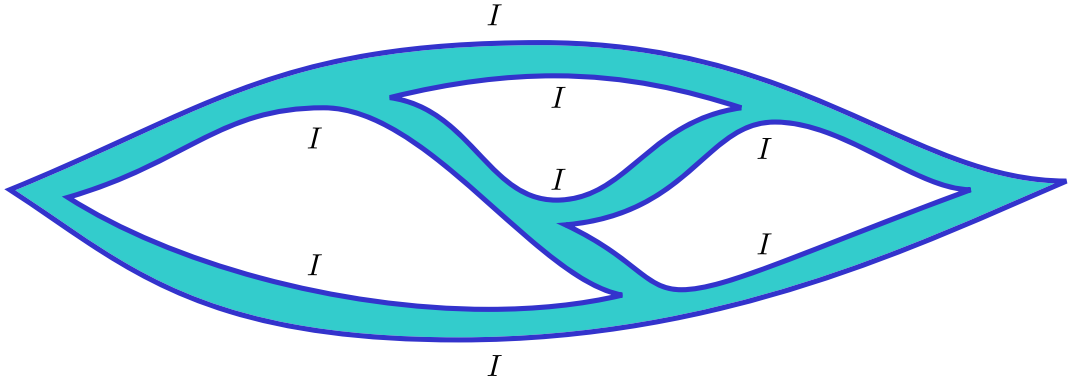

$$\mathfig{.9}{deligne/intervals}$$
\caption{Little bigons, thought of as encoding surgeries}\label{delfig1}\end{figure}
The surgeries correspond to the $k$ bigon-shaped ``holes".
We remove the bottom interval of each little bigon and replace it with the top interval.
To convert this topological operation to an algebraic one, we need, for each hole, an element of
$\hom(\bc^C_*(I_{\text{bottom}}), \bc^C_*(I_{\text{top}}))$, which is homotopy equivalent to $Hoch^*(C, C)$.
So for each fixed configuration we have a map
\[
	 \hom(\bc^C_*(I), \bc^C_*(I))\otimes\cdots
	\otimes \hom(\bc^C_*(I), \bc^C_*(I))  \to  \hom(\bc^C_*(I), \bc^C_*(I)) .
\]
If we deform the configuration, corresponding to a 1-chain in $C_*(LD_k)$, we get a homotopy
between the maps associated to the endpoints of the 1-chain.
Similarly, higher-dimensional chains in $C_*(LD_k)$ give rise to higher homotopies.

We emphasize that in $\hom(\bc^C_*(I), \bc^C_*(I))$ we are thinking of $\bc^C_*(I)$ as a module
for the $A_\infty$ 1-category associated to $\bd I$, and $\hom$ means the 
morphisms of such modules as defined in 
\S\ref{ss:module-morphisms}.

It should now be clear how to generalize this to higher dimensions.
In the sequence-of-surgeries description above, we never used the fact that the manifolds
involved were 1-dimensional.
So we will define, below, the operad of $n$-dimensional surgery cylinders, analogous to mapping
cylinders of homeomorphisms (Figure \ref{delfig2}).
\begin{figure}[t]
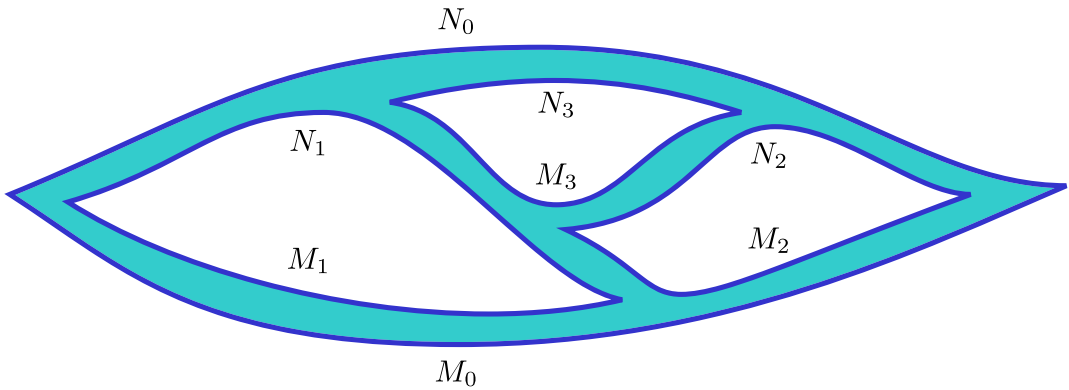

$$\mathfig{.9}{deligne/manifolds}$$
\caption{An $n$-dimensional surgery cylinder}\label{delfig2}
\end{figure}
(Note that $n$ is the dimension of the manifolds we are doing surgery on; the surgery cylinders
are $n{+}1$-dimensional.)

An $n$-dimensional surgery cylinder ($n$-SC for short) consists of:
\begin{itemize}
\item ``Lower" $n$-manifolds $M_0,\ldots,M_k$ and ``upper" $n$-manifolds $N_0,\ldots,N_k$,
with $\bd M_i = \bd N_i = E_i$ for all $i$.
We call $M_0$ and $N_0$ the outer boundary and the remaining $M_i$'s and $N_i$'s the inner
boundaries.
\item Additional manifolds $R_1,\ldots,R_{k}$, with $\bd R_i = E_0\cup \bd M_i = E_0\cup \bd N_i$.
\item Homeomorphisms 
\begin{eqnarray*}
	f_0: M_0 &\to& R_1\cup M_1 \\
	f_i: R_i\cup N_i &\to& R_{i+1}\cup M_{i+1}\;\; \mbox{for}\, 1\le i \le k-1 \\
	f_k: R_k\cup N_k &\to& N_0 .
\end{eqnarray*}
Each $f_i$ should be the identity restricted to $E_0$.
\end{itemize}
We can think of the above data as encoding the union of the mapping cylinders $C(f_0),\ldots,C(f_k)$,
with $C(f_i)$ glued to $C(f_{i+1})$ along $R_{i+1}$
(see Figure \ref{xdfig2}).
\begin{figure}[t]
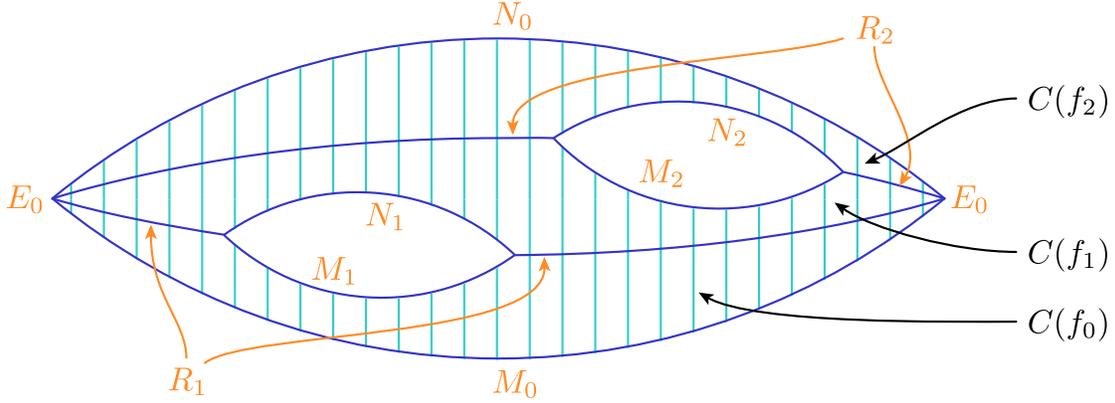

$$\mathfig{.9}{deligne/mapping-cylinders}$$
\caption{An $n$-dimensional surgery cylinder constructed from mapping cylinders}\label{xdfig2}
\end{figure}
We regard two such surgery cylinders as the same if there is a homeomorphism between them which is the 
identity on the boundary and which preserves the 1-dimensional fibers coming from the mapping
cylinders.
More specifically, we impose the following two equivalence relations:
\begin{itemize}
\item If $g: R_i\to R'_i$ is a homeomorphism which restricts to the identity on 
$\bd R_i = \bd R'_i = E_0\cup \bd M_i$, we can replace
\begin{eqnarray*}
	(\ldots, R_{i-1}, R_i, R_{i+1}, \ldots) &\to& (\ldots, R_{i-1}, R'_i, R_{i+1}, \ldots) \\
	(\ldots, f_{i-1}, f_i, \ldots) &\to& (\ldots, g\circ f_{i-1}, f_i\circ g^{-1}, \ldots),
\end{eqnarray*}
leaving the $M_i$ and $N_i$ fixed.
(Keep in mind the case $R'_i = R_i$.)
(See Figure \ref{xdfig3}.)
\begin{figure}[t]
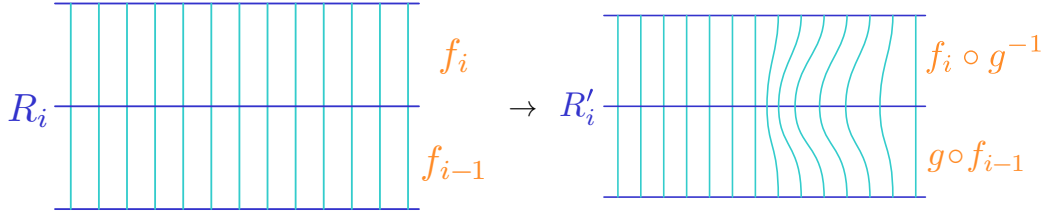

$$\mathfig{.4}{deligne/dfig3a} \to \mathfig{.4}{deligne/dfig3b} $$
\caption{Conjugating by a homeomorphism.}
\label{xdfig3}
\end{figure}
\item If $M_i = M'_i \du M''_i$ and $N_i = N'_i \du N''_i$ (and there is a
compatible disjoint union of $\bd M = \bd N$), we can replace
\begin{eqnarray*}
	(\ldots, M_{i-1}, M_i, M_{i+1}, \ldots) &\to& (\ldots, M_{i-1}, M'_i, M''_i, M_{i+1}, \ldots) \\
	(\ldots, N_{i-1}, N_i, N_{i+1}, \ldots) &\to& (\ldots, N_{i-1}, N'_i, N''_i, N_{i+1}, \ldots) \\
	(\ldots, R_{i-1}, R_i, R_{i+1}, \ldots) &\to& 
						(\ldots, R_{i-1}, R_i\cup M''_i, R_i\cup N'_i, R_{i+1}, \ldots) \\
	(\ldots, f_{i-1}, f_i, \ldots) &\to& (\ldots, f_{i-1}, {\rm{id}}, f_i, \ldots) .
\end{eqnarray*}
(See Figure \ref{xdfig1}.)
\begin{figure}[t]
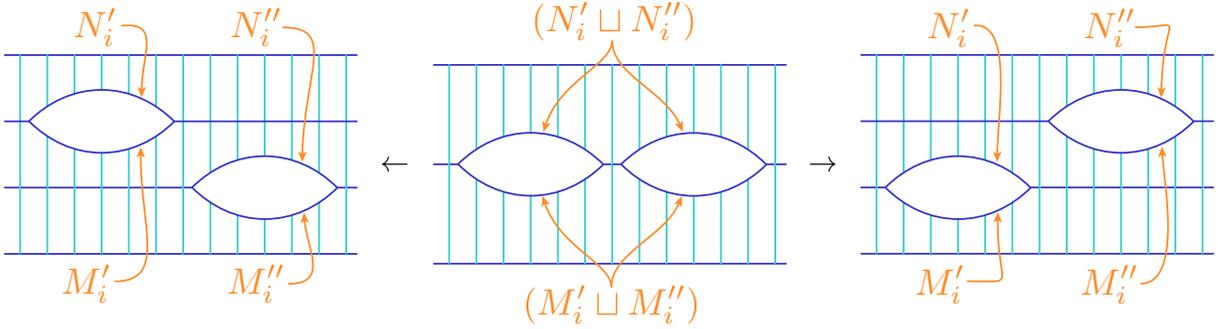

$$\mathfig{.3}{deligne/dfig1a} \leftarrow \mathfig{.3}{deligne/dfig1b} \rightarrow \mathfig{.3}{deligne/dfig1c}$$
\caption{Changing the order of a surgery.}\label{xdfig1}
\end{figure}
\end{itemize}

Note that the second equivalence increases the number of holes (or arity) by 1.
We can make a similar identification with the roles of $M'_i$ and $M''_i$ reversed.
In terms of the ``sequence of surgeries" picture, this says that if two successive surgeries
do not overlap, we can perform them in reverse order or simultaneously.

There is a colored operad structure on $n$-dimensional surgery cylinders, given by gluing the outer boundary
of one cylinder into one of the inner boundaries of another cylinder.
We leave it to the reader to work out a more precise statement in terms of $M_i$'s, $f_i$'s etc.

For fixed $\ol{M} = (M_0,\ldots,M_k)$ and $\ol{N} = (N_0,\ldots,N_k)$, we let
$SC^n_{\ol{M}\ol{N}}$ denote the topological space of all $n$-dimensional surgery cylinders as above.
(Note that in different parts of $SC^n_{\ol{M}\ol{N}}$ the $M_i$'s and $N_i$'s
are ordered differently.)
The topology comes from the spaces
\[
	\Homeo(M_0\to R_1\cup M_1)\times \Homeo(R_1\cup N_1\to R_2\cup M_2)\times
			\cdots\times \Homeo(R_k\cup N_k\to N_0)
\]
and the above equivalence relations.
We will denote the typical element of $SC^n_{\ol{M}\ol{N}}$ by $\ol{f} = (f_0,\ldots,f_k)$.

\medskip

The $n$-SC operad contains the little $n{+}1$-balls operad.
Roughly speaking, given a configuration of $k$ little $n{+}1$-balls in the standard
$n{+}1$-ball, we fiber the complement of the balls by vertical intervals
and let $M_i$ [$N_i$] be the southern [northern] hemisphere of the $i$-th ball.
More precisely, let $x_1,\ldots,x_{n+1}$ be the coordinates of $\r^{n+1}$.
Let $z$ be a point of the $k$-th space of the little $n{+}1$-balls operad, with
little balls $D_1,\ldots,D_k$ inside the standard $n{+}1$-ball.
We assume the $D_i$'s are ordered according to the $x_{n+1}$ coordinate of their centers.
Let $\pi:\r^{n+1}\to \r^n$ be the projection corresponding to $x_{n+1}$.
Let $B\sub\r^n$ be the standard $n$-ball.
Let $M_i$ and $N_i$ be $B$ for all $i$.
Identify $\pi(D_i)$ with $B$ (a.k.a.\ $M_i$ or $N_i$) via translations and dilations (no rotations).
Let $R_i = B\setmin \pi(D_i)$.
Let $f_i = \rm{id}$ for all $i$.
We have now defined a map from the little $n{+}1$-balls operad to the $n$-SC operad,
with contractible fibers.
(The fibers correspond to moving the $D_i$'s in the $x_{n+1}$ 
direction while keeping them disjoint.)

Another familiar subspace of the $n$-SC operad is $\Homeo(M_0\to N_0)$, which corresponds to 
case $k=0$ (no holes).
In this case the surgery cylinder is just a single mapping cylinder.

\medskip

Let $\ol{f} \in SC^n_{\ol{M}\ol{N}}$.
As usual, fix a system of field $\cF$ and let $\bc_*$ denote the blob complex construction based on $\cF$.
Let $\hom(\bc_*(M_i), \bc_*(N_i))$ denote the morphisms from $\bc_*(M_i)$ to $\bc_*(N_i)$,
as modules of the $A_\infty$ 1-category $\bc_*(E_i)$ (see \S\ref{ss:module-morphisms}).
We will define a map
\[
	p(\ol{f}): \hom(\bc_*(M_1), \bc_*(N_1))\ot\cdots\ot\hom(\bc_*(M_k), \bc_*(N_k))
				\to \hom(\bc_*(M_0), \bc_*(N_0)) .
\]
Given $\alpha_i\in\hom(\bc_*(M_i), \bc_*(N_i))$, we define 
$p(\ol{f})(\alpha_1\ot\cdots\ot\alpha_k)$ to be the composition
\[
	\bc_*(M_0)  \stackrel{f_0}{\to} \bc_*(R_1\cup M_1)
				 \stackrel{\id\ot\alpha_1}{\to} \bc_*(R_1\cup N_1)
				 \stackrel{f_1}{\to} \bc_*(R_2\cup M_2) \stackrel{\id\ot\alpha_2}{\to}
				 \cdots  \stackrel{\id\ot\alpha_k}{\to} \bc_*(R_k\cup N_k)
				 \stackrel{f_k}{\to} \bc_*(N_0)
\]
(Recall that the maps $\id\ot\alpha_i$ were defined in \S\ref{ss:module-morphisms}.)
It is easy to check that the above definition is compatible with the equivalence relations
and also the operad structure.
We can reinterpret the above as a chain map
\[
	p: C_0(SC^n_{\ol{M}\ol{N}})\ot \hom(\bc_*(M_1), \bc_*(N_1))\ot\cdots\ot\hom(\bc_*(M_k), \bc_*(N_k))
				\to \hom(\bc_*(M_0), \bc_*(N_0)) .
\]
The main result of this section is that this chain map extends to the full singular
chain complex $C_*(SC^n_{\ol{M}\ol{N}})$.

\begin{thm}
\label{thm:deligne}
There is a collection of chain maps
\[
	C_*(SC^n_{\ol{M}\ol{N}})\otimes \hom(\bc_*(M_1), \bc_*(N_1))\otimes\cdots\otimes 
\hom(\bc_*(M_{k}), \bc_*(N_{k})) \to  \hom(\bc_*(M_0), \bc_*(N_0))
\]
which satisfy the operad compatibility conditions, up to coherent homotopy.
On $C_0(SC^n_{\ol{M}\ol{N}})$ this agrees with the chain map $p$ defined above.
When $k=0$, this coincides with the $C_*(\Homeo(M_0\to N_0))$ action of \S\ref{sec:evaluation}.
\end{thm}

The ``up to coherent homotopy" in the statement is due to the fact that the isomorphisms of 
\ref{lem:bc-btc} and \ref{thm:gluing} are only defined up to a contractible set of homotopies.

If, in analogy to Hochschild cochains, we define elements of $\hom(\bc_*(M), \bc_*(N))$
to be ``blob cochains", we can summarize the above proposition by saying that the $n$-SC operad acts on
blob cochains.
As noted above, the $n$-SC operad contains the little $n{+}1$-balls operad, so this constitutes
a higher dimensional version of the Deligne conjecture for Hochschild cochains and the little 2-disks operad.

\begin{proof}
As described above, $SC^n_{\ol{M}\ol{N}}$ is equal to the disjoint
union of products of homeomorphism spaces, modulo some relations.
By Theorem \ref{thm:CH} and the Eilenberg-Zilber theorem, we have for each such product $P$
a chain map
\[
	C_*(P)\otimes \hom(\bc_*(M_1), \bc_*(N_1))\otimes\cdots\otimes 
\hom(\bc_*(M_{k}), \bc_*(N_{k})) \to  \hom(\bc_*(M_0), \bc_*(N_0)) .
\]
It suffices to show that the above maps are compatible with the relations whereby
$SC^n_{\ol{M}\ol{N}}$ is constructed from the various $P$'s.
This in turn follows easily from the fact that
the actions of $C_*(\Homeo(\cdot\to\cdot))$ are local (compatible with gluing) and associative
(up to coherent homotopy).
\end{proof}

We note that even when $n=1$, the above theorem goes beyond an action of the little disks operad.
$M_i$ could be a disjoint union of intervals, and $N_i$ could connect the end points of the intervals
in a different pattern from $M_i$.
The genus of the surface associated to the surgery cylinder could be greater than zero.

\appendix


\section{The method of acyclic models}  \label{sec:moam}

In this section we recall the method of acyclic models for the reader's convenience. The material presented here is closely modeled on  \cite[Chapter 4]{MR0210112}.
We use this method throughout the paper (c.f. Theorem \ref{thm:product}, Theorem \ref{thm:gluing} and Theorem \ref{thm:map-recon}), as it provides a very convenient way to show the existence of a chain map with desired properties, even when many non-canonical choices are required in order to construct one, and further to show the up-to-homotopy uniqueness of such maps.

Let $F_*$ and $G_*$ be chain complexes.
Assume $F_k$ has a basis $\{x_{kj}\}$
(that is, $F_*$ is free and we have specified a basis).
(In our applications, $\{x_{kj}\}$ will typically be singular $k$-simplices or 
$k$-blob diagrams.)
For each basis element $x_{kj}$ assume we have specified a ``target" $D^{kj}_*\sub G_*$.

We say that a chain map $f:F_*\to G_*$ is {\it compatible} with the above data (basis and targets)
if $f(x_{kj})\in D^{kj}_*$ for all $k$ and $j$.
Let $\Compat(D^\bullet_*)$ denote the subcomplex of maps from $F_*$ to $G_*$
such that the image of each higher homotopy applied to $x_{kj}$ lies in $D^{kj}_*$.

\begin{thm}[Acyclic models]  \label{moam-thm}
Suppose 
\begin{itemize}
\item $D^{k-1,l}_* \sub D^{kj}_*$ whenever $x_{k-1,l}$ occurs in $\bd x_{kj}$
with non-zero coefficient;
\item $D^{0j}_0$ is non-empty for all $j$; and
\item $D^{kj}_*$ is $(k{-}1)$-acyclic (i.e.\ $H_{k-1}(D^{kj}_*) = 0$) for all $k,j$ .
\end{itemize}
Then $\Compat(D^\bullet_*)$ is non-empty.
If, in addition,
\begin{itemize}
\item $D^{kj}_*$ is $m$-acyclic for $k\le m \le k+i$ and for all $k,j$,
\end{itemize}
then $\Compat(D^\bullet_*)$ is $i$-connected.
\end{thm}

\begin{proof}
(Sketch)
This is a standard result; see, for example, \cite[Chapter 4]{MR0210112}.

We will build a chain map $f\in \Compat(D^\bullet_*)$ inductively.
Choose $f(x_{0j})\in D^{0j}_0$ for all $j$
(possible since $D^{0j}_0$ is non-empty).
Choose $f(x_{1j})\in D^{1j}_1$ such that $\bd f(x_{1j}) = f(\bd x_{1j})$
(possible since $D^{0l}_* \sub D^{1j}_*$ for each $x_{0l}$ in $\bd x_{1j}$
and $D^{1j}_*$ is 0-acyclic).
Continue in this way, choosing $f(x_{kj})\in D^{kj}_k$ such that $\bd f(x_{kj}) = f(\bd x_{kj})$
We have now constructed $f\in \Compat(D^\bullet_*)$, proving the first claim of the theorem.

Now suppose that $D^{kj}_*$ is $k$-acyclic for all $k$ and $j$.
Let $f$ and $f'$ be two chain maps (0-chains) in $\Compat(D^\bullet_*)$.
Using a technique similar to above we can construct a homotopy (1-chain) in $\Compat(D^\bullet_*)$
between $f$ and $f'$.
Thus $\Compat(D^\bullet_*)$ is 0-connected.
Similarly, if $D^{kj}_*$ is $(k{+}i)$-acyclic then we can show that $\Compat(D^\bullet_*)$ is $i$-connected.
\end{proof}


\section{Adapting families of maps to open covers}  \label{sec:localising}

In this appendix we prove some results about adapting families of maps to open covers.
These results are used in Lemma \ref{small-top-blobs} and Theorem \ref{thm:map-recon}.

Let $X$ and $T$ be topological spaces, with $X$ compact.
Let $\cU = \{U_\alpha\}$ be an open cover of $X$ which affords a partition of
unity $\{r_\alpha\}$.
(That is, $r_\alpha : X \to [0,1]$; $r_\alpha(x) = 0$ if $x\notin U_\alpha$;
for fixed $x$, $r_\alpha(x) \ne 0$ for only finitely many $\alpha$; and $\sum_\alpha r_\alpha = 1$.)
Since $X$ is compact, we will further assume that $r_\alpha = 0$ (globally) 
for all but finitely many $\alpha$.

Consider  $C_*(\Maps(X\to T))$, the singular chains on the space of continuous maps from $X$ to $T$.
$C_k(\Maps(X \to T))$ is generated by continuous maps
\[
	f: P\times X \to T ,
\]
where $P$ is some convex linear polyhedron in $\r^k$.
Recall that $f$ is {\it supported} on $S\sub X$ if $f(p, x)$ does not depend on $p$ when
$x \notin S$, and that $f$ is {\it adapted} to $\cU$ if 
$f$ is supported on the union of at most $k$ of the $U_\alpha$'s.
A chain $c \in C_*(\Maps(X \to T))$ is adapted to $\cU$ if it is a linear combination of 
generators which are adapted.

\begin{lemma} \label{basic_adaptation_lemma}
Let $f: P\times X \to T$, as above.
Then there exists
\[
	F: I \times P\times X \to T
\]
such that the following conditions hold.
\begin{enumerate}
\item $F(0, \cdot, \cdot) = f$.
\item We can decompose $P = \cup_i D_i$ so that
the restrictions $F(1, \cdot, \cdot) : D_i\times X\to T$ are all adapted to $\cU$.
\item If $f$ has support $S\sub X$, then
$F: (I\times P)\times X\to T$ (a $k{+}1$-parameter family of maps) also has support $S$.
Furthermore, if $Q$ is a convex linear subpolyhedron of $\bd P$ and $f$ restricted to $Q$
has support $S' \subset X$, then
$F: (I\times Q)\times X\to T$ also has support $S'$.
\item Suppose both $X$ and $T$ are smooth manifolds, metric spaces, or PL manifolds, and 
let $\cX$ denote the subspace of $\Maps(X \to T)$ consisting of immersions or of diffeomorphisms (in the smooth case), 
bi-Lipschitz homeomorphisms (in the metric case), or PL homeomorphisms (in the PL case).
 If $f$ is smooth, Lipschitz or PL, as appropriate, and $f(p, \cdot):X\to T$ is in $\cX$ for all $p \in P$
then $F(t, p, \cdot)$ is also in $\cX$ for all $t\in I$ and $p\in P$.
\end{enumerate}
\end{lemma}


\begin{proof}
Our homotopy will have the form
\eqar{
    F: I \times P \times X &\to& X \\
    (t, p, x) &\mapsto& f(u(t, p, x), x)
}
for some function
\eq{
    u : I \times P \times X \to P .
}

First we describe $u$, then we argue that it makes the conclusions of the lemma true.

For each cover index $\alpha$ choose a cell decomposition $K_\alpha$ of $P$
such that the various $K_\alpha$ are in general position with respect to each other.
If we are in one of the cases of item 4 of the lemma, also choose $K_\alpha$
sufficiently fine as described below.

\def\jj{\tilde{L}}
Let $L$ be a common refinement of all the $K_\alpha$'s.
Let $\jj$ denote the handle decomposition of $P$ corresponding to $L$.
Each $i$-handle $C$ of $\jj$ has an $i$-dimensional tangential coordinate and,
more importantly for our purposes, a $k{-}i$-dimensional normal coordinate.
We will typically use the same notation for $i$-cells of $L$ and the 
corresponding $i$-handles of $\jj$.

For each (top-dimensional) $k$-cell $C$ of each $K_\alpha$, choose a point $p(C) \in C \sub P$.
If $C$ meets a subpolyhedron $Q$ of $\bd P$, we require that $p(C)\in Q$.
(It follows that if $C$ meets both $Q$ and $Q'$, then $p(C)\in Q\cap Q'$.
Ensuring this is possible corresponds to some mild constraints on the choice of the $K_\alpha$.)

Let $D$ be a $k$-handle of $\jj$.
For each $\alpha$ let $C(D, \alpha)$ be the $k$-cell of $K_\alpha$ which contains $D$
and let $p(D, \alpha) = p(C(D, \alpha))$.

For $p \in D$ we define
\eq{
    u(t, p, x) = (1-t)p + t \sum_\alpha r_\alpha(x) p(D, \alpha) .
}
(Recall that $P$ is a convex linear polyhedron, so the weighted average of points of $P$
makes sense.)

Thus far we have defined $u(t, p, x)$ when $p$ lies in a $k$-handle of $\jj$.
We will now extend $u$ inductively to handles of index less than $k$.

Let $E$ be a $k{-}1$-handle.
$E$ is homeomorphic to $B^{k-1}\times [0,1]$, and meets
the $k$-handles at $B^{k-1}\times\{0\}$ and $B^{k-1}\times\{1\}$.
Let $\eta : E \to [0,1]$, $\eta(x, s) = s$ be the normal coordinate
of $E$.
Let $D_0$ and $D_1$ be the two $k$-handles of $\jj$ adjacent to $E$.
There is at most one index $\beta$ such that $C(D_0, \beta) \ne C(D_1, \beta)$.
(If there is no such index, choose $\beta$
arbitrarily.)
For $p \in E$, define
\eq{
    u(t, p, x) = (1-t)p + t \left( \sum_{\alpha \ne \beta} r_\alpha(x) p(D_0, \alpha)
            + r_\beta(x) (\eta(p) p(D_0, \beta) + (1-\eta(p)) p(D_1, \beta)) \right) .
}

Now for the general case.
Let $E$ be a $k{-}j$-handle.
Let $D_0,\ldots,D_a$ be the $k$-handles adjacent to $E$.
There is a subset of cover indices $\cN$, of cardinality $j$, 
such that if $\alpha\notin\cN$ then
$p(D_u, \alpha) = p(D_v, \alpha)$ for all $0\le u,v \le a$.
For fixed $\beta\in\cN$ let $\{q_{\beta i}\}$ be the set of values of 
$p(D_u, \beta)$ for $0\le u \le a$.
Recall the product structure $E = B^{k-j}\times B^j$.
Inductively, we have defined functions $\eta_{\beta i}:\bd B^j \to [0,1]$ such that
$\sum_i \eta_{\beta i} = 1$ for all $\beta\in \cN$.
Choose extensions of $\eta_{\beta i}$ to all of $B^j$.
Via the projection $E\to B^j$, regard $\eta_{\beta i}$ as a function on $E$.
Now define, for $p \in E$,
\begin{equation}
\label{eq:u}
    u(t, p, x) = (1-t)p + t \left(
            \sum_{\alpha \notin \cN} r_\alpha(x) p(D_0, \alpha)
                + \sum_{\beta \in \cN} r_\beta(x) \left( \sum_i \eta_{\beta i}(p) \cdot q_{\beta i} \right)
             \right) .
\end{equation}

This completes the definition of $u: I \times P \times X \to P$. 
The formulas above are consistent: for $p$ at the boundary between a $k-j$-handle and 
a $k-(j+1)$-handle the corresponding expressions in Equation \eqref{eq:u} agree, 
since one of the normal coordinates becomes $0$ or $1$. 
Note that if $Q\sub \bd P$ is a convex linear subpolyhedron, then $u(I\times Q\times X) \sub Q$.

\medskip

Next we verify that $u$ affords $F$ the properties claimed in the statement of the lemma.

Since $u(0, p, x) = p$ for all $p\in P$ and $x\in X$, $F(0, p, x) = f(p, x)$ for all $p$ and $x$.
Therefore $F$ is a homotopy from $f$ to something.

\medskip

Next we show that for each handle $D$ of $J$, $F(1, \cdot, \cdot) : D\times X \to X$
is a singular cell adapted to $\cU$.
Let $k-j$ be the index of $D$.
Referring to Equation \eqref{eq:u}, we see that $F(1, p, x)$ depends on $p$ only if 
$r_\beta(x) \ne 0$ for some $\beta\in\cN$, i.e.\ only if
$x\in \bigcup_{\beta\in\cN} U_\beta$.
Since the cardinality of $\cN$ is $j$ which is less than or equal to $k$,
this shows that $F(1, \cdot, \cdot) : D\times X \to X$ is adapted to $\cU$.

\medskip

Next we show that $F$ does not increase supports.
If $f(p,x) = f(p',x)$ for all $p,p'\in P$,
then 
\[
	F(t, p, x) = f(u(t,p,x),x) = f(u(t',p',x),x) = F(t',p',x)
\]
for all $(t,p)$ and $(t',p')$ in $I\times P$.
Similarly, if $f(q,x) = f(q',x)$ for all $q,q'\in Q\sub \bd P$,
then 
\[
	F(t, q, x) = f(u(t,q,x),x) = f(u(t',q',x),x) = F(t',q',x)
\]
for all $(t,q)$ and $(t',q')$ in $I\times Q$.
(Recall that we arranged above that $u(I\times Q\times X) \sub Q$.)

\medskip

Now for claim 4 of the lemma.
Assume that $X$ and $T$ are smooth manifolds and that $f$ is a smooth family of diffeomorphisms.
We must show that we can choose the $K_\alpha$'s and $u$ so that $F(t, p, \cdot)$ is a 
diffeomorphism for all $t$ and $p$.
It suffices to 
show that the derivative $\pd{F}{x}(t, p, x)$ is non-singular for all $(t, p, x)$.
We have
\eq{
    \pd{F}{x} = \pd{f}{x} + \pd{f}{p} \pd{u}{x} .
}
Since $f$ is a family of diffeomorphisms and $X$ and $P$ are compact, 
$\pd{f}{x}$ is non-singular and bounded away from zero.
Also, since $f$ is smooth $\pd{f}{p}$ is bounded.
Thus if we can insure that $\pd{u}{x}$ is sufficiently small, we are done.
It follows from Equation \eqref{eq:u} above that $\pd{u}{x}$ depends on $\pd{r_\alpha}{x}$
(which is bounded)
and the differences amongst the various $p(D_0,\alpha)$'s and $q_{\beta i}$'s.
These differences are small if the cell decompositions $K_\alpha$ are sufficiently fine.
This completes the proof that $F$ is a homotopy through diffeomorphisms.

If we replace ``diffeomorphism" with ``immersion" in the above paragraph, the argument goes
through essentially unchanged.

Next we consider the case where $f$ is a family of bi-Lipschitz homeomorphisms.
Recall that we assume that $f$ is Lipschitz in the $P$ direction as well.
The argument in this case is similar to the one above for diffeomorphisms, with
bounded partial derivatives replaced by Lipschitz constants.
Since $X$ and $P$ are compact, there is a universal bi-Lipschitz constant that works for 
$f(p, \cdot)$ for all $p$.
By choosing the cell decompositions $K_\alpha$ sufficiently fine,
we can insure that $u$ has a small Lipschitz constant in the $X$ direction.
This allows us to show that $F(t, p, \cdot)$ has a bi-Lipschitz constant
close to the universal bi-Lipschitz constant for $f$.

Since PL homeomorphisms are bi-Lipschitz, we have established this last remaining case of claim 4 of the lemma as well.
\end{proof}


\begin{lemma} \label{extension_lemma_c}
Let $\cX_*$ be any of $C_*(\Maps(X \to T))$ or singular chains on the 
subspace of $\Maps(X\to T)$ consisting of immersions, diffeomorphisms, 
bi-Lipschitz homeomorphisms, or PL homeomorphisms.
Let $G_* \subset \cX_*$ denote the chains adapted to an open cover $\cU$
of $X$.
Then $G_*$ is a strong deformation retract of $\cX_*$.
\end{lemma}
\begin{proof}
It suffices to show that given a generator $f:P\times X\to T$ of $\cX_k$ with
$\bd f \in G_{k-1}$ there exists $h\in \cX_{k+1}$ with $\bd h = f + g$ and $g \in G_k$.
This is exactly what Lemma \ref{basic_adaptation_lemma}
gives us.
More specifically, let $\bd P = \sum Q_i$, with each $Q_i\in G_{k-1}$.
Let $F: I\times P\times X\to T$ be the homotopy constructed in Lemma \ref{basic_adaptation_lemma}.
Then $\bd F$ is equal to $f$ plus $F(1, \cdot, \cdot)$ plus the restrictions of $F$ to $I\times Q_i$.
Part 2 of Lemma \ref{basic_adaptation_lemma} says that $F(1, \cdot, \cdot)\in G_k$,
while part 3 of Lemma \ref{basic_adaptation_lemma} says that the restrictions to $I\times Q_i$ are in $G_k$.
\end{proof}

\medskip

Topological (merely continuous) homeomorphisms are conspicuously absent from the 
list of classes of maps for which the above lemma hold.
The $k=1$ case of Lemma \ref{basic_adaptation_lemma} for plain, continuous homeomorphisms 
is more or less equivalent to Corollary 1.3 of \cite{MR0283802}.
We suspect that the proof found in \cite{MR0283802} of that corollary can be adapted to many-parameter families of
homeomorphisms, but so far the details have alluded us.

\noop{

\medskip

\nn{do we want to keep the following?}

\nn{ack! not easy to adapt (pun) this old text to continuous maps (instead of homeos, as
in the old version); just delete (\\noop) it all for now}

The above lemmas remain true if we replace ``adapted" with ``strongly adapted", as defined below.
The proof of Lemma \ref{basic_adaptation_lemma} is modified by
choosing the common refinement $L$ and interpolating maps $\eta$
slightly more carefully.
Since we don't need the stronger result, we omit the details.

Let $X$, $T$ and $\cU$ be as above.
A $k$-parameter family of maps $f: P \times X \to T$ is
{\it strongly adapted to $\cU$} if there is a factorization
\eq{
    P = P_1 \times \cdots \times P_m
}
(for some $m \le k$)
and families of homeomorphisms
\eq{
    f_i :  P_i \times X \to T
}
such that
\begin{itemize}
\item each $f_i$ is supported on some connected $V_i \sub X$;
\item the sets $V_i$ are mutually disjoint;
\item each $V_i$ is the union of at most $k_i$ of the $U_\alpha$'s,
where $k_i = \dim(P_i)$; and
\item $f(p, \cdot) = g \circ f_1(p_1, \cdot) \circ \cdots \circ f_m(p_m, \cdot)$
for all $p = (p_1, \ldots, p_m)$, for some fixed $g:X\to T$.
\end{itemize}

}


\section{Comparing \texorpdfstring{$n$}{n}-category definitions}
\label{sec:comparing-defs}

In \S\ref{sec:example:traditional-n-categories(fields)} we showed how to construct
a disk-like  $n$-category from a traditional $n$-category; the morphisms of the 
disk-like  $n$-category are string diagrams labeled by the traditional $n$-category.
In this appendix we sketch how to go the other direction, for $n=1$ and 2.
The basic recipe, given a disk-like $n$-category $\cC$, is to define the $k$-morphisms
of the corresponding traditional $n$-category to be $\cC(B^k)$, where
$B^k$ is the {\it standard} $k$-ball.
One must then show that the axioms of \S\ref{ss:n-cat-def} imply the traditional $n$-category axioms.
One should also show that composing the two arrows (between traditional and disk-like $n$-categories)
yields the appropriate sort of equivalence on each side.
Since we haven't given a definition for functors between disk-like $n$-categories, we do not pursue this here.

We emphasize that we are just sketching some of the main ideas in this appendix ---
it falls well short of proving the definitions are equivalent.


\subsection{1-categories over \texorpdfstring{$\Set$ or $\Vect$}{Set or Vect}}
\label{ssec:1-cats}
Given a disk-like $1$-category $\cX$ we construct a $1$-category in the conventional sense, $c(\cX)$.
This construction is quite straightforward, but we include the details for the sake of completeness, 
because it illustrates the role of structures (e.g. orientations, spin structures, etc) 
on the underlying manifolds, and 
to shed some light on the $n=2$ case, which we describe in \S \ref{ssec:2-cats}.

Let $B^k$ denote the \emph{standard} $k$-ball.
Let the objects of $c(\cX)$ be $c(\cX)^0 = \cX(B^0)$ and the morphisms of $c(\cX)$ be $c(\cX)^1 = \cX(B^1)$.
The boundary and restriction maps of $\cX$ give domain and range maps from $c(\cX)^1$ to $c(\cX)^0$.

Choose a homeomorphism $B^1\cup_{pt}B^1 \to B^1$.
Define composition in $c(\cX)$ to be the induced map $c(\cX)^1\times c(\cX)^1 \to c(\cX)^1$ 
(defined only when range and domain agree).
By isotopy invariance in $\cX$, any other choice of homeomorphism gives the same composition rule.
Also by isotopy invariance, composition is strictly associative.

Given $a\in c(\cX)^0$, define $\id_a \deq a\times B^1$.
By extended isotopy invariance in $\cX$, this has the expected properties of an identity morphism.

We have now defined the basic ingredients for the 1-category $c(\cX)$.
As we explain below, $c(\cX)$ might have additional structure corresponding to the
unoriented, oriented, Spin, $\text{Pin}_+$ or $\text{Pin}_-$ structure on the 1-balls used to define $\cX$.

For 1-categories based on unoriented balls, 
there is a map $\dagger:c(\cX)^1\to c(\cX)^1$
coming from $\cX$ applied to an orientation-reversing homeomorphism (unique up to isotopy) 
from $B^1$ to itself.
(Of course our $B^1$ is unoriented, i.e.\ not equipped with an orientation.
We mean the homeomorphism which would reverse the orientation if there were one;
$B^1$ is not oriented, but it is orientable.)
Topological properties of this homeomorphism imply that 
$a^{\dagger\dagger} = a$ ($\dagger$ is order 2), $\dagger$ reverses domain and range, and $(ab)^\dagger = b^\dagger a^\dagger$
($\dagger$ is an anti-automorphism).
Recall that in this context 0-balls should be thought of as equipped with a germ of a 1-dimensional neighborhood.
There is a unique such 0-ball, up to homeomorphism, but it has a non-identity automorphism corresponding to reversing the
orientation of the germ.
Consequently, the objects of $c(\cX)$ are equipped with an involution, also denoted $\dagger$.
If $a:x\to y$ is a morphism of $c(\cX)$ then $a^\dagger: y^\dagger\to x^\dagger$.

For 1-categories based on oriented balls, there are no non-trivial homeomorphisms of 0- or 1-balls, and thus no 
additional structure on $c(\cX)$.

For 1-categories based on Spin balls,
the nontrivial spin homeomorphism from $B^1$ to itself which covers the identity
gives an order 2 automorphism of $c(\cX)^1$.
There is a similar involution on the objects $c(\cX)^0$.
In the case where there is only one object and we are enriching over complex vector spaces, this
is just a super algebra.
The even elements are the $+1$ eigenspace of the involution on $c(\cX)^1$, 
and the odd elements are the $-1$ eigenspace of the involution.

For 1-categories based on $\text{Pin}_-$ balls,
we have an order 4 antiautomorphism of $c(\cX)^1$.
For 1-categories based on $\text{Pin}_+$ balls,
we have an order 2 antiautomorphism and also an order 2 automorphism of $c(\cX)^1$,
and these two maps commute with each other.
In both cases there is a similar map on objects.

\noop{
\medskip

In the other direction, given a $1$-category $C$
(with objects $C^0$ and morphisms $C^1$) we will construct a disk-like
$1$-category $t(C)$.

If $X$ is a 0-ball (point), let $t(C)(X) \deq C^0$.
If $S$ is a 0-sphere, let $t(C)(S) \deq C^0\times C^0$.
If $X$ is a 1-ball, let $t(C)(X) \deq C^1$.
Homeomorphisms isotopic to the identity act trivially.
If $C$ has extra structure (e.g.\ it's a *-1-category), we use this structure
to define the action of homeomorphisms not isotopic to the identity
(and get, e.g., an unoriented disk-like 1-category).

The domain and range maps of $C$ determine the boundary and restriction maps of $t(C)$.

Gluing maps for $t(C)$ are determined by composition of morphisms in $C$.

For $X$ a 0-ball, $D$ a 1-ball and $a\in t(C)(X)$, define the product morphism 
$a\times D \deq \id_a$.
It is not hard to verify that this has the desired properties.

\medskip

The compositions of the constructions above, $$\cX\to c(\cX)\to t(c(\cX))$$ 
and $$C\to t(C)\to c(t(C)),$$ give back 
more or less exactly the same thing we started with.  

As we will see below, for $n>1$ the compositions yield a weaker sort of equivalence.
} 

\medskip

Similar arguments show that modules for disk-like 1-categories are essentially
the same thing as traditional modules for traditional 1-categories.

\subsection{Pivotal 2-categories}
\label{ssec:2-cats}
Let $\cC$ be a disk-like 2-category.
We will construct from $\cC$ a traditional pivotal 2-category $D$.
(The ``pivotal" corresponds to our assumption of strong duality for $\cC$.)

We will try to describe the construction in such a way that the generalization to $n>2$ is clear,
though this will make the $n=2$ case a little more complicated than necessary.

Before proceeding, we must decide whether the 2-morphisms of our
pivotal 2-category are shaped like rectangles or bigons.
Each approach has advantages and disadvantages.
For better or worse, we choose bigons here.

Define the $k$-morphisms $D^k$ of $D$ to be $\cC(B^k) \trans E$, where $B^k$ denotes the standard
$k$-ball, which we also think of as the standard bihedron (a.k.a.\ globe).
(For $k=1$ this is an interval, and for $k=2$ it is a bigon.)
Since we are thinking of $B^k$ as a bihedron, we have a standard decomposition of the $\bd B^k$
into two copies of $B^{k-1}$ which intersect along the ``equator" $E \cong S^{k-2}$.
Recall that the subscript in $\cC(B^k) \trans E$ means that we consider the subset of $\cC(B^k)$
whose boundary is splittable along $E$.
This allows us to define the domain and range of morphisms of $D$ using
boundary and restriction maps of $\cC$.

Choosing a homeomorphism $B^1\cup B^1 \to B^1$ defines a composition map on $D^1$.
This is not associative, but we will see later that it is weakly associative.

Choosing a homeomorphism $B^2\cup B^2 \to B^2$ defines a ``vertical" composition map 
on $D^2$ (Figure \ref{fzo1}).
Isotopy invariance implies that this is associative.
We will define a ``horizontal" composition later.

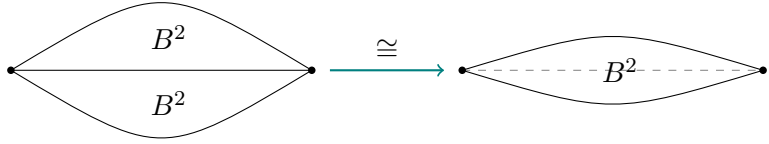
\begin{figure}[t]
\centering
\begin{tikzpicture}

\newcommand{\vertex}{node[circle,fill=black,inner sep=1pt] {}}
\newcommand{\nsep}{1.8}

\node[outer sep=\nsep](A) at (0,0) {
\begin{tikzpicture}
	\draw (0,0) coordinate (p1);
	\draw (4,0) coordinate (p2);
	\draw (2,1.2) coordinate (pu);
	\draw (2,-1.2) coordinate (pd);

	\draw (p1) .. controls (pu) .. (p2) .. controls (pd) .. (p1);
	\draw (p1)--(p2);

	\draw (p1) \vertex;
	\draw (p2) \vertex;
	
	\node at (2.1, .44) {$B^2$};
	\node at (2.1, -.44) {$B^2$};
	
\end{tikzpicture}
};

\node[outer sep=\nsep](B) at (6,0) {
\begin{tikzpicture}
	\draw (0,0) coordinate (p1);
	\draw (4,0) coordinate (p2);
	\draw (2,.6) coordinate (pu);
	\draw (2,-.6) coordinate (pd);

	\draw (p1) .. controls (pu) .. (p2) .. controls (pd) .. (p1);
	\draw[help lines, dashed] (p1)--(p2);

	\draw (p1) \vertex;
	\draw (p2) \vertex;
	
	\node at (2.1,0) {$B^2$};
	
\end{tikzpicture}
};

\draw[->, thick, blue!50!green] (A) -- node[black, above] {$\cong$} (B);

\end{tikzpicture}
\caption{Vertical composition of 2-morphisms}
\label{fzo1}
\end{figure}

Given $a\in D^1$, define $\id_a = a\times I \in D^2$ (pinched boundary).
Extended isotopy invariance for $\cC$ shows that this morphism is an identity for 
vertical composition.

Given $x\in C^0$, define $\id_x = x\times B^1 \in C^1$.
We will show that this 1-morphism is a weak identity.
This would be easier if our 2-morphisms were shaped like rectangles rather than bigons.

In showing that identity 1-morphisms have the desired properties, we will
rely heavily on the extended isotopy invariance of 2-morphisms in $\cC$.
Extended isotopy invariance implies that adding a product collar to a 2-morphism of $\cC$ has no effect,
and by cutting and regluing we can insert (or delete) product regions in the interior of 2-morphisms as well.
Figure \ref{fig:product-regions} shows some examples.
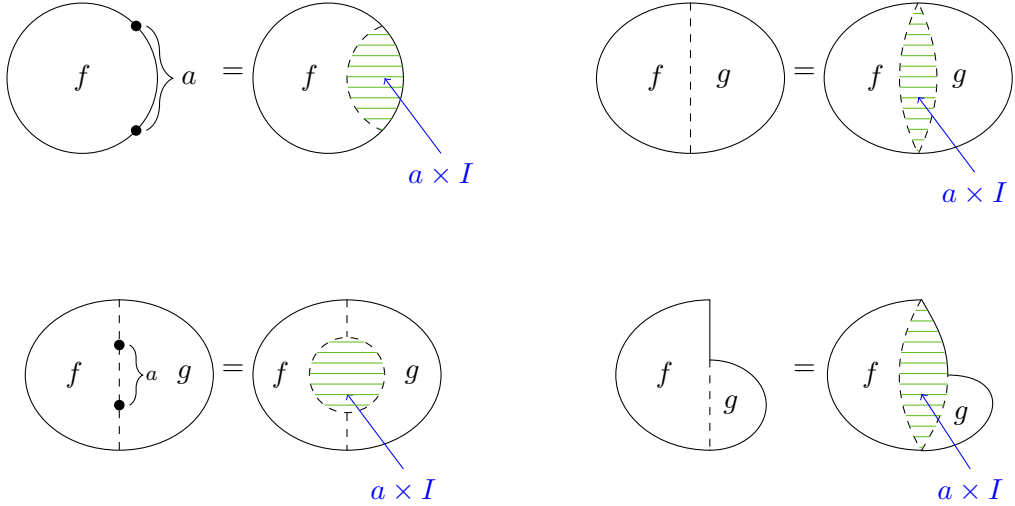
\begin{figure}[t]
\begin{align*}
\begin{tikzpicture}[baseline]
\node[draw] (c) at (0,0) [circle through = {(1,0)}] {$f$};
\node (d) at (c.east) [circle through = {(0.25,0)}] {};
\foreach \n in {1,2} {
	\node (p\n) at (intersection \n of c and d) {};
	\fill (p\n) circle (2pt);
}
\begin{scope}[decoration={brace,amplitude=10,aspect=0.5}]
	\draw[decorate] (p2.east) -- node[right=2ex] {$a$} (p1.east);
\end{scope}
\end{tikzpicture} & = 
\begin{tikzpicture}[baseline]
\node[draw] (c) at (0,0) [circle through = {(1,0)}] {};
\begin{scope}
\path[clip] (c) circle (1);
\node[draw,dashed] (d) at (c.east) [circle through = {(0.25,0)}] {};
\foreach \n in {1,2} {
	\node (p\n) at (intersection \n of c and d) {};
}
\node[left] at (c) {$f$};
\path[clip] (d) circle (0.75);
\foreach \y in {1,0.86,...,-1} {
	\draw[green!50!brown] (0,\y)--(1,\y);
}
\end{scope}
\draw[->,blue] (1.5,-1) node[below] {$a \times I$} -- (0.75,0);
\end{tikzpicture} &
\begin{tikzpicture}[baseline]
\node[draw] (c) at (0,0) [ellipse, minimum height=2cm,minimum width=2.5cm] {};
\draw[dashed] (c.north) -- (c.south);
\node[right=6] at (c) {$g$};
\node[left=6] at (c) {$f$};
\end{tikzpicture} & =
\begin{tikzpicture}[baseline]
\node[draw] (c) at (0,0) [ellipse, minimum height=2cm,minimum width=2.5cm] {};
\node[right=9] at (c) {$g$};
\node[left=9] at (c) {$f$};
\draw[dashed] (c.north) to[out=-115,in=115] (c.south) to[out=65,in=-65] (c.north);
\begin{scope}
\path[clip] (c.north) to[out=-115,in=115] (c.south) to[out=65,in=-65] (c.north);
\foreach \y in {1,0.86,...,-1} {
	\draw[green!50!brown] (-1,\y)--(1,\y);
}
\end{scope}
\draw[->,blue] (.75,-1.25) node[below] {$a \times I$} -- (0,-0.25);
\end{tikzpicture} \\
\begin{tikzpicture}[baseline]
\node[draw] (c) at (0,0) [ellipse, minimum height=2cm,minimum width=2.5cm] {};
\draw[dashed] (c.north) -- (c.south);
\node[right=18] at (c) {$g$};
\node[left=10] at (c) {$f$};
\fill (0,0.4) node (p1) {} circle (2pt);
\fill (0,-0.4) node (p2) {} circle (2pt);
\begin{scope}[decoration={brace,amplitude=5,aspect=0.5}]
	\draw[decorate] (p1.east) -- node[right=0.5ex] {\scriptsize $a$} (p2.east);
\end{scope}
\end{tikzpicture} & =
\begin{tikzpicture}[baseline]
\node[draw] (c) at (0,0) [ellipse, minimum height=2cm,minimum width=2.5cm] {};
\node[draw,dashed] (d) at (0,0) [circle, minimum height=1cm,minimum width=1cm] {};
\draw[dashed] (c.north) -- (d.north) (d.south) -- (c.south);
\node[right=18] at (c) {$g$};
\node[left=18] at (c) {$f$};
\draw[->,blue] (.75,-1.25) node[below] {$a \times I$} -- (0,-0.25);
\clip (0,0) circle (0.5cm);
\foreach \y in {1,0.86,...,-1} {
	\draw[green!50!brown] (-1,\y)--(1,\y);
}
\end{tikzpicture} &
\begin{tikzpicture}[baseline]
\begin{scope}
\clip (-1.3,-2) rectangle (0,2);
\node[draw] (c) at (0,0) [ellipse, minimum height=2cm,minimum width=2.5cm] {};
\end{scope}
\begin{scope}
\clip (1,-2) rectangle (0,2);
\node[draw] (d) at (0,-0.4) [ellipse, minimum height=1.2cm, minimum width=1.5cm] {};
\end{scope}
\draw (c.north) -- (d.north);
\draw[dashed] (d.north) -- (d.south);
\node[right=8,below=4] at (c) {$g$};
\node[left=10] at (c) {$f$};
\end{tikzpicture} & =
\begin{tikzpicture}[baseline]
\begin{scope}
	\clip (-1.3,-2) rectangle (0,2);
	\node[draw] (c) at (0,0) [ellipse, minimum height=2cm,minimum width=2.5cm] {};
\end{scope}
\draw[dashed] (c.north) to[out=-120,in=120] (c.south);
\draw[dashed] (c.south) to[out=60,in=-90] (0.35,0) node (m) {};
\draw (m.center) to[out=90,in=-60] (c.north);
\begin{scope}
	\clip (c.north) to[out=-120,in=120] (c.south) to[out=60,in=-90] (m) to[out=90,in=-60] (c.north);
	\foreach \y in {1,0.86,...,-1} {
		\draw[green!50!brown] (-1,\y)--(1,\y);
	}
\end{scope}
\draw (m.center) .. controls +(1,0) and +(1,0) .. (c.south);
\node[right=15,below=8] at (c) {$g$};
\node[left=12] at (c) {$f$};
\draw[->,blue] (.65,-1.25) node[below] {$a \times I$} -- (0,-0.25);
\end{tikzpicture}
\end{align*}
\caption{Examples of inserting or deleting product regions.}
\label{fig:product-regions}
\end{figure}

Let $a: y\to x$ be a 1-morphism.
Define 2-morphsims $a \to a\bullet \id_x$ and $a\bullet \id_x \to a$
as shown in Figure \ref{fzo2}.
\begin{figure}[t]
\centering
\begin{tikzpicture}
\newcommand{\rr}{6}
\newcommand{\vertex}{node[circle,fill=black,inner sep=1pt] {}}
\newcommand{\namedvertex}[1]{node[circle,fill=black,inner sep=1pt] (#1) {}}

\node(A) at (0,0) {
\begin{tikzpicture}
\node[red,left] at (0,0)  {$y$};
\draw (0,0) \vertex arc (-120:-105:\rr) node[red,below] {$a$} arc(-105:-90:\rr) \vertex node[red,below](x2) {$x$};
\draw (0,0) \vertex arc (120:105:\rr) node[red,above] {$a$} arc (105:90:\rr) \vertex node[red,above](x1) {$x$} -- (x2);
\begin{scope}
	\path[clip] (0,0) arc (-120:-60:\rr) arc (60:120:\rr);
	\foreach \x in {0,0.24,...,3} {
		\draw[green!50!brown] (\x,1) -- (\x,-1);
	}
\end{scope}
\draw[red, decorate,decoration={brace,amplitude=5pt}] ($(x1)+(0.2,-0.2)$) -- ($(x2)+(0.2,0.2)$) node[midway, xshift=0.7cm] {$x \times I$};
\end{tikzpicture}
};

\node(B) at (-4,-4) {
\begin{tikzpicture}
\node[red,left] at (0,0) {$y$};
\draw (0,0) \vertex 
	arc (120:105:\rr) node[red,above] {$a$}
	arc (105:90:\rr) node[red,above] {$x$} \namedvertex{x1};
\draw (0,0)
	arc (-120:-90:\rr) node[red,below] {$a$}
	arc (-90:-61:\rr) \namedvertex{x2} node[red,right] {$x$};
\draw (x1) -- node[red, above=3pt] {$x \times I$} (x2);
\begin{scope}
	\path[clip] (0,0) arc (-120:-60:\rr) arc (60:120:\rr);
	\foreach \x in {0,0.48,...,9} {
		\draw[green!50!brown] (\x/4,1) -- (\x,-1);
	}
\end{scope}
\end{tikzpicture}
};

\node(C) at (4,-4) {
\begin{tikzpicture}[y=-1cm]
\node[red,left] at (0,0) {$y$};
\draw (0,0) \vertex 
	arc (120:105:\rr) node[red,below] {$a$}
	arc (105:90:\rr) node[red,below] {$x$} \namedvertex{x1};
\draw (0,0)
	arc (-120:-90:\rr) node[red,above] {$a$}
	arc (-90:-61:\rr) \namedvertex{x2} node[red,right] {$x$};
\draw (x1) -- node[red, below=3pt] {$x \times I$} (x2);
\begin{scope}
	\path[clip] (0,0) arc (-120:-60:\rr) arc (60:120:\rr);
	\foreach \x in {0,0.48,...,9} {
		\draw[green!50!brown] (\x/4,1) -- (\x,-1);
	}
\end{scope}
\end{tikzpicture}
};

\draw[->] (A) -- (B);
\draw[->] (A) -- (C);
\end{tikzpicture}
\caption{Producing weak identities from half pinched products}
\label{fzo2}
\end{figure}
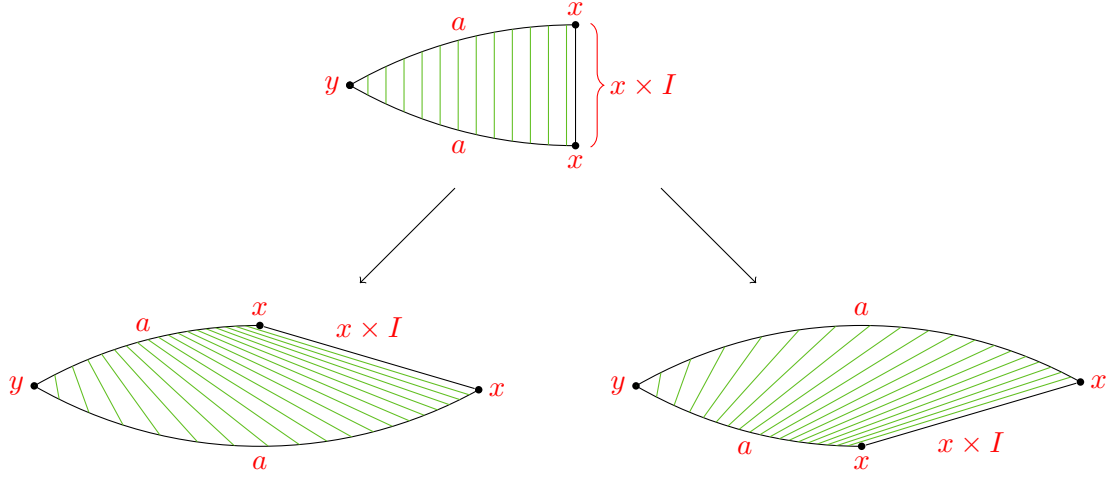
As suggested by the figure, these are two different reparameterizations
of a half-pinched version of $a\times I$
(i.e.\ two different homeomorphisms from the half-pinched $I\times I$ to the standard bigon).
We must show that the two compositions of these two maps give the identity 2-morphisms
on $a$ and $a\bullet \id_x$, as defined above.
Figure \ref{fzo3} shows one case.
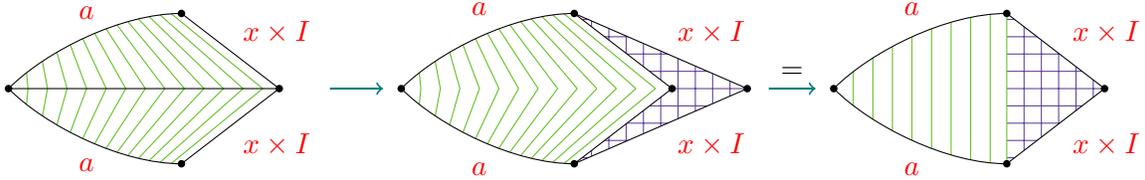
\begin{figure}[t]
\centering
\begin{tikzpicture}

\newcommand{\vertex}{node[circle,fill=black,inner sep=1pt] {}}
\newcommand{\nsep}{1.8}

\node(A) at (0,0) {
\begin{tikzpicture}

	\draw (0,0) coordinate (p1);
	\draw (3.6,0) coordinate (p2);
	\draw (2.3,1) coordinate (p3);
	\draw (2.3,-1) coordinate (p4);
	
	\begin{scope}
		\clip (p1) 	.. controls +(.5,-.5) and +(-.8,0)  .. (p4) -- 
				(p2) -- (p3) .. controls +(-.8,0) and +(.5,.5) .. (p1);
		\foreach \x in {0,0.26,...,4} {
			\draw[green!50!brown] (\x,0) -- (intersection cs: first line={(p2)--(p3)}, second line={(0,0)--(0,1)});
			\draw[green!50!brown] (\x,0) -- (intersection cs: first line={(p2)--(p4)}, second line={(0,0)--(0,1)});
		}
	\end{scope}
	
	\draw (p1) .. controls ($(p1) + (.5,.5)$) and ($(p3) + (-.8,0)$) .. node[red, above=3pt] {$a$} (p3);
	\draw (p1) .. controls ($(p1) + (.5,-.5)$) and ($(p4) + (-.8,0)$) .. node[red, below=3pt] {$a$} (p4);
	\draw (p3) -- node[red, above=7pt, right=1pt] {$x \times I$} (p2);
	\draw (p4) -- node[red, below=7pt, right=1pt] {$x \times I$} (p2);
	\draw (p1) -- (p2);
	
	\draw (p1) \vertex;
	\draw (p2) \vertex;
	\draw (p3) \vertex;
	\draw (p4) \vertex;

\end{tikzpicture}
};

\node[outer sep=\nsep](B) at (5.5,0) {
\begin{tikzpicture}

	\draw (0,0) coordinate (p1);
	\draw (3.6,0) coordinate (p2);
	\draw (2.3,1) coordinate (p3);
	\draw (2.3,-1) coordinate (p4);
	\draw (4.6,0) coordinate (p2b);
	
	\begin{scope}
		\clip (p1) 	.. controls +(.5,-.5) and +(-.8,0)  .. (p4) -- 
				(p2) -- (p3) .. controls +(-.8,0) and +(.5,.5) .. (p1);
		\foreach \x in {0,0.26,...,4} {
			\draw[green!50!brown] (\x,0) -- (intersection cs: first line={(p2)--(p3)}, second line={(0,0)--(0,1)});
			\draw[green!50!brown] (\x,0) -- (intersection cs: first line={(p2)--(p4)}, second line={(0,0)--(0,1)});
		}
	\end{scope}
	
	\begin{scope}
		\clip (p3)--(p2)--(p4)--(p2b)--cycle;
		\draw[blue!50!brown, step=.23] ($(p4)+(0,-1)$) grid +(3,3);
	\end{scope}
	
	\draw (p1) .. controls ($(p1) + (.5,.5)$) and ($(p3) + (-.8,0)$) .. node[red, above=3pt] {$a$} (p3);
	\draw (p1) .. controls ($(p1) + (.5,-.5)$) and ($(p4) + (-.8,0)$) .. node[red, below=3pt] {$a$} (p4);
	\draw (p3) -- (p2);
	\draw (p4) -- (p2);
	\draw (p3) -- node[red, above=7pt, right=1pt] {$x \times I$} (p2b);
	\draw (p4) -- node[red, below=7pt, right=1pt] {$x \times I$} (p2b);
	
	\draw (p1) \vertex;
	\draw (p2) \vertex;
	\draw (p3) \vertex;
	\draw (p4) \vertex;
	\draw (p2b) \vertex;

\end{tikzpicture}
};

\node[outer sep=\nsep](C) at (11,0) {
\begin{tikzpicture}

	\draw (0,0) coordinate (p1);
	\draw (2.3,0) coordinate (p2);
	\draw (2.3,1) coordinate (p3);
	\draw (2.3,-1) coordinate (p4);
	\draw (3.6,0) coordinate (p2b);
	
	\begin{scope}
		\clip (p1) 	.. controls +(.5,-.5) and +(-.8,0)  .. (p4) -- 
				(p2) -- (p3) .. controls +(-.8,0) and +(.5,.5) .. (p1);
		\foreach \x in {0,0.26,...,4} {
			\draw[green!50!brown] (\x,-1) -- (\x,1);
		}
	\end{scope}
	
	\begin{scope}
		\clip (p3)--(p2)--(p4)--(p2b)--cycle;
		\draw[blue!50!brown, step=.23] ($(p4)+(0,-1)$) grid +(3,3);
	\end{scope}
	
	\draw (p1) .. controls ($(p1) + (.5,.5)$) and ($(p3) + (-.8,0)$) .. node[red, above=3pt] {$a$} (p3);
	\draw (p1) .. controls ($(p1) + (.5,-.5)$) and ($(p4) + (-.8,0)$) .. node[red, below=3pt] {$a$} (p4);
	\draw[green!50!brown] (p3) -- (p4);
	\draw (p3) -- node[red, above=7pt, right=1pt] {$x \times I$} (p2b);
	\draw (p4) -- node[red, below=7pt, right=1pt] {$x \times I$} (p2b);
	
	\draw (p1) \vertex;
	\draw (p3) \vertex;
	\draw (p4) \vertex;
	\draw (p2b) \vertex;

\end{tikzpicture}
};

\draw[->, thick, blue!50!green] (A) -- (B);
\draw[->, thick, blue!50!green] (B) -- node[black, above] {$=$} (C);

\end{tikzpicture}
\caption{Composition of weak identities, 1}
\label{fzo3}
\end{figure}
In the first step we have inserted a copy of $(x\times I)\times I$.
Figure \ref{fzo4} shows the other case.

\newcommand{\vertex}{node[circle,fill=black,inner sep=1pt] {}}

\begin{figure}[t]
\centering
\begin{tikzpicture}

\newcommand{\nsep}{1.8}

\clip (-4,-1.25)--(12,-1.25)--(16,1.25)--(-1,1.25)--cycle;

\node[outer sep=\nsep](A) at (0,0) {
\begin{tikzpicture}
	\draw (0,0) coordinate (p1);
	\draw (4,0) coordinate (p2);
	\draw (2.4,0) coordinate (p2a);
	\draw (2,1.2) coordinate (pu);
	\draw (2,-1.2) coordinate (pd);

	\begin{scope}
		\clip (p1) .. controls (pu) .. (p2) .. controls (pd) .. (p1);
		\foreach \t in {0,.065,...,1} {
			\draw[green!50!brown] ($(p1)!\t!(p2a)$) -- +(90 - \t*90 + \t*6 : 4);
			\draw[green!50!brown] ($(p1)!\t!(p2a)$) -- +(-90 + \t*90 - \t*6 : 4);
		}
		\draw[dashed] ($(p2a) + (-.6,3)$)--(p2a)--($(p2a) + (-.6,-3)$);
	\end{scope}

	\draw (p1) .. controls (pu) .. (p2) .. controls (pd) .. (p1);
	\draw (p1)--(p2);

	\draw (p1) \vertex;
	\draw (p2) \vertex;
	\draw (p2a) \vertex;
\end{tikzpicture}
};

\node[outer sep=\nsep](B) at (5,0) {
\begin{tikzpicture}
	\draw (0,0) coordinate (p1);
	\draw (4,0) coordinate (p2);
	\draw (2.4,0) coordinate (p2a);
	\draw (2,1.2) coordinate (pu);
	\draw (2,-1.2) coordinate (pd);

	\begin{scope}
	\clip (-.1,3)--($(p2a) + (-.6,3)$)--(p2a)--($(p2a) + (-.6,-3)$)--(-.1,-3)--cycle;

	\begin{scope}
		\clip (p1) .. controls (pu) .. (p2) .. controls (pd) .. (p1);
		\foreach \t in {0,.065,...,1} {
			\draw[green!50!brown] ($(p1)!\t!(p2a)$) -- +(90 - \t*90 + \t*6 : 4);
			\draw[green!50!brown] ($(p1)!\t!(p2a)$) -- +(-90 + \t*90 - \t*6 : 4);
		}
		\draw ($(p2a) + (-.6,3)$)--(p2a)--($(p2a) + (-.6,-3)$);
	\end{scope}

	\draw (p1) .. controls (pu) .. (p2) .. controls (pd) .. (p1);
	\end{scope}

	\begin{scope}
		\clip (p1) .. controls (pu) .. (p2) .. controls (pd) .. (p1);
		\draw ($(p2a) + (-.6,3)$)--(p2a)--($(p2a) + (-.6,-3)$);
	\end{scope}

	\draw (p1) \vertex;
	\draw (p2a) \vertex;
\end{tikzpicture}
};

\node[outer sep=\nsep](C) at (9,0) {
\begin{tikzpicture}
	\draw (0,0) coordinate (p1);
	\draw (4,0) coordinate (p2);
	\draw (2,1.2) coordinate (pu);
	\draw (2,-1.2) coordinate (pd);

	\begin{scope}
		\clip (p1) .. controls (pu) .. (p2) .. controls (pd) .. (p1);
		\foreach \t in {0,.045,...,1} {
			\draw[green!50!brown] ($(p1)!\t!(p2) + (0,2)$) -- +(0,-4);
		}
	\end{scope}

	\draw (p1) .. controls (pu) .. (p2) .. controls (pd) .. (p1);

	\draw (p1) \vertex;
	\draw (p2) \vertex;
\end{tikzpicture}
};

\draw[->, thick, blue!50!green] (A) -- (B);
\draw[->, thick, blue!50!green] ($(B) + (1,0)$) -- node[black, above] {$=$} (C);

\end{tikzpicture}
\caption{Composition of weak identities, 2}
\label{fzo4}
\end{figure}
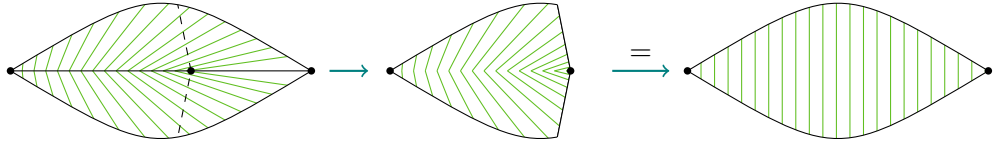
We notice that a certain subset of the disk is a product region and remove it.

Given 2-morphisms $f$ and $g$, we define the horizontal composition $f *_h g$ to be any of the four
equal 2-morphisms in Figure \ref{fzo5}.
Figure \ref{fig:horizontal-compositions-equal} illustrates part of the proof that these four 2-morphisms are equal.
Similar arguments show that horizontal composition is associative.
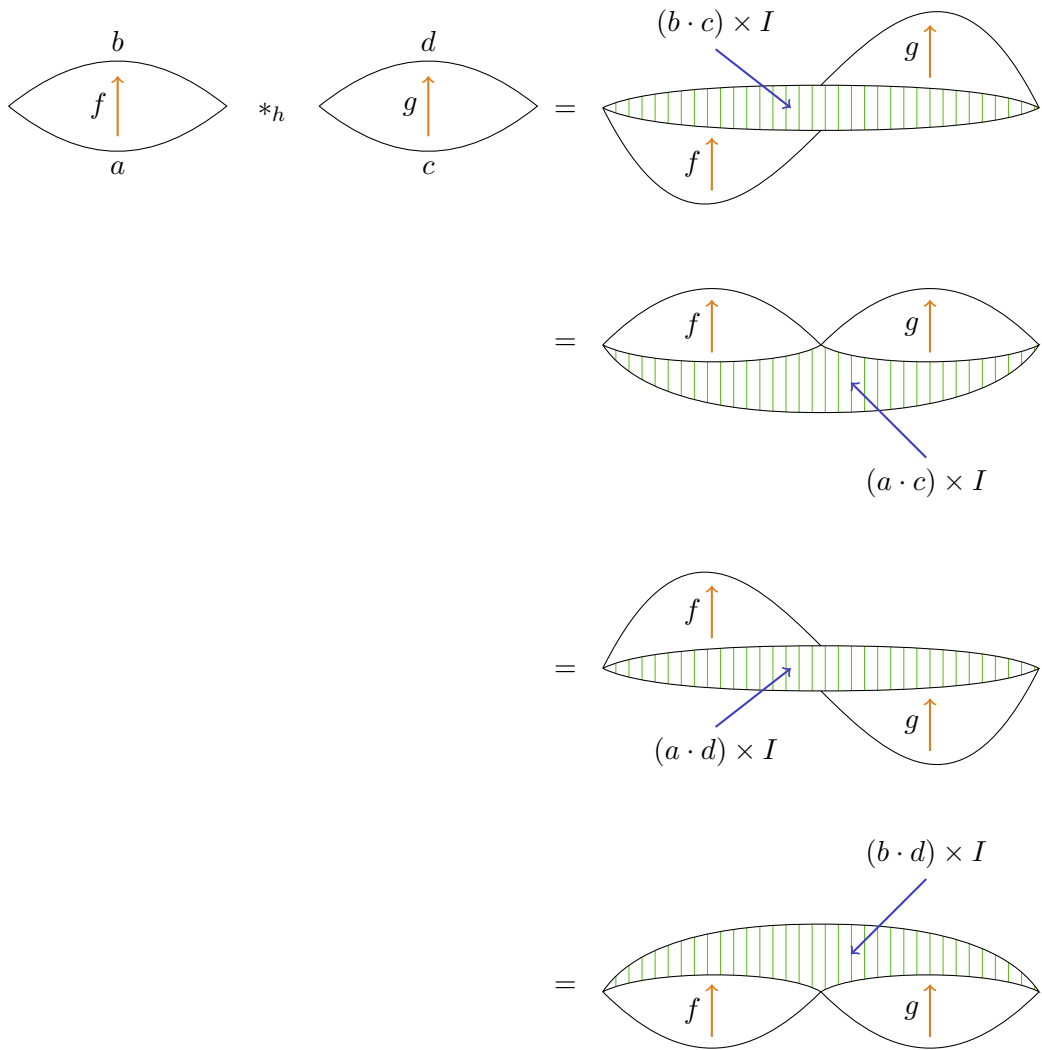
\begin{figure}[t]
\begin{align*}
\raisebox{-.9cm}{
\begin{tikzpicture}
	\draw (0,0) .. controls +(1,.8) and +(-1,.8) .. node[above] {$b$} (2.9,0)
				.. controls +(-1,-.8) and +(1,-.8) .. node[below] {$a$} (0,0);
	\draw[->, thick, orange!50!brown] (1.45,-.4)--  node[left, black] {$f$} +(0,.8);
\end{tikzpicture}}
\;\;\;*_h\;\;
\raisebox{-.9cm}{
\begin{tikzpicture}
	\draw (0,0) .. controls +(1,.8) and +(-1,.8) .. node[above] {$d$} (2.9,0)
				.. controls +(-1,-.8) and +(1,-.8) .. node[below] {$c$} (0,0);
	\draw[->, thick, orange!50!brown] (1.45,-.4)--  node[left, black] {$g$} +(0,.8);
\end{tikzpicture}}
\;&=\;
\raisebox{-1.9cm}{
\begin{tikzpicture}
	\draw (0,0) coordinate (p1);
	\draw (5.8,0) coordinate (p2);
	\draw (2.9,.3) coordinate (pu);
	\draw (2.9,-.3) coordinate (pd);
	\begin{scope}
		\clip (p1) .. controls +(.6,.3) and +(-.5,0) .. (pu)
					.. controls +(.5,0) and +(-.6,.3) .. (p2)
					.. controls +(-.6,-.3) and +(.5,0) .. (pd)
					.. controls +(-.5,0) and +(.6,-.3) .. (p1);
		\foreach \t in {0,.03,...,1} {
			\draw[green!50!brown] ($(p1)!\t!(p2) + (0,2)$) -- +(0,-4);
		}
	\end{scope}
	\draw (p1) .. controls +(.6,.3) and +(-.5,0) .. (pu)
				.. controls +(.5,0) and +(-.6,.3) .. (p2)
				.. controls +(-.6,-.3) and +(.5,0) .. (pd)
				.. controls +(-.5,0) and +(.6,-.3) .. (p1);
	\draw (p1) .. controls +(1,-2) and +(-1,-1) .. (pd);
	\draw (p2) .. controls +(-1,2) and +(1,1) .. (pu);
	\draw[->, thick, orange!50!brown] (1.45,-1.1)--  node[left, black] {$f$} +(0,.7);
	\draw[->, thick, orange!50!brown] (4.35,.4)--  node[left, black] {$g$} +(0,.7);
	\draw[->, thick, blue!75!yellow] (1.5,.78) node[black, above] {$(b\cdot c)\times I$} -- (2.5,0);
\end{tikzpicture}} \displaybreak[1] \\
\;&=\;
\raisebox{-2.1cm}{
\begin{tikzpicture}
	\draw (0,0) coordinate (p1);
	\draw (5.8,0) coordinate (p2);
	\draw (2.9,0) coordinate (pu);
	\draw (2.9,-.9) coordinate (pd);
	\begin{scope}
		\clip (p1) .. controls +(.6,-.3) and +(-.5,-.3) .. (pu)
					.. controls +(.5,-.3) and +(-.6,-.3) .. (p2)
					.. controls +(-.6,-.9) and +(.5,0) .. (pd)
					.. controls +(-.5,0) and +(.6,-.9) .. (p1);
		\foreach \t in {0,.03,...,1} {
			\draw[green!50!brown] ($(p1)!\t!(p2) + (0,2)$) -- +(0,-4);
		}
	\end{scope}
	\draw  (p1) .. controls +(.6,-.3) and +(-.5,-.3) .. (pu)
					.. controls +(.5,-.3) and +(-.6,-.3) .. (p2)
					.. controls +(-.6,-.9) and +(.5,0) .. (pd)
					.. controls +(-.5,0) and +(.6,-.9) .. (p1);
	\draw (p1) .. controls +(1,1) and +(-1,1) .. (pu);
	\draw (p2) .. controls +(-1,1) and +(1,1) .. (pu);
	\draw[->, thick, orange!50!brown] (1.45,-0.1)--  node[left, black] {$f$} +(0,.7);
	\draw[->, thick, orange!50!brown] (4.35,-0.1)--  node[left, black] {$g$} +(0,.7);
	\draw[->, thick, blue!75!yellow] (4.3,-1.5) node[black, below] {$(a\cdot c)\times I$} -- (3.3,-0.5);
\end{tikzpicture}} \displaybreak[1] \\
\;&=\;
\raisebox{-1.9cm}{
\begin{tikzpicture}[y=-1cm]
	\draw (0,0) coordinate (p1);
	\draw (5.8,0) coordinate (p2);
	\draw (2.9,.3) coordinate (pu);
	\draw (2.9,-.3) coordinate (pd);
	\begin{scope}
		\clip (p1) .. controls +(.6,.3) and +(-.5,0) .. (pu)
					.. controls +(.5,0) and +(-.6,.3) .. (p2)
					.. controls +(-.6,-.3) and +(.5,0) .. (pd)
					.. controls +(-.5,0) and +(.6,-.3) .. (p1);
		\foreach \t in {0,.03,...,1} {
			\draw[green!50!brown] ($(p1)!\t!(p2) + (0,2)$) -- +(0,-4);
		}
	\end{scope}
	\draw (p1) .. controls +(.6,.3) and +(-.5,0) .. (pu)
				.. controls +(.5,0) and +(-.6,.3) .. (p2)
				.. controls +(-.6,-.3) and +(.5,0) .. (pd)
				.. controls +(-.5,0) and +(.6,-.3) .. (p1);
	\draw (p1) .. controls +(1,-2) and +(-1,-1) .. (pd);
	\draw (p2) .. controls +(-1,2) and +(1,1) .. (pu);
	\draw[<-, thick, orange!50!brown] (1.45,-1.1)--  node[left, black] {$f$} +(0,.7);
	\draw[<-, thick, orange!50!brown] (4.35,.4)--  node[left, black] {$g$} +(0,.7);
	\draw[->, thick, blue!75!yellow] (1.5,.78) node[black, below] {$(a\cdot d)\times I$} -- (2.5,0);
\end{tikzpicture}} \displaybreak[1] \\
\;&=\;
\raisebox{-1.0cm}{
\begin{tikzpicture}[y=-1cm]
	\draw (0,0) coordinate (p1);
	\draw (5.8,0) coordinate (p2);
	\draw (2.9,0) coordinate (pu);
	\draw (2.9,-.9) coordinate (pd);
	\begin{scope}
		\clip (p1) .. controls +(.6,-.3) and +(-.5,-.3) .. (pu)
					.. controls +(.5,-.3) and +(-.6,-.3) .. (p2)
					.. controls +(-.6,-.9) and +(.5,0) .. (pd)
					.. controls +(-.5,0) and +(.6,-.9) .. (p1);
		\foreach \t in {0,.03,...,1} {
			\draw[green!50!brown] ($(p1)!\t!(p2) + (0,2)$) -- +(0,-4);
		}
	\end{scope}
	\draw  (p1) .. controls +(.6,-.3) and +(-.5,-.3) .. (pu)
					.. controls +(.5,-.3) and +(-.6,-.3) .. (p2)
					.. controls +(-.6,-.9) and +(.5,0) .. (pd)
					.. controls +(-.5,0) and +(.6,-.9) .. (p1);
	\draw (p1) .. controls +(1,1) and +(-1,1) .. (pu);
	\draw (p2) .. controls +(-1,1) and +(1,1) .. (pu);
	\draw[<-, thick, orange!50!brown] (1.45,-0.1)--  node[left, black] {$f$} +(0,.7);
	\draw[<-, thick, orange!50!brown] (4.35,-0.1)--  node[left, black] {$g$} +(0,.7);
	\draw[->, thick, blue!75!yellow] (4.3,-1.5) node[black, above] {$(b\cdot d)\times I$} -- (3.3,-0.5);
\end{tikzpicture}} 
\end{align*}
\caption{Horizontal composition of 2-morphisms}
\label{fzo5}
\end{figure}
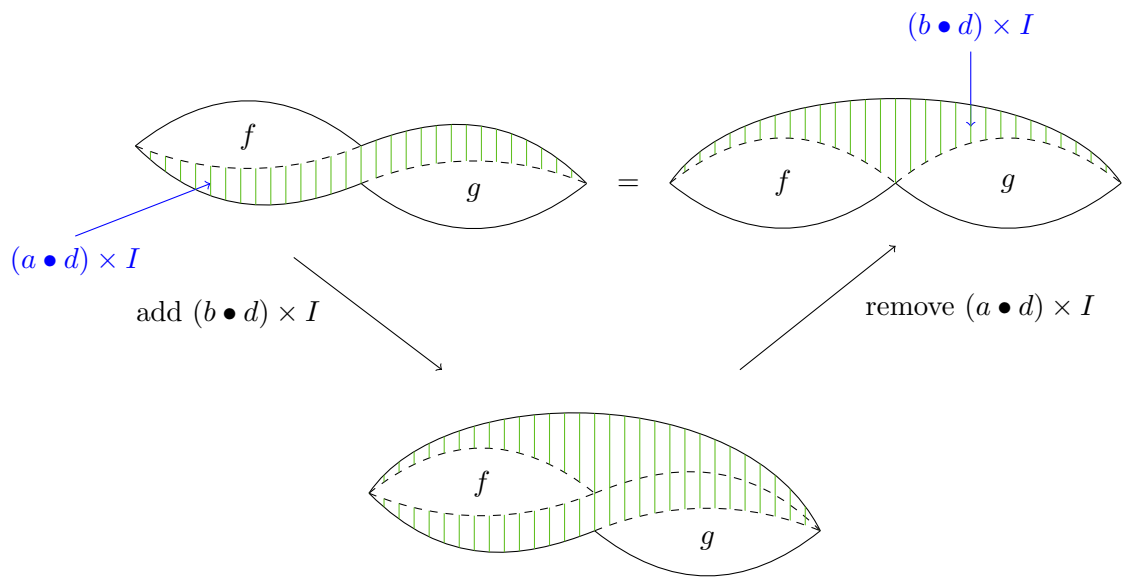
\begin{figure}[t]
$$
\begin{tikzpicture}
\node (fg1) at (0,0) {
\begin{tikzpicture}[baseline=-0.6cm]
\path (0,0) coordinate (f1);
\path (3,0) coordinate (f2);
\path (3,-0.5) coordinate (g1);
\path (6,-0.5) coordinate (g2);
\node at (1.5,0.125) {$f$};
\node at (4.5,-0.625) {$g$};
\draw (f1) .. controls +(1,.8) and +(-1,.8) .. (f2);
\draw[dashed] (f1) .. controls +(1,-.4) and +(-1,-.4) .. (f2);
\draw (f1) .. controls +(1,-1) and +(-1,-.4) .. (g1);
\draw (g1) .. controls +(1,-.8) and +(-1,-.8) .. (g2);
\draw[dashed] (g1) .. controls +(1,.4) and +(-1,.4) .. (g2);
\draw (f2) .. controls +(1,.4) and +(-1,1) .. (g2);
\draw[blue,->] (-0.8,-1.2) node[below] {$(a \bullet d) \times I$} -- (1,-0.5) ;
\path[clip] (f1) .. controls +(1,-.4) and +(-1,-.4) .. (f2)
                    .. controls +(1,.4) and +(-1,1) .. (g2)
                    .. controls +(-1,.4) and +(1,.4) .. (g1)
                    .. controls +(-1,-.4) and +(1,-1) .. (f1);
\foreach \x in {0,0.2, ..., 6} {
	\draw[green!50!brown] (\x,-2) -- + (0,4);
}
\end{tikzpicture}
};
\node (fg2) at (4,-4) {
\begin{tikzpicture}[baseline=-0.1cm]
\path (0,0) coordinate (f1);
\path (3,0) coordinate (f2);
\path (3,-0.5) coordinate (g1);
\path (6,-0.5) coordinate (g2);
\node at (1.5,0.125) {$f$};
\node at (4.5,-0.625) {$g$};
\draw[dashed] (f1) .. controls +(1,.8) and +(-1,.8) .. (f2);
\draw[dashed] (f1) .. controls +(1,-.4) and +(-1,-.4) .. (f2);
\draw (f1) .. controls +(1,-1) and +(-1,-.4) .. (g1);
\draw (g1) .. controls +(1,-.8) and +(-1,-.8) .. (g2);
\draw[dashed] (g1) .. controls +(1,.4) and +(-1,.4) .. (g2);
\draw[dashed] (f2) .. controls +(1,.4) and +(-1,1) .. (g2);
\draw (f1) .. controls +(1,1.5) and +(-1,2)..(g2);
\begin{scope}
\path[clip] (f1) .. controls +(1,-.4) and +(-1,-.4) .. (f2)
                    .. controls +(1,.4) and +(-1,1) .. (g2)
                    .. controls +(-1,.4) and +(1,.4) .. (g1)
                    .. controls +(-1,-.4) and +(1,-1) .. (f1);
\foreach \x in {0,0.2, ..., 6} {
	\draw[green!50!brown] (\x,-2) -- + (0,4);
}
\end{scope}
\begin{scope}
\path[clip] (f1) ..  controls +(1,1.5) and +(-1,2).. (g2)
		      .. controls +(-1,1) and +(1,.4) .. (f2)
		      .. controls +(-1,.8) and + (1,.8) .. (f1);
\foreach \x in {0,0.2, ..., 6} {
	\draw[green!50!brown] (\x,-2) -- + (0,4);
}
\end{scope}
\end{tikzpicture}
};
\node (fg3) at (8,0) {
\begin{tikzpicture}[baseline=-2.45cm]
\path (0,0) coordinate (f1);
\path (3,0) coordinate (f2);
\path (3,0) coordinate (g1);
\path (6,0) coordinate (g2);
\node at (1.5,0) {$f$};
\node at (4.5,0) {$g$};
\draw[dashed] (f1) .. controls +(1,.8) and +(-1,.8) .. (f2);
\draw (f1) .. controls +(1,-.8) and +(-1,-.8) .. (f2);
\draw (g1) .. controls +(1,-.8) and +(-1,-.8) .. (g2);
\draw[dashed] (g1) .. controls +(1,.8) and +(-1,.8) .. (g2);
\draw (f1) .. controls +(1,1.5) and +(-1,1.5)..(g2);
\draw[blue,->] (4,1.75) node[above] {$(b \bullet d) \times I$}-- + (0,-1);
\begin{scope}
\path[clip] (f1) ..  controls +(1,1.5) and +(-1,1.5).. (g2)
		      .. controls +(-1,.8) and +(1,.8) .. (f2)
		      .. controls +(-1,.8) and + (1,.8) .. (f1);
\foreach \x in {0,0.2, ..., 6} {
	\draw[green!50!brown] (\x,-2) -- + (0,4);
}
\end{scope}
\end{tikzpicture}
};
\draw[->] ($(fg1.south)+(0,0.5)$) -- node[left=0.5cm] {add $(b \bullet d) \times I$} (fg2);
\draw[->] (fg2) -- node[right=0.5cm] {remove $(a \bullet d) \times I$} ($(fg3.south)+(0,1.75)$);
\path (fg1) -- node {$=$} (fg3);
\end{tikzpicture}
$$
\caption{Part of the proof that the four different horizontal compositions of 2-morphisms are equal.}
\label{fig:horizontal-compositions-equal}
\end{figure}

Given 1-morphisms $a$, $b$ and $c$ of $D$, we define the associator from $(a\bullet b)\bullet c$ to $a\bullet(b\bullet c)$
as in Figure \ref{fig:associator}.
This is just a reparameterization of the pinched product $(a\bullet b\bullet c)\times I$ of $\cC$.
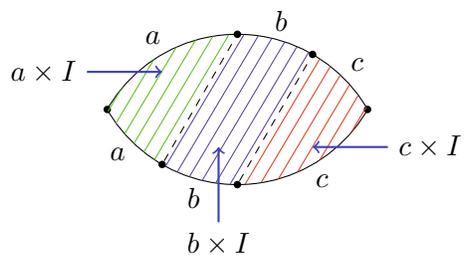
\begin{figure}[t]
$$
\begin{tikzpicture}
\node[circle,fill=black,inner sep=1pt] at (1.73,0) {};
\node[circle,fill=black,inner sep=1pt] at (-1.73,0) {};
\begin{scope}[yshift=-1cm]
\path[clip] (0,0) circle (2);
\begin{scope}[yshift=2cm]
\draw (0,0) circle (2);
\node[circle,fill=black,inner sep=1pt] (L2) at (-90:2) {};
\node[circle,fill=black,inner sep=1pt] (L1) at (-120:2) {};
\node at (-60:2.25) {$c$};
\node at (-105:2.25) {$b$};
\node at (-135:2.25) {$a$};
\end{scope}
\end{scope}
\begin{scope}[yshift=1cm]
\path[clip] (0,0) circle (2);
\begin{scope}[yshift=-2cm]
\node at (120:2.25) {$a$};
\node at (75:2.25) {$b$};
\node at (45:2.25) {$c$};
\draw (0,0) circle (2);
\node[circle,fill=black,inner sep=1pt] (U1) at (90:2) {};
\node[circle,fill=black,inner sep=1pt] (U2) at (60:2) {};
\end{scope}
\end{scope}
\draw[dashed] (L1) -- (U1);
\draw[dashed] (L2) -- (U2);
\begin{scope}
\path[clip] (0,1) circle (2);
\path[clip] (0,-1) circle (2);
		\foreach \t in {-2.5,-2.3,...,-1.2} {
			\draw[green!50!brown] (\t,-1) -- +(1.19,2);
		}
		\foreach \t in {-1.1,-0.9,...,0} {
			\draw[blue!50!brown] (\t,-1) -- +(1.19,2);
		}
		\foreach \t in {0.1,0.3,...,2.5} {
			\draw[red!50!brown] (\t,-1) -- +(1.19,2);
		}
\end{scope}
\draw[->, thick, blue!75!yellow] (-2,0.5) node[black, left] {$a\times I$} -- (-1,0.5);
\draw[->, thick, blue!75!yellow] (-0.25,-1.5) node[black, below] {$b\times I$} -- (-0.25,-0.5);
\draw[->, thick, blue!75!yellow] (2,-0.5) node[black, right] {$c\times I$} -- (1,-0.5);
\end{tikzpicture}
$$
\caption{An associator.}
\label{fig:associator}
\end{figure}

Let $x,y,z$ be objects of $D$ and let $a:x\to y$ and $b:y\to z$ be 1-morphisms of $D$.
We have already defined above 
structure maps $u:a\bullet \id_y\to a$ and $v:\id_y\bullet b\to b$, as well as an associator
$\alpha: (a\bullet \id_y)\bullet b\to a\bullet(\id_y\bullet b)$, as shown in
Figure \ref{fig:ingredients-triangle-axiom}.
(See also Figures \ref{fzo2} and \ref{fig:associator}.)
We now show that $D$ satisfies the triangle axiom, which states that $u\bullet\id_b$ 
is equal to the vertical composition of $\alpha$ and $\id_a\bullet v$.
(Both are 2-morphisms from $(a\bullet \id_y)\bullet b$ to $a\bullet b$.)
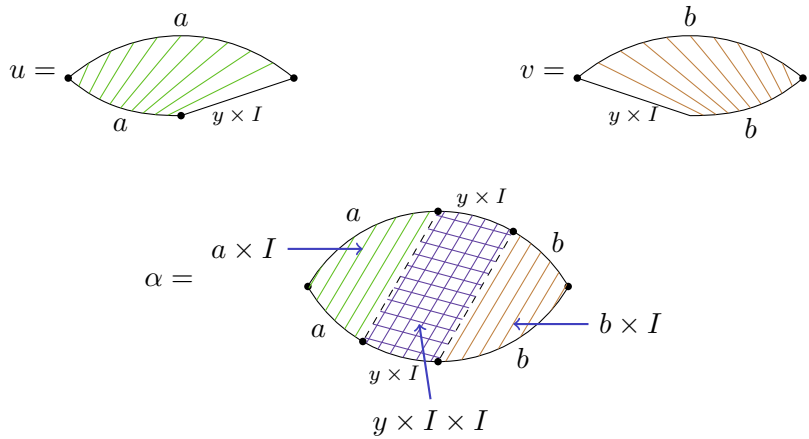
\begin{figure}[t]
\begin{align*}
u & =
\begin{tikzpicture}[baseline]
\coordinate (P) at (0,0);
\coordinate (Q) at (3,0);
\coordinate (R) at (1.5,-0.5);
\draw (P) \vertex to[out=40,in=140] node[above] {$a$} (Q) \vertex -- node[below] {\scriptsize $y \times I$} (R) \vertex to[out=-180,in=-40] node[below] {$a$} (P);
\clip (P) to[out=40,in=140] (Q) -- (R) to[out=-180,in=-40] (P);
\foreach \x in {1,...,9} {
	\path (P) to[out=40,in=140] node[pos=\x/10] (PQ\x) {} (Q);
	\path (P) to[out=-40,in=-180] node[pos=\x/10] (PR\x) {} (R);
	\draw[green!50!brown] (PQ\x.center) -- (PR\x.center);
}
\end{tikzpicture} &
v & = 
\begin{tikzpicture}[baseline]
\coordinate (P) at (0,0);
\coordinate (Q) at (3,0);
\coordinate (R) at (1.5,-0.5);
\draw (P) \vertex to[out=40,in=140] node[above] {$b$} (Q) \vertex to[out=-140,in=0] node[below] {$b$} (R) -- node[below] {\scriptsize $y \times I$}  (P);
\clip (P) to[out=40,in=140] (Q) to[out=-140,in=0] (R) -- (P);
\foreach \x in {1,...,9} {
	\path (P) to[out=40,in=140] node[pos=\x/10] (PQ\x) {} (Q);
	\path (R) to[out=0,in=-140] node[pos=\x/10] (RQ\x) {} (Q);
	\draw[brown] (PQ\x.center) -- (RQ\x.center);
}
\end{tikzpicture}
\end{align*}
\vspace{-2cm}
\begin{align*}
\alpha & = 
\begin{tikzpicture}[baseline]
\node[circle,fill=black,inner sep=1pt] at (1.73,0) {};
\node[circle,fill=black,inner sep=1pt] at (-1.73,0) {};
\begin{scope}[yshift=-1cm]
\path[clip] (0,0) circle (2);
\begin{scope}[yshift=2cm]
\draw (0,0) circle (2);
\node[circle,fill=black,inner sep=1pt] (L2) at (-90:2) {};
\node[circle,fill=black,inner sep=1pt] (L1) at (-120:2) {};
\node at (-60:2.25) {$b$};
\node at (-105:2.25) {\scriptsize $y \times I$};
\node at (-135:2.25) {$a$};
\end{scope}
\end{scope}
\begin{scope}[yshift=1cm]
\path[clip] (0,0) circle (2);
\begin{scope}[yshift=-2cm]
\node at (120:2.25) {$a$};
\node at (75:2.25) {\scriptsize $y \times I$};
\node at (45:2.25) {$b$};
\draw (0,0) circle (2);
\node[circle,fill=black,inner sep=1pt] (U1) at (90:2) {};
\node[circle,fill=black,inner sep=1pt] (U2) at (60:2) {};
\end{scope}
\end{scope}
\draw[dashed] (L1) -- (U1);
\draw[dashed] (L2) -- (U2);
\begin{scope}
\path[clip] (0,1) circle (2);
\path[clip] (0,-1) circle (2);
		\foreach \t in {-2.5,-2.3,...,-1.2} {
			\draw[green!50!brown] (\t,-1) -- +(1.19,2);
		}
		\foreach \t in {-1.1,-0.9,...,0} {
			\draw[blue!50!brown] (\t,-1) -- +(1.19,2);
		}
		\foreach \t in {0.1,0.3,...,2.5} {
			\draw[brown] (\t,-1) -- +(1.19,2);
		}
	\foreach \x in {0,1,...,10} {
		\path (L1) to[out=60,in=-120] node[pos=\x/10] (p1\x) {}(U1);
		\path (L2) to[out=60,in=-120] node[pos=\x/10] (p2\x) {} (U2);
		\draw[blue!50!brown] (p1\x.center) -- (p2\x.center);
	}
\end{scope}
\draw[->, thick, blue!75!yellow] (-2,0.5) node[black, left] {$a\times I$} -- (-1,0.5);
\draw[->, thick, blue!75!yellow] (-0.1,-1.5) node[black, below] {$y\times I \times I$} -- (-0.25,-0.5);
\draw[->, thick, blue!75!yellow] (2,-0.5) node[black, right] {$b\times I$} -- (1,-0.5);
\end{tikzpicture}
\end{align*}
\vspace{-1cm}
\caption{Ingredients for the triangle axiom.}
\label{fig:ingredients-triangle-axiom}
\end{figure}

The horizontal compositions $u *_h \id_b$ and $\id_a *_h  v$ are shown in Figure \ref{fig:horizontal-composition}
(see also Figure \ref{fzo5}).
The vertical composition of $\alpha$ and $\id_a *_h  v$ is shown in Figure \ref{fig:vertical-composition}.
Figure \ref{fig:adding-a-collar} shows that we can add a collar to $u *_h \id_b$ so that the result differs from
Figure  \ref{fig:vertical-composition} by an isotopy rel boundary.
Note that here we have used in an essential way the associativity of product morphisms (Axiom \ref{axiom:product}.3) 
as well as compatibility of product morphisms with fiber-preserving maps (Axiom \ref{axiom:product}.1).
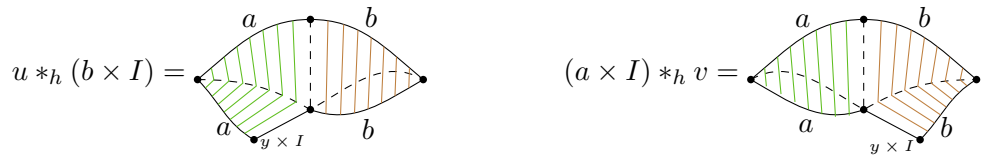
\begin{figure}[t]
\begin{align*}
u *_h (b \times I) & = 
\begin{tikzpicture}[baseline]
\coordinate (L) at (0,0);
\coordinate (R) at (3,0);
\coordinate (T) at (1.5,0.8);
\coordinate (M) at (1.5,-0.4);
\coordinate (B) at (0.75,-0.8);
\path (L) \vertex (T) \vertex (R) \vertex (M) \vertex (B) \vertex;
\draw (L)  to[out=40,in=180] node[above] {$a$} (T) 
		to[out=0,in=140] node[above] {$b$} (R) 
		to[out=-150,in=-20] node[below] {$b$} (M) 
		-- node[below] {\tiny $y \times I$} (B) 
		to[out=150,in=-45] node[below] {$a$} (L);
\draw[dashed] (L) to[out=0,in=150] (M)
			     to[out=30,in=150] (R);
\draw[dashed] (T) -- (M);
\foreach \n in {0,...,7} {
	\path (L) to[out=-45,in=150] node[coordinate,pos=\n/8] (LB\n) {} (B);
	\path (L) to[out=0,in=150] node[coordinate,pos=\n/8] (LM\n) {} (M);
	\path (L) to[out=40,in=180] node[coordinate,pos=\n/8] (LT\n) {} (T);
	\draw[green!50!brown] (LB\n) -- (LM\n) -- (LT\n);
}
\foreach \n in {1,...,7} {
	\path (M) to[out=-20,in=-150] node[coordinate,pos=\n/8] (MR\n) {} (R);
	\path (T) to[out=0,in=140] node[coordinate,pos=\n/8] (TR\n) {} (R);
	\draw[brown] (MR\n) -- (TR\n);
}
\end{tikzpicture} &
(a \times I) *_h v & = 
\begin{tikzpicture}[baseline]
\coordinate (L) at (0,0);
\coordinate (R) at (3,0);
\coordinate (T) at (1.5,0.8);
\coordinate (M) at (1.5,-0.4);
\coordinate (B) at (2.25,-0.8);
\path (L) \vertex (T) \vertex (R) \vertex (M) \vertex (B) \vertex;
\draw (L)  to[out=40,in=180] node[above] {$a$} (T) 
		to[out=0,in=140] node[above] {$b$} (R)
		to[out=-150,in=45] node[below] {$b$} (B) 
		-- node[below=2] {\tiny $y \times I$} (M) 
		to[out=-160,in=-30] node[below] {$a$} (L);
\draw[dashed] (L) to[out=30,in=150] (M)
			     to[out=30,in=180] (R);
\draw[dashed] (T) -- (M);
\foreach \n in {1,...,7} {
	\path (B) to[out=45,in=-150] node[coordinate,pos=\n/8] (BR\n) {} (R);
	\path (M) to[out=30,in=180] node[coordinate,pos=\n/8] (MR\n) {} (R);
	\path (T) to[out=0,in=140] node[coordinate,pos=\n/8] (TR\n) {} (R);
	\draw[brown] (BR\n) -- (MR\n) -- (TR\n);
}
\foreach \n in {1,...,7} {
	\path (L) to[out=-20,in=-150] node[coordinate,pos=\n/8] (LM\n) {} (M);
	\path (L) to[out=40,in=180] node[coordinate,pos=\n/8] (LT\n) {} (T);
	\draw[green!50!brown] (LM\n) -- (LT\n);
}
\end{tikzpicture} 
\end{align*}
\caption{Horizontal compositions in the triangle axiom.}
\label{fig:horizontal-composition}
\end{figure}
\begin{figure}[t]
\vspace{-1.5cm}
\begin{align*}
\begin{tikzpicture}
\node[circle,fill=black,inner sep=1pt] (A) at (1.73,0) {};
\node[circle,fill=black,inner sep=1pt] (B) at (-1.73,0) {};
\draw[dashed] (A) -- (B);
\node[circle,fill=black,inner sep=1pt] (C) at (0,0) {};
\node[circle,fill=black,inner sep=1pt] (D) at (0.8,0) {};
\begin{scope}[yshift=-1cm]
\path[clip] (0,0) circle (2);
\begin{scope}[yshift=2cm]
\draw (0,0) circle (2);
\node[circle,fill=black,inner sep=1pt] (L2) at (-90:2) {};
\node[circle,fill=black,inner sep=1pt] (L1) at (-120:2) {};
\end{scope}
\end{scope}
\begin{scope}[yshift=1cm]
\path[clip] (0,0) circle (2);
\begin{scope}[yshift=-2cm]
\draw (0,0) circle (2);
\node[circle,fill=black,inner sep=1pt] (U) at (90:2) {};
\end{scope}
\end{scope}
\begin{scope}
\path[clip] (0,1) circle (2);
\path[clip] (0,-1) circle (2);
\foreach \n in {1,...,6} {
	\path (B) to[out=0,in=-180] node[coordinate,pos=\n/6] (BC\n) {} (C);
	\draw[green!50!brown] (BC\n) -- +(0,1);
	\draw[green!50!brown] (BC\n) -- +(-1.5,-1);
}
\foreach \n in {0,...,5} {
	\path (C) to[out=0,in=-180] node[coordinate,pos=\n/5] (CD\n) {} (D);
	\path (L1) to[out=-30,in=-180] node[coordinate,pos=\n/5] (L1L2\n) {} (L2);
	\draw[blue!50!brown] (CD\n) -- (L1L2\n);
}
\foreach \n in {0,...,5} {
	\path (L1) to[out=40,in=-140] node[coordinate,pos=\n/5] (L1C\n) {} (C);
	\path (L2) to[out=50,in=-130] node[coordinate,pos=\n/5] (L2D\n) {} (D);
	\draw[blue!50!brown] (L1C\n) -- (L2D\n);
}
\foreach \n in {1,3,...,9} {
	\path(C) to[out=30,in=150] node[coordinate,pos=\n/10] (CA\n) {} (A);
	\path(D) to[out=0,in=180] node[coordinate,pos=\n/10] (DA\n) {} (A);
	\path(L2) to[out=0,in=-120] node[coordinate,pos=\n/10] (L2A\n) {} (A);
	\draw[brown] (CA\n) -- +(0,1);
	\draw[brown] (CA\n) to[out=0] (DA\n) -- (L2A\n);
}
\draw[dashed] (U) -- (C) -- (L1) (L2) -- (D);
\end{scope}
\end{tikzpicture}
\end{align*}
\vspace{-2.5cm}
\caption{Vertical composition in the triangle axiom.}
\label{fig:vertical-composition}
\end{figure}
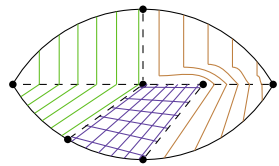
\begin{figure}[t]
\begin{align*}
\begin{tikzpicture}[baseline]
\coordinate (L) at (0,0);
\coordinate (R) at (3,0);
\coordinate (T) at (1.5,0.8);
\coordinate (M) at (1.5,-0.4);
\coordinate (B) at (0.75,-0.8);
\path (L) \vertex (T) \vertex (R) \vertex (M) \vertex (B) \vertex;
\draw (L)  to[out=40,in=180]  (T) 
		to[out=0,in=140]  (R) 
		to[out=-150,in=-20]  (M) 
		-- (B) 
		to[out=150,in=-45]  (L);
\draw[dashed] (T) -- (M);
\foreach \n in {0,...,7} {
	\path (L) to[out=-45,in=150] node[coordinate,pos=\n/8] (LB\n) {} (B);
	\path (L) to[out=0,in=150] node[coordinate,pos=\n/8] (LM\n) {} (M);
	\path (L) to[out=40,in=180] node[coordinate,pos=\n/8] (LT\n) {} (T);
	\draw[green!50!brown] (LB\n) -- (LM\n) -- (LT\n);
}
\foreach \n in {1,...,7} {
	\path (M) to[out=-20,in=-150] node[coordinate,pos=\n/8] (MR\n) {} (R);
	\path (T) to[out=0,in=140] node[coordinate,pos=\n/8] (TR\n) {} (R);
	\draw[brown] (MR\n) -- (TR\n);
}
\coordinate (B') at ($(B)+(0.125,-0.25)$);
\coordinate (M') at ($(M)+(0.125,-0.25)$);
\coordinate (R') at ($(R)+(0.125,-0.25)$);
\coordinate (X) at ($(M')+(0,-0.6)$);
\path (X) \vertex;
\draw[dashed] (M') -- (X);
\draw[clip] (B') to[out=-30,in=180] (X) to[out=0,in=-135] (R') to[out=-150,in=-20] (M') -- cycle;
\foreach \n in {1,...,7} {
	\path (M') to[out=-20,in=-150] node[coordinate,pos=\n/8] (M'R'\n) {} (R');
	\draw[brown] (M'R'\n) -- +(0,-1);
}
\foreach \n in {1,...,5} {
	\path (B') to[out=-30,in=180] node[coordinate,pos=\n/6] (B'X\n) {} (X);
	\path (M') to[out=-90,in=90] node[coordinate,pos=\n/6] (M'X\n) {} (X);
	\path (B') to[out=30,in=-150] node[coordinate,pos=\n/6] (B'M'\n) {} (M');
	\draw[blue!50!brown] (B'X\n) -- (M'X\n);
	\draw[blue!50!brown] (B'M'\n) -- (B'X\n);
}
\end{tikzpicture} 
\xrightarrow{\text{collar}}
\begin{tikzpicture}[baseline]
\coordinate (L) at (0,0);
\coordinate (R) at (3,0);
\coordinate (T) at (1.5,0.8);
\coordinate (M) at (1.5,-0.4);
\coordinate (B) at (0.75,-0.8);
\path (L) \vertex (T) \vertex (R) \vertex (M) \vertex (B) \vertex;
\draw (L)  to[out=40,in=180]  (T) 
		to[out=0,in=140]  (R);
\draw[dashed] (R) to[out=-150,in=-20]  (M) 
		-- (B); 
\draw	(B)	to[out=150,in=-45]  (L);
\draw[dashed] (T) -- (M);
\foreach \n in {0,...,7} {
	\path (L) to[out=-45,in=150] node[coordinate,pos=\n/8] (LB\n) {} (B);
	\path (L) to[out=0,in=150] node[coordinate,pos=\n/8] (LM\n) {} (M);
	\path (L) to[out=40,in=180] node[coordinate,pos=\n/8] (LT\n) {} (T);
	\draw[green!50!brown] (LB\n) -- (LM\n) -- (LT\n);
}
\foreach \n in {1,...,7} {
	\path (M) to[out=-20,in=-150] node[coordinate,pos=\n/8] (MR\n) {} (R);
	\path (T) to[out=0,in=140] node[coordinate,pos=\n/8] (TR\n) {} (R);
	\draw[brown] (MR\n) -- (TR\n);
}
\coordinate (B') at (B);
\coordinate (M') at (M);
\coordinate (R') at (R);
\coordinate (X) at ($(M')+(0,-0.6)$);
\path (X) \vertex;
\draw[dashed] (M') -- (X);
\draw (B') to[out=-30,in=180] (X) to[out=0,in=-135] (R');
\path[clip] (B') to[out=-30,in=180] (X) to[out=0,in=-135] (R') to[out=-150,in=-20] (M') -- cycle;
\foreach \n in {1,...,7} {
	\path (M') to[out=-20,in=-150] node[coordinate,pos=\n/8] (M'R'\n) {} (R');
	\draw[brown] (M'R'\n) -- +(0,-1);
}
\foreach \n in {1,...,5} {
	\path (B') to[out=-30,in=180] node[coordinate,pos=\n/6] (B'X\n) {} (X);
	\path (M') to[out=-90,in=90] node[coordinate,pos=\n/6] (M'X\n) {} (X);
	\path (B') to[out=30,in=-150] node[coordinate,pos=\n/6] (B'M'\n) {} (M');
	\draw[blue!50!brown] (B'X\n) -- (M'X\n);
	\draw[blue!50!brown] (B'M'\n) -- (B'X\n);
}
\end{tikzpicture} 
\end{align*}
\caption{Adding a collar in the proof of the triangle axiom.}
\label{fig:adding-a-collar}
\end{figure}
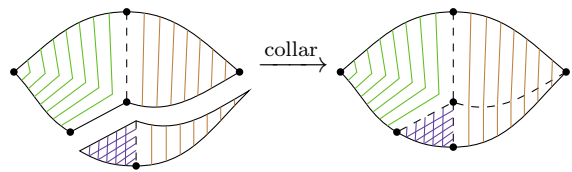

\subsection{\texorpdfstring{$A_\infty$}{A-infinity} 1-categories}
\label{sec:comparing-A-infty}
In this section, we make contact between the usual definition of an $A_\infty$ category 
and our definition of a disk-like $A_\infty$ $1$-category, from \S \ref{ss:n-cat-def}.

\medskip

Given a disk-like $A_\infty$ $1$-category $\cC$, we define an ``$m_k$-style" 
$A_\infty$ $1$-category $A$ as follows.
The objects of $A$ are $\cC(pt)$.
The morphisms of $A$, from $x$ to $y$, are $\cC(I; x, y)$
($\cC$ applied to the standard interval with boundary labeled by $x$ and $y$).
For simplicity we will now assume there is only one object and suppress it from the notation.
Henceforth $A$ will also denote its unique morphism space.

A choice of homeomorphism $I\cup I \to I$ induces a chain map $m_2: A\otimes A\to A$.
We now have two different homeomorphisms $I\cup I\cup I \to I$, but they are isotopic.
Choose a specific 1-parameter family of homeomorphisms connecting them; this induces
a degree 1 chain homotopy $m_3:A\ot A\ot A\to A$.
Proceeding in this way we define the rest of the $m_i$'s.
It is straightforward to verify that they satisfy the necessary identities.

\medskip

In the other direction, we start with an alternative conventional definition of an $A_\infty$ algebra:
an algebra $A$ for the $A_\infty$ operad.
(For simplicity, we are assuming our $A_\infty$ 1-category has only one object.)
We are free to choose any operad with contractible spaces, so we choose the operad
whose $k$-th space is the space of decompositions of the standard interval $I$ into $k$
parameterized copies of $I$.
Note in particular that when $k=1$ this implies a $C_*(\Homeo(I))$ action on $A$.
(Compare with Example \ref{ex:e-n-alg} and the discussion which precedes it.)
Given a non-standard interval $J$, we define $\cC(J)$ to be
$(\Homeo(I\to J) \times A)/\Homeo(I\to I)$,
where $\beta \in \Homeo(I\to I)$ acts via $(f, a) \mapsto (f\circ \beta\inv, \beta_*(a))$.
Note that $\cC(J) \cong A$ (non-canonically) for all intervals $J$.
We define a $\Homeo(J)$ action on $\cC(J)$ via $g_*(f, a) = (g\circ f, a)$.
The $C_*(\Homeo(J))$ action is defined similarly.

Let $J_1$ and $J_2$ be intervals, and let $J_1\cup J_2$ denote their union along a single boundary point.
We must define a map $\cC(J_1)\ot\cC(J_2)\to\cC(J_1\cup J_2)$.
Choose a homeomorphism $g:I\to J_1\cup J_2$.
Let $(f_i, a_i)\in \cC(J_i)$.
We have a parameterized decomposition of $I$ into two intervals given by
$g\inv \circ f_i$, $i=1,2$.
Corresponding to this decomposition the operad action gives a map $\mu: A\ot A\to A$.
Define the gluing map to send $(f_1, a_1)\ot (f_2, a_2)$ to $(g, \mu(a_1\ot a_2))$.
Operad associativity for $A$ implies that this gluing map is independent of the choice of
$g$ and the choice of representative $(f_i, a_i)$.

It is straightforward to verify the remaining axioms for a disk-like $A_\infty$ 1-category.



\bibliographystyle{alpha}
\bibliography{bibliography/bibliography}

This paper is available online at \arxiv{1009.5025}, and at
\url{http://tqft.net/blobs},
and at \url{http://canyon23.net/math/}.

\end{document}